\documentclass[12pt,a4paper, reqno]{amsart}
\usepackage{}
\usepackage{txfonts}
\usepackage{color}
\usepackage{empheq}
\usepackage{CJK}
\usepackage{amssymb}
\usepackage[T5, T1]{fontenc}
\usepackage{bbm}
\usepackage{amsfonts}
\usepackage{amsxtra}
\usepackage{amsmath}
\usepackage{amssymb}
\usepackage{amscd}
\usepackage{amsthm}
\usepackage{mathrsfs}
\usepackage{lipsum}
\usepackage{a4wide}
\usepackage{url}
\usepackage{mathscinet}
\usepackage{mathdots}
\usepackage[all]{xypic}
\usepackage{graphicx}
\usepackage{amsmath}
\usepackage{fourier}
\usepackage{hyperref}
\makeatletter
\makeatletter

\numberwithin{equation}{section}
\newtheorem*{theorem*}{Theorem}
\newtheorem{lemma}{Lemma}[section]
\newtheorem{proposition}[lemma]{Proposition}
\newtheorem{remark}[lemma]{Remark}
\newtheorem{assumption}[lemma]{Assumption}
\newtheorem{example}[lemma]{Example}
\newtheorem{theorem}[lemma]{Theorem}
\newtheorem{definition}[lemma]{Definition}

\newtheorem{corollary}[lemma]{Corollary}

\newtheorem*{question}{Question}

\baselineskip 20pt \textwidth 18cm  \theoremstyle{plain}

\newcommand{\JH}{\operatorname{JH}}
\newcommand{\Aut}{\operatorname{Aut}}
\newcommand{\End}{\operatorname{End}}
\newcommand{\Hom}{\operatorname{Hom}}
\renewcommand{\Im}{\operatorname{Im}}
\newcommand{\Ker}{\operatorname{Ker}}

\newcommand{\Irr}{\operatorname{Irr}}

\newcommand{\spann}{\operatorname{span}}
\newcommand{\cInd}{\operatorname{c-Ind}}
\newcommand{\Ind}{\operatorname{Ind}}
\newcommand{\Rep}{\operatorname{Rep}}
\newcommand{\Res}{\operatorname{Res}}
\newcommand{\res}{\operatorname{res}}

\newcommand{\Ha}{\operatorname{H}}
\newcommand{\Id}{\operatorname{Id}}
\newcommand{\Oa}{\operatorname{O}}
\newcommand{\Ext}{\operatorname{Ext}}

\newcommand{\C}{\mathbb C}

\newcommand{\R}{\mathbb R}
\newcommand{\Z}{\mathbb Z}

\newcommand{\coker}{\operatorname{coker}}

\newcommand{\GL}{\operatorname{GL}}
\newcommand{\GU}{\operatorname{GU}}
\newcommand{\GSp}{\operatorname{GSp}}
\newcommand{\Mp}{\operatorname{Mp}}
\newcommand{\GMp}{\operatorname{GMp}}

\newcommand{\Sp}{\operatorname{Sp}}

\newcommand{\U}{\operatorname{U}}

\newcommand{\GO}{\operatorname{GO}}
\newcommand{\Ga}{\operatorname{G}}
\newcommand{\Ua}{\operatorname{U}}

\newcommand{\diag}{\operatorname{diag}}

\newcommand{\id}{\operatorname{Id}}
\newcommand{\nnn}{\operatorname{N}}

\newcommand{\stab}{\operatorname{Stab}}
\newcommand{\Stab}{\operatorname{Stab}}
\newcommand{\supp}{\operatorname{supp}}
\newcommand{\tr}{\operatorname{Tr}}
\newcommand{\Tr}{\operatorname{Tr}}
\newcommand{\Trd}{\operatorname{Trd}}
\newcommand{\Nrd}{\operatorname{Nrd}}

\begin{document}
\title{On the local theta representation}
\author{Chun-Hui Wang}
\address{School of Mathematics and Statistics\\Wuhan University \\Wuhan, 430072,
P.R. CHINA}
\email{cwang2014@whu.edu.cn}
\begin{abstract}
We study the algebraic framework in which one can define, in  the manner of the theta correspondence, a correspondence between representations of two locally profinite groups  $H_1$, $H_2$. In particular, we examine when and how such a correspondence can be extended to bigger groups $G_1$, $G_2$ containing $H_1$, $H_2$ respectively as normal subgroups.  As an application,  we   discuss     the theta correspondence for   a reductive  dual pair of   the  similitude  groups  in the non-archimedean case.
\end{abstract}
\maketitle

\setcounter{tocdepth}{1}
\setcounter{secnumdepth}{6}
\tableofcontents

\section*{Introduction}
The celebrated local theta or Howe correspondence relates representations of two groups $G_1$, $G_2$ which form a dual pair inside  a symplectic group $\Sp(F)$ or  its  metaplectic cover   group $\Mp(F)$ over a local field $F$.   The Weil  representation $\omega$ of $\Mp(F)$ can then be restricted to $G_1\times G_2$ and the    correspondence is between irreducible quotients of $\omega|_{G_1}$ and irreducible quotients of $\omega|_{G_2}$.  To put it in a general perspective,  in this text we  propose to study a kind of representation of  a direct product of two locally profinite groups,  based on the representation-theoretic  consideration  of   this  correspondence.  It is inspired from the works of   Barthel \cite{Bar1}, Gan-Tantono\cite{GanT} and Roberts \cite{Rob1} on the study of local Howe correspondences for  the similitude groups.   Our  original motivation is to generate their results  largely to various reductive  dual pairs of   similitude  groups over  a non-archimedean local field $F$.  In \cite{Bar1}  Barthel  defined  the  Metaplectic cover group $\GMp(W)$ of $\GSp(W)$, and also explained  the difficulty  to study    Howe correspondences in this case. Next,  Roberts  in  \cite{Rob1} definitely  studied  local  theta correspondences for   certain  symplectic-orthogonal  reductive  dual pairs of  similitude  groups, and then Gan-Tantono  \cite{GanT}  studied the cases of  their inner forms.
These papers provided  some original  ideas and methods, in particular examples to  this text, although to achieve  our main results we need to use  a lot of results on smooth  representations of locally profinite groups.

 To simply our introduction, let us  take up the notation and conventions of the  next section in advance.  Let $(\Pi, V)$ be a smooth representation  of a direct product of  two locally profinite groups $G_1$, $ G_2$.  We  only work with  the case that all irreducible smooth  representations of $G_1$, $G_2$, and $G_1 \times G_2$ are supposed to be admissible.     It is not hard to see that there are   two canonical projections $p_1: \mathcal{R}_{G_1 \times G_2}(\Pi) \longrightarrow \mathcal{R}_{G_1}(\Pi) $, and $p_2: \mathcal{R}_{G_1 \times G_2}(\Pi) \longrightarrow \mathcal{R}_{G_2}(\Pi)$, with the images $\mathcal{R}^0_{G_1}(\Pi)$ and  $\mathcal{R}^0_{G_2}(\Pi)$ respectively.   We call $(\Pi, V)$  a \emph{theta representation} of $G_1 \times G_2$  if
 \begin{itemize}
 \item[(1)] the representation satisfies the graph property meaning that both $p_1$, $p_2$ are injective,
 \item[(2)] the restriction of $\Pi$ to $G_1 \times G_2$ is multiplicity-free, i.e.   $m_{G_1 \times G_2}(\Pi, \pi_1 \otimes \pi_2) \leq 1$, for all $\pi_1 \otimes \pi_2 \in \Irr(G_1 \times G_2)$,  and
 \item[(3)] for $1 \leq  \alpha \neq \beta \leq 2$, the greatest $\pi_{\alpha}$-isotypic component $V_{\pi_{\alpha}} \simeq \pi_{\alpha} \otimes \Theta_{\pi_{\alpha}}$ is a finitely generated  representation of $G_{\alpha} \times G_{\beta}$ .
 \end{itemize}
  One such representation gives, the \emph{Howe correspondence} in the general sense,  between the sets  $\mathcal{R}^0_{G_1}(\Pi)$ and   $\mathcal{R}^0_{G_2}(\Pi)$, grouped in the graphic set $\mathcal{R}_{G_1 \times G_2}(\Pi)$.
 It also gives rise to  another  associated maps from $\left\{ \pi_{\alpha}\right\}$ to  the Jordan-H\"older multiset $\JH(\Theta_{\pi_{\alpha}})$.

 In the above  definition, we will call  $(\Pi,  G_1\times G_2, V)$ a theta representation of finite length if each $\Theta_{\pi_{\alpha}}$ has finite length;  call it  a general theta representation if it  only satisfies  the conditions (1)(2); call it a general theta representation with respect to a subset $\mathcal{I}$ of $\Irr(G_1\times  G_2) $ if we only consider  irreducible representations $\pi_1\otimes \pi_2 \in \mathcal{I}$;  the last one  is extremely interesting when  there exists a non-denegenate $G_1\times G_2$-invariant Hermitian form on $V$, i.e. $(\Pi,  G_1\times G_2, V)$ is a preunitary representation.

One purpose of this paper is to   provide some incipient results for this kind of representations. Assume now that  $H_1, H_2$ are two  closed normal subgroups of $G_1$, $G_2$ respectively such that $G_1/H_1 \simeq G_2/H_2$ under a  map $\gamma$  with the graph $\Gamma/{(H_1 \times H_2)}$. Suppose that all  irreducible smooth representations of $G_i$, $H_i$ are admissible, for $i=1, 2$.  Let $(\rho, W)$ be a smooth representation of $\Gamma$.  Our  main results are the following:

 \begin{theorem}\label{theorema1}
 Suppose that $G_1/H_1$ is an  abelian discrete group  .
  \begin{itemize}
  \item[(1)] Suppose that  $\mathcal{R}_{H_i}(\pi_i) \neq \emptyset$ for every $\pi_i \in \Irr(G_i)$,   $i=1,2$. If  the representation $\Res_{H_1 \times H_2}^{\Gamma} \rho$  of $H_1 \times H_2$ is a theta  representation, then so is  the representation $\cInd_{\Gamma}^{G_1 \times G_2} \rho$ of $G_1 \times G_2$.
   \item[(2)] If the representation $\cInd_{\Gamma}^{G_1 \times G_2} \rho$ of $G_1 \times G_2$ is a
   theta representation of finite length, then  $\Res_{H_1 \times H_2}^{\Gamma} \rho$  satisfies the graph property.  Moreover     if  for $i=1,2$, assume (a)  $\Ext_{G_i}^1(\pi_i, \pi_i)=0$, for any  $\pi_1\otimes \pi_2\in \mathcal{R}_{G_1\times G_2}(\pi)$, (b) $\Rep(H_i)$ is locally noetherian, then  $\Res_{H_1 \times H_2}^{\Gamma} \rho$   is a  theta representation of finite length.
   \end{itemize}
\end{theorem}
 \begin{theorem}\label{theoremb}
 Suppose that $G_i/H_i$ is a compact group, and assume the category $\Rep(H_i)$ is locally noetherian,  for $i=1, 2$.
   \begin{itemize}
   \item[(1)] If the representation $\Res_{H_1 \times H_2}^{\Gamma} \rho$ of $H_1 \times H_2$ is a theta  representation, then so is  the representation $\cInd_{\Gamma}^{G_1 \times G_2} \rho$ of $G_1 \times G_2$.
   \item[(2)] Suppose that  $\mathcal{L}_{G_i}(\Ind_{H_i}^{G_i}\sigma_i) \neq \emptyset$,   for every $\sigma_i \in \Irr(H_i)$,  $i=1,2$. If  the representation $\cInd_{\Gamma}^{G_1 \times G_2} \rho$   of $G_1 \times G_2$ is a theta representation, then so is the representation  $\Res_{H_1 \times H_2}^{\Gamma} \rho$ of $H_1 \times H_2$.
   \end{itemize}
 \end{theorem}

 Now let  $\delta_{\Gamma\setminus G_1\times G_2} =\frac{\Delta_{G_1\times G_2}}{\Delta_{\Gamma}}$;  let $\widehat{H}_i$  resp.  $\widehat{G}_i$ denote the set of all equivalence classes of irreducible unitary representations of $H_i$ resp. $G_i$  and $\Irr_u(H_i)$  resp. $\Irr_u(G_i)$  the set of all equivalence classes of irreducible preunitary smooth representations of $H_i$  resp. $G_i$.   For each $i$  assume (1) $G_i$, $H_i$  are  groups of \emph{type I}, (2) $\widehat{H_i}/G_i$ is \emph{ countably separated}, (3) For any $\omega \in \widehat{H_i}$, the orbit $\{ \omega^g \mid g\in G_i\}$ is  \emph{ countable}, (4) For any $\omega\in \widehat{H_i}$,  the cardinality of $ \{ \lambda \in \widehat{G_i} \mid m_{H_i}(\lambda, \omega)\neq 0\}$ is \emph{countable}, (5) there exists an open  subgroup $O$ of $G$, such that   $\Ha^2(O, \C^{\times})$ only contains  elements of finite order.     Let $(\rho, W)$ be a smooth preunitary representation of $\Gamma$.   Assume $W$ is a second countable vector space, and $G_i$, $H_i$ all are second countable groups.
  \begin{theorem}
     \begin{itemize}
   \item[(1)]
If $\Res_{H_1 \times H_2}^{\Gamma} \rho$  is a general theta representation of  $H_1 \times H_2$  with respect to $\Irr_u(H_1) \times \Irr_u(H_2)$, then so is  the representation $\cInd_{\Gamma}^{G_1 \times G_2} (\delta_{\Gamma \setminus (G_1\times G_2)}^{1/2}\otimes \rho)$ of $G_1 \times G_2$  with respect to $\Irr_u(G_1) \times \Irr_u(G_2)$.
   \item[(2)] Suppose that  $m_{H_i}(\lambda_i, \omega_i)<+\infty$,   for  $\lambda_i\in \Irr_u(G_i),  \omega_i \in \Irr_u(H_i)$,  $i=1,2$. If  $\cInd_{\Gamma}^{G_1 \times G_2}(\delta_{\Gamma \setminus (G_1\times G_2)}^{1/2}\otimes \rho)$   of $G_1 \times G_2$ is a general theta representation  with respect to $\Irr_u(G_1) \times \Irr_u(G_2)$, then so is   $\Res_{H_1 \times H_2}^{\Gamma} \rho$ of $H_1 \times H_2$  with respect to $\Irr_u(H_1) \times \Irr_u(H_2)$.
   \end{itemize}
     \end{theorem}
 Keep the notations, and consider the situation that  $H_i$ is not a normal subgroup of $G_i$.  In this case,  set $H=H_1\times H_2$, $G=G_1\times G_2$. Let  $\Delta=\{s=(s_1, s_2)\in G\}$, containing $1$, be a complete  set of representatives for  $H\setminus G/H$. Assume $\Delta$ is a countable set. For any $s\in \Delta$, $s\neq 1$,   assume: (1) $H_s\cap H$ is a normal subgroup of $H$, (2) $H/(H_s\cap H)$  is  not compact, (3) up to $H_s\cap H$-conjugacy there  exists at least one and at most  a finite number of maximal open compact subgroups in $H$, (4)   for each maximal open compact  subgroup $K$of $H$, for each positive integer  $n$, the set  $\mathcal{N}(K)_n=\{ K^i \mid K^i \lhd K, [K: K^i]=n\}$ has finite cardinality. Let $(\sigma, U)$ be  a smooth representation of $H$,   set $\pi= \cInd_{H}^{G}\sigma$.  Assume $U$ is a second countable vector space, and $G$, $H$ both are second countable groups.
  \begin{proposition}
  Assume $G/H$ is compact.
  \begin{itemize}
\item[(1)] Assume  that $H$ is an open subgroup of $G$. If $\sigma$ is a  general theta representation of $H$, then  so is  the representation $\pi$ of $G$.
\item[(2)] Assume:  (1)  the category $\Rep(H)$ is locally noetherian,  (2)  for any open compact subgroup $K_1$ of $H$, assume $\epsilon_{K_1} \ast \mathcal{H}(H)\ast \epsilon_{ K_1}$ can be generated by $\epsilon_{K_1}$ and  a finite number of $\epsilon_{h}$'s,  (3)  $(\sigma, U)$ is an admissible preunitary representation of $H$. If $\sigma$ is a  general theta representation of $H$,  then so is  the representation $\pi$ of $G$.
\end{itemize}
\end{proposition}

To  show  those results,  we use many fine results on representations of $p$-adic groups  established in Bernstein-Zelevinsky \cite{BernZ}, Bushnell-Henniart \cite{BushH}, Casselman \cite{Cass}, Mackey  \cite{Ma}, and we deem  them as our basic references.  The proofs   proceed by using the Clifford-Mackey theory about the behaviour of  the restriction of irreducible representations of a locally profinite group to its certain invariant subgroups.   Indeed under our assumptions, we essentially  only work  with  these irreducible representations of $G_i$, whose restrictions to $H_i$ are semi-simple.   With an application, we discuss in board generalities about Howe correspondences for  the similitude groups in the last section. It is a very difficult  problem to  give the explicit correspondences and connect them with the related subjects. However one can see   many favorable and interesting   research  works in local  and global  cases, for examples   Gan-Ichino\cite{GanI}, Harris-Kudla- Sweet\cite{HKS}, Ginzburg-Rallis-Soudry \cite{GiRaSo}, Mao-Rallis\cite{MaR}, etc.

\section{Preliminaries}
\subsection{Notation and Conventions}\label{notation}
We  shall  follow the notion and  conventions of \cite{BernZ}, \cite{BushH}, \cite{Cass} on  the subject of   complex representations of locally profinite groups. In the whole text,  locally profinite group will be assumed \emph{$\sigma$-compact}, meaning   a union of countably many compact sets.
Let $(\pi, V)$ be a representation of a locally profinite group $G$. Call $\pi$ \emph{smooth} if the stabilizer of every vector in $V$ is open,  \emph{admissible} if its $K$-invariant subspace is finite-dimensional for any compact open subgroup $K$ of $G$.  If $H$ is a closed subgroup of $G$ and $(\sigma, W)$ is a smooth representation of $H$, we use the notions of \emph{induced representation}:

$\Ind_{H}^G \sigma=\left\{ \right. f: G \longrightarrow W\mid $ (a)  $f(hg)= \sigma(h) f(g)$, for $h\in H, g\in G$, (b) there is a compact open subgroup  $K_f$  of $G$ (depending on  $f$) such that  $f$  is right  $K_f$-invariant\}\\
and \emph{compact induced representation}:

$\cInd_{H}^G \sigma=\left\{\right. f: G \longrightarrow W\mid f$  satisfies the above conditions (a), (b), and also (c) that $f$  is compactly supported modulo $H$\}. Let $S(G)$, or  $C_c^{\infty}(G)$ denote  the space   of locally constant functions with compact support. Let  $S^{\ast}(G)$ denote the set of $\C$-linear functions on $S(G)$; such functions are called \emph{distributions}. The so-called \emph{Dirac distribution} $\delta_g$ at a point $g$, is defined by $\langle\delta_g, f \rangle=f(g)$, for all $f\in S(G)$.     Recall that  a  left Haar measure $\mu_G$ of $G$  acting on $S(G)$ is defined  by $\langle \mu_G, f\rangle := \int_G f(x) d\mu_G(x)$, for $f(x)\in S(G)$.
As is known that there is a unique character $\Delta_G: G \longrightarrow \R_{>0}^{\ast}$, called the \emph{modulus} of $G$, such that $\Delta_G(g) \int_G f(xg)d\mu_G(x)=\int_G f(x) d\mu_G(x)$, for $f(x)\in S(G)$.   In particular, when $\Delta_G \equiv 1_G$,  $G$ is called \emph{unimodular}.

$S(G)$, when imposed the canonical convolution $\ast$ defined by $f_1 \ast f_2(x)=\int_G f_1(y) f_2(y^{-1}x) d\mu_G(y)$ for $f_1,f_2 \in S(G)$,  will be called  the \emph{Hecke algebra} of $G$, denoted by $\mathcal{H}(G)$ from now on.  For    a compact open subgroup  $K$ of $G$,  one kind of idempotent element $\epsilon_K $ in $\mathcal{H}(G)$ is defined by
    $$\epsilon_K(g)=\left\{ \begin{array}{ll}
\mu_G(K)^{-1} & \textrm{ if } g\in K,\\
0&  \textrm{ otherwise. }
\end{array} \right.$$
We then write $\mathcal{H}(G,K)$ for  the unit  algebra $\epsilon_K\ast \mathcal{H}(G) \ast \epsilon_K$. $\Rep(G)$ will denote the category of all smooth representations of $G$, and $\Irr(G)$ will denote the set of equivalence classes of irreducible smooth representations of $G$. If $(\sigma, W) \in \Rep(G)$,  let $(\check{\sigma},   \check{W})$ denote its \emph{contragredient} representation. If $\pi \in \Rep(G)$, we will  let $\mathcal{R}_G(\pi)= \left\{ \rho \in \Irr(G) \mid \Hom_G(\pi, \rho) \neq 0\right\}$, $\mathcal{L}_G(\pi)= \left\{ \rho \in \Irr(G) \mid \Hom_G(\rho,\pi) \neq 0\right\}$, and define  $m_G(\pi, \rho)=  \dim_{\C}\Hom_G(\pi, \rho)$.  The symbol $\rho \prec  \pi$ means that $\rho$  is a sub-representation of $\pi$.

 In the whole paper, the representations  will be assumed smooth unless otherwise stated.
\subsection{Some results on representations}

This section is devoted to recalling some well-known results in \cite{BernZ}, \cite{BushH}, \cite{Cass} and proving some consequences for convenient use.
 We will let $H$ be a closed subgroup of a locally profinite group  $G$, $\Delta_G$(resp. $\Delta_H$)  the modulus of $G$(resp. $H$).  Fix   an element  $(\pi, V) \in \Rep(G)$, and an element $(\rho, W) \in \Rep(H)$.
\begin{lemma}\label{unimodular}
\begin{itemize}
\item[(1)] If $H$ is an open  subgroup of $G$, then $\Delta_H=\Delta_G|_H$.
\item[(2)] If $H$ is a normal subgroup of $G$, and $G/H$ is a unimodular group, then $\Delta_H=\Delta_G|_H$.
\end{itemize}
\end{lemma}
\begin{proof}
1) In the known exact sequence
$0 \longrightarrow S^{\ast}(G\backslash  H)  \longrightarrow S^{\ast}(G) \stackrel{ i_{H}^{\ast}}{\longrightarrow} S^{\ast}(H) \longrightarrow 0$,
 the map $i_H^{\ast}$ sends a left Haar measure $\mu_G$  of $G$ to that of $H$.  For an element $f\in S(H) \subset S(G), h\in H$ we have
$$\Delta_G(h) \int_{H} f(xh) d i_H^{\ast}(\mu_G)(x)= \Delta_G(h) \int_{G} f(xh) d \mu_G(x)= \int_G f(x) d\mu_G(x)= \int_H f(x) di_H^{\ast}(\mu_G)(x),$$
 so $\Delta_G|_H=\Delta_H$.\\
2) Let $\mu_H$  be a left Haar measure of $H$ and $\mu_{G/H}$  a Haar measure of $G/H$. Then there is a \emph{well-defined} $\C$-linear map:
$$ - : S(G) \longrightarrow S(G/H); \qquad f \longmapsto  \Big(  \overline{f}(gH):= \int_H f(gh) d\mu_H(h)\Big).$$
 Now we define an element $\mu_G \in S^{\ast}(G)$  by
$\langle \mu_G, f \rangle:= \langle \mu_{G/H}, \overline{f}\rangle= \int_{G/H} \overline{f}(\overline{g}) \mu_{G/H} (\overline{g})$, for all $f\in S(G)$.
Define the left and right actions of $G$ on itself by
$\rho_G(g_0) (g)=g_0g$  and  $\gamma_G(g_0)(g) =gg_0^{-1}$ respectively, and   extend them conventionally onto the sets $S(G)$ and $S^{\ast}(G)$.  For $g_0 \in G, f\in S(G)$, we then have
$$\langle \rho_G(g_0) \mu_G, f \rangle= \langle \mu_G, \rho_G(g^{-1}_0) f \rangle=\langle \mu_{G/H}, \overline{\rho_G(g_0^{-1})f}\rangle=\langle \mu_{G/H}, \rho_{G/H}(\overline{g_0}^{-1}) \overline{f}\rangle
 = \langle \mu_{G/H}, \overline{f} \rangle = \langle \mu_{G}, f \rangle.$$
This implies that $\mu_G$ is a left Haar measure of $G$. On the other hand, for $h \in H$, we have
$$ \langle \Delta_G(h) \mu_G, f \rangle= \langle \gamma_G(h) \mu_G, f\rangle =\langle \mu_G , \gamma_G(h^{-1}) f\rangle =\langle \mu_{G/H}, \overline{\gamma_G(h^{-1}) f } \rangle= \langle \mu_{G/H} , \Delta_H(h) \overline{f}\rangle= \Delta_H(h) \langle \mu_G, f\rangle,$$
which  shows that $\Delta_G|_H=\Delta_H$.
\end{proof}
\begin{remark}\label{generalmod}
By the general result on locally compact groups,  if $H$ is a normal subgroup of $G$,  then $\Delta_H=\Delta_G|_H$.
\end{remark}
\begin{proof}
The proof is   more complicated than the above discussion,  and one can see \cite[ pp. 205-206]{HR}.
\end{proof}
\begin{remark}
If $G$ is an abelian group, a simple group, or a union of compact groups, then it is unimodular.
\end{remark}

\begin{theorem}[Frobenius reciprocity]
\begin{itemize}
\item[(1)] $\Hom_G\big( \pi, \Ind_H^G \rho\big) \simeq \Hom_H\big( \Res_H^G \pi, \rho\big).$
\item[(2)] $\Hom_G\big( \cInd_H^G \rho, \check{\pi}\big) \simeq \Hom_H \big( \frac{\Delta_H}{\Delta_G} \rho, (\Res_H^G \pi)^{\vee}\big).$
\end{itemize}
\end{theorem}
\begin{proof}
See \cite[pp. 23-24]{BernZ}.
\end{proof}
\begin{lemma}[{\cite[p. 23]{BernZ}}]\label{thedualityQ}
$(\cInd_H^G \rho)^{\vee} \simeq \Ind_H^G (\tfrac{\Delta_G}{\Delta_H} \check{\rho})$.
\end{lemma}

\begin{lemma}
Let $(\pi, V)$ be an admissible smooth representation of $G$.
\begin{itemize}
\item[(1)] If $H$ is an open subgroup of $G$, then $\Res_H^G \pi$ is    also admissible.
\item[(2)] Let $H_1$ be a closed subgroup of $G$, and $H_1 \supseteq H$. If $\Res_H^G \pi$ is admissible, so is $\Res_{H_1}^G \pi$.
\item[(3)] If $H$ is a normal subgroup of $G$, then $V^{H}$ is an admissible smooth $\frac{G}{H}$-module.
\item[(4)] Let $ K_1 \lhd K_2$ be two two open compact subgroups of $G$, and assume $\mathcal{R}_{K_2}(\Ind_{K_1}^{K_2} 1)=\{ (\lambda_i, U_i)\in  \Irr(K_2/K_1)\mid i=1, \cdots, n\}$. Let $V^{\lambda_i}$ denote the $\lambda_i$-isotypic  component of $\Res^G_{K_2} \pi$. Then $V^{K_1} = \oplus_{i=1}^nV^{\lambda_i}$ as $K_1$-modules.
\end{itemize}
\end{lemma}
\begin{proof}
Parts (1)(2) are straightforward. For (3), clearly $V^{H}$ is a smooth $G/H$-module.  Note that the inverse image of any open compact subgroup $\overline{K}$ of $G/H$ in $G$, denoted by $K$,  is an open subgroup  of $G$. Let $K_1$ be  an open compact subgroup  of $K$ with the  image $\overline{K_1}$ in $G/H$. Then $(V^{H})^{\overline{K_1}}=V^{K_1H} \subseteq V^{K_1}$; this implies the part (3).   In (4), $V^{\lambda_i} \simeq n_i U_i$ , so each vector in $V^{\lambda_i}$ is $K_1$-fixed, and $V^{\lambda_i} \subseteq V^{K_1}$.  On the other hand, by part (3), $V^{K_1} \simeq \sum_{i=1}^n m_i U_i$ as $\frac{K_2}{K_1}$-modules, so $V^{K_1} \subseteq \oplus_{i=1}^n V^{\lambda_i}$.
\end{proof}
\begin{lemma}\label{directsums}
Let $\sigma \simeq \oplus_{i\in I} \sigma_i$ be a smooth representation of $G$.
\begin{itemize}
\item[(1)] $ \oplus_{i\in I} \check{\sigma}_i \hookrightarrow \check{\sigma} \hookrightarrow \prod_{i\in I} \check{\sigma}_i$;
\item[(2)] If $\sigma$ is an   admissible representation,   then $\check{\sigma } \simeq \oplus_{i\in I} \check{\sigma}_i$.
\end{itemize}
\end{lemma}
\begin{proof}
 1) As is known that $\sigma^{\ast} \simeq \prod_{i\in I}\sigma_i^{\ast} \supseteq \oplus_{i\in I} \sigma_i^{\ast}$.  Considering  their smooth  parts, we get the result. \\
 2) Each factor $\sigma_i$ is also an admissible  representation of $G$ and  there is a  $G$-embedding $\oplus_{i\in I} \check{\sigma}_i \hookrightarrow \check{\sigma }$.  Then by investigating their $K$-invariant parts, as $K$ runs through  open compact subgroups of $G$, we obtain the result.
 \end{proof}

\begin{lemma}\label{admissiblerepresentations}
If $\Res_H^G\pi$ is an admissible smooth representation of $H$, then $\big(\Res_H^G\pi\big)^{\vee}  \simeq \Res_H^G\check{\pi}.$
\end{lemma}
\begin{proof}
One-side inclusion $\Res_{H}^G \check{\pi} \hookrightarrow  (\Res_H^G \pi)^{\vee}$ is clear.  It is sufficient to show that $[(\Res_H^G V)^{\vee})]^{K\cap H}$ belongs to $\Res_H^G \check{V}$ for any open compact subgroup $K$ of $G$.  By definition, the set $[(\Res_H^G V)^{\vee})]^{K\cap H}$ consists of the $\C$-linear functions $f: V^{K\cap H} \oplus V[K \cap H] \longrightarrow \C$ such that $f|_{V[K \cap H]}=0$, where $V[K\cap H]=\left\{\sum c_i(\pi(g_i) v_i-v_i)\mid c_i \in \C, v_i \in V, g_i\in K\cap H\right\}$. Suppose now that $V^{K\cap H}$ is linearly spanned by $v_1, \cdots, v_n$ in $V$; let $U_0$ be an open compact subgroup of $\cap \Stab_{G}(v_i)$  such that it contains $K\cap H$ (for instance,  $U_0= \cap_{i} \Stab_{G}(v_i)\cap K $). By \cite[ p.15, Prop.]{BushH}, $V= \oplus_{\sigma\in \hat{U_0}} V^{\sigma}$, $V^{\sigma}$ being the $\sigma$-isotropic components of $V$. Since $V^{K\cap H}$ has finite dimension, there exist only finite $\sigma_1$, $\cdots$, $\sigma_n$, such that each $V^{\sigma_i}|_{K\cap H}$ contains the trivial representation of $K\cap H$. Assume now that  $V^{\sigma_i}$ is spanned by elements $v_1^{(i)},  \cdots, v_{n_i}^{(i)}$ in $V$. Let $U_1= \cap_{i,j}\Stab_{G}(v_j^{(i)}) \cap U_0$, be an open  subgroup of $G$.  Then $f\in \big( \Res_H^G \check{V}\big)^{U_1} \subseteq \Res_H^G \check{V}$.
\end{proof}

\begin{corollary}\label{frobeniusreciprocityregular}
If $H$, $G$ are two groups satisfying any condition in  Lmm.\ref{unimodular} and $\Res_H^G\pi$ is an admissible smooth representation of $H$, then $\Hom_G\big( \cInd_H^G \rho, \pi\big) \simeq \Hom_H\big( \rho, \Res_H^G\pi\big)$.
\end{corollary}
\begin{lemma}\label{thedualityequa}
Let $(\pi_1, V_1)$ be a smooth representation  of $G$, and $f:  V_1 \longrightarrow V$ is a $G$-morphism. If the induced map $\check{f}:  \check{\pi} \longrightarrow \check{\pi}_1$ is  an isomorphism, then $\pi_1 \simeq \pi$.
\end{lemma}
\begin{proof}
Applying the contragredient operator to the  short exact sequence of $G$-modules   $0\longrightarrow \ker(f) \longrightarrow V_1 \longrightarrow V $ shows that  $(\ker(f))^{\vee}=0$. Since $0=[(\ker(f))^{\vee}]^{K} \simeq [\ker(f)^{K}]^{\ast}$, for any open compact subgroup $K$ of $G$, and $\ker(f)=\cup_{K}(\ker(f))^{K}$, we obtain $\ker(f)=0$. Similarly, the $\coker(f) $ is also zero.
\end{proof}

\begin{lemma}\label{homeo}
Let $G_1$ be a closed subgroup of $G$ such that the canonical map
$e: H\cap G_1 \setminus G_1 \longrightarrow H \setminus G$ is bijective. Then $e$ is  homeomorphic.\footnote{This result uses the $\sigma$-compact condition.}
\end{lemma}
\begin{proof}
The result can be deduced from \cite[p.7, Coro.]{BernZ}  by considering the right action of $G_1$ on $H\setminus G$ and by taking $x_0=[H] \in H\setminus G$ there.
\end{proof}
\begin{lemma}\label{twocompactsubgroups}
\begin{itemize}
\item[(1)] Let $K_1$, $K_2$ be two compact subsets of $G$. Then $K_1\ltimes K_2=\{ xyx^{-1} \mid x\in K_1, y\in K_2 \}$ is also a compact subset of $G$.
\item[(2)]  Suppose now that
\begin{itemize}
\item[(a)] $K_1$, $K_2$ both are   compact subgroups of $E$, for an open compact subgroup $E$ of $G$, and
\item[(b)] $K_2$ is also an open subset of $G$.
\end{itemize}
  Then $K_0= \cap_{k \in K_1} k K_2 k^{-1}$ is an open subgroup of $K_2$ as well as $E$.
\end{itemize}
\end{lemma}
\begin{proof}
1) Let us consider the continuous map: $G \times G \longrightarrow G;  (x, y) \longmapsto xyx^{-1}.$
Then $K_1\ltimes K_2$ is just the image of the compact subset $K_1 \times K_2$.\\
2) Note that $E\backslash K_0=\cup_{k\in K_1} k (E \backslash K_2)k^{-1}$. Since  $E \backslash K_2$ is also a compact set, applying the above (1) shows that  $E\backslash K_0$ is also closed. So $K_0$ is an open subgroup  of $E$ as well as $G$.
\end{proof}
\begin{proposition}\label{restricition}
Let $G_1$ be a closed subgroup of $G$  such that the canonical map
$e: H_1\setminus G_1 \longrightarrow H \setminus G$ is  homeomorphic, where $H_1= H \cap G_1$. Then
$\Res_{G_1}^G \big( \cInd_H^G \rho\big) \simeq \cInd_{H_1}^{G_1} \big(\Res_{H_1}^H \rho\big).$
\end{proposition}
\begin{proof}
 Let  $K_1  $ be an open compact subgroup of $G_1$. Let  $\Omega=\left\{ g_i \in G_1\right\}_{i\in I} $ be  a set of representatives for  $H_1\setminus G_1 / K_1$ as well as $H\setminus G / K_1$. For each $g_1 \in \Omega$, we write  $K_{1_{g_1^{-1}}}=g_1 K_1 g_1^{-1}$. By \cite[p.22, Lmm.]{BernZ}, there exists a bijection:
$$i: (\cInd_{H_1}^{G_1} \rho)^{K_1} \longrightarrow \mathcal{K}_1=\left\{ f: \Omega \longrightarrow W \mid  f(g_1) \in W^{K_{1_{g_1^{-1}}}\cap H_1} \textrm{ for } g_1 \in \Omega \textrm{ and  the support of } f \textrm{ is  a finite set}\right\}. $$
 Here, $i$ is the restriction of functions from $G_1$ to $\Omega$. On the other hand, for $\varphi \in (\cInd_H^GW)^{K_1}$,  $h \in H$, $g_1 \in \Omega$, $ k_1\in K_1$,  we have
  $\varphi(hg_1k_1) = \rho(h) \varphi(g_1)$,  and $\varphi(g_1) \in W^{K_{1_{g_1^{-1}}}\cap H_1}$ by observing $K_{1_{g_1^{-1}}} \cap H_1 =K_{1_{g_1^{-1}}} \cap H$. Recall that  $\supp(\varphi) \subseteq HK$ for some compact set $K$ of $G$.  Note that  the  collection  $\left \{H \setminus  Hg_1K_1\mid  g_1 \in\Omega\right\}$ is  an open cover of $H \setminus G$, so it is clear that  the compact set $H\setminus  HK$ has finite subcover.  In this way, we verify that  $\varphi|_{\Omega}$  belongs to the above set $\mathcal{K}_1$.

 Next, for  $f \in \mathcal{K}_1$,  we define  a function $\varphi_f$ from $G$ to $W$ by $ \varphi_f (g)= \rho(h) f(g_1)$  for $ g=hg_1k_1$ with $ h\in H, g_1 \in \Omega, k_1 \in K_1.$
 To show $\varphi_f$  belongs to $(\cInd_{H}^{G} \rho)^{K_1}$ it suffices to verify that $\varphi_f$ is  $K$-invariant for an open compact subgroup $K$ of $G$. For then we can  replace $K_1$ by its subgroup and may   assume $E_0 \cap G_1 \subseteq K_1 \subseteq E_1 \cap G_1$ for some open compact subgroups $E_0\subseteq E_1$ of $G$. Suppose now that  $\supp(f) \cap \Omega= \left\{ g_1, \cdots, g_m\right\}$ and  $f(g_i)=v_{g_i}$ lies in $ W^{K_{1_{g_1^{-1}}}\cap H_1}$.  We may and do take  open compact subgroups $F_i$ of $G$ such that
$v_{g_i}\in W^{F_{i_{g_i^{-1}}}\cap H}$ and $ F_i \subseteq E_0.$
 Suppose now that
 $Hg_i (F_i \cap G_1) \supseteq Hg_iL_i$, for some open compact subgroups $L_i$ of $ F_i$ and  $G$. Now we define a new open compact subgroup $K$ of $G$ by
$K:= \cap_{i=1}^m L_i$, which
satisfies
 $Hg_iK\subseteq Hg_iL_i \subseteq Hg_i(F_i \cap G_1).$
 For  $k \in K$, when  decomposed as
 $k=g_i^{-1}h_i g_i l_i$, for $h_i \in F_{i_{g_i^{-1}}}\cap H$, $l_i  \in F_i\cap G_1 \subseteq E_0\cap G_1 \subseteq K_1$,  we  have
$$\varphi_f(g_ik)=\varphi_f(h_ig_il_i)=\varphi_f(h_ig_i)=\rho(h_i) f(g_i)=\rho(h_i) v_{g_i}=\varphi_f(g_i).$$
We also need to discuss  the other $g\in \Omega$ besides those $g_i$. For this purpose let us consider a  smaller subgroup $K_0$ of $K$ given by
$K_0= \cap_{k_1 \in K_1 } k_1^{-1} K k_1$. Note that $K$, $K_1$ both are subgroups of $E_1$. By Lmm.\ref{twocompactsubgroups} (2), $K_0$ is an open compact subgroup of $G$ satisfying $K_0K_1=K_1K_0$. Then
  $ Hg_i K_1 \subseteq Hg_i K_1 K_0=Hg_iK_0 K_1 \subseteq Hg_iL_i K_1 \subseteq Hg_i K_1,$ and  $Hg_i K_1=Hg_iK_1 K_0$.

    For $g_0 \in \Omega \setminus \left\{ g_1, \cdots, g_m\right\}$, we have
$Hg_0K_1K_0 \cap Hg_i K_1K_0 = \emptyset$, for $i=1, \cdots, m$. Otherwise, for some $i_0$, $Hg_0K_1K_0 \cap Hg_{i_0} K_1K_0 \neq  \emptyset$, contradicting to $Hg_{i_0} K_1K_0=Hg_{i_0} K_1$ and $Hg_0K_1 \cap  Hg_{i_0} K_1=\emptyset$. So $\varphi_f(g_0k_0)=0=\varphi_f(g_0)$, for $k_0 \in K_0$. All in all, we have
$\varphi_f(hgk_1k_0)=\varphi_f\big(hgk_0(k_0^{-1}k_1k_0)\big)=\varphi_f(hg)$, for all $g\in \Omega, k_1 \in K_1, k_0 \in K_0$.

By the above discussion,    the canonical restriction  from
 $ \Res_{G_1}^G\big(\cInd_H^G W\big)$  to $\cInd_{H_1}^{G_1} W$ given by $f  \longrightarrow f|_{G_1}$ is bijective. This  completes the proof.
\end{proof}

\begin{corollary}\label{theopenrestriction}
Under the conditions of the above proposition, if $G_1$ is  an open   subgroup of $G$,  then $\Res_{G_1}^G \big( \Ind_H^G \rho\big) \simeq \Ind_{H_1}^{G_1} \big(\Res_{H_1}^H \rho\big)$.
\end{corollary}
\begin{proof}
We follow the similar procedure as  above, and keep the notations, but assume that $\mathcal{K}_1=\left\{ f: \Omega \longrightarrow W \mid  f(g_1) \in W^{K_{1_{g_1}}\cap H_1}\right\}$. Analogously, the canonical restriction from $(\Ind_H^G W)^{K_1}$ to $(\Ind_{H_1}^{G_1} W)^{K_1}$ given by $f  \longrightarrow f|_{G_1}$ is well-defined and injectivity. Note that now $K_1$ is  an open compact subgroup of $G$.  In view of the  proof, the surjectivity is also clear.
\end{proof}

We  close  this section by recording  some consequences of   \cite[p.19, Lmm.]{BushH}.  For $(\rho, W) \in \Rep(H)$, we write   $\rho^G=\cInd_H^G  \rho$.  For any open compact subgroup $K$ of $G$,  let $ \Delta$ be a complete set of representatives for $H \setminus G/K$.  For $s\in \Delta$, let $H_s=s^{-1}Hs$, and set $\rho^{s}(x)= \rho(sxs^{-1})$, for $x\in H_s\cap K$.
\begin{lemma}\label{therestriction1}
   $\Res_{K}^{G} \rho^G \simeq \oplus_{s\in \Delta}\cInd_{H_s\cap K}^K\rho^{s}$.
\end{lemma}
 \begin{proof}
For any  $s\in \Delta$, there  exists a canonical $K\cap H_s$-morphism $\cInd_H^G \rho \longrightarrow \rho^s; f \longmapsto f(s)$. So it induces a $K$-morphism  $A_s: \rho^G \longrightarrow \cInd_{H_s\cap K}^K   \rho^s=\Ind_{H_s\cap K}^K  \rho^s$. Therefore we obtain a $K$-morphism $A=\oplus_{s\in \Delta} A_s: \rho^G \longrightarrow  \prod_{s\in \Delta}\cInd_{H_s\cap K}^K \rho^{s}$.  Since for any $f\in \rho^G$, $supp f\subseteq \cup_{i=1}^n Hs_i K$ for certain $s_i\in \Delta$, the above mapping  $A$  factors through $\oplus_{s\in \Delta}\cInd_{H_s\cap K}^K \rho^{s} \hookrightarrow \prod_{s\in \Delta}\cInd_{H_s\cap K}^K \rho^{s}$.  Hence we obtain $A=\oplus_{s\in \Delta} A_s: \rho^G \longrightarrow  \oplus_{s\in \Delta}\cInd_{H_s\cap K}^K \rho^{s}$. We first  show that  $A$ is injective. If $A(f_1)=A(f_2)$, for $f_1, f_2\in \rho^G$, then $A_s(f_i)(k)=  f_i(sk)$, and $f_1(sk)=f_2(sk)$ for any $k\in K$. So $f_1|_{HsK}=f_2|_{HsK}$ for  any $s\in \Delta$, and   $f_1=f_2$. Secondly,  assume $\sum_{i=1}^nt_{s_i}\in \sum_{i=1}^n\cInd_{H_{s_i}\cap K}^K \rho^{s_i}$. Then there exist open compact subgroups $K_{s_i}$ of $K$ such that $t_{s_i}$ is $K_{s_i}$-invariant. We now define an element $f: G \longrightarrow W$ as follows:
$  f|_{Hs_iK} (h s_ik)=\rho(h) t_{s_i}(k)$,  for $ h\in H, k\in K$;
it is well-defined because for  $h_1, h_2\in H$, $k_1, k_2 \in K$, if  $h_1{s_i}k_1=h_2s_ik_2$, i.e.  $k_1= s_i^{-1}h_1^{-1}h_2 s_ik_2$, we have
 $\rho(h_1)t_{s_i}(k_1)=\rho(h_1)t_{s_i}(s_i^{-1}h_1^{-1} h_2s_i k_2)= \rho(h_1)\rho^{s_i}(s_i^{-1}h_1^{-1} h_2 s_i) t_{s_i}(k_2)=\rho(h_2) t_{s_i}(k_2)$.  Clearly $f$ is $\cap_{i=1}^n K_{s_i} $-invariant, and  $A_{s_i}(f)=t_{s_i}$.
 \end{proof}
\begin{lemma}\label{compactadm}
Keep the notations. If $\rho$ is admissible and $G/H$ is compact, then $\rho^G$ is also admissible.
\end{lemma}
\begin{proof}
Under the hypothesis,  assume $\{s_1, \cdots, s_m\}$ is a complete set of representatives for  $H \setminus G/K$.  Clearly $\rho^{s_i}$ is also an admissible representation of $H_{s_i}$, and $m_{K\cap H_{s_i}}( \rho^{s_i}, \mathbb{C}) =\dim  [\rho^{s_i}]^{K\cap H_{s_i}}< \infty$. Hence $\dim [\rho^G]^K=\sum_{i=1}^mm_{K\cap H_{s_i}}( \rho^{s_i}, \mathbb{C}) =\sum_{i=1}^m \dim  [\rho^{s_i}]^{K\cap H_{s_i}} < \infty$.
\end{proof}
Assume now $H,J$ are two   open subgroups of $G$. Let  $\Delta=\{s_i\in G\}_{i\in I}$ be a complete  set of representatives for  $H\setminus G/J$, and then $\{s^{-1}\mid s\in \Delta\}$ forms  a complete set of representatives for $J\setminus G/H$.  For $s\in \Delta$, let $H_{s}=s^{-1}Hs$, and set $\rho^{s}(x)= \rho(sxs^{-1})$, for $x\in H_s\cap J$.

\begin{lemma}\label{therestriction1}
\begin{itemize}
\item[(1)] There is an $H$-monomorphism $W \longrightarrow \cInd_{H}^{G} W; w\longmapsto f_w$ with the image, denoted by $\mathcal{W}$, where $f_w(1)=w$, and $\supp f_w \subseteq H$.
\item[(2)] $\Res_{J}^{G} \rho^G
 \simeq \oplus_{s\in \Delta}\cInd_{H_s\cap J}^J\rho^{s}$.
\end{itemize}
\end{lemma}
\begin{proof}
Part (1) is the result of \cite[p.19, Lmm.]{BushH}.  Now  $\cInd_{H}^{G} W\simeq \oplus_{g\in [G/H]} g\mathcal{W}$. Let $\mathcal{W}_{s}$ be the vector space generated by those $g\mathcal{W}$, $g\in Js^{-1}H/H$. Clearly $\mathcal{W}_{s}$ is $J$-stable, and $\mathcal{W}_{s}\simeq \oplus_{g\in [J/H_s\cap J]} gs^{-1}\mathcal{W}$. Therefore $\mathcal{W}_{s}\simeq  \cInd_{H_s\cap J}^Js^{-1}\mathcal{W}$, and $\Res_{J}^{G} \rho^G
 \simeq \oplus_{s\in \Delta}\cInd_{H_s\cap J}^J\rho^{s}$.
 \end{proof}
 \begin{lemma}\label{finitem}
Keep the notations, and assume $J=H$,  $1\in \Delta$.  For any $s\in \Delta$, $s\neq 1$, if   the index $[H: H_s\cap H]$ is  infinite, then
$\Hom_G(\cInd_{H}^G \sigma_1, \cInd_{H}^G \sigma_2)\simeq \Hom_H(\sigma_1, \sigma_2)$, for  a finite dimensional   representation  $(\sigma_1,W_1)$ of $H$,
and a smooth representation $(\sigma_2,W_2)$ of $H$.
\end{lemma}
\begin{proof}
By Frobenius reciprocity for open subgroups in \cite[p.20,Prop.]{BushH}, $\Hom_G(\cInd_{H}^G \sigma_1, \cInd_{H}^G \sigma_2)\simeq \Hom_H(\sigma_1, \cInd_{H}^G \sigma_2)\simeq \Hom_H(\sigma_1, \oplus_{s\in \Delta}\cInd_{H_s\cap H}^H(\sigma_2)^{s})\hookrightarrow \prod_{s\in \Delta}\Hom_H(\sigma_1, \cInd_{H_s\cap H}^H(\sigma_2)^{s})$. Let $\{e_1, e_2, \cdots, e_n\}$  be a basis of $ W_1$.   For $1\neq s\in \Delta$, if $0\neq A\in \Hom_H(\sigma_1, \cInd_{H_s\cap H}^H(\sigma_2)^{s})$, then $A(e_i) \in \cInd_{H_s\cap H}^H (W_2)^s\simeq \oplus_{t\in \Sigma_s} t(s^{-1} \mathcal{W}_2)$, where $\Sigma_s$ is a complete set of representatives for $H/[H_s\cap H]$.   So there  exists a finite natural number $m>0$, such that all $A(e_i)\in \oplus_{j=1}^m t_j (s^{-1} \mathcal{W}_2)$, for some $t_j\in\Sigma_s $. Denote  $\mathcal{W}_J= \oplus_{j=1}^m t_j (s^{-1} \mathcal{W}_2)$; clearly  $A(W_1)\subseteq \mathcal{W}_J$. Notice that for $t\in H$, $A(te_i) =tA(e_i) \in t\mathcal{W}_J$.

Assume $0\neq A(e_1)=\oplus_{j=1}^m c_{j} t_j s^{-1} w_j$, for $c_{j} \in \mathbb{C}$ with   $c_{j'}\neq 0$, and  non-zero  vectors $ w_j\in \mathcal{W}_2$.  Let  $t_0=t_{m+1} t_{j'}^{-1}$.  Then $A(t_0e_1)= t_0A(e_1)=\oplus_{j\neq j'} c_{j} t_0 t_j s^{-1} w_j \oplus c_{j'}t_{m+1} s^{-1}w_{j'}$.  Note that for different $j$, $t_{0} t_j H_s\cap H $ belongs to different left $H_s\cap H$-cosets in $H/H_s\cap H$. Hence
$A(t_0e_1) \notin  \mathcal{W}_J$;  this makes a contradiction. Therefore $\Hom_H(\sigma_1, \cInd_{H_s\cap H}^H(\sigma_2)^{s})=0$, for any $s\in \Delta$ with $s\neq 1$, and the result follows.
\end{proof}
\begin{lemma}\label{crossse}
Let $H$ be a closed normal subgroup of $G$, and $X=\frac{G}{H}$. Then there exists a continuous cross section $\kappa: X \longrightarrow G$.
\end{lemma}
\begin{proof}
Under the  $\sigma$-compact hypothesis on $G$, there exists a family $K_1\subseteq K_2 \subseteq \cdots $ of compact subsets  of $G$, such that $G=\cup_{n} K_n$. Let $U$ be an open profinite subgroup of $G$. Then $K_n\subseteq \cup_{a\in K_n} a U$, so $(K_n\setminus K_{n-1}) \subseteq K_n\subseteq \cup  a_i U$, for some finite set $a_i$. Hence replacing $K_n$ by  $\cup  a_i U$, we assume   each  $K_i=\cup_{j\in I} a_{j}U$, for a finite index set $I$. In particular, $K_i$ is an open compact set.

By \cite[Section 1.2, Prop.1]{Se}, for the profinite group $U$,  there exists a continuous cross section from $\frac{U}{H}$ to $U$. By Prop.\ref{restricition}, the canonical map $\iota_U: U \longrightarrow \frac{UH}{H}$,  induces a  topological group isomorphism: $\overline{i}_U: \frac{U}{U\cap H} \simeq  \frac{UH}{H}$. Hence there exists a continuous cross section $\kappa_U: \frac{UH}{H} \longrightarrow U$.

For above $K_n=\sqcup_{i=1}^{k_n} a_{i} U$, $K_nH/H =\cup_{i=1}^{k_n} a_{i} U H/H$.  If $a_{i}u_1H=a_{j}u_2H$, then $a_{i}u_1h_1=a_{j}u_2h_2$, and $a_{i}=a_{j}u_2h_2h_1^{-1}u_1^{-1}$. Hence for any $u\in U$,
$$a_{i}u H =a_{j}u_2h_2h_1^{-1}u_1^{-1} uH=a_{j}[u_2 u_1^{-1}u] \cdot[u^{-1}u_1h_2h_1^{-1}u_1^{-1} u]H \subseteq a_{j} UH.$$
By duality, if $(a_{i} U H/H) \cap (a_{j} U H/H) \neq \emptyset$, then $a_{i} U H/H= a_{j} U H/H $. Assume $K_nH/H = \sqcup_{i\in I_n} a_{i} U H/H$, for some $I_n \subseteq \{1, \cdots, k_n\}$.  For such $i$, there exists a continuous cross section $\kappa_{U,i}:  a_{i} U H/H \longrightarrow a_{i} U$ induced by $\kappa_{U}$; then a continuous cross section $$\kappa_{n}: K_nH/H =\sqcup_{i\in I_n} a_{i} U H/H \longrightarrow \sqcup_{i\in I_n} a_{i} U  \subseteq \cup_{i=1}^{k_n} a_{i} U=K_n.$$
Note that $K_n\setminus K_{n-1}=\sqcup_{j} b_{j} U$, for some finite set $b_{j}$. Hence there also exists a continuous cross section $ \kappa_{n,n-1}: (K_n\setminus K_{n-1})H/H \longrightarrow K_n\setminus K_{n-1}$. Now $(K_nH/H) \setminus  (K_{n-1}H/H)  \subseteq (K_n\setminus K_{n-1})H/H $. The restriction of $\kappa_{n,n-1}$ to $(K_nH/H) \setminus  (K_{n-1}H/H)$ is also a continuous map. By induction, assume that we  construct a family of  continuous cross sections $\kappa_{i}: K_iH/H \longrightarrow K_i$, for $1\leq i\leq n-1$, such that $\kappa_{i}|_{K_{i-1}H/H}=\kappa_{i-1}$. Then combining with the map  $\kappa_{n,n-1}$ on $(K_nH/H) \setminus  (K_{n-1}H/H) $, we get $\kappa_{n} $.  Finally we can let $\kappa=\lim \kappa_n=\cup \kappa_n$.

\end{proof}
\section{Projective representations of locally profinite groups}
In this section, we shall give some basic results about smooth projective representations of locally profinite groups. Our main references are \cite{BushH}, \cite{CuRe}, \cite{Ma1}.

\subsection{}
Let $G$ be a $\sigma$-compact, locally profinite group with an identity element $1_G$.  Let $\mathcal {X}_G$ denote the set  of  all continuous maps $f:  G \longrightarrow \C^{\times}$ such that $f(1)=1$, and write  $\mathcal{X}(G)$ for the set of all characters of $G$.
\begin{definition}\label{thedePro}\footnote{When $G$ is a finite group, the  definition is compatible with the classical one.  }
 A \emph{ smooth $\alpha$-projective representation} $(\pi, V)$ of $G$ is a map $\pi: G \longrightarrow \Aut_{\C}(V)$, for a $\C$-vector space $V$, such that
  \begin{itemize}
  \item[(1)]  $\pi(g_1) \pi(g_2)= \alpha(g_1, g_2) \pi(g_1g_2)$ for a (normalized) $2$-cocycle $\alpha(-,-)$ in the continuous cohomology $\Ha^2(G, \C^{\times})$(\emph{cf}. \cite{AM});
  \item[(2)] for each vector  $0 \neq v\in V$, there exist an open neighborhood  $U_v$ of $1_G$, and a continuous  map $\chi_v: U_v \longrightarrow \C^{\times}$ satisfying  $\pi(g) v=\chi_v(g)v$, for all $g\in U_v$.
  \end{itemize}
  \end{definition}
 \begin{remark}\label{itsretoK}
\begin{itemize}
 \item[(1)]  Let $K_v$ be an open compact subgroup of $U_v$. Then
 $\alpha(g_1,g_2)=\chi_v^{-1}(g_1g_2) \chi_v(g_1) \chi_v(g_2)$, for $g_1,g_2 \in K_v$, i.e. the restriction of $[\alpha]$ to $K_v$ is trivial.
 \item[(2)] Under the above situation, $\pi_v: K_v \longrightarrow \Aut_{\C}(V); g \longmapsto \pi(g) \chi_v^{-1}(g)$ is a honest representation of $K_v$. Moreover,  this representation is smooth.
   \end{itemize}
   \end{remark}
   \begin{proof}
   Let us check the last statement of Part (2).  For any $0 \neq w\in V$, there is an open compact subgroup $K_w \subseteq K_v$, and a continuous map $\chi_w: K_w \longrightarrow \C^{\times}$ such that (1) $\pi(h) w= \chi_w(h) w$, for $h\in K_w$; (2) $\alpha(h_1,h_2)=\chi_w^{-1}(h_1h_2) \chi_w(h_1) \chi_w(h_2)$,  for $h_1, h_2 \in K_w$; (3)  $ \pi_w: K_w \longrightarrow \Aut_{\C}(V); h \longmapsto \pi(h) \chi_w^{-1}(h)$ is a representation of $K_w$. Note that $\chi_v|_{K_w}$ differs from $ \chi_w$ by a character $\chi_{v, w}$ of $K_w$, so the kernel of $\chi_{v,w}$ is an open subgroup $U$ of $K_w$. It follows that the stabilizer $\Stab_{K_v}(w)$ of $w$ in the representation $(\pi_v, V)$  of $K_v$  contains that $U$.
   \end{proof}
 \begin{remark}\label{trivalK}
For a class $[c]$ of finite order in $\Ha^2(G, \C^{\times})$, there exists an open compact subgroup $K$ of $G$ such that the restriction  of $[c]$ to $K$ is trivial.
\end{remark}
\begin{proof}
Assume that $c^n(g_1, g_2)=1$, for any $g_i\in G$. Then $c(g_1,g_2)=  e^{\frac{2k\pi i}{n}}$, for some $k=0, \cdots, n-1$. Since $c(-,-)$ is a continuous function, $c^{-1}(1)$ is an open subset of $G\times G$.  Hence such $K$ exists.
\end{proof}

\begin{remark}\label{trivalK}
If we change above $G$ by its  one open  subgroup, the result also holds.   In this situation, smooth projective representations of locally profinite groups will be compared with usual projective representations of locally compact groups.
\end{remark}
For simplicity, we can take the following assumption:
\begin{assumption}
There exists an open subgroup $O$ of $G$ such that $\Ha^2(O, \C^{\times})$ only contains elements of finite order.
\end{assumption}

A projective \emph{$G$-morphism} between two smooth projective representations $(\pi_1, V_1)$ and $(\pi_2,V_2)$ of $G$ is just a $\C$-linear map $F: V_1 \longrightarrow V_2$ such that
\begin{equation}\label{morproj}
F(\pi_1(g)v)=\mu(g) \pi_2(g) F(v)
\end{equation}
holds for  all $g\in G$, all $v\in V_1$, and some $\mu \in \mathcal{X}_G$.  Let $\Hom_G^{\mu}(\pi_1, \pi_2)$ or $\Hom_G^{\mu}(V_1, V_2)$ denote the $\C$-linear space of all those morphisms, and let  $\Hom^{\mathcal{X}_{G}}_G(V_1, V_2)$ or $\Hom_G(V_1, V_2)$  be the union of $\Hom_G^{\mu}(V_1, V_2)$ as $\mu$ runs over all elements in $\mathcal{X}_G$. By observation, if every  $V_i \neq 0$, then  $\Hom_G(\pi_1, \pi_2)=0$, unless the two $2$-cocycles related to $(\pi_1, V_1)$ and $(\pi_2, V_2)$ represent the same class in $\Ha^2(G, \C^{\times})$. We call $(\pi_1,V_1)$ a projective  \emph{sub-representation} of $(\pi_2,V_2)$ if there exists an injective morphism in $\Hom_G(V_1,V_2)$. If  $V_1 \neq 0$, and  $(\pi_1,V_1)$ has no nonzero proper projective sub-representation, we call  $(\pi_1, V_1)$  \emph{ irreducible}. Two irreducible smooth projective representations $(\pi_1, V_1)$, $(\pi_2, V_2)$ of $G$ are \emph{projectively equivalent},  if there  exists a bijective $\C$-linear map in $\Hom_G(\pi_1, \pi_2)$ (its inverse is also a projective $G$-morphism.). In particular, when this bijective map lies in $\Hom_G^{1}(V_1, V_2)$,  $1$ being the  trivial map in $\mathcal{X}_G$,  we will say that  $(\pi_1, V_1)$, $(\pi_2, V_2)$  are \emph{linearly equivalent}. For two projective representations $(\pi_1, V_1), (\pi_2, V_2)$ of $G$, we can also define their inner product projective representation $(\pi_1\otimes \pi_2, V_1\otimes V_2)$ of $G$.

\begin{lemma}[Schur's Lemma]\label{froben}
Let $(\pi_1,V_1)$, $(\pi_2, V_2)$ be two projectively equivalent irreducible projective representations of $G$. Then:
 \begin{itemize}
\item[(1)] $\dim \Hom_G^{\mu} (V_1, V_2) \leq 1$, for every $\mu \in \mathcal{X}_G$;
\item[(2)] There exists certain $\mu_0 \in \mathcal{X}_G$, such that $\dim \Hom_G^{\mu_0} (V_1, V_2)=1$;
\item[(3)] If $\dim \Hom_G^{\mu_0} (V_1, V_2)=\dim \Hom_G^{\mu_1} (V_1, V_2)=1$, then $\mu_1=\mu_0 \chi$, for some $\chi\in \mathcal{X}(G)$.
\end{itemize}
\end{lemma}
\begin{proof}
First there exists at least a  non-zero bijective $G$-morphism $\varphi \in \Hom^{\mu_0}_G(\pi_1, \pi_2) $, for certain $\mu_0 \in \mathcal{X}_G$, and  $\varphi^{-1}\circ \phi\in \Hom_G^{1}(V_1, V_1)$, for any $\phi\in  \Hom^{\mu_0}_G(\pi_1, \pi_2) $. Next, similar to the proof of  the standard   Schur's Lemma (e.g. \cite[p. 21]{BushH}), we can assert that $ \dim \Hom_G^{1}(V_1, V_1)=1$, so the second result follows.  If $0\neq \psi \in \Hom_G^{\mu}(V_1, V_2)$, for some $\mu \in \mathcal{X}_G$. By the irreducible property,  $\psi $ is a bijective $G$-morphism, and $\dim\Hom_G^{\mu}(V_1, V_2)=1$ as shown above. For (3), assume the normalized $2$-cocycle attached to $(\pi_1, V_1)$ is $\alpha(-,-)$.  Let $0 \neq \phi\in \Hom_G^{\mu_1}(V_1, V_2), 0\neq \varphi \in \Hom_G^{\mu_0}(V_1, V_2)$, and $g_1, g_2\in G$, $0\neq v\in V_1$. Set $\mu_0^{-1}\mu_1=\chi\in \mathcal{X}_{G}$, and $F= \varphi^{-1}\circ \phi$. Then
\[\chi(g_1g_2)\alpha(g_1, g_2)^{-1}\pi_1(g_1)\pi_1(g_2)F(v)= F\big(\pi_1(g_1g_2)v\big)\]\[=  F\big( \alpha(g_1, g_2)^{-1} \pi_1(g_1) \pi_1(g_2) v\big)=\alpha(g_1, g_2)^{-1} \chi(g_1)\chi(g_2) \pi_1(g_1)\pi_1(g_2)F(v),\] so $\chi(g_1g_2)=\chi(g_1)\chi(g_2)$.
\end{proof}
\begin{corollary}
For any irreducible ordinary representation $(\pi, V)$ of $G$, let $\mathcal{O}(\pi)=\{ \chi \in \mathcal{X}(G) \mid \pi\otimes \chi \simeq \pi\}$; then the set $\End_G^{\mathcal{X}_G}(\pi)=\cup_{\chi\in \mathcal{O}(\pi)} \mathbb{C}_{\chi}$,  each $\mathbb{C}_{\chi} =\mathbb{C}$.
\end{corollary}
  Let $H$ be a closed  subgroup of $G$, and  let $ (\sigma, W)$ be a smooth $\omega$-projective representations of  $H$,  attached to a normalized $2$-cocycle $\omega(-,-)\in \Ha^2(H, \mathbb{C}^{\times})$. Assume that  $\Omega(-,-)$ is a normalized $2$-cocycle in $\Ha^2(G, \mathbb{C}^{\times})$ extending $\omega(-,-)$. Now let $X$ be  a linear space consisting of all  functions $ f: G \longrightarrow W $ such that (a) $f(hg)=\Omega^{-1}(h, g) \sigma(h)f(g)$, for $h\in H, g\in G$, (b) there is a compact open subgroup  $K_f$  of $G$, and a continuous function $\chi $ from $K_f$ to $\mathbb{C}^{\times}$, satisfying $f(xg)=\Omega^{-1}(x, g)\chi(g)f(x)$ for  $g\in K_f, x\in G$. Then we define a homomorphism $\Sigma:   G \longrightarrow \Aut_{\mathbb{C}}(X)$ by  $[\Sigma(g)f](x)=\Omega(x,g)f(xg))$
  for $g, x \in G$, $f\in X$.  Then $\Sigma(gkg^{-1}) [\Sigma(g)f](x)=[\Omega(gkg^{-1}, g) \Omega(g,k)^{-1} \chi(k)] (\Sigma(g)f)(x)$, for $k\in K_f$, so $\Sigma(g)f\in X$. It can be also checked that $\Sigma(g_1)\Sigma(g_2)=\Omega(g_1, g_2)\Sigma(g_1g_2)$ for $g_1, g_2\in G$. Hence the pair $(\Sigma, X)$ provides a projective representation, called  \emph{ projective  induced representation} of $G$  from $\sigma$, and it is denoted $\Ind_{H, \omega}^{G,\Omega} \sigma$.  We also consider the space $X_c$ which consists of all  functions   $f\in X$ such that $f$  is compactly supported modulo $H$. Then the space $X_c$ is $G$-stable, and it provides a projective representation of $G$, called  \emph{ projective induced representation with  compact supports}, denoted by $\cInd_{H, \omega}^{G,\Omega} \sigma$.

Assume now that $(\pi, V)$  is  a smooth projective representation of $G$,  attached to the above  $2$-cocycle $\Omega(-,-)$. Then the restriction of $(\pi, V)$ to $H$ is also a smooth projective representation, and it is denoted by $\Res_{H, \omega}^{G,\Omega}  \pi$ or $\Res_{H}^{G}  \pi$. For $\chi \in\mathcal{X}_{G}$, let us define $\Omega_{\chi}(g_1, g_2)=\Omega(g_1, g_2) \chi(g_1)^{-1}\chi(g_2)^{-1}\chi(g_1g_2)$, for $g_i\in G$, and  let $(\pi_{\chi}, V_{\chi}=V)$ be a $\Omega_{\chi}$-projective representation of $G$, defined by $g \longrightarrow \pi(g) \chi(g)^{-1}$, for $g\in G$.
\begin{theorem}[Frobenius reciprocity]\label{FRr}
  $\Hom_G^{\chi}\big( \pi, \Ind_{H, \omega}^{G,\Omega} \sigma\big) \simeq \Hom_H^{\chi}\big( \Res_{H, \omega}^{G,\Omega}  \pi, \sigma\big)$, for  $\chi\in \mathcal{X}(G)\subseteq \mathcal{X}(H)$.
\end{theorem}
\begin{proof}
We follow the proof in \cite[p.18]{BushH}. Firstly there is a canonical $H$-morphism $\alpha_{\sigma}:  \Ind_{H, \omega}^{G,\Omega} \sigma \longrightarrow W; f\longmapsto f(1)$. We then get a canonical map from $ \Hom_G^{\chi}\big( \pi, \Ind_{H, \omega}^{G,\Omega} \sigma\big) $ to $ \Hom_H^{\chi}\big( \Res_{H, \omega}^{G,\Omega}  \pi, \sigma\big)$ defined by $\phi \longmapsto \alpha_{\sigma}\circ \phi$. On the other hand, if $f: V \longrightarrow W$ is an $H$-morphism in  $ \Hom_H^{\chi}\big( \Res_{H, \omega}^{G,\Omega}  \pi, \sigma\big)$, then we can define $\beta_{\pi}(f): V \longrightarrow \Ind_{H, \omega}^{G,\Omega} W$ as $[\beta_{\pi}(f)_v](g)=\chi(g)^{-1}f(\pi(g)v)$, for $v\in V$; it is well-defined because  $[\beta_{\pi}(f)_v](hg)=\chi(hg)^{-1}f(\pi(hg)v)=\Omega^{-1}(h,g) \sigma(h)\chi(g)^{-1}f(\pi(g)v)=\Omega^{-1}(h,g)\sigma(h)[\beta_{\pi}(f)_v](g)=[\Sigma(h)\beta_{\pi}(f)_v](g)$, for $h\in H$, $g\in G$. Moreover, for $g, g_1 \in G$, we have
$$\beta_{\pi}(f)_{\pi(g)v}(g_1)=\chi(g_1)^{-1}f(\pi(g_1)\pi(g) v)=\Omega(g_1, g)\chi(g)\beta_{\pi}(f)_{v}(g_1g)=\chi(g)\Sigma(g)\beta_{\pi}(f)_v(g_1);$$
this implies  that $\beta_{\pi}(f)_{\pi(g)v}= \chi(g)\Sigma(g)\beta_{\pi}(f)_v$. Hence $\beta_{\pi}(f) \in \Hom_G^{\chi}\big( \pi, \Ind_{H, \omega}^{G,\Omega} \sigma\big) $, and it can be checked that   $\beta_{\pi}$ is an inverse morphism of $\alpha_{\sigma}$.
\end{proof}

\begin{corollary}\label{FRr1}
  $\Hom_G^{\chi}\big( \pi, \Ind_{H, \omega_{\chi}}^{G,\Omega_{\chi}} \sigma_{\chi}\big) \simeq \Hom_H^{\chi}\big( \Res_{H, \omega}^{G,\Omega}  \pi, \sigma_{\chi}\big)$, for  the general $\chi\in \mathcal{X}_{G}\subseteq \mathcal{X}_{H}$.
\end{corollary}
\begin{proof}
Let $\iota_{\chi} \in \Hom_G^{\chi}( \pi, \pi_{\chi})$ simply defined  by $v\longrightarrow v$, for $v\in V$.  Then $ \Hom^{1}_G(\pi_{\chi},
\Ind_{H, \omega_{\chi}}^{G,\Omega_{\chi}} \sigma_{\chi}) \simeq \Hom_G^{\chi}\big( \pi,  \Ind_{H, \omega_{\chi}}^{G,\Omega_{\chi}} \sigma_{\chi}\big); \phi \longrightarrow \phi\circ \iota_{\chi}$,  $\Hom_H^{1}\big( \Res_{H, \omega_{\chi}}^{G,\Omega_{\chi}}
 \pi_{\chi}, \sigma_{\chi}\big)\simeq
\Hom_H^{\chi}\big( \Res_{H, \omega}^{G,\Omega}  \pi, \sigma_{\chi}\big)$. By the above theorem, we get  the result.
\end{proof}
\subsection{}
For a compact open subgroup $K$ of $G$, we let  $\mathcal {X}_K$ denote the set  of all  continuous maps $f:  K \longrightarrow \C^{\times}$ such  that $f(1_K)=1$, and for $\chi \in \mathcal{X}_K$, let  $V^{K, \chi}=\left\{ v\in V \mid \pi(g)
v=\chi(g)v \textrm{ for all } g\in K\right\}$.  Note that $V=\cup_{K} \cup_{\chi \in \mathcal{X}_K} V^{K, \chi}$ as $K$ runs over all open compact subgroups of $G$.  Let $V[K, \chi]$ denote  the linear space spanned by
 $\pi(k)v-\chi(k)v$ for  $ v\in V, k\in K$.  Then the following result comes from   Remark \ref{itsretoK}:

\begin{corollary}\label{theKchipart}
If  $V^{K, \chi} \neq 0$, for an open compact subgroup $K$ of $G$,$\chi \in \mathcal{X}_K$, then
\begin{itemize}
\item[(1)] there is a smooth  representation $(\pi_{\chi}, V)$ of $K$, defined by $k \longmapsto \pi(k) \chi^{-1}(k)$ for $k\in K$,
\item[(2)] $V^{K, \chi}$ is just the $K$-invariant part of the above  $(\pi_{\chi}, V)$,
\item[(3)] $V[K, \chi]=\left\{ \sum_{i}\pi_{\chi}(k_i)v_i-v_i\mid k_i \in K, v_i \in V\right\}$.
\end{itemize}
 The following result is analogue of Cor.2 in  \cite[p.16]{BushH}.
\end{corollary}
\begin{lemma}\label{thedecompositionofKv}
Let $(\pi, V)$ be a smooth projective representation of $G$. Then $V=V^{K, \chi} \oplus V[K, \chi]$.
\end{lemma}
\begin{proof}
Assume $V\neq 0$. If $V^{K, \chi}\neq 0$,  the result arises from Cor.\ref{theKchipart}, and  \cite[p.16, Cor.2 ]{BushH}. If $V^{K, \chi}=0$,  we take a non-zero $v\in V$, such that $v\in V^{K_v, \chi_v}$ for some open compact subgroup $K_v \subseteq K$ and $\chi_v \in \mathcal{X}_{K_v}$. If there exists an element $g\in K_v$ such that $ \chi_v(g)-\chi(g) =\frac{1}{c}$ for some $c\in \C^{\times}$, then $\pi(g) ( cv)-\chi(g) (cv)= (\chi_v(g)-\chi(g))(cv)=v \in V[K, \chi]$. Otherwise $\chi_v=\chi|_{K_v}$. By Lmm. \ref{twocompactsubgroups}, we may and do assume that $K_v$ is a normal subgroup of $K$, so that $K/K_v$ is  a finite group. Then $(\pi|_K, V)$  is  projectively isomorphic to  another projective representation $(\pi_{\chi}, V)$ of $ K$,  defined by $k \longmapsto \pi(k) \chi(k)^{-1}$, for $k\in K$. Moreover $\pi_{\chi}|_{K_v}$ is a honest representation, whose ${K_v}$-invariant part  induces   a projective representation    of $K/K_v $; let us  denote it  by $(\sigma_v, V^{K_v})$. Let $(\sigma_v, W)$ be an irreducible constituent of $(\sigma_v, V^{K_v})$ containing $v$. By hypothesis,   $W$ is spanned by those $\pi_{\chi}(g_i) v_i-v_i$ for $g_i \in K$, $v_i \in W$ because $\left\{ \sum \pi_{\chi} (g_i) v_i -v_i \right\}$ is nonzero and $K$-stable. This proves the last case.
\end{proof}

 Keep the notations.   On the linear dual space $V^{\ast}$ of $V$, we define an action  of $G$ by the relation $\langle \pi^{\ast}(g) v^{\ast}, \pi(g)v\rangle =\langle v^{\ast},  v\rangle$, for $g\in G$, $v\in V$, $v^{\ast}\in \check{V}$. Denote by  $\check{V}=\cup_{K} \cup_{\chi \in \mathcal{X}_K} (V^{\ast})^{K, \chi}$ as $K$ runs over all open compact subgroups of $G$, and $\chi \in \mathcal{X}_K$. Then the  above action of $G$ on  the subspace  $\check{V}$ of  $V^{\ast}$ shall give a smooth projective representation of $G$, called the \emph{ contragredient projective representation} of $(\pi, V)$, denoted by $(\check{\pi}, \check{V})$ from now on. One says that $(\pi, V)$ is \emph{admissible} if the space $V^{K, \chi}$ is finite-dimensional for any open compact subgroup $K$ of $G$, and any $\chi \in \mathcal{X}_K$.  In this situation,   by Lmm.\ref{thedecompositionofKv} we have
\begin{lemma}
\begin{itemize}
\item[(1)] $\check{V}^{K, \chi^{-1}} \simeq (V^{K, \chi})^{\ast}$.
\item[(2)] $(\pi, V)$ is linearly equivalent to $(\check{\check{\pi}}, \check{\check{V}})$.
\end{itemize}
\end{lemma}
\begin{proof}
1) $\check{V}^{K, \chi^{-1}}$ consists of the elements $f: V \longrightarrow \C$ subject to the condition that $f(\pi(k^{-1}) v- \chi(k^{-1}) v)=0$, for all $k\in K$, and  $v\in V$, i.e. $f|_{V[K, \chi]}=0$, so $f\in (V^{K, \chi})^{\ast}$ by Lmm.\ref{thedecompositionofKv}.\\
2) There is a canonical a projective $G$-morphism in $\Hom_G^{1}\big( V, \check{\check{V}}\big)$ defined as $\iota: V  \longrightarrow \check{\check{V}}; v \longmapsto (\check{v} \longmapsto \langle \check{v}, v \rangle)$. And it maps $V^{K, \chi}$ bijectively to $\big( \check{V}^{K, \chi^{-1}}\big)^{\ast} \simeq (V^{K, \chi})^{\ast \ast} \simeq \check{\check{V}}^{K, \chi}$.
\end{proof}
Let us also present some  results on projective representations \emph{for later use},   analogue of the  results   in   \cite[Chap. 1]{BushH}.
\begin{lemma}\label{twoprojectiverepresesII}
Let $(\pi_1, V_1)$, $(\pi_2, V_2)$ be two smooth projective representations of $G$. Then
there is a bijection between $\Hom_G( \pi_1, \check{ \pi}_2)$ and $\Hom_G( \pi_1 \otimes \pi_2, \C)$ by sending $\Hom_G^{\mu} (\pi_1, \check{\pi}_2)$ to $\Hom^{\mu}_G(\pi_1 \otimes \pi_2, \C)$, for $\mu\in \mathcal{X}_G$.
\end{lemma}
\begin{proof}
If the associated classes of $(\pi_1, V_1)$ and $(\check{\pi}_2, \check{V}_2)$ are not the same, then both sides vanish. Otherwise the bijection $f\longleftrightarrow g$ is well determined by
$\langle f(v_1), v_2\rangle =  g(v_1 \otimes v_2)$
for $v_1 \in V_1, v_2 \in V_2$.
\end{proof}
\begin{lemma}\label{twoprojectiverepreses}
\begin{itemize}
\item[(1)] $ \Hom_G^{\mu} (\pi_1, \check{\pi}_2)\simeq  \Hom_G^{\mu} (\pi_2, \check{\pi}_1)$;
\item[(2)]   If $(\pi_2, V_2)$ is admissible, then   $\Hom_G^{\mu} (\pi_1, \pi_2) \simeq \Hom_G^{\mu}(\pi_1 \otimes \check{\pi}_2, \C)$.
\end{itemize}
\end{lemma}

Keep the notations of Cor.\ref{FRr1}. Recall the notations: $\Omega_{\chi}(-, -)$,  $(\sigma_{\chi}, W_{\chi})$.
\begin{lemma}\label{isp}
There exists a projective isomorphism  $\alpha_{\chi} \in \Hom^{\chi}_{G}\big(\Ind_{H, \omega}^{G,\Omega} \sigma,  \Ind_{H, \omega_{\chi}}^{G,\Omega_{\chi}} \sigma_{\chi}\big)$, defined by $f(g) \longrightarrow f(g) \chi^{-1}(g)$. Moreover $\alpha_{\chi}$ sends $\cInd_{H, \omega}^{G,\Omega} \sigma$ onto $\cInd_{H, \omega_{\chi}}^{G,\Omega_{\chi}} \sigma_{\chi}$.
\end{lemma}
\begin{proof}
 For $f\in \Ind_{H, \omega}^{G,\Omega} W$, $h\in H$, $g\in G$,  $\alpha_{\chi}(f)(hg)= f(hg)\chi^{-1}(hg)= \Omega^{-1}(h,g) [\sigma(h) f](g)\chi^{-1}(hg)=\Omega^{-1}_{\chi}(h,g) \chi^{-1}(g) [\sigma_{\chi}(h)f](g)=\Omega^{-1}_{\chi}(h,g) \sigma_{\chi}(h) [\alpha_{\chi}(f)](g)$, so $\alpha_{\chi}(f) \in \Ind_{H, \omega_{\chi}}^{G,\Omega_{\chi}} \sigma_{\chi}$.

 Set $\Sigma=\Ind_{H, \omega}^{G,\Omega} \sigma$, $\Sigma_{\chi}= \Ind_{H, \omega_{\chi}}^{G,\Omega_{\chi}} \sigma_{\chi}$. Then for $g, g_1\in G$, $\alpha_{\chi}[\Sigma(g_1) f](g)= f(gg_1) \Omega(g, g_1) \chi^{-1}(g)=\alpha_{\chi}(f) (gg_1) \Omega_{\chi}(g,g_1) \chi(g_1)=\chi(g_1) \Sigma_{\chi}(g_1)[\alpha_{\chi}(f)](g)$, so $\alpha_{\chi}$ is well-defined.  Clearly $\alpha_{\chi}$ is a bijective map, and  the last assertion also holds.
\end{proof}

Let $K$ be an open compact subgroup  of $G$, and   let $ \Delta$ be a complete set of representatives for $H \setminus G/K$.  For $s\in \Delta$, let $K_{s^{-1}}=sKs^{-1}$,   $\lambda_{\chi, s} (h)=\Omega_{\chi}^{-1}(s,s^{-1}h) \Omega_{\chi}(s^{-1}h,s)$, for $h\in H$. Let
 $\mathcal{K}=\{ f: \Delta \longrightarrow W_{\chi}\mid f(s) \in W_{\chi}^{H\cap K_{s^{-1}},\lambda_{\chi,s}}\}$, and $\mathcal{K}_c=\{ f \in \mathcal{K}\mid \supp f \textrm{ is a finite set }\}$.

 \begin{lemma}
 Assume $[\Ind_{H, \omega_{\chi}}^{G, \Omega_{\chi}} \sigma_{\chi}]^{K, 1} \neq 0$. Then there exists a bijection $\res_K: [\Ind_{H, \omega_{\chi}}^{G, \Omega_{\chi}} \sigma_{\chi}]^{K, 1} \longrightarrow \mathcal{K}; f\longmapsto f|_{\Delta}$, which sends $[\cInd_{H, \omega_{\chi}}^{G, \Omega_{\chi}} \sigma_{\chi}]^{K, 1}$ onto
 $ \mathcal{K}_c$.
  \end{lemma}
  \begin{proof}
  For any $0\neq f\in  [\Ind_{H, \omega_{\chi}}^{G, \Omega_{\chi}} \sigma_{\chi}]^{K, 1}$,  $s\in \Delta$, and  $ h\in H \cap sKs^{-1}$, we have
  \begin{equation}\label{chi1}
   \sigma_{\chi}(h) f(s) \Omega_{\chi}^{-1}(h,s)=f(hs)=f(s \cdot s^{-1}hs)= \Omega_{\chi}^{-1}(s, s^{-1}hs) f(s)
   \end{equation}
  Note that $\Omega_{\chi}(h,s)\Omega_{\chi}^{-1}(s, s^{-1}hs)=\Omega_{\chi}^{-1}(s,s^{-1}h) \Omega_{\chi}(s^{-1}h,s)=\lambda_{\chi, s}(h)$. Hence $f(s)\in W_{\chi}^{H\cap sKs^{-1}, \lambda_{\chi,s}}$. Conversely for any $f\in \mathcal{K}$, we can  extend it to a function $F: G \longrightarrow W_{\chi}$ in the following way: for $h\in H, s\in \Delta, k\in K$,  $F|_{HsK} (hsk)=\Omega_{\chi}^{-1}(h,sk)  \Omega_{\chi}^{-1}(s,k) \sigma_{\chi}(h)f(s)$. Clearly  $F|_{\Delta}=f$. So it reduces to  check that $F(-)\in [\Ind_{H, \omega_{\chi}}^{G, \Omega_{\chi}} \sigma_{\chi}]^{K, 1}$. By Remark  \ref{itsretoK},  $\Omega_{\chi}(k, k_1)=1$, for $k, k_1\in K$.  For $h, h_1\in H$, $k, k_1\in K$,
   \begin{align}
   F(h_1h sk) &= \Omega_{\chi}^{-1}(h_1h, sk) \Omega_{\chi}^{-1}(s,k) \sigma_{\chi}(h_1)\sigma_{\chi}(h) \Omega^{-1}_{\chi}(h_1, h) f(s)\label{FH1}\\
   &=\Omega_{\chi}^{-1}(h_1, hsk) \sigma_{\chi}(h_1) F(hsk)\label{FH2}
  \end{align}
 and
  \begin{align}
  F(hskk_1)&=\Omega_{\chi}^{-1}(h, skk_1)\Omega_{\chi}^{-1}(s, kk_1) \sigma_{\chi}(h) f(s)\label{FH3} \\
  & = \Omega_{\chi}^{-1}(hsk,k_1)\Omega_{\chi}^{-1}(h,sk) \Omega_{\chi}(sk,k_1) \Omega_{\chi}^{-1}(s,kk_1)\sigma_{\chi}(h)f(s)\label{FH4}\\
  &= \Omega_{\chi}^{-1}(hsk,k_1)\Omega_{\chi}^{-1}(h,sk) \Omega^{-1}_{\chi}(s,k)\sigma_{\chi}(h)f(s)\label{FH5}\\
  &= \Omega_{\chi}^{-1}(hsk,k_1) F(hsk).\label{FH6}
    \end{align}
  If $h_1sk_1=hsk$, then $h^{-1}h_1=skk_1^{-1}s^{-1} \in H\cap sKs^{-1}$,  so by (\ref{chi1}), $F(h^{-1}h_1s)=F(skk_1)$, and then by (\ref{FH1})-(\ref{FH2}),
  \begin{align*}
  F(h_1s) & =F(hh^{-1}h_1s)=\Omega_{\chi}^{-1}(h,h^{-1}h_1s)\sigma_{\chi}(h) F(h^{-1}h_1s)\\
  & =\Omega_{\chi}^{-1}(h,skk_1^{-1}) \sigma_{\chi}(h) F(skk_1^{-1})=F(hskk_1^{-1}),
  \end{align*}
  and then by (\ref{FH3})-(\ref{FH6}),
  \begin{equation*}
  F(h_1sk_1)=\Omega_{\chi}^{-1}(h_1s, k_1) F(h_1s)=\Omega_{\chi}^{-1}(hskk_1^{-1}, k_1)  F(hskk_1^{-1})=F(hsk).
    \end{equation*}

      \end{proof}
Let us   go back to $(\Sigma, \Ind_{H, \omega}^{G, \Omega} W)$.    Let   $\lambda_{s} (h)=\Omega^{-1}(s,s^{-1}h) \Omega(s^{-1}h,s) \chi(s^{-1}hs)$, for $h\in H$. Let
 $\mathcal{K}^{\chi}=\{ f: \Delta \longrightarrow W\mid f(s) \in W^{H\cap K_{s^{-1}},\lambda_s}\}$, and $\mathcal{K}^{\chi}_c=\{ f \in \mathcal{K}^{\chi}\mid \supp f \textrm{ is a finite set }\}$.
   \begin{lemma}\label{KKK}
 Assume $[\Ind_{H, \omega}^{G, \Omega} \sigma]^{K, \chi} \neq 0$. Then there exists a bijection $\res_{K, \chi}: [\Ind_{H, \omega}^{G, \Omega} \sigma]^{K, \chi} \longrightarrow \mathcal{K}^{\chi}; f\longmapsto f|_{\Delta}$, which sends $[\cInd_{H, \omega}^{G, \Omega} \sigma]^{K, \chi}$ onto
 $ \mathcal{K}^{\chi}_c$.
\end{lemma}
\begin{proof}
For $v\in [\Ind_{H, \omega}^{G, \Omega} \sigma]^{K, \chi} $, $k\in K$, by Lmm.\ref{isp},  $\alpha_{\chi}(v) \chi(k)=\alpha_{\chi}(\Sigma(k)v)=\Sigma_{\chi}(k) \alpha_{\chi}(v) \chi(k)$, so $\alpha_{\chi}(v)\in
[\Ind_{H, \omega_{\chi}}^{G, \Omega_{\chi}}\sigma_{\chi}]^{K, 1}$. For $s\in \Delta$, and  $ h\in H \cap sKs^{-1}$, we have $\sigma_{\chi}(h) [\alpha_{\chi}(v)(s)]=\lambda_{\chi, s}(h) [\alpha_{\chi}(v)(s)]$. By calculation, we obtain
$$\chi^{-1}(h) \sigma(h) v(s)\chi^{-1}(s)=\chi^{-1}(s) v(s)\Omega^{-1}(s,s^{-1}h) \Omega(s^{-1}h,s)\chi^{-1}(h)\chi(s^{-1}hs).$$
Hence $\sigma(h) v(s)= v(s)\Omega^{-1}(s,s^{-1}h) \Omega(s^{-1}h,s)\chi(s^{-1}hs)= \lambda_s(h) v(s)$, and $v(s) \in W^{H\cap sKs^{-1}, \lambda_s}$. The results then hold.
\end{proof}

 Recall that $\delta_{H\setminus G}=\frac{\Delta_G}{\Delta_H}$, and $\nu_{H\setminus G}$ is a positive semi-invariant measure on $H\setminus G$. The following result is  just the projective version  of  the duality theorem in \cite[p.32 ]{BushH}, and we shall follow that proof.

\begin{lemma}\label{duality1}
 $[\cInd_{H, \omega^{-1}}^{G,\Omega^{-1}} (\delta_{H\setminus G} \otimes  \check{\sigma})]^{\vee} \simeq \Ind_{H, \omega}^{G,\Omega} \sigma$.
\end{lemma}
\begin{proof}
1) For $\Phi \in \Ind_{H, \omega}^{G,\Omega} W, \phi \in  \cInd_{H, \omega^{-1}}^{G,\Omega^{-1}}  (\delta_{H\setminus G} \otimes \check{W})$, the function  $g \longrightarrow f(g)= \langle \Phi(g), \phi(g)\rangle$ lies in $C^{\infty}_c(H\setminus G, \delta_{H\setminus G})$. So there exists  a $G$-invariant pairing
$$  P:  \Ind_{H, \omega}^{G,\Omega} W \times \cInd_{H, \omega^{-1}}^{G,\Omega^{-1}}  (\delta_{H\setminus G} \otimes \check{W} )   \longrightarrow \mathbb{C};  (\Phi,\phi) \longmapsto \int_{H\setminus G}  \langle \Phi(g), \phi(g)\rangle d\nu_{H\setminus G}(\dot{g})$$  which   defines a map $P \in \Hom_{G}(\Ind_{H, \omega}^{G,\Omega} W \otimes\cInd_{H, \omega^{-1}}^{G,\Omega^{-1}}  (\delta_{H\setminus G} \otimes \check{W}) , \mathbb{C})$; by Lmm.\ref{twoprojectiverepresesII}, the map $P$ will induce a linear $G$-morphism $\iota: \Ind_{H, \omega}^{G,\Omega} W   \longrightarrow [\cInd_{H, \omega^{-1}}^{G,\Omega^{-1}}  (\delta_{H\setminus G} \otimes \check{W})]^{\vee}$. \\
2) Assume now $\{[\cInd_{H, \omega^{-1}}^{G,\Omega^{-1}}  (\delta_{H\setminus G} \otimes \check{W})]^{\vee}\}^{K, \chi} \simeq \{ [\cInd_{H, \omega^{-1}}^{G,\Omega^{-1}}  (\delta_{H\setminus G} \otimes \check{W})]^{K, \chi^{-1}} \}^{\ast}\neq 0$. As a consequence,  $\Omega_{\chi}(k_1, k_2)=1$, for $k_i\in K$. In this situation, the result of Lmm. \ref{KKK} also holds, i.e. there exists a bijection from $ [\Ind_{H, \omega}^{G, \Omega} \sigma]^{K, \chi} $ to  $ \mathcal{K}^{\chi}$.  For each $s\in \Delta$, let  $\mathcal{W}_s^{\chi}$ denote a basis of  the space  $  W^{H\cap K_{s^{-1}}, \lambda_s}$. Then for each $w\in \mathcal{W}_s^{\chi}$,  there exists a unique function $f_{s, w}\in  [\Ind_{H, \omega}^{G, \Omega} \sigma]^{K, \chi} $ such that $f_{s, w}(s)=w$, and $\supp f_{s,w}=HsK$.  Moreover those $f_{s, w}$'s form a basis of $ [\Ind_{H, \omega}^{G, \Omega} \sigma]^{K, \chi}|_{HsK} $.
Notice that $[W^{H\cap K_{s^{-1}}, \lambda_s}]^{\ast} \simeq  [\delta_{H\setminus G} \otimes \check{\sigma}]^{H\cap K_{s^{-1}},\lambda^{-1}_s}$. We now let $\check{\mathcal{W}}_s^{\chi} $ denote a basis of  $[W^{H\cap K_{s^{-1}}, \lambda_s}]^{\ast} $.
Similarly, for each $\check{w}\in \check{ \mathcal{W}}_s^{\chi}$,  there exists a unique function $f_{s, \check{w}}\in  \big(\cInd_{H, \omega^{-1}}^{G,\Omega^{-1}}  (\delta_{H\setminus G} \otimes \check{W})\big)^{K, \chi^{-1}} $ such that $f_{s, \check{w}}(s)=\check{w}$, and $\supp f_{s,\check{w}}=HsK$.  Then for $s_1, s_2\in \Delta$, $P(f_{s_1, w}, f_{s_2, \check{w}})=$ $\left\{\begin{array}{cc}  \nu_{H\setminus G}(Hs_1K) &  \textrm { if  } Hs_1K=Hs_2K,\\ 0 & \textrm{ otherwise. } \end{array} \right.$
Here $\nu_{H\setminus G}(Hs_1K)>0$, so $ [\Ind_{H, \omega}^{G,\Omega} W]^{K,\chi}   \longrightarrow \{[\cInd_{H, \omega^{-1}}^{G,\Omega^{-1}}  (\delta_{H\setminus G} \otimes \check{W})]^{\vee}\}^{K,\chi}$ is bijective, and $\iota$ is surjective.  If assume $[\Ind_{H, \omega}^{G,\Omega} \sigma]^{K, \chi} \neq 0$, the above proof also shows that $\iota$ is injective.
\end{proof}
\begin{lemma}
Let $(\sigma, W)$ be an $\omega^{-1}$-projective representation of $H$, $(\pi, V)$ an $\Omega$-projective representation of $G$. Then
$\Hom_G^{\chi}\big(\cInd_{H, \omega^{-1}_{\chi^{-1}}}^{G,\Omega^{-1}_{\chi^{-1}}} \sigma_{\chi^{-1}},  \check{\pi}\big) \simeq \Hom_H^{\chi}\big(\delta^{-1}_{H\setminus G}\otimes \sigma_{\chi^{-1}}, (\Res_{H, \omega}^{G,\Omega}\pi)^{\vee}\big)$, for  $\chi\in \mathcal{X}_{G}\subseteq \mathcal{X}_{H}$.
\end{lemma}
\begin{proof}
By Corollaries  \ref{FRr1}, \ref{twoprojectiverepreses}, Lmm.\ref{duality1},
 \begin{align*}
 &\Hom_G^{\chi}\big(\cInd_{H, \omega^{-1}_{\chi^{-1}}}^{G,\Omega^{-1}_{\chi^{-1}}} \sigma_{\chi^{-1}},  \check{\pi}\big)
 \simeq \Hom_G^{\chi}\big(\pi,  [\cInd_{H, \omega^{-1}_{\chi^{-1}}}^{G,\Omega^{-1}_{\chi^{-1}}} \sigma_{\chi^{-1}}]^{\vee}\big) \\
&\simeq \Hom_G^{\chi}\big(\pi,  \Ind_{H, \omega_{\chi}}^{G,\Omega_{\chi}}  (\delta_{H\setminus G} \otimes(\check{\sigma})_{\chi})\big)
  \simeq \Hom_H^{\chi}\big( \Res_{H, \omega}^{G,\Omega}  \pi, \delta_{H\setminus G} \otimes (\check{\sigma})_{\chi}\big)\\
& \simeq \Hom_H^{\chi}\big(\delta^{-1}_{H\setminus G}\otimes \sigma_{\chi^{-1}}, (\Res_{H, \omega}^{G,\Omega}\pi)^{\vee}\big).
\end{align*}
\end{proof}
  For $s\in \Delta$, let $H_s=s^{-1}Hs$, and set $\sigma^{s}(k)= \sigma(sks^{-1})$, for $k\in H_s\cap K$.  Let us also define a continuous function $\chi_s: g\in G \longrightarrow  \Omega(gs^{-1},s) \Omega^{-1}(s, gs^{-1})$,\footnote{Here the $\chi_s$ is just the function $g^{-1}_s$ given  by Mackey in \cite[p.276, Lmm.4.2]{Ma1}. From the proof of the next lemma \ref{decomMa}, we can see that Mackey's lemma a priori  holds.}  and  two $2$-cocycles
  $\Omega_{\chi_s}(g_1, g_2)= \Omega(g_1, g_2) \chi^{-1}_s(g_1)\chi^{-1}_s(g_1)\chi_s(g_1g_2)$,  $\Omega^s(g_1, g_2)=\Omega(sg_1s^{-1}, sg_2s^{-1})$ for $g_1, g_2\in G$.
Recall $ \Sigma_c=\cInd_{H, \omega}^{G,\Omega} \sigma$.
\begin{lemma}\label{Mackeycc}
$\Omega^s(g_1, g_2)= \Omega_{\chi_s^{-1}}(g_1, g_2)$,  and  $[\Omega^s]_{\chi_s}(g_1, g_2)=\Omega(g_1, g_2)$, for $g_1, g_2 \in G$.
\end{lemma}
\begin{proof}
The first statement is just the result of Lmm.4.2 in \cite{Ma1}. The second assertion is another way to write this result.
\end{proof}
  \begin{lemma}\label{decomMa}
 $\Res_{K}^G [\cInd_{H, \omega}^{G,\Omega} \sigma] \simeq  \oplus_{s\in \Delta} \cInd_{H_s\cap K, [\omega^s]_{\chi_s}}^{K,[\Omega^s]_{\chi_s}} [\sigma^s]_{\chi_s}\simeq \oplus_{s\in \Delta} \cInd_{H_s\cap K, \omega}^{K,\Omega} [\sigma^s]_{\chi_s}$, linear isomorphisms.
\end{lemma}
\begin{proof}
1)  For any  $s\in \Delta$, there  exists a canonical $\chi_{s}^{-1}$-projective  $K\cap H_s$-morphism $\cInd_{H, \omega}^{G,\Omega} \sigma \longrightarrow \sigma^s; f \longmapsto f(s)$, because
for $k\in K\cap H_s$, $[\Sigma_c(k)f](s)=f(sk) \Omega(s,k)=f(sks^{-1} s)\Omega(s,k)=\sigma(sks^{-1}) f(s) \Omega^{-1}(sks^{-1}, s) \Omega(s,k)=\sigma^s(k) f(s) \Omega^{-1}(ks^{-1},s) \Omega(s, ks^{-1})=\sigma^s(k) f(s) \chi^{-1}_{s}(k)$. By Frobenius reciprocity (Cor.\ref{FRr1}),  it induces  a $\chi^{-1}_{s}$-projective  $K$-morphism  $A_s: \cInd_{H, \omega}^{G,\Omega} \sigma  \longrightarrow \cInd_{H_s\cap K, \omega^s}^{K, \Omega^s}   \sigma^s=\Ind_{H_s\cap K, \omega^s}^{K, \Omega^s} \sigma^s$.  Applying the result of Lmm.\ref{isp}, we obtain a morphism $\alpha_{\chi_s} \in \Hom^{\chi_s}_{K}\big( \cInd_{H_s\cap K, \omega^s}^{K, \Omega^s}   \sigma^s, \cInd_{H_s\cap K, [\omega^s]_{\chi_s}}^{K,[\Omega^s]_{\chi_s}} [\sigma^s]_{\chi_s}\big)$. Then $\alpha_{\chi_s} \circ A_s\in  \Hom^{1}_{K}\big(\cInd_{H, \omega}^{G,\Omega} \sigma, \cInd_{H_s\cap K, [\omega^s]_{\chi_s}}^{K,[\Omega^s]_{\chi_s}} [\sigma^s]_{\chi_s}\big)$.
Therefore we obtain a linear  $K$-morphism $\alpha\circ A=\oplus_{s\in \Delta} \alpha_{\chi_s}\circ A_s: \cInd_{H, \omega}^{G,\Omega} \sigma  \longrightarrow  \prod_{s\in \Delta}\cInd_{H_s\cap K, [\omega^s]_{\chi_s}}^{K,[\Omega^s]_{\chi_s}} [\sigma^s]_{\chi_s}$.  Since for any $f\in \cInd_{H, \omega}^{G,\Omega} \sigma$, $supp f\subseteq \cup_{i=1}^n Hs_i K$ for certain $s_i\in \Delta$, the above mapping  $A$  factors through $\oplus_{s\in \Delta}\cInd_{H_s\cap K, [\omega^s]_{\chi_s}}^{K,[\Omega^s]_{\chi_s}} [\sigma^s ]_{\chi_s} \hookrightarrow \prod_{s\in \Delta}\cInd_{H_s\cap K, [\omega^s]_{\chi_s}}^{K,[\Omega^s]_{\chi_s}} [\sigma^s]_{\chi_s}$.  Hence we obtain $\alpha\circ A=\oplus_{s\in \Delta} \alpha_{\chi_s}\circ A_s: \cInd_{H, \omega}^{G,\Omega} \sigma   \longrightarrow  \oplus_{s\in \Delta} \cInd_{H_s\cap K, [\omega^s]_{\chi_s}}^{K,[\Omega^s]_{\chi_s}} [\sigma^s]_{\chi_s}\simeq \oplus_{s\in \Delta} \cInd_{H_s\cap K, \omega}^{K,\Omega} [\sigma^s]_{\chi_s}$.

2)  We first  show that  $\alpha\circ A$ is injective. If $\alpha\circ A(f_1)=\alpha\circ A(f_2)$, for $f_1, f_2\in\cInd_{H, \omega}^{G,\Omega} \sigma$, then $A_s(f_1)=A_s(f_2)$. More precisely $A_s(f_i) (k)= \Omega(s,k) \chi_s(k) f_i(sk)$, and $f_1(sk)=f_2(sk)$ for any $k\in K$. So $f_1|_{HsK}=f_2|_{HsK}$ for  any $s\in \Delta$, and   $f_1=f_2$.  Secondly,  assume $\sum_{i=1}^nt_{s_i}\in \sum_{i=1}^n\cInd_{H_s\cap K, \omega^{s_i}}^{K, \Omega^{s_i}}   \sigma^{s_i}$. Then there exist open compact subgroups $K_{s_i}$ of $K$ such that $t_{s_i}$ is $(K_{s_i}, \xi_{s_i})$-invariant. We now define an element $f: G \longrightarrow W$ as follows:
$  f|_{Hs_iK} (h s_ik)=\sigma(h) \Omega^{-1}(h,s_ik) \Omega^{-1}(s_i,k) \chi^{-1}_{s_i}(k)t_{s_i}(k)$,  for $ h\in H, k\in K$;
it is well-defined because for  $h_1, h_2\in H$, $k_1, k_2 \in K$, if  $h_1{s_i}k_1=h_2s_ik_2$, i.e.  $k_1= s_i^{-1}h_1^{-1}h_2 s_ik_2$, we have
 \begin{align*}
  &f|_{Hs_iK} (h_1s_ik_1)\\
  &= \sigma(h_1)\Omega^{-1}(h_1,s_ik_1) \Omega^{-1}(s_i,k_1) \chi^{-1}_{s_i}(k_1)t_{s_i}(k_1)\\
  &= \Omega^{-1}(h_1,s_ik_1) \Omega^{-1}(s_i,k_1) \chi^{-1}_{s_i}(k_1) \sigma(h_1)t_{s_i}( s_i^{-1}h_1^{-1}h_2 s_ik_2)\\
   &= \Omega^{-1}(h_1,s_ik_1) \Omega^{-1}(s_i,k_1) \chi^{-1}_{s_i}(k_1) \Omega^{s_i}(s_i^{-1}h_1^{-1} h_2 s_i, k_2)^{-1}\sigma(h_1)\sigma^{s_i}(s_i^{-1}h_1^{-1} h_2 s_i) t_{s_i}(k_2) \\
    & =\Omega^{-1}(h_1,s_ik_1) \Omega^{-1}(s_i,k_1) \chi^{-1}_{s_i}(k_1) \Omega^{s_i}(s_i^{-1}h_1^{-1} h_2 s_i, k_2)^{-1}\sigma(h_1)\sigma(h_1^{-1}h_2)t_{s_i}(k_2)\\
   & =[?]  \sigma(h_2)t_{s_i}(k_2),
 \end{align*}
 where $[?]=\Omega^{-1}(h_2,s_ik_2) \Omega^{-1}(s_i,k_2)\chi^{-1}_{s_i}(k_2)$ by the next lemma.  Now let  $K_f=\cap_{i=1}^n K_{s_i} $. Then two  $\xi_{s_i}|_{K_f}$, $\xi_{s_j}|_{K_f} $ will differ by a character of $K_f$; this character will be trivial on some open compact subgroup $K_{ij}$ of $ K_f$. Therefore $\xi_{s_i}|_{K_{ij}}=\xi_{s_j}|_{K_{ij}}$, and $f$ is $(\cap_{ij} K_{ij}, \xi_{i})$-invariant. Clearly $A_{s_i}(f)=t_{s_i}$. The proof is  completed.
  \end{proof}
 \begin{lemma}
 The above $[?]=\Omega^{-1}(h_2,s_ik_2) \Omega^{-1}(s_i,k_2)\chi^{-1}_{s_i}(k_2)$.
  \end{lemma}
  \begin{proof}
   \begin{align*}
   [?]&=\Omega^{-1}(h_1,s_ik_1) \Omega^{-1}(s_i,k_1) \chi^{-1}_{s_i}(k_1) \Omega^{s_i}(s_i^{-1}h_1^{-1} h_2 s_i, k_2)^{-1}\Omega(h_1, h_1^{-1}h_2)\\
   &=\Omega^{-1}(h_1,h_1^{-1}h_2s_ik_2) \Omega(h_1, h_1^{-1}h_2)\Omega(s_i,   k_1)^{-1}\Omega^{s_i}(k_1k_2^{-1}, k_2)^{-1}  \chi^{-1}_{s_i}(k_1)\\
   & =\Omega^{-1}(h_2, s_ik_2)  \Omega(h_1^{-1}h_2, s_ik_2)   \Omega^{s_i}(k_1k_2^{-1}, k_2)^{-1} \Omega(s_i,   k_1)^{-1} \chi^{-1}_{s_i}(k_1)\\
   &=\Omega^{-1}(h_2, s_ik_2)  \Omega^{s_i}(k_1k_2^{-1}, k_2s_i)   \Omega^{s_i}(k_1k_2^{-1}, k_2)^{-1}\Omega(s_i,   k_1)^{-1}  \chi^{-1}_{s_i}(k_1)\\
   &=\Omega^{-1}(h_2, s_ik_2)  \Omega^{s_i}(k_2, s_i)^{-1}   \Omega^{s_i}(k_1, s_i)\Omega(s_i,   k_1)^{-1} \chi^{-1}_{s_i}(k_1)\\
   &=\Omega^{-1}(h_2, s_ik_2)  \Omega^{s_i}(k_2, s_i)^{-1}  \Omega(s_ik_1s_i^{-1}, s_i)\Omega(s_i,   k_1)^{-1}  \Omega^{-1}(k_1s_i^{-1},s_i) \Omega(s_i, k_1s_i^{-1})\\
   &=\Omega^{-1}(h_2, s_ik_2)  \Omega^{s_i}(k_2, s_i)^{-1}\\
   &=\Omega^{-1}(h_2,s_ik_2) \Omega^{-1}(s_i,k_2)\chi^{-1}_{s_i}(k_2).
               \end{align*}
         \end{proof}
\subsection{Connection with   covering groups }\label{Algebraicg}
Let $F$ be a  non-archimedean local field with finite residue field,  $\mu_F$ the group of roots of unit in $F$ (a cyclic group of finite order). Let  $\textbf{G}$ be  a split, simple,  simply-connected algebraic group over $F$. Denote by $G=\textbf{G}(F)$ the $F$-points of $\textbf{G}$. By the works of \cite{D} \cite{Mat} \cite{Mo}, for  any  $2$-cocycle $\alpha(-,-)$ in the continuous cohomology $\Ha^2(\textbf{G}(F), \C^{\times})$, there exists a Steinberg cocycle $b(-,-)\in \Ha^2(\textbf{G}(F), \mu_F)$, and $\lambda\in \Hom(\mu_F, \mathbb{C}^{\times})$, such that $[\alpha]=[\lambda\circ b]\in \Ha^2(\textbf{G}(F), \C^{\times})$. To the $2$-cocyle $b(-,-)$,   is associated  a central extension of $\textbf{G}(F)$ by $\mu_F$,  expounded as
$$0\longrightarrow \mu_F \longrightarrow \widetilde{\textbf{G}(F)} \longrightarrow \textbf{G}(F)  \longrightarrow 1.$$
The extension group $\widetilde{\textbf{G}(F)}$ is also locally profinite, and one can think of the group law  being given by $$[g, t] \cdot [g', t']=[gg', b(g,g')tt'], \quad\quad g,g'\in \textbf{G}(F), t,t'\in \mu_F.$$
Now let $(\pi, V)$ be a smooth projective representation of $\textbf{G}(F)$, attached to the above $2$-cocycle $\lambda\circ b$. Assume now $\alpha(-,-)=\lambda\circ b(-,-)$.
\begin{lemma}\label{liftings}
 $(\pi, V)$  can lift uniquely to  a smooth representation $\widetilde{\pi}$ of $\widetilde{\textbf{G}(F)}$, such that   $\widetilde{\pi}|_{\mu_F}\simeq \lambda $.
 \end{lemma}
\begin{proof}
Let us  define  $\widetilde{\pi}$ as $\widetilde{\pi}([g, t])v=\lambda(t)\pi(g)v$, for $g\in\textbf{G}(F)$, $t\in \mu_F$, $v\in V$. For $[g, t], [g', t'] \in \widetilde{\textbf{G}(F)} $,
$$\widetilde{\pi}([g, t]\cdot[g', t'])= \widetilde{\pi}([gg', b(g,g')tt'])= \lambda(tt')\alpha(g,g') \pi(gg')=\widetilde{\pi}([g, t])\widetilde{\pi}([g', t'])$$
Moreover, for  $0\neq v\in V$, let  $K_v$, $U_v$ be the notions in Definition \ref{thedePro}, Remark \ref{itsretoK}; then the restriction of
$[\alpha(-,-)]$ to $K_v$ is trivial, and  $\alpha(g, g')=\chi^{\flat}_v(gg')^{-1} \chi^{\flat}_v(g) \chi^{\flat}_v(g')$, for $g, g' \in K_v$,  $\chi^{\flat}$ being certain continuous function from $ K_v$ to $\mathbb{C}^{\times}$. Assume the cardinality  of $\mu_F$ is just $n$, and let $\mu_n=\langle e^{\frac{2\pi i}{n}}\rangle\subseteq \mathbb{C}^{\times}$.   Then by composing $\chi^{\flat}_v$ with the canonical projection $\mathbb{C}^{\times} \longrightarrow\mathbb{C}^{\times}/\mu_n$, we obtain a character $\overline{\chi^{\flat}_v}$ from $K_v$ to $\mathbb{C}^{\times}/\mu_n$. Hence the kernel of $\overline{\chi^{\flat}_v}$ is an open  subgroup of $K_v$. Since $\ker \overline{\chi^{\flat}_v}=\cup_{t\in \mu_n} [\chi^{\flat}_v]^{-1}(t)$,  $\ker \chi^{\flat}_v$ is also  an open set of $K_v$ as well as $G$. So  $\widetilde{\pi}$ is  a smooth representation of $\widetilde{\textbf{G}(F)}$. The uniqueness follows from the fact that $\Hom(\textbf{G}(F), \mu_F)=0$.
\end{proof}
Let $(\pi_1, V_1)$, $(\pi_2, V_2)$ be two smooth projective representations of $\textbf{G}(F)$, attached to the  $2$-cocycle $\alpha(-,-)$.
Let $(\widetilde{\pi_1}, V_1)$,  $(\widetilde{\pi_2}, V_2)$  be their lifting representations of $\widetilde{\textbf{G}(F)}$ respectively  as described in Lmm.\ref{liftings}.
\begin{lemma}
  $(\pi_1, V_1)$, $(\pi_2, V_2)$ are linearly equivalent if and only if $\widetilde{\pi_1}\simeq \widetilde{\pi_2}$ as ordinary $\widetilde{\textbf{G}(F)}$-modules.
\end{lemma}
\begin{proof}
Assume first that $(\pi_1, V_1)$, $(\pi_2, V_2)$ are projectively equivalent by a $\textbf{G}(F)$-morphism $\varphi\in \Hom_{\textbf{G}(F)}^{1}(V_1, V_2)$. Then $\varphi\big(\widetilde{\pi_1}([g, t])v\big)=\varphi\big( \lambda(t)\pi_1(g)v\big)=\lambda(t)\pi_2(g)\varphi(v)=\widetilde{\pi_2}([g, t]) \varphi(v)$, i.e., $\varphi\in \Hom_{\widetilde{\textbf{G}(F)}}\big( V_1, V_2\big)$.  It is clear that the other side  also holds.
\end{proof}

\section{Abstract Howe correspondences}
\subsection{$G$}First of all  let $G$ be  a locally profinite group, $(\rho, V)$  a smooth representation of $G$.
For $(\pi, W) \in \Irr(G)$, we define
$V[\pi]= \cap_{f\in \Hom_G(V,W)} \ker(f)$.  The set $V_{\pi}=V/{V[\pi]}$ is called the \emph{greatest $\pi$-isotypic quotient} of $V$ with a canonical map $V \stackrel{p}{\longrightarrow} V_{\pi},$ which  satisfies the universal property: For any $G$-homomorphism $f$ from $V$ to $W$, it factors uniquely through $p$ as in the commutative diagram
$\xymatrix{ V \ar[r]^-{p}\ar[dr]_(.5){f}&  V_{\pi}\ar[d]^-{\overline{f}}\\
     &                                W}$.
Note that $\Hom_{G}(V,W)=0$ if and only if $V_{\pi}=0.$ In particular, if $\pi=1_G$, then $V_{\pi}$ is just the $G$-coinvariant set $V_G$ of $V$ and $V[\pi]=V[G]$ is generated by   vectors $\rho(g)v-v$ for all $g\in G$, $v\in V$.
\begin{proposition}\label{quotientzero}
If  $(\rho, V)$ is finitely generated, then $(\rho, V)=0$ if and only if $\mathcal{R}_G(\rho)=\emptyset$.
\end{proposition}
\begin{proof}
See \cite[p.16, Lmm.]{BernZ}.
\end{proof}
\begin{proposition}\label{finitegenerated}
Let $H$ be a closed subgroup of $G$.
\begin{itemize}
\item[(1)] If $H$ is also open, and $(\sigma, U)$ is a finitely generated smooth representation of $H$, then $\cInd_H^G \sigma$ is finitely generated.
\item[(2)] If $H\backslash G$ is compact, and $(\rho, V)$ is a  finitely generated smooth representation of $G$, then $\Res_H^G \rho$ is finitely generated.
\end{itemize}
\end{proposition}
\begin{proof}
1) Since $H$ is open, the compact induction $\cInd_H^G\sigma$ is just $\C[G] \otimes_{\C[H]} \pi$; hence the result follows.\\
2) Let $\{ v_1, \cdots, v_n\}$ be the set of generators of $V$ as a $G$-module. Choose an open compact subgroup $K$ of $G$ such that $e_K \ast v_j=v_j$ for $1\leq j\leq n$. By assumption $H \setminus G$ is compact, so there exists a finite number of  elements $g_1, \cdots, g_m$ of $G$ such that $G=\cup_{i=1}^m Hg_i K$. Therefore the representation $\Res_H^G\rho$ is generated by those $\rho(g_i) v_j$, $ i=1,  \cdots,  m, j=1,  \cdots,  n$.
\end{proof}
\begin{definition}
\begin{itemize}
\item[(1)] If $m_G(\rho, \pi)$ is finite for all $\pi\in \Irr(G)$, we will call $\rho$ a \textbf{representation with  finite (quotient) multiplicity }.
\item[(2)] If $m_G(\rho, \pi)\leq 1$ for all $\pi\in \Irr(G)$, we will call $\rho$ \textbf{multiplicity-free}.
\end{itemize}
\end{definition}
\begin{lemma}\label{typeimpliquequotientadmissible}
Let $(\rho, V)$ be a finitely generated smooth representation of $G$, and suppose that all the irreducible representations of $G$ are admissible. Then $\rho$ is a representation with  finite  multiplicity.
\end{lemma}
\begin{proof}
Assume that $V$ is generated by  elements $v_1, \cdots, v_n$ as a $G$-module. Let $(\pi, W) \in \Irr(G)$ and  $f\in \Hom_G(V, W)$. Then  for $v=\sum_{i=1}^n\sum_{j=1}^m c_{ij} \rho(g_j)v_i \in V$ we have
$$f(v)= f(\sum_{i=1}^n \sum_{j=1}^m c_{ij}\rho(g_j)v_i)=\sum_{i=1}^n\sum_{j=1}^m c_{ij} \pi(g_j)f(v_i),$$
which  is determined by its values at the points $v_1, \cdots, v_n$. We choose   an open compact subgroup $K$ of $G$  fixing  all the vectors $v_1, \cdots, v_n$; then $f(v_i)$ must take the value in $W^{K}$  for every $i$. Under the admissible assumption, the vector space $W^K$ is finite-dimensional, so $m_G(V, W) \leq n \dim W^K < +\infty.$
\end{proof}
\begin{lemma}\label{thequotient}
Under the above situation, let  $(\pi, W) \in \mathcal{R}_G(\rho)$, and suppose $m_G(\rho, \pi)=m <+\infty$. Then $V_{\pi}$ is a semi-simple $G$-module of finite length with the Jordan-H\"older set $\left\{ \pi\right\}$.
\end{lemma}
\begin{proof}
Let $f_1, \cdots, f_m$ be a set of $\C$-linear independent functions in $\Hom_G(V, W)$. Then $\prod_{i=1}^m f_i: V \longrightarrow \prod_{i=1}^m W$ is a $G$-morphism with the kernel $\cap_{i=1}^m \ker f_i$. Note that every $g\in \Hom_G(V, W)$ is equal to $\sum_{i=1}^m c_i f_i$, for some $c_i \in \C$. So $\ker g \supseteq \cap_{i=1}^m \ker f_i$, $V[\pi]= \cap_{i=1}^m \ker f_i$, and the result is proved.
\end{proof}
\subsubsection{Representations with  finite  multiplicity}
Let $F$ be a  non-archimedean local field with finite residue field, $\textbf{G}$ a connected reductive group over $F$. Denote by $G=\textbf{G}(F)$ the $F$-points of $\textbf{G}$.  Let $\textbf{P}$ be a parabolic $F$-subgroup of $\textbf{G}$ admitting a  $F$-Levi decomposition $\textbf{P}=\textbf{M}\textbf{N}$ (here $\textbf{M}$ is a connected reductive $F$-group and $\textbf{N}$ is the unipotent radical of $\textbf{P}$). Following \cite{Bern0} we write $\overline{\textbf{P}}$  for the parabolic subgroup opposite to $\textbf{P}$ with the Levi decomposition $\overline{\textbf{P}}=\textbf{M }\overline{\textbf{N}}$. Denote by $P=\textbf{P}(F)$, $\overline{P}=\overline{\textbf{P}}(F)$, $M=\textbf{M}(F)$, $N=\textbf{N}(F)$, $\overline{N}= \overline{\textbf{N}}(F)$. (\emph{cf}. \cite{Spring}, pp. 13-14).

Let $(\pi, V)$ be a smooth representation of $G$. The $N$-coinvariant space $V_N$ inherits a smooth representation $\pi_N$ of $M$, called the \textbf{Jacquet module } of $( \pi, V)$ at $N$. Define the Jacquet functor $J_N: \Rep(G) \longrightarrow \Rep(M)$ by $J_N(V)=V_N$.
 Let $(\sigma, W)$ be a smooth representation of $M$, viewed also as a smooth representation of $P$. Then we can define the parabolically induced functor $\Ind_{P\supset M}^G: \Rep(M) \longrightarrow \Rep(G);$
$ W \longmapsto \Ind_P^G W$.

Let us  recall some fundamental but difficultly achieved   results on the subject of  the  complex representations of $p$-adic reductive groups. \footnote{For different definitions,  in principle we always choose a much narrow   one and  leave the reader  to judge the proper one.  One can  read the  book \cite{Re}, which  systematically treats this part. } For the proofs,  one can consult   \cite[p.18, Theorem]{Bern0},  \cite[p.60, Theorem 6.3.10]{Cass}  and \cite[Theorem]{Bern3} respectively.
\begin{theorem}\label{lesfoncteurstypefini}

The functors $\Ind_{P\supset M}^G, J_N$ both map finitely generated (resp. admissible) representations into finitely generated (resp. admissible) representations.
\end{theorem}
\begin{theorem}\label{howetheorem}
Let $(\pi, V)$ be a smooth  representation of $G$. Then the following conditions are equivalent:
\begin{itemize}
\item[(1)] The $G$-space $V$ has finite length.
\item[(2)] $\pi$ is admissible and finitely generated.
\end{itemize}
\end{theorem}
\begin{theorem}\label{Gadmissible}
All the smooth irreducible representations of $G$ are admissible.
\end{theorem}
\begin{corollary}\label{longueurfiniedeJacquet}
The  functors $J_N$ and $\Ind_{P \supset M}^G$ both  map    smooth representations of finite length into  smooth representations  of finite length.
\end{corollary}
\begin{proof}
This comes from  Theorems \ref{lesfoncteurstypefini}, \ref{howetheorem}.
\end{proof}
The following unexpected theorem is due to Bernstein.
\begin{theorem}[{\cite[Main theorem]{Bern0}}]\label{bernsteintheorem}
Let $\rho \in \Rep(M)$, $\pi \in \Rep(G)$. Then
$\Hom_G\Big( \Ind_{P \supset M}^G\tfrac{\Delta_G}{\Delta_P} \rho,   \pi\Big) \simeq \Hom_M \Big(\rho, \pi_{\overline{N}}\Big)$.
\end{theorem}
\begin{lemma}\label{longueurfiniequotient}
If $(\pi, V)$ is a smooth representation of  $G$  with  finite   multiplicity, and $(\rho, W)$ is  a smooth representation of $G$ of finite length, then $m_G(\pi, \rho) < +\infty.$
\end{lemma}
\begin{proof}
If $0=W_0 \subseteqq W_1 \subseteqq \cdots \subseteqq W_s=W$ is a complete chain of $ \mathcal{H}(G)$-modules in $W$, then there is an exact sequence
$1 \longrightarrow W_{s-1} \longrightarrow W \longrightarrow W/{W_{s-1}} \longrightarrow 1$;
applying the left exact functor $\Hom_G(V, -)$ on it we obtain
$1 \longrightarrow \Hom_G( V, W_{s-1}) \longrightarrow \Hom_G(V, W) \longrightarrow \Hom_G(V, W/{W_{s-1}})$.
It follows that
$m_G(V,W) \leqq m_G(V,W_{s-1}) + m_G(V, W/{W_{s-1}})$.
By induction, we get
$m_G(V,W) \leqq \sum_{i=1}^s m_G(V, W_i/{W_{i-1}}) < +\infty$.
\end{proof}
\begin{lemma}\label{compo}
Under the conditions of the above lemma, for $\pi_1 \in \Irr(G)$, if $m_G(\pi, \pi_1)=m$ and $m_G(\pi_1, \rho)=n$, for some positive integers $m,n$, then $m_G(\pi, \rho) \geq \max\left\{ m, n\right\}$.
\end{lemma}
\begin{proof}
Assume first that $m\geq n$. Let $f_1, \cdots, f_m$ be the $\C$-linear independent functions in $\Hom_G(\pi, \pi_1)$ and  $0\neq g\in \Hom_G(\pi_1, \rho)$. Then $g\circ f_1, \cdots, g\circ f_m$ all belong to $\Hom_G(\pi, \rho)$ and   are  $\C$-linear independent. So the result  holds for $m\geq n$. The similar proof also works for the case $n > m$.
\end{proof}
\begin{lemma}\label{compodouble}
The similar result also holds if we replace the above $\pi_1$ by a finite direct sum of different irreducible representations $\sigma_1$, $\cdots$, $\sigma_k$ of $G$. More precisely  if assume  $m_i=m_G(\pi, \sigma_i)>0$, $ n_i=m_G(\sigma_i, \rho)>0$, then $m_G(\pi, \rho)\geq \max \left\{m= \sum_{i=1}^k m_i, n= \sum^k_{i=1} n_i\right\}$.
\end{lemma}
\begin{proof}
The proof is similar as above. For example assume $m\geq n$. We may take $0 \neq g_i\in \Hom_G(\sigma_i, \rho)$, so that $g=\oplus_{i=1}^k g_i $ is an injective morphism  from $\oplus_{i=1}^k \sigma_i $ to $\rho$.
\end{proof}
\begin{proposition}
The functors $\Ind_{P \supset M}^G$ and $J_N$ preserve the class of   smooth representations with  finite   multiplicity.
\end{proposition}
\begin{proof}
1) Let $(\pi, V)$ be a smooth representation of $M$ with  finite   multiplicity and $(\rho, W) \in \Irr(G)$. Theorem \ref{bernsteintheorem} tells us that
$\Hom_G\Big(\Ind_{P\supset M}^G \pi, \rho\Big) \simeq \Hom_M\Big(\tfrac{\Delta_P}{\Delta_G}\pi, \rho_{\overline{N }}\Big)$.
By Cor.\ref{longueurfiniedeJacquet},  $\rho_{\overline{N}}$ has finite length. By Lmm.\ref{longueurfiniequotient}, the dimension of  $\Hom_M\Big(\tfrac{\Delta_P}{\Delta_G}\pi, \rho_{\overline{N}}\Big)$ is finite. So  the result for $\Ind_{P\supset M}^G$ holds.\\
2) Now let $(\pi, V)$ be a smooth representation of  $G$ with  finite   multiplicity and $(\rho, W) \in \Irr(M)$. By virtue of Frobenius reciprocity, we have
$\Hom_M(J_N(\pi), \rho) \simeq \Hom_G(\pi, \Ind_{P \supset M}^G\rho)$.
The result then follows from  Cor.\ref{longueurfiniedeJacquet} and Lmm.\ref{longueurfiniequotient}.
\end{proof}
\subsection{$G_1\times G_2$}
Let us write  $G_1, G_2$ for  two locally profinite groups,  and  let $(\pi, S)$ be a smooth representation of $G_1 \times G_2$. We are interested in the relationship of the sets $\mathcal{R}_{G_1 \times G_2}(S)$, $\mathcal{R}_{G_1}(S)$ and $ \mathcal{R}_{G_2}(S)$. Let us  recall two technical lemmas proved by Waldspurger in \cite[pp. 45-46]{MVW}.
\begin{lemma}\label{waldspurger1}
Let $(\pi_1, V_1)$ be an irreducible \emph{admissible}  representation of $G_1$, $(\pi_2, V_2)$ a smooth representation of $G_2$. If a vector subspace  $W $  of $ V_1 \otimes V_2$ is  $G_1 \times G_2$-invariant, then there is a unique(up to isomorphism) $G_2$-subspace $V_2'$ of $V_2$ such that $W \simeq V_1 \otimes V_2'.$
\end{lemma}

\begin{lemma}\label{waldspurger2}
Let $(\pi_1, V_1)$ be an irreducible \emph{admissible} representation of $G_1$, $(\sigma, W)$ a smooth representation of $G_1 \times G_2$. Suppose that $\cap \ker(f)=0$ for all $f\in \Hom_{G_1}(W, V_1)$. Then there is a unique(up to isomorphism) smooth representation  $(\pi_2', V_2')$ of $G_2$ such that $\sigma \simeq \pi_1 \otimes \pi_2'$.
\end{lemma}
Now let $(\pi_1, V_1)$ be an irreducible admissible representation of $G_1$, $S_{\pi_1}=S/{S[\pi_1]}$   the greatest $\pi_1$-isotypic quotient. By Lmm.\ref{waldspurger2}, $S_{\pi_1} \simeq \pi_1 \otimes \pi_2',$ and $\pi_2' \simeq  \big( \check{V}_1 \otimes S_{\pi_1}\big)_{G_1}.$ Passaging to  the $\C$-linear dual of $\pi_2'$, we get the following  isomorphisms of $G_2$-modules:
$$\pi_2^{'\ast} \simeq \Hom_{G_1}(\check{V}_1\otimes S_{\pi_1}, \C) \simeq \Hom_{G_1}(S_{\pi_1}, V_1) \simeq \Hom_{G_1}(S, V_1)\simeq \Hom_{G_1}(\check{V}_1\otimes S, \C).$$
Moreover  considering their $G_2$-smooth parts,  we get
$(\pi_2')^{\vee} \simeq \Hom_{G_1}(S, V_1)^{\infty} \simeq \Hom_{G_1}(\check{V}_1\otimes S, \C)^{\infty}$.
In this way, we can see that $(\pi_2')^{\vee}$ is more easy to approach than $\pi_2'$ itself.

Let us begin to prove another statement in \cite{Rob1} about the quotient of the smooth representation $(\pi,S)$ of $G_1 \times G_2$.
\begin{lemma}\label{quotientdedeuxgroupes}
Follow the above notations, and suppose that $(\pi_2, V_2)$ is  a smooth representation of $G_2$. Then
\begin{itemize}
\item[(1)]  $\Hom_{G_1 \times G_2}( S, V_1 \otimes V_2) \simeq \Hom_{G_1\times G_2}(S_{\pi_1}, V_1 \otimes V_2)$.
\item[(2)]   $\Hom_{G_2}( \pi_2', \pi_2) \simeq \Hom_{G_1 \times G_2}( \pi_1 \otimes \pi_2', \pi_1 \otimes \pi_2)$.
\end{itemize}
\end{lemma}
\begin{proof}
(1) Let $\mathcal{A}$ be a basis of the vector space $V_2$. For an element $e \in V_2$, we will denote the canonical projection
$V_1 \otimes V_2 \longrightarrow V_1 \otimes e$
by $p_e$. For $f\in \Hom_{G_1 \times G_2} \big( S, V_1 \otimes V_2 \big)$, the composing map $p_e \circ f $ belongs to $\Hom_{G_1}(S, V_1 )$.  Clearly $\cap_{e\in \mathcal{A}} \ker (p_e\circ f) = \ker(f)$. It follows that
$$S[\pi_1] = \cap_{g\in \Hom_{G_1}(S, V_1)} \ker(g) \subseteq S[\pi_1 \otimes \pi_2] = \cap_{f\in \Hom_{G_1 \times G_2}(S, V_1 \otimes V_2) }\ker(f).$$
Hence by definition every map $f\in \Hom_{G_1 \times G_2}(S, V_1 \otimes V_2)$ needs  to factor through $S_{\pi_1} \longrightarrow V_1 \otimes V_2$.\\
(2) The isomorphism is given by $\varphi \longrightarrow 1 \otimes \varphi$. This map is well-defined and injective.  It suffices to check the surjection. Let  $0 \neq \varphi' \in \Hom_{G_1 \times G_2}(V_1 \otimes V_2', V_1 \otimes V_2)$ and $0 \neq e_2' \in V_2'$.  Let $\mathcal{A}=\{e_i\}_{i\in I}$ be a basis of $V_2$ and $V_{2,i}=\C e_i$ for $i\in I$.  Namely $V_1  \otimes V_2 \simeq  \oplus_{i\in I} V_1 \otimes V_{2,i}$, which can be viewed as a  sub-space of $\prod_{i\in I} V_1 \otimes V_{2,i}$. We will denote the projection  from $\prod_{i\in I} V_1 \otimes V_{2,i}$ to $V_1 \otimes V_{2,i}$ by $p_i$.  Through $\varphi'$ and
 $V_1 \otimes V_2 \longrightarrow  \prod_{i\in I} V_1 \otimes V_{2,i} \stackrel{p_i}{\longrightarrow} V_1 \otimes V_{2,i}$, we get a $G_1$-homomorphism
$\varphi_i': V_1 \otimes e_2' \longrightarrow  V_1 \otimes V_{2,i}$.
Since $\pi_1$ is admissible, by virtue of Schur's lemma the map $\varphi_i'$ is given by
 $\sum_k v_k \otimes e_2' \longmapsto \sum_k v_k \otimes c_ie_i,$ for some  $c_i \in \C$. On the other hand $\prod_{i\in I} \varphi_i': V_1 \otimes e_2' \longrightarrow \prod_{i\in I} V_1 \otimes V_{2,i}$ has to factor through $V_1 \otimes e_2' \longrightarrow V_1 \otimes V_2$, so  $\varphi'_i =0$ for all but a finite number of indices  $i$. Therefore we can define a map
 $\varphi_{e_2'}: \C e_2' \longrightarrow V_2; \varphi_{e_2'}(e_2')= \sum_{i\in I} c_i e_i, $ which  satisfies
 $\varphi'|_{V_1 \otimes e_2'}= 1 \otimes \varphi_{e_2'}$.
In this way, for any non-zero element $v_2' \in V_2'$ we construct  a map $\varphi_{v_2'}: \C v_2' \longrightarrow V_2$. For $v_2'=0$, we can simply let $\varphi_{v_2'}=0$.  Then  these maps  satisfy
\begin{itemize}
\item[(i)] $\varphi'|_{V_1 \otimes v_2'}= 1 \otimes \varphi_{v_2'}$, for $v_2' \in V_2'$, and
\item[(ii)] $\varphi_{\alpha v_2' + \beta v_2''}(\alpha v_2' + \beta v_2'')=\varphi_{\alpha v_2'}(\alpha v_2')+ \varphi_{\beta v_2''}(\beta v_2'')=\alpha \varphi_{v_2'}(v_2')+ \beta \varphi_{v_2''}(v_2'')$, for  $\alpha, \beta \in \C$, $v_2', v_2'' \in V_2'$.
\end{itemize}
So we  can define a map  $\varphi$ from $ V_2'$ to  $V_2$ as
$\sum_i v_{2,i}' \longmapsto\sum_i\varphi_{v_{2,i}'}(v_{2,i}')$.
It is well-defined and $\C$-linear satisfying  $\varphi'= 1 \otimes \varphi$, which forces  $\varphi$  to be $G_2$-equivariant, i.e. $\varphi \in \Hom_{G_2}(V_2', V_2)$.
\end{proof}
\subsection{Theta representation}
Keep the above notations. Assume now that every  irreducible smooth representation of $G_i$ is admissible, $i=1, 2$. According to  \cite[ p.20, Prop.]{BernZ}, every smooth irreducible representation of $G_1 \times G_2$ has the unique(up to isomorphism) form $\pi_1 \otimes \pi_2$ for $\pi_i \in \Irr(G_i)$, $i=1,2$.
\begin{proposition}\label{lissedetypefini}
Let $(\pi, S)$ be a finitely generated smooth representation of $G_1 \times G_2$.
\begin{itemize}
\item[(1)] $\pi$ is a smooth representation with  finite   multiplicity.
\item[(2)] $\mathcal{R}_{G_1 \times G_2}(S)=\emptyset$ if and only if $(\pi, S)=0$.
\item[(3)] For $\pi_1 \in \Irr(G_1)$, let $S_{\pi_1}$ denote the greatest $\pi_1$-isotypic quotient of $\pi$. If $S_{\pi_1} \simeq \pi_1 \otimes \pi_2'$, then $\pi_2'$ is a finitely generated smooth representation of $G_2$.
\end{itemize}
\end{proposition}
\begin{proof}
(1) and (2) follow from Props.\ref{typeimpliquequotientadmissible}, \ref{quotientzero} respectively. For (3) there is
$S_{\pi_1} \simeq S/{S[\pi_1]}\simeq \pi_1 \otimes \pi_2'.$
By hypothesis, $\pi_1 \otimes \pi_2'$ is  generated by a  set $\{ v_1^{(1)} \otimes v_2^{'(1)}, \cdots , v_1^{(n)} \otimes v_2^{'(n)}\}$ as a $G_1 \times G_2$-module. Since  $(\pi_1, V_1)$ is an irreducible admissible representation of $G_1$, applying Lmm.\ref{waldspurger1} we know that  $\pi_2'$(up to isomorphism) is generated by $v_2^{'(1)}, \cdots ,v_2^{'(n)}$ as a $G_2$-module.
\end{proof}
\begin{lemma}
Let $(\pi, S)$ be an admissible smooth representation of  $G_1 \times G_2$, such that $S_{\pi_1} \neq 0$, for some $\pi_1 \in \Irr(G_1)$.  If we  write $S_{\pi_1} \simeq \pi_1 \otimes \pi_2'$,  then $\pi_2'$ is also an admissible smooth representation of $G_2$.
\end{lemma}
\begin{proof}
By definition, there is an exact sequence of $G_1  \times G_2$-modules: $1 \longrightarrow S_0 \longrightarrow S \longrightarrow S_{\pi_1} \simeq \pi_1 \otimes \pi_2' \longrightarrow 1$. So $S_{\pi_1}$ is an admissible $G_1 \times G_2$-module. By hypothesis, $\pi_1$ is admissible, which implies the result.
\end{proof}
\begin{proposition}\label{grapheprojectif}
Let $(\pi, S)$ be a finitely generated smooth representation of $G_1 \times G_2$.
\begin{itemize}
\item[(1)] If $\pi_1 \otimes \pi_2 \in \mathcal{R}_{G_1 \times G_2}(\pi)$, then $\pi_1 \in \mathcal{R}_{G_1}(\pi)$.
\item[(2)] If $\pi_1 \in \mathcal{R}_{G_1}(\pi)$, then there is $\pi_2 \in \mathcal{R}_{G_2}(\pi)$ such that $\pi_1 \otimes \pi_2 \in \mathcal{R}_{G_1 \times G_2}(\pi)$.
\end{itemize}
\end{proposition}
\begin{proof}
1) Let $(\pi_1 \otimes \pi_2, G_1 \times G_2, V_1 \otimes V_2) \in \mathcal{R}_{G_1 \times G_2}(\pi)$ which  means that there is a surjective map $V \stackrel{f}{\longrightarrow} V_1 \otimes V_2$. Take an element $0 \neq e_2 \in V_2$ and  denote the canonical projection $V_2 \longrightarrow \C e_2$ by $p_{e_2}$. Composing $f$ with $1 \otimes p_{e_2}$  gives  a non-trivial map from $V$ to $V_1$, i.e. $\pi_1 \in \mathcal{R}_{G_1}(\pi)$.\\
2) Suppose that $(\pi_1, V_1) \in \mathcal{R}_{G_1}(\pi)$. Thus the greatest $\pi_1$-isotypic quotient $S_{\pi} \simeq \pi_1 \otimes \pi_2'$ is non-trivial, which implies that $\pi_2'$ is also non-trivial. By Prop.\ref{lissedetypefini} (3),  $\pi_2'$ is finitely generated and $\mathcal{R}_{G_2}(\pi_2') \neq 0$. By Lmm.\ref{quotientdedeuxgroupes}, there is a bijection between $\mathcal{R}_{G_1 \times G_2}(S_{\pi_1})$ and $\mathcal{R}_{G_2}(\pi_2')$. So there is an irreducible representation $(\pi_2, V_2)$ of $G_2$ such that $\pi_1 \otimes \pi_2 \in \mathcal{R}_{G_1 \times G_2}(\pi)$.
\end{proof}
Now we consider the general case. Let $(\pi, S)$ be a smooth representation of $G_1 \times G_2$. The result in Prop.\ref{grapheprojectif} (1) also holds. So there are two canonical projections
$$p_i: \mathcal{R}_{G_1 \times G_2}(\pi) \longrightarrow \mathcal{R}_{G_i}(\pi); \pi_1 \otimes \pi_2 \longmapsto \pi_i, \quad  i=1,2.$$
 From now on, we will denote their images by $\mathcal{R}^0_{G_i}(\pi)$ for $i=1,2$.
 \begin{corollary}
 If $(\pi, S)$ is a finitely generated smooth representation of the group $G_1 \times G_2$, then the above maps $p_1$, $p_2$  both are surjective.
 \end{corollary}
 When $p_1$(resp. $p_2$) is injective,  there is a unique irreducible representation $\pi_2^{(1)} \in \mathcal{R}_{G_2}(\pi)$(resp. $\pi_1^{(2)}\in \mathcal{R}_{G_1}(\pi)$) such that $\pi_1 \otimes \pi_2^{(1)} \in \mathcal{R}_{G_1 \times G_2}(\pi)$(resp. $\pi_1^{(2)} \otimes \pi_2 \in \mathcal{R}_{G_1 \times G_2}(\pi)$), so that we obtain  two  canonical mappings $\theta_1: \mathcal{R}_{G_1}^0(\pi) \longrightarrow \mathcal{R}_{G_2}^0(\pi); \pi_1 \longmapsto \pi_2^{(1)}$ (resp. $\theta_2: \mathcal{R}_{G_2}^0(\pi) \longrightarrow \mathcal{R}_{G_1}^0(\pi); \pi_2 \longmapsto \pi_1^{(2)}$). Namely  $(\mathcal{R}_{G_1 \times G_2}(\pi), p_i)$ is the graph of the theta map $\theta_i$  for $i=1, 2$ respectively.
\begin{definition}\label{thethetarepre}
If $p_1$ and $p_2$ both are injective,  $\pi$ is also multiplicity-free, and  $\pi_{\sigma_i} \simeq \sigma_i \otimes \Theta_{\sigma_i}$ is a finitely generated smooth representation of $G_i \times G_j$ for $1 \leq i \neq j \leq 2$, we will call $\pi$  a \textbf{theta} representation of $G_1 \times G_2$. In this situation,   the theta maps $\theta_1$, $\theta_2$ both are bijective and $\theta_1=\theta_2^{-1}$. So  we get a correspondence between $\mathcal{R}_{G_1}^0(\pi)$ and $\mathcal{R}_{G_2}^0(\pi)$, called  the \textbf{Howe correspondence} or the \textbf{theta correspondence}.
\end{definition}

\begin{remark}
\begin{itemize}
 \item[(1)] If $p_1$, $p_2$ both are injective, we will say that  $\pi$ \emph{satisfies the property of graph} in future.
\item[(2)] In Definition \ref{thethetarepre}, we also have another two   correspondences: $\sigma_i \stackrel{ \Theta_{\sigma_i}}{\longleftrightarrow }V_{\sigma_i}$, for $i=1,2$.  In some simple cases,  the representation $\pi$ may be reconstructed  by those  $\Theta_{\sigma_i}$ for all $\sigma_1 \otimes \sigma_2 \in \mathcal{R}_{G_1 \times G_2}(\pi)$.  For us,  we mainly care about the Howe correspondences and  limit ourself to study the representation $\pi$, with some finiteness conditions on  its greatest $\sigma_i$-isotypic quotients.
\item[(3)] In the above definition, if $\Theta_{\sigma_i}$  is not required to be finitely generated, we will call $\pi$ a general theta representation of $G_1\times G_2$.
 \item[(4)] If the greatest $\sigma_i$-quotient $\pi_{\sigma_i} \simeq \sigma_i \otimes \Theta_{\sigma_i}$ is a smooth representation of $G_i \times G_j$ of finite length, we call $(\pi, V)$ a theta representation of $G_1 \times G_2$ of finite length. In this case $\Theta_{\sigma_i}$ is an indecomposable representation of $G_i$ by the next lemma \ref{mulitione}.
\end{itemize}
\end{remark}
\begin{lemma}\label{mulitione}
If $(\pi, V)$ is a multiplicity-free representation of $G$ of finite length, and $\mathcal{R}_{G}(\pi)$ has only one element, then $\pi$ is an indecomposable representation of $G$.
\end{lemma}
\begin{proof}
If $V=V_1 \oplus V_2$, then either the case that $V_1$ and $V_2$ have different irreducible  quotient representations, or the case that $V_1$ and $V_2$ have the  same quotient representation whose  multiplicity in $V$ is bigger than $2$; both cases contradict to  the hypotheses.
\end{proof}
Let us finish this section by proposing  some simple properties for such representations.
\begin{lemma}\label{thegrquofpro}
Let $(\pi_1, V_1)$, $(\pi_2, V_2)$ be two smooth representations of $G_1$, $G_2$ respectively. Then:
\begin{itemize}
\item[(1)] $(V_1 \otimes  V_2)[G_1]= V_1[G_1] \otimes V_2$, and $(V_1 \otimes  V_2)[G_2]=V_1 \otimes (V_2[G_2])$,
\item[(2)] $(V_1 \otimes V_2)_{G_1 \times G_2} \simeq V_{1_{G_1}} \otimes V_{2_{G_2}}$,
\item[(3)] $(V_1 \otimes V_2)_{\sigma_1 \times \sigma_2} \simeq V_{1_{\sigma_1}} \otimes V_{2_{\sigma_2}}$, for $(\sigma_i, W_i) \in \Irr(G_i)$.
\end{itemize}
\end{lemma}
\begin{proof}
1) Let us verify the first assertion.  Let $\left\{ e_i\right\}_{i\in I}$ be a basis of $V_2$. So $ \Hom_{G_1}\big( V_1 \otimes V_2, \C\big) \simeq \prod_{i\in I} \Hom_{G_1}\big( V_1 \otimes e_i, \C\big); f \longmapsto (f_i)$, for $f_i =f|_{V_1 \otimes e_i}$, and $\ker(f) \supseteq \sum_{i\in I} \ker f_i$. It follows that $$(V_1 \otimes V_2)[G_1] = \cap_{f\in \Hom_{G_1}\big( V_1 \otimes V_2, \C\big)} \ker f$$
$$\supseteq \cap_{f\in \Hom_{G_1}(V_1 \otimes V_2, \C)} \sum_{i\in I} \ker f_i \supseteq \sum_{i\in I} \cap_{g_i \in \Hom_{G_1} \big( V_1 \otimes e_i, \C\big)} \ker g_i =V_1[G_1] \otimes V_2.$$
Conversely,  if  $\sum_{i=1}^n v_1^{(i)} \otimes e_i \in (V_1 \otimes V_2)[G_1]$, we have $f\big( \sum_{i=1}^n v_1^{(i)} \otimes e_i\big)=0$, for any $f\in \Hom_{G_1}(V_1 \otimes V_2, \C)$, i.e. $\sum_{i=1}^n f_i (v_1^{(i)} \otimes e_i)=0$, where  $f_i=f|_{V_1 \otimes e_i}$. Since $f_i$ can be any element in $\Hom_{G_1}\big(V_1 \otimes e_i, \C\big)$, in particular the zero element, we assert that each $v_1^{(j)}\otimes e_j$ belongs to $\ker f_j$, hence to $\cap_{f_j\in \Hom_{G_1}\big( V_1 \otimes e_j, \C\big)} \ker f_j=V_1[G_1] \otimes e_j$.  No doubt that the previous $\sum_{i=1}^n v_1^{(i)} \otimes e_i \in V_1[G_1] \otimes V_2$.\\
2) From the definition, we know that $(V_1 \otimes V_2)[G_1 \times G_2]$ is linearly spanned by $v_1 \otimes v_2 -\pi_1(g_1) v_1 \otimes \pi_2(g_2) v_2$, for all $v_i \in V_i$, $g_i \in G_i$. Writing $v_1 \otimes v_2 -\pi_1(g_1) v_1 \otimes \pi_2(g_2) v_2$ in its equal form $v_1 \otimes (v_2 -\pi_2(g_2)v_2) + (v_1-\pi_1(g_1) v_1) \otimes \pi_2(g_2) v_2$, shows that  $\big( V_1 \otimes V_2\big)[G_1 \times G_2]= V_1[G_1]\otimes V_2 + V_1 \otimes (V_2[G_2])$. Notice that $\big( V_1[G_1] \otimes V_2\big) \cap \big(V_1 \otimes (V_2[G_2])\big) \supseteq V_1[G_1] \otimes V_2[G_2]$. On the other hand, assuming that $ v=\sum_{i=1}^n v_1^{(i)}\otimes v_2^{(i)}$, for some nonzero linearly independent elements  $v_1^{(i)}\in V_1[G_1]$ and some nonzero elements $v_2^{(i)} \in V_2$, belongs to the above  left-hand side set.  Then $f( v)=0$ for all $f\in \Hom_{G_2}\big( V_1 \otimes V_2, \C\big)$. By considering $f|_{v_1^{(i)} \otimes V_2}$, we see $v_2^{(i)} \in V_2[G_2]$. It then follows that $\big( V_1[G_1] \otimes V_2\big) \cap \big( V_1 \cap (V_2[G_2])\big) =V_1[G_1] \otimes V_2[G_2]$. Now
$$(V_1 \otimes V_2)_{G_1 \times G_2} \simeq V_1 \otimes V_2 /{(V_1 \otimes V_2 [G_1 \times G_2])}$$
$$\simeq V_{1} \otimes V_2 /{\Big(V_1[G_1] \otimes V_2 + V_1 \otimes (V_2[ G_2])\Big)} \simeq \frac{V_1\otimes V_2/{V_1[G_1] \otimes V_2}}{\Big(V_1[G_1] \otimes V_2 + V_1 \otimes (V_2[G_2])\Big)/{V_1[G_1] \otimes V_2}}$$
$$\simeq V_{1_{G_1}} \otimes V_2 /{\big(V_{1_{G_1}} \otimes V_2[G_2]\big)} \simeq V_{1_{G_1}} \otimes V_{2_{G_2}}.$$
3) Note that there exists a canonical surjective map $f: (V_1\otimes V_2)_{\sigma_1 \otimes \sigma_2}  \longrightarrow (V_1)_{\sigma_1} \otimes (V_2)_{\sigma_2}$. Moreover $[(V_1\otimes V_2)_{\sigma_1 \otimes \sigma_2}]^{\ast} \simeq \Hom_{G_1\times G_2}( \check{W}_1\otimes\check{W}_2 \otimes V_1\otimes V_2,\mathbb{C}) \simeq \Hom_{G_1\times G_2}(\check{W}_1\otimes V_1 \otimes \check{W}_2\otimes V_2, \mathbb{C})  \simeq [V_{1_{\sigma_1}} \otimes V_{2_{\sigma_2}}]^{\ast}$; considering  their smooth parts, we see that $\check{f}$ is an isomorphism, and then $f$ is an isomorphism.
\end{proof}
\begin{lemma}
Let $G_1, \cdots, G_{2n}$ be locally profinite groups. If the representation  $(\pi_i, V_i)$ of $G_i \times G_{n+i}$ is   a theta representation for $1 \leq i \leq n$, then so is the representation $\otimes_{i=1}^n \pi_i$ of the group $(G_1 \times \cdots \times G_n) \times (G_{n+1} \times \cdots \times G_{2n})$.
\end{lemma}
\begin{proof}
By induction, it is sufficient to assume that $n=2$. Suppose that $(\sigma_1 \otimes \cdots \otimes \sigma_4, W_1 \otimes \cdots \otimes W_4) \in \mathcal{R}_{G_1 \times \cdots \times G_4} \big( \pi_1 \otimes \pi_2\big)$.   By the result of Lmm.\ref{thegrquofpro}(2), we have $\Hom_{G_1 \times \cdots \times G_4} \big( \pi_1 \otimes \pi_2, \sigma_1 \otimes \cdots \otimes \sigma_4\big) \simeq \Hom_{\C} \big( (V_1 \otimes \check{W}_1 \otimes \check{W}_3)_{G_1 \times G_3} \otimes (V_2 \otimes \check{W}_2 \otimes \check{W}_4)_{G_2 \times G_4}, \C\big)$.  By assumption, $(V_1 \otimes \check{W}_1 \otimes \check{W}_3)_{G_1 \times G_3}$, and $(V_2 \otimes \check{W}_2 \otimes \check{W}_4)_{G_2 \times G_4}$ both have one dimension, so does their tensor product. Hence $m_{G_1 \times \cdots \times G_4} \big( \pi_1 \otimes \pi_2, \sigma_1 \otimes \cdots \otimes \sigma_4\big)=1$.  Suppose now that $\sigma_1 \otimes \sigma_2 \otimes \sigma_3' \otimes \sigma_4' \in \mathcal{R}_{G_1 \times \cdots \times G_4} (\pi_1 \otimes \pi_2)$. Then $\sigma_1 \otimes \sigma_3' \in \mathcal{R}_{G_1 \times G_3}(\pi_1 \otimes \pi_2)=\mathcal{R}_{G_1 \times G_3}(\pi_1)$, and  it follows that $\sigma_3' \simeq \sigma_3$. Similarly $\sigma_4'\simeq \sigma_4$.
By symmetry, the property of graph holds for $\pi_1 \otimes \pi_2$. Now $(\pi_1 \otimes \pi_2)_{\sigma_1 \otimes \sigma_2} \simeq (\pi_{1_{\sigma_1}}) \otimes (\pi_{2_{\sigma_2}})$ by Lmm.\ref{thegrquofpro}(3);  this  isomorphism keeps the $G_3 \times G_4$-module structure. Hence  the former  representation $(\pi_1 \otimes \pi_2)_{\sigma_1 \otimes \sigma_2}$ of $G_3 \times G_4$ is finitely generated. The similar result also holds for the representation $(\pi_1 \otimes \pi_2)_{\sigma_3 \otimes \sigma_4} $ of $G_1 \times G_2$.This finishes the proof.
\end{proof}

Let $G_1, G_2, H$  be locally profinite groups. Suppose now  that $H$ is an abelian group.  Let $\gamma$ be an automorphism of $H$, and $\pi$ a smooth representation of $G_1 \times G_2 \times H$. Via the homomorphism $(G_1 \times H) \times (G_2 \times H) \longrightarrow G_1\times G_2 \times H$,
$[(g_1, h_1), (g_2,h_2)] \longmapsto (g_1g_2, h_1\gamma(h_2))$,  we obtain a smooth representation $\widetilde{\pi}$ of $(G_1 \times H) \times (G_2 \times H)$.
\begin{lemma}\label{troisgroupesdebigrapheforte}
If  $\pi|_{G_1 \times G_2}$ is a theta representation,  so is  $\widetilde{\pi}$.
\end{lemma}
\begin{proof}
By observation, the multiplicity-free property  also  holds for  $\tilde{\pi}$.  Suppose now  $(\pi_1 \otimes \chi_1) \otimes (\pi_2 \otimes \chi_2) \in \mathcal{R}_{(G_1 \times H) \times (G_2 \times H)}(\widetilde{\pi})$, and let $0 \neq F \in \Hom_{(G_1\times H) \times (G_2 \times H)}(\pi, \, (\pi_1 \otimes \chi_1) \otimes (\pi_2 \otimes \chi_2)).$
By definition, we have
$$F(\pi((g_1\otimes g_2), h\gamma(h'))v)=\pi_1(g_1) \otimes \pi_2(g_2) F(v) \chi_1(h) \chi_2(h'),  \quad v\in V, g_i \in G_i, h,h' \in H.$$
Substituting $g_1=g_2=1$, $h'=\gamma^{-1}(h^{-1})$ shows that
$F(v)= F(v) \chi_1(h) \chi_2\big(\gamma^{-1}(h^{-1})\big)$ for all $v\in V$.
As $F\neq 0$ and  $\gamma$ is an isomorphism, we get $\chi_2=\chi_1^{\gamma^{-1}}$,  where $\chi_1^{\gamma^{-1}}(h):= \chi_1\big(\gamma(h)\big)$, for  $h\in H$. If we write $\theta_{\pi}$ for the theta map of  $\pi|_{G_1 \times G_2}$, then there is  a bijection from
$ \mathcal{R}_{G_1 \times H_1}^0(\widetilde{\pi})$ to $ \mathcal{R}_{G_2 \times H_2}^0(\widetilde{\pi})$, just  given by
$ \pi_1 \otimes \chi_1 \longmapsto \theta_{\pi}(\pi_1) \otimes \chi_1^{\gamma^{-1}}$.  Recall $\widetilde{\pi}_{\pi_1 \otimes \chi} \simeq \frac{V}{\cap_{f\in \Hom_{G_1 \times H}( \widetilde{\pi}, \pi_1 \otimes \chi_1)}\ker f}$, and  $\pi_{\pi_1 } \simeq \frac{V}{\cap_{g\in \Hom_{G_1}( \pi, \pi_1)}\ker g}$.  Hence there exists a \emph{surjective} $G_1 \times G_2$-morphism from  $\pi_{\pi_1}$ to $\widetilde{\pi}_{\pi_1 \otimes \chi_1}$. If we write $\widetilde{\pi}_{\pi_1 \otimes \chi} \simeq (\pi_1 \otimes \chi) \otimes \Theta_{\pi_1 \otimes \chi}$, then  $\Theta_{\pi_1 \otimes \chi}$ is a finitely generated representation of $G_2$ as well as $G_2 \times H$.
\end{proof}
\begin{remark}\label{remarktroisgroupesdebigrapheforte}
The above result also holds for the  theta representation of finite length.
\end{remark}
\begin{proof}
We follow the notations. It suffices to show that the greatest $\pi_1 \otimes \chi_1$-isotypic quotient  space $\widetilde{\pi}_{\pi_1 \otimes \chi_1}$ is a representation of $G_2 \times H$ of finite length. Let us consider the $G_2 \times H$-smooth part of $\Hom_{G_1 \times H}\big( \widetilde{\pi}_{\pi_1 \otimes \chi_1}, \pi_1 \otimes \chi_1\big)$.
Recall that $\Hom_{G_1 \times H}\big( \widetilde{\pi}_{\pi_1 \otimes \chi_1}, \pi_1 \otimes \chi_1\big)\simeq \Hom_{G_1 \times H}\big( \widetilde{\pi}, \pi_1 \otimes \chi_1\big)$, and it follows that  $H$ acts canonically on the latter $\Hom$-space via $\chi_1^{-1}\circ \gamma$. Therefore it suffices to extract the $G_2$-smooth part of $\Hom_{G_1 \times H}\big( \widetilde{\pi}, \pi_1 \otimes \chi_1\big)$. Now $\Hom_{G_1 \times H}\big( \widetilde{\pi}, \pi_1 \otimes \chi_1\big) \simeq \Hom_{G_1}\big( (\widetilde{\pi} \otimes \chi_1^{-1})_{H}, \pi_1\big)\hookrightarrow  \Hom_{G_1}\big( \pi, \pi_1\big)$, and this process keeps the $G_2$-module structure. Hence  the representation $(\widetilde{\pi}_{\pi_1 \otimes\chi_1})^{\vee}$ of $G_2 \times H$  has finite length,  so does  $\widetilde{\pi}_{\pi_1 \otimes \chi_1}$ itself. By symmetry, the similar result is still valid for $\widetilde{\pi}_{\pi_2 \otimes \chi_1\circ \gamma}$.
\end{proof}
\begin{remark}\label{ouvertmorphisme}
If the above map
$(G_1 \times H) \times (G_2 \times H) \longrightarrow G_1\times G_2 \times H$
 factors through
 $(G_1 \times H) \times (G_2 \times H) \longrightarrow G_1H \times G_2H,$
 for open surjective homomorphisms $p_i : G_i \times H \longrightarrow
 G_i H$, then the result in  Lemmas \ref{troisgroupesdebigrapheforte} also holds for the analogous representation of   $ G_1H \times G_2H$.
\end{remark}
\begin{proof}
This follows from the fact that each irreducible representation of $G_i H$ can be identified with an irreducible representation of  $G_i \times H $ trivially at $ker(p_i)$.
\end{proof}

\section{The Clifford-Mackey theory}\label{theCliffordMacheytheory}
In this section, we study  Clifford-Mackey theory in our case. We will let $G$ be a locally profinite group, and let $H$ be a closed  subgroup  of $G$.  Suppose that \emph{all} irreducible representations of $G$, $H$ are \emph{admissible}.
 \subsection{} In the first subsection we assume that $H$ is an open normal subgroup of $G$,  $G/H$ is an abelian discrete group.
\begin{theorem}[Clifford-Mackey]\label{cliffordadmissible}
Let $(\pi, V) \in \Irr(G)$. Suppose  $\mathcal{R}_H(\pi) \neq \emptyset$. Then:
\begin{itemize}
\item[(1)] $\Res_{H}^G\pi$ is  a semi-simple representation with  finite   multiplicities.
\item[(2)] If $\sigma_1, \sigma_2 \in \mathcal{R}_H(\pi)$, then there is an element $g\in G$ such that $\sigma_2 \simeq \sigma_1^g$, where $\sigma_1^g (h): =\sigma_1(ghg^{-1})$ for $h\in H$.
\item[(3)] There is a positive integer $m$ such that $\Res_H^G\pi\simeq \sum_{\sigma\in \mathcal{R}_H(\pi)} m \sigma$.
\item[(4)] Let $(\sigma, W) $ be an irreducible  constituent  of $\Res_H^G\pi$. Then:
 \begin{itemize}
\item[(a)] $I_G^0(\sigma)=\left\{ g\in G \mid g(W)=W\right\}$ is an open normal subgroup of $G$. For two irreducible constituents  $(\sigma_1, W_1)$, $(\sigma_2, W_2)$ of $(\Res_H^G \pi, V)$, we have $I_G^0(\sigma_1)= I_G^0(\sigma_2)$, denoted by $\widetilde{H}^0$.   Moveover, $\sigma$ is extendible to $\widetilde{H}^0$.
\item[(b)] $I_G(\sigma)=\left\{ g\in G \mid \sigma^g \simeq \sigma\right\}$ is an open normal subgroup of $G$. For any $\sigma_1$, $\sigma_2 \in \mathcal{R}_H(\pi)$, we have $I_G(\sigma_1)= I_G(\sigma_2)$, denoted by $\widetilde{H}$.
\item[(c)] The isotypic component $m\sigma$ of $\sigma$ in $\Res_H^G \pi$ is an irreducible smooth representation of $\widetilde{H}$, denoted by $(\widetilde{\sigma}, \widetilde{W})$.
\end{itemize}
\item[(5)]
 $\Res_{\widetilde{H}}^G\pi\simeq  \oplus_{\widetilde{\sigma} \in \mathcal{R}_{\widetilde{H}}(\pi)} \widetilde{\sigma}$ with $\widetilde{\sigma}|_H \simeq m \sigma$. The action of   $G/{\widetilde{H}}$ on the set $\mathcal{R}_{\widetilde{H}}(\pi)$ is simply  transitive.
\item[(6)] $\pi\simeq \cInd_{\widetilde{H}}^G \widetilde{\sigma}$ for any $\widetilde{\sigma} \in \mathcal{R}_{\widetilde{H}}(\pi)$.
\item[(7)] $\cInd_{\widetilde{H}}^G \widetilde{\sigma}\simeq \Ind_{\widetilde{H}}^G \widetilde{\sigma}$.
\end{itemize}
\end{theorem}
\begin{proof}
1) Let $(\sigma, W) \in \mathcal{R}_H(\pi)$,  $\Omega=\left\{ g_i  \in G\right\}$ a complete set of  coset representatives  of $G/H$. By the contragredient  duality,  $(\check{\sigma}, \check{W})$ is a sub-representation of $(\Res_H^G \check{\pi}, \check{V})$(\emph{cf}. Lmm.\ref{admissiblerepresentations}). The vector space  $\sum_{g_i\in \Omega} \check{\pi}(g_i) \check{W}$ is $G$-invariant, equalling to $\check{V}$. Thus $\Res_{H}^G \check{\pi}$ is semi-simple and contains an irreducible factor representation. It follows that $(\Res_H^G\pi, V)$ is semi-simple as well (lemma \ref{admissiblerepresentations}). Let   $K$  be an open compact subgroup of $H$ such that   the  finite-dimensional vector space $W^K$ is nonzero. By Frobenius reciprocity, we have the relation of dimensions: $m_H(\pi, \sigma) \leq m_H(\pi, \Ind_K^HW^K) \leq m_K(\pi, W^K) < +\infty.$\\
2) Every irreducible sub-representation of $(\Res_H^G \pi, V)$ is isomorphic with $(\Res_H^G \pi, \pi(g_i) W)$ for some $g_i\in \Omega$, and  $(\Res_H^G \pi, \pi(g_i)(W)) \simeq  (\sigma^{g_i^{-1}}, W)$, so the part (2) is clear.\\
3) Let $\sigma_1, \sigma_2$ be two elements in $\mathcal{R}_H(\pi)$. Then there is an element $g\in G$ such that $\sigma_2 \simeq \sigma_1^g$, and $m_H( \pi, \sigma_1)=m_H(\pi^g, \sigma_1^g)=m_H(\pi, \sigma_2)=m$, for some positive integer $m$.\\
4)  The group  $I_G^0(\sigma)$ containing  $H$ is  an open normal subgroup of $G$. For $(\sigma_1,W_1), (\sigma_2, W_2) \in \mathcal{R}_H(\pi)$, there exists $g \in G$ such that $W_1=g(W_2)$.  Then the map from $I_{G}^0(\sigma_1)$ to $I_{G}^0(\sigma_2)$, defined by $h \longrightarrow g^{-1}hg$, is bijective.   It follows that the two normal subgroups  $I_{G}^0(\sigma_1)$ and $I_{G}^0(\sigma_2)$ of $G$ coincide.  The similar proof  works for (b).  By observation, the $\sigma$-isotypic component $\widetilde{\sigma}$($\simeq m\sigma$) is an irreducible representation of  $\widetilde{H}$.\\
5)   Applying the result (1) to $\widetilde{H}$ shows that  $\Res_{\widetilde{H}}^G\pi$ is semi-simple. This will yield a decomposition
$\Res_{\widetilde{H}}^G\pi\simeq \oplus_{\sigma\in \mathcal{R}_{H}(\pi)} \widetilde{\sigma},$   where $ \widetilde{\sigma}|_H \simeq m \sigma$.  Namely, $\widetilde{\sigma_1} \ncong \widetilde{\sigma_2}$ if $\sigma_1 \ncong \sigma_2 \in \mathcal{R}_H(\pi)$.  For $\widetilde{\sigma_1}, \widetilde{\sigma_2} \in \mathcal{R}_{\widetilde{H}}(\pi)$,  we can find  $g \in G$ such that $\widetilde{\sigma_1} \simeq \widetilde{\sigma_2}^g$. On the other hand,  if $\widetilde{\sigma} \simeq \widetilde{\sigma}^g$, a priori $\sigma\simeq \sigma^g$  so that $g\in \widetilde{H}$. In this way we verify  that the action of $G/{\widetilde{H}}$ on $\mathcal{R}_{\widetilde{H}}(\pi)$ is  simply transitive.\\
6)   Let $\Lambda=\left\{ g_i \right\}_{i\in I}$ be a set of  representatives for $G/{\widetilde{H}}$. Then  $\widetilde{V}=\sum_{g_i\in \Lambda} \pi(g_i) \widetilde{W}$ is $G$-invariant, and $\widetilde{V}=V$. By Frobenius reciprocity, we have $\alpha: \Hom_{\widetilde{H}} (\widetilde{\sigma}, \pi) \stackrel{ \sim }{\longrightarrow } \Hom_G(\cInd_{\widetilde{H}}^G \widetilde{\sigma}, \pi)$, which is of dimension $1$. By the explicit construction in \cite[p.24]{BernZ}, the map $\alpha(\Id_{\widetilde{\sigma}})$ shall give a $G$-isomorphism from $\cInd_{\widetilde{H}}^G(\widetilde{\sigma})$ to $\pi$.\\
7) Under the admissible assumption, $\check{\pi}$ is also an irreducible representation of $G$. Hence $\check{\pi} \simeq  \Ind_{\widetilde{H}}^G \check{\widetilde{\sigma}}= \cInd_{\widetilde{H}}^G \check{\widetilde{\sigma}}$, and $\check{\check{\pi}} =  (\cInd_{\widetilde{H}}^G \check{\widetilde{\sigma}})^{\vee} \simeq \Ind_{\widetilde{H}}^{G}  \widetilde{\sigma}$ for the reason that  $\widetilde{\sigma}$ is an admissible representation of $H$ as well as $\widetilde{H}$.
\end{proof}
\begin{corollary}\label{semisimple}
Keep the above notations. Suppose now that $H_1$ is a closed  subgroup of $\widetilde{H}$ and $H_1 \supseteq H$ . Then $\mathcal{R}_{H_1}(\pi) \neq \emptyset$ and $\Res_{H_1}^G \pi$ is semi-simple as well.
\end{corollary}
\begin{proof}
Let $(\sigma, W)$ be an irreducible constituent of $(\Res_H^G \pi, V)$. The action of  $H_1$  on $W$ produces a finitely generated representation of $H_1$, denoted by $(\sigma_1, W_1)$. This representation  admits an exact sequence of $\mathcal{H}(H_1)$-modules:
$1 \longrightarrow U_1 \longrightarrow W_1 \longrightarrow U \longrightarrow 1,$
for  an irreducible quotient representation $(\rho, U)$ of $H_1$ and  a sub-representation $(\rho_1, U_1)$ of $H_1$. As we know, $\Res_H^{H_1} \sigma_1$ ($\subseteq  m\sigma$) is semi-simple. It follows that $\Res_{H}^{H_1} \rho_1 \simeq m_1 \sigma$ for certain $m_1$ smaller than $ m$.  Note that  $U_1$ is also a finitely generated $\mathcal{H}(H_1)$-module.($U_1|_H \simeq m_2 \sigma$) By induction on  $m$, finally we can find an irreducible sub-representation of $(\sigma_1, W_1)$ or $(\Res_{H_1}^G\pi, V)$. The proving process of the theorem \ref{cliffordadmissible} (1) shall give the result.
\end{proof}

\begin{corollary}\label{IndCidn}
Under the conditions of Theorem \ref{cliffordadmissible}, let $\chi \in \Irr(G/H)$; then $\cInd_{\widetilde{H}}^G (\widetilde{\sigma} \otimes  \chi|_{\widetilde{H}}) \simeq \big( \cInd_{\widetilde{H}}^G \widetilde{\sigma} \big) \otimes \chi$, and $\cInd_{\widetilde{H}}^G (\widetilde{\sigma} \otimes  \chi|_{\widetilde{H}}) = \Ind_{\widetilde{H}}^G ( \widetilde{\sigma} \otimes \chi|_{\widetilde{H}})$.
\end{corollary}
\begin{proof}
Let $\widetilde{\Delta}=\left\{ g_i \in G\right\}_{i\in I}$, assumed  to  contain $1$,   be a complete set of coset representatives of  $G/{\widetilde{H}}$. By Frobenius reciprocity,  we have  $\alpha: \Hom_G\big( \cInd_{\widetilde{H}}^G (\widetilde{\sigma} \otimes  \chi|_{\widetilde{H}}), \big( \cInd_{\widetilde{H}}^G \widetilde{\sigma} \big) \otimes \chi\big) \simeq \Hom_{\widetilde{H}}\big( \widetilde{\sigma} \otimes \chi|_{\widetilde{H}},\sum_{g\in \widetilde{\Delta}} \widetilde{\sigma}^{g} \otimes \chi|_{\widetilde{H}}\big)$. Then $\alpha^{-1}(\Id_{\widetilde{\sigma} \otimes \chi|_{\widetilde{H}}})$ shall give a $G$-morphism from  $\cInd_{\widetilde{H}}^G (\widetilde{\sigma} \otimes  \chi|_{\widetilde{H}})$  to $\big(  \cInd_{\widetilde{H}}^G \widetilde{\sigma} \big) \otimes \chi$.  By   investigating their restrictions to $\widetilde{H}$, we see that the morphism  is bijective.
The second assertion follows from  Theorem \ref{cliffordadmissible} (7)  by replacing $ \widetilde{\sigma}$ with $ \widetilde{\sigma} \otimes \chi|_{\widetilde{H}}$.
\end{proof}
\begin{corollary}\label{lastsubgroup}
Under the conditions of Theorem \ref{cliffordadmissible},  there exists a normal subgroup $H_m$ of $G$ such that
\begin{itemize}
\item[(1)] $H \subseteq H_m \subseteq \widetilde{H}$,
\item[(2)] $H_m/H$ is finitely generated,
\item[(3)] $\Res_{H_m}^G\pi$ is multiplicity-free.
\end{itemize}
\end{corollary}
\begin{proof}
Suppose $\widetilde{\sigma}|_H=\pi(g_1)( W)   \oplus \cdots \oplus \pi(g_m)(W)$ for some $g_1, \cdots, g_m \in G$. We let $H_m$ be the subgroup of $G$ generated by $H$ and these  $g_1, g_2, \cdots, g_m$.  Clearly $H_m/H$ is finitely generated.  By definition, $\widetilde{\sigma}$ is an irreducible $\C[H_m]$-module, which forces $\Res_{H_m}^G\pi$ to be multiplicity-free.
\end{proof}
\begin{proposition}\label{twograph}
For  $(\pi_1, V_1)$, $(\pi_2, V_2) \in \Irr(G)$,  we have:
\begin{itemize}
\item[(1)] $\mathcal{R}_H(\pi_1) \cap \mathcal{R}_H(\pi_2) \neq \emptyset$ only if $\mathcal{R}_H(\pi_1)= \mathcal{R}_H(\pi_2) \neq \emptyset$.
\item[(2)] If $\mathcal{R}_H(\pi_1)=\mathcal{R}_H(\pi_2) \neq \emptyset$, then $\pi_1 \simeq \pi_2 \otimes \chi_{G/H}$ for some character $\chi_{G/H} $ of $G/H$.
\end{itemize}
\end{proposition}
\begin{proof}
1) By symmetry,  we only check one-side inclusion. Let $(\sigma,W)\in \mathcal{R}_H(\Res_H^G\pi_1)\cap \mathcal{R}_H(\Res_H^G\pi_2)$. For $\sigma' \in \mathcal{R}_H(\pi_1)$, by Theorem \ref{cliffordadmissible} there exists $g\in G$ such that $\sigma^g \simeq \sigma'$. Hence $m_H(\pi_2, \sigma')=m_H(\pi_2, \sigma^g)=m_H( \pi_2^g, \sigma^g)=m_H(\pi_2, \sigma)$; this implies that $\sigma' \in \mathcal{R}_H(\pi_2)$, so  $\mathcal{R}_H(\pi_1) \subseteq \mathcal{R}_H(\pi_2)$.

2) For simplicity, we identity $(\sigma, W)$ as an irreducible constituent of $(\Res_H^G \pi_1, V_1)$ as well as $(\Res_H^G \pi_2, V_2)$.   Let $\widetilde{H}$ be the open normal subgroup of $G$ defined as in Theorem \ref{cliffordadmissible}(4)  for the above $\sigma$. Let  $(\widetilde{\sigma_1}, \widetilde{V_1})$, $(\widetilde{\sigma_2}, \widetilde{V_2})$ be the $\sigma$-isotrypic components  of $\Res_{\widetilde{H}}^G\pi_1$ and  $\Res_{\widetilde{H}}^G \pi_2$ respectively.  On  $\Hom_H(\widetilde{\sigma_1}, \widetilde{\sigma_2})$, we  impose a natural $\widetilde{H}/H$-action defined as follows:
    $[\overline{g} \varphi](v_1)= \varphi^{\overline{g}}(v_1):= \widetilde{\sigma_2}(g)\varphi(\widetilde{\sigma_1}(g^{-1})v_1)$,  for $ \varphi \in \Hom_H(\widetilde{\sigma_1}, \widetilde{\sigma_2})$, $\overline{g} \in \widetilde{H}/H$, $v_1\in \widetilde{V_1}$. Here, $g\in \widetilde{H}$ is a representative of $\overline{g} $.
By Theorem \ref{cliffordadmissible}, we have
 $\Res_H^{\widetilde{H}}\widetilde{\sigma_1}= \oplus_{i=1}^{m_1} \pi_1(g_i) W$, for some suitable  $g_1=1, g_2, \cdots, g_{m_1}$ in $\widetilde{H}$,
  so that we can construct an  element  $f\in \Hom_H(\widetilde{\sigma_1}, \widetilde{\sigma_2})$ by $f|_{\pi_1(g_i)(W)}(\pi_1(g_i)w)=\pi_2(g_i)w$, for $w\in W$. Write $\mathcal{F}=\{ \sum_{i} c_i f^{\overline{g_i}} \mid \overline{g_i} \in \widetilde{H}/H, c_i \in \C \},$
an $\widetilde{H}/H$-module of finite dimension. Let us show that  $\mathcal{F}$ is actually a \emph{smooth} representation of $\widetilde{H}/H$. Fix $0 \neq w_0 \in W$ and  let $ K =\cap_{i=1}^{m_1}\big( \Stab_{\widetilde{H}}(\pi_1(g_i) w_0) \cap \Stab_{\widetilde{H}}(\pi_2(g_i) w_0)\big)$.  For $k\in K$ we denote its image in $\widetilde{H}/H$ by $\overline{k}$. Then for $w= \sum_{j=1}^n c_j\pi_1(h_j) w_0 \in W,$ we have
$$f^{\overline{k}}(\pi_1(g_i)w)=\sum_{j=1}^n c_j\pi_2(g_ih_jg_i^{-1}) f^{\overline{k}}\big( \pi_1(g_i) w_0\big)= \sum_{j=1}^n c_j \pi_2(g_ih_jg_i^{-1}) f\big(\pi_1(g_i)  w_0\big)= f(\pi_1(g_i)w).$$
 Hence $\Stab_{\widetilde{H}/H}(f) \supseteq  K$  is an open subgroup of $\widetilde{H}/H$. Similarly,  $\Stab_{\widetilde{H}/H}(f^{\overline{g}}) \supseteq g^{-1}Kg$ is also open for $\overline{g} \in \widetilde{H}/H$.  So   $\mathcal{F}$  is  smooth  and contains a sub-representation $(\chi_{\widetilde{H}/H}, U)$ of $\widetilde{H}/H$.  Any nonzero element $ F\in U$  lies inside  $ \Hom_{\widetilde{H}}(\chi_{\widetilde{H}/H}\otimes\widetilde{\sigma_1}, \widetilde{\sigma_2})$, so we  conclude that $\widetilde{\sigma_2} \simeq \widetilde{\sigma_1} \otimes \chi_{\widetilde{H}/H}$. Now the character $\chi_{\widetilde{H}/H}\in \Hom(\widetilde{H}/H, \C^{\times})$ can extend  to a  continuous homomorphism $\chi_{G/H}$ from $G/H$ to $\C^{\times}$, since $\C^{\times}$ is a divisible  group and $\widetilde{H}$ is open.   By replacing $\pi_1$ with $\pi_1 \otimes \chi_{G/H}$, we may assume $\mathcal{R}_{\widetilde{H}}(\pi_1) \cap \mathcal{R}_{\widetilde{H}}(\pi_2) \neq \emptyset$ and  the above $\chi_{\widetilde{H}/H}$ is trivial. The result then follows from Theorem \ref{cliffordadmissible} (6).
\end{proof}

\begin{proposition}\label{dimesionsmall1}
Let  $(\pi, V)$ be a smooth  representation of $G$ with  finite   multiplicity. Let $(\pi_1, V_1) \in \mathcal{R}_G(\pi)$ such that $\mathcal{R}_{H}(\pi_1) \neq \emptyset$.
\begin{itemize}
\item[(1)] $\mathcal{R}_H(\pi_1) \subseteq \mathcal{R}_H(\pi)$.
\item[(2)] $m_H( \pi, \sigma_1)=m_H(\pi, \sigma_2)$ for $\sigma_1, \sigma_2 \in \mathcal{R}_H(\pi_1)$.
\item[(3)] If $m_H( \pi, \sigma) \leq 1$ for all $\sigma\in \mathcal{R}_H(\pi_1)$, then $m_G(\pi, \pi_1) \leq 1$.
\end{itemize}
\end{proposition}
\begin{proof}
(1) is obvious and (2) follows from Theorem \ref{cliffordadmissible} (2). For (3) we take the subgroup $\widetilde{H}$ of $G$ for the representation $\pi_1$ as defined  in Theorem \ref{cliffordadmissible}(4). Then $\Res_{\widetilde{H}}^G\pi_1\simeq \oplus_{\widetilde{\sigma} \in \mathcal{R}_{\widetilde{H}}(\pi_1)} \widetilde{\sigma}$, where $\widetilde{\sigma}|_{H}= m \sigma$ for some  $\sigma \in \mathcal{R}_H(\pi_1)$. We first  show that $m_{\widetilde{H}}(\pi, \widetilde{\sigma}) \leq 1$. If $f, g \in \Hom_{\widetilde{H}}(\pi, \widetilde{\sigma})$, and $0 \neq p \in \Hom_H(\widetilde{\sigma}, \sigma)$, then $p \circ f$, $p \circ g \in \Hom_H(\pi, \sigma)$. This means $p \circ f$ is proportional to $ p \circ g$, in other words, $ p\circ g=c p\circ f$ for some $c\in \C^{\times}$.  The map  $g-cf \in \Hom_{\widetilde{H}}(\pi, \widetilde{\sigma})$  is either surjective or zero; as $p\circ (g-cf)=0$, it has to be zero. Hence $m_{\widetilde{H}}(\pi, \widetilde{\sigma}) \leq 1$. As before, the set $\Hom_{\widetilde{H}}(\pi, \pi_1)$ is a $G/{\widetilde{H}}$-module. By the decomposition of $\Res_{\widetilde{H}}^G\pi_1$, we have  $\Hom_{\widetilde{H}}(\pi, \pi_1) \hookrightarrow \prod_{\widetilde{\sigma} \in \mathcal{R}_{\widetilde{H}}(\pi_1)} \Hom_{\widetilde{H}}(\pi, \widetilde{\sigma})$. We denote the canonical map from $\Hom_{\widetilde{H}}(\pi, \pi_1)$ to $\Hom_{\widetilde{H}}(\pi, \widetilde{\sigma})$ by $p_{\widetilde{\sigma}}$.  Each $F \in \Hom_{\widetilde{H}}(\pi, \pi_1)$ is determined uniquely by the family $\{ p_{\widetilde{\sigma}} \circ F\}_{\widetilde{\sigma} \in \mathcal{R}_{\widetilde{H}}(\pi_1)}$ and $G/{\widetilde{H}}$ acts transitively on  $\{ p_{\widetilde{\sigma}} \circ F\}_{\widetilde{\sigma} \in \mathcal{R}_{\widetilde{H}}(\pi_1)}$.   Since $\Hom_G(\pi, \pi_1) \simeq \Hom_{\widetilde{H}}(\pi, \pi_1)^{G/{\widetilde{H}}}$,  finally   $\dim\Hom_G(\pi, \pi_1)= \dim \Hom_{\widetilde{H}}(\pi, \pi_1)^{G/{\widetilde{H}}} = \dim\Hom_{\widetilde{H}}(\pi, \widetilde{\sigma}) \leq 1$ as required.
\end{proof}

\begin{lemma}\label{cycliccase}
Under the situation of Theorem \ref{cliffordadmissible}, if  $G/H$ is a cyclic group, then $\Res_H^G\pi$ is multiplicity-free.
\end{lemma}
\begin{proof}
 Keep the notations in the theorem \ref{cliffordadmissible}. By hypothesis, the subgroup  $\widetilde{H}/H$ is also cyclic generated by one element $ \overline{s}$ with a  representative  $s$ in $\widetilde{H}$. Since $\sigma^s \simeq \sigma$, there exists a $\C$-linear map $A: W \longrightarrow W$ such that $\sigma^s(h) A=A \sigma(h) \textrm{ for all } h \in H$.\footnote{In case $\# \widetilde{H}/H =n < \infty$, and $s^n=h_0\in H$,  we have  $A^n \sigma(h)=\sigma^{s^n}(h) A^n=\sigma(h_0)\sigma(h) \sigma(h_0^{-1})A^n$, for all $h\in H$. By Schur's Lemma, $\sigma(h_0)=cA^n$, for certain $c\in \mathbb{C}^{\times}$. Hence, we can replace the above $A$ so that the constant number $c=1$.}Then there is a \emph{well-defined} $\widetilde{H}$-homomorphism $\widetilde{\sigma}': \widetilde{H}\longrightarrow \Aut(W); s^i h \longmapsto A^{i} \sigma(h)$.
 In fact, $\widetilde{\sigma}'$ is an irreducible smooth representation of  $\widetilde{H}$ because $H$ is open.  Consequently $\widetilde{\sigma'}|_{H} \simeq \sigma$. By Prop. \ref{twograph}, we get $\widetilde{\sigma} \simeq \widetilde{\sigma}' \otimes \chi_{\widetilde{H} /H} $ for some character $\chi_{\widetilde{H}/H}$ of $\widetilde{H}/H$, so it forces  $m=1$.
\end{proof}

 \begin{lemma}\label{towerofnormalsubgroups}
Under the situation of Theorem \ref{cliffordadmissible},  there exists a tower of normal subgroups of $G$:  $H=H_0 \lhd H_1 \lhd \cdots \lhd  H_n \lhd  H_{n+1}= G$,  such  that
\begin{itemize}
\item[(1)] $H_{i+1}/H_i$ is a cyclic group, for $i=0, \cdots, n-1$,
\item[(2)] $\mathcal{R}_{H_i}(\pi) \neq \emptyset$, for $i=0, \cdots, n$,
\item[(2)] for each $i$ and $\sigma_{i+1} \in \mathcal{R}_{H_{i+1}}(\pi)$,  $\Res_{H_i}^{H_{i+1}}\sigma_{i+1}$ is multiplicity-free.
\end{itemize}
 \end{lemma}
 \begin{proof}
 We can take $H_{n}$ to be the group $H_m$ as defined in  Cor.\ref{lastsubgroup};  by the part (2) there, $H_n/H_0$  is an abelian group generated by $m$ elements, so it is isomorphic to a  direct sum of cyclic groups $F_1 \oplus F_2 \oplus  \cdots \oplus F_n$. By Lmm.\ref{cycliccase}, we only need to let $H_i$ be the inverse image of $F_1 \oplus \cdots \oplus F_{i}$ in $G$. Then these $H_i$ satisfy the desired conditions.
\end{proof}

\subsection{}
In this second subsection, we assume that $H$ is a closed normal subgroup  of $G$ with  cocompact quotient. The main regular  results of this subsection  have already obtained by  Silberger in \cite{Sil} or by  Henniart in \cite{Hen}, but for completeness we reproduce them again.    We fix an element  $(\pi, V) \in \Irr(G)$.  Assume the category $\Rep(H)$ is locally noetherian.(cf. \cite[\S 4]{Bern0} )
\begin{lemma}\label{thequotient1}
$\mathcal{R}_H(\pi) \neq \emptyset$, and $m_H(\pi, \sigma)< +\infty$, for $(\sigma, W) \in  \mathcal{R}_H(\pi)$.
\end{lemma}
\begin{proof}
See Prop.\ref{finitegenerated}(2) and Lmm.\ref{typeimpliquequotientadmissible}.
\end{proof}
\begin{lemma}\label{comiso}
For $(\sigma, W)\in \mathcal{R}_H(\pi)$, there exists an open compact group $K$ of $G$ such that $\sigma^k \simeq \sigma$, for  $k\in K$.
\end{lemma}
\begin{proof}
Let $ f: V\longrightarrow W$  be a non-zero $H$-morphism.   Assume that  $\ker(f)$  is generated by   vectors $v_1, \cdots, v_m$  as an  $H$-module. Let  $K$  be an open compact subgroup of $G$ such that $K \subseteq \cap_{i=1}^m \Stab_G(v_i)$.     For  any  $k\in K$,  $v=\sum_{i=1}^m c_i h_i v_i\in \ker(f)$ with $c_i\in \mathbb{C}, h_i \in H$,  we have
$kv=\sum_{i=1}^m c_i kh_i v_i=\sum_{i=1}^m c_i  kh_i k^{-1}v_i$.  Hence $kv\in \ker(f)$.  So there is a canonical $\mathbb{C}$-linear map $\pi(k): V/\ker(f) \longrightarrow V/\ker(f)$, and $\pi(k) \pi(h)=\pi(khk^{-1}) \pi(k)$, for $h\in H$. Hence $\sigma^k\simeq \sigma$.
\end{proof}
\begin{lemma}\label{ssmcheck}
$(\Res_H^G \check{\pi}, \check{V})$ is a semi-simple representation with finite  multiplicity.
\end{lemma}
\begin{proof}
Assume $(\sigma, W) \in \mathcal{R}_H(\pi)$, and let $f: V \longrightarrow W $ be  a non-zero $H$-morphism.    Given the open compact subgroup $K$  of $G$ in the proof of lemma \ref{comiso}, we let $W'$ be the  $K$-complement of $\ker(f)$ in $V$. Then $f: W' \longrightarrow W$ is a bijective  $K\cap H$-morphism.  Applying the contragredient duality to $f$,   we get an $H$-embedding  $\check{f}: \check{\sigma} \hookrightarrow (\Res_H^G \pi)^{\vee}$.  Given another  open compact subgroup $K_1 \subseteq K$, we have     $$\check{f}: \check{\sigma}^{K_1\cap H} \simeq ( \sigma^{K_1\cap H})^{\ast} \hookrightarrow  [ (\Res_H^G \pi)^{\vee}]^{K_1 \cap H}\simeq ( \pi^{K_1\cap H})^{\ast},$$
which stems from  $$f: \pi^{K_1\cap H}=[\ker(f) \oplus W']^{K_1\cap H}\simeq \ker(f)^{K_1\cap H}\oplus W'^{K_1\cap H} \longrightarrow W^{K_1\cap H}.$$
Here $K_1\cap H$ is a normal subgroup of $K_1$, and $\ker(f)^{K_1\cap H}$, $W'[K_1\cap H]$ both are  $K_1$-stable.  Let   $\{w_1', \cdots, w_n'\}$  be a basis of $W'^{K_1\cap H}$ . Then the image of
$\check{\sigma}^{K_1\cap H}$  in $(\Res_H^G \pi)^{\vee}$ is $ \cap_{i=1}^n\Stab_{G}(w'_i)\cap  K_1$-stable, and it  lies in $\Res_H^G\check{ \pi}$.   Therefore   $\check{f}: \check{\sigma} \hookrightarrow (\Res_H^G \pi)^{\vee}$
factors through $\Res_H^G \check{\pi} \hookrightarrow (\Res_H^G \pi)^{\vee}$.  So we can identify $ (\check{\sigma}, \check{W})$ as an irreducible constituent  of  $(\Res_H^G \check{\pi}, \check{V})$. Let  $\Delta=\{ g\in G\}$ be a coset representatives of $G/H$. Then  $\sum_{g\in \Delta} \check{\pi}(g) \check{W}$ is also $G$-invariant, and coincides with $\check{V}$. Moreover $m_{H}(\check{\pi}, \check{\sigma})< +\infty$ by Lmm.\ref{thequotient1}.
\end{proof}
\begin{lemma}\label{comind}
For $(\sigma, W)\in \Irr(H)$, $(\pi, V)\in \Irr(G)$,  $\Hom_G(\cInd_H^G \sigma, \pi)\simeq \Hom_H(\sigma, \pi)$.
\end{lemma}
\begin{proof}
By Frobenius reciprocity, $\Hom_G(\cInd_H^G \sigma, \pi)\simeq \Hom_H\big(\sigma, (\Res_H^G \check{\pi})^{\vee}\big)$. By the above proof,  any $f\in \Hom_H\big(\sigma, (\Res_H^G \check{\pi})^{\vee}\big)$ has to factor through $\Res_{H}^G \pi \hookrightarrow (\Res_H^G \check{\pi})^{\vee}$.
\end{proof}
 \begin{question}
 If   $\Rep(H)$ is not assumed  to be locally noetherian, what the proper condition needs    to add, so that  the similar   result also  holds ?
\end{question}

\begin{remark}\label{notcompactfrb}
If $H$ is not assumed to be a normal subgroup of $G$, but   for any open compact subgroup $K_H$ of $H$, assume that there exists a finite number of elements  $x_1, \cdots, x_n \in H$  such that $\mathcal{H}(H, K_H)=\epsilon_{K_H}  \ast \mathcal{H}(H) \ast \epsilon_{K_H} $ is  an algebra which can be  generated by $\epsilon_{K_H}$, $\epsilon_{x_1}, \cdots, \epsilon_{x_n}$, \footnote{When  $H$ is a $p$-adic reductive group, the condition is satisfied. (cf. \cite[p.27, Corollaire 3.4]{BernD})}  then $\Hom_G(\cInd_H^G \sigma, \pi)\simeq \Hom_H(\delta_{H\setminus G}^{-1} \otimes \sigma, \pi)$, for $ (\sigma, W)\in \Irr(H)$, $(\pi, V)\in \Irr(G)$.
\end{remark}
\begin{proof}

By Frobenius reciprocity, $ \Hom_G\big(\cInd_H^G \sigma, \pi\big)\simeq \Hom_H\big(\delta_{H\setminus G}^{-1} \otimes \sigma, (\Res_H^G \check{\pi})^{\vee}\big) \simeq
\Hom_{H}\big( \check{\pi}, \delta_{H\setminus G}\otimes  \check{\sigma}\big) $.  Let    $0\neq f \in \Hom_H\big(\delta_{H\setminus G}^{-1} \otimes \sigma, (\Res_H^G \check{\pi})^{\vee}\big)$, and the corresponding $\check{f}\in \Hom_{H}\big( \check{\pi}, \delta_{H\setminus G}\otimes  \check{\sigma}\big)$.   Assume $\ker(\check{f})$ is generated by $\check{v}_1, \cdots, \check{v}_n$ as an $H$-module.  Let $K$ be an open compact subgroup of $\cap_{i=1}^n \Stab_G (\check{v}_i)$.

 Consider $K_H=K\cap H$. For simplicity, assume $1\in \{ x_1, \cdots, x_n\}$. Consider the continuous map $\eta: G \times H \longrightarrow G\times H; (g,h) \longrightarrow (g, ghg^{-1})$. Then $X_i=\eta^{-1}( G \times [x_i (K\cap H)]^c)\cap [G \times x_i (K\cap H) ]$ is a closed  subset of $G \times x_i (K\cap H)$, where $[x_i (K\cap H)]^c$ denotes the complement of $x_i (K\cap H)$ in $H$.  Let $p_1: G \times  x_i (K\cap H) \longrightarrow G$ be the canonical projection. By  the tube lemma in topology, $p_1(X_i) $ is a closed subset of $G$. We let $U_i=G\setminus p_1(X_i)$; it contains   $1_G$, and for any $t\in U_i$, $t x_i (K\cap H)\subseteq  x_i (K\cap H)t$, in particular for $x_i=1$, $t\in U_i$, $t (K\cap H)\subseteq (K\cap H)t$. Let $K_0$ an open compact subgroup  of $\cap_{i=1}^n U_i \cap K \subseteq G$.   For  $k\in K_0$, $t\in K\cap H$, and  any open compact subgroup $T\subseteq K\cap H\subseteq H$, we have (1) $\epsilon_{k}\ast \epsilon_{x_i}=\epsilon_{x_i} \ast \epsilon_{h} \ast \epsilon_{ k} $, for some $h\in K\cap H$, (2) $\epsilon_{k} \ast \epsilon_{K\cap H} = \epsilon_{K\cap H} \ast \epsilon_{k}$, (3) $\epsilon_{k}\ast \epsilon_{t}=\epsilon_{ktk^{-1}} \ast \epsilon_{k}$, (4)  $\epsilon_{k}\ast \epsilon_{T}=\epsilon_{kTk^{-1}} \ast \epsilon_{k}$ (here $kTk^{-1} \subseteq K\cap H$);  hence  for  $\epsilon_T \ast \epsilon_t \ast \epsilon_{x_i} \ast \epsilon_{K\cap H}\in \mathcal{H}(H)\ast \epsilon_{K\cap H}$,  $\epsilon_k\ast  \epsilon_T \ast \epsilon_t \ast \epsilon_{x_i} \ast \epsilon_{K\cap H}=\epsilon_{kTk^{-1}} \ast \epsilon_{ktk^{-1}} \ast \epsilon_{x_i} \ast \epsilon_{h} \ast \epsilon_{K\cap H}\ast \epsilon_{ k} \in \mathcal{H}(H)\ast \epsilon_{K\cap H}\ast \epsilon_k$. So $\epsilon_k\ker(\check{f})=\epsilon_k \ast \mathcal{H}(H) \ker(\check{f})=\sum_{i=1}^m \epsilon_k  \ast \mathcal{H}(H) \ast \epsilon_{K\cap H} \check{v}_i\subseteq \sum_{i=1}^m  \mathcal{H}(H) \ast \epsilon_k   \check{v}_i \subseteq \ker(\check{f})$.

We now   let $\check{W'}$ be the  $K$-complement of $\ker(\check{f})$ in $\check{V}$.  Given another  open compact subgroup $K_1 \subseteq K_0$, we have     $$f: (\delta_{H\setminus G}^{-1} \otimes \sigma)^{K_1\cap H}  \hookrightarrow  [ (\Res_H^G \check{\pi})^{\vee}]^{K_1 \cap H}\simeq ( \check{\pi}^{K_1\cap H})^{\ast},$$
which stems from  $$\check{f}: \check{\pi}^{K_1\cap H}=[\ker(\check{f}) \oplus \check{W'}]^{K_1\cap H}\simeq \ker(\check{f})^{K_1\cap H}\oplus \check{W'}^{K_1\cap H} \longrightarrow \check{W}^{K_1\cap H}.$$
Note that $\ker(\check{f})= \ker(\check{f})^{K_1\cap H} \oplus \ker(\check{f})[K_1\cap H]$. Let   $\{\check{w}_1', \cdots, \check{w}_m'\}$  be a basis of $\check{W'}^{K_1\cap H}$ . Then the image of
$(\delta_{H\setminus G}^{-1} \otimes \sigma)^{K_1\cap H}$  in $(\Res_H^G \check{\pi})^{\vee}$ is $ \cap_{i=1}^m\Stab_{G}(\check{w}'_i)\cap  K_1$-stable, and it  lies in $\Res_H^G\check{ \check{\pi}}\simeq \Res_H^G \pi$.     Therefore   $f: \delta_{H\setminus G}^{-1} \otimes \sigma \hookrightarrow (\Res_H^G \check{\pi})^{\vee}$
factors through $\Res_H^G \pi\hookrightarrow (\Res_H^G \check{\pi})^{\vee}$.
 \end{proof}
 Go back to the normal  case.
\begin{lemma}\label{thequotientofrp}
\begin{itemize}
\item[(1)] $\Res_H^G \pi$ is a semi-simple representation with finite  multiplicity.
\item[(2)] If $\sigma_1, \sigma_2 \in \mathcal{R}_H(\pi)$, then there is an element $g\in G$ such that $\sigma_2 \simeq \sigma_1^g$, where $\sigma_1^g (h): =\sigma_1(ghg^{-1})$ for $h\in H$.
\item[(3)] There is a positive integer $m$ such that $\Res_H^G\pi\simeq \sum_{\sigma\in \mathcal{R}_H(\pi)} m \sigma$.
\end{itemize}
\end{lemma}
\begin{proof}
By Lmm.\ref{ssmcheck}, $0 \neq m_H(\check{\pi}, \check{\sigma}) \simeq m_G(\check{\pi}, \Ind_{H}^G\check{\sigma})=m_{G}(\cInd_H^G \sigma, \pi) =m_H(\sigma, \pi)$.  By the similar proof of  Theorem \ref{cliffordadmissible},  we obtain the  results (1)---(3).
\end{proof}
\begin{remark}\label{remarkcont}
Keep the notations. Then $\check{\pi}|_{H} \simeq \oplus_{\sigma \in \mathcal{R}_{H}(\pi)} m \check{\sigma}$.

\end{remark}
\begin{proof}
It follows from  $m_{H}(\check{\pi}, \check{\sigma})=m_G(\check{\pi}, \Ind_{H}^G\check{\sigma})= m_{G}(\cInd_H^G \sigma, \pi) =m_H(\sigma, \pi)$.
\end{proof}
In the following, we assume that $(\sigma, W)$ is an irreducible constituent of $(\Res_H^G \pi, V)$.  Let  $I_G(\sigma)=\{ g\in G\mid \sigma^g\simeq \sigma\}$, and $I_G^0(\sigma)=\left\{ g\in G \mid \pi(g)(W)=W\right\}$. The $\sigma$-isotypic component of $(\Res_H^G \pi, V)$ is an irreducible  $I_G(\sigma)$-module,  denoted by $(\widetilde{\sigma}, \widetilde{W})$.
 \begin{lemma}\label{opennormal}
Both  $I_G^0(\sigma)$, $I_G(\sigma)$ are open subgroups of $G$. Moveover, $(\sigma, W)$ is extendible to $I_G^0(\sigma)$, and $\pi\simeq \cInd_{I_G(\sigma)}^G \widetilde{\sigma}$.
 \end{lemma}
\begin{proof}
 1) Let  $0 \neq w_0\in W$ and  $K_{w_0}=\Stab_G(w_0)$. For $g\in K_{w_0}, h \in H$, we have
$\pi(g) \sigma(h) w_0=\pi(ghg^{-1}) \pi(g) w_0=\sigma(ghg^{-1}) w_0$; this means that $g$ stabilizes $W$, so $I_G(\sigma)$,  $I_G^0(\sigma)$ contains  $K_{w_0}$, and  both are  open subgroups of $G$. \\
2)  Since $I_G(\sigma)/H$ is an open subgroup of the compact group $G/H$, $[G : I_G(\sigma)]$ has finite cardinality.  By Frobenius reciprocity, we have $\Hom_G( \pi, \cInd_{I_G(\sigma)}^G \widetilde{\sigma}) \neq 0$. On the other hand, $\Hom_{G}(\cInd_{I_G(\sigma)}^G \widetilde{\sigma}, \cInd_{I_G(\sigma)}^G \widetilde{\sigma}) \simeq \Hom_{I_G(\sigma)}(\widetilde{\sigma},  \cInd_{I_G(\sigma)}^G \widetilde{\sigma}) $. By  the structure of $\cInd_{I_G(\sigma)}^G \widetilde{\sigma}$ as described in \cite{BushH}, we have $\Res_{H}^{G} \cInd_{I_G(\sigma)}^G \widetilde{\sigma} \simeq \sum_{g\in G/{I_G(\sigma)}} g \widetilde{W}$\footnote{Notice that $g\widetilde{W}$ perhaps is not $I_G(\sigma)$-stable.}. Any non-zero $f\in\Hom_{I_G(\sigma)}(\widetilde{\sigma},  \cInd_{I_G(\sigma)}^G \widetilde{\sigma})$, is also an $H$-morphism, and then has  image  in $\widetilde{\sigma}$. Therefore $m_{I_G(\sigma)}(\widetilde{\sigma},  \cInd_{I_G(\sigma)}^G \widetilde{\sigma})=1$, and $\pi \simeq \cInd_{I_G(\sigma)}^G \widetilde{\sigma}$.
\end{proof}

\begin{remark}\label{finitelengths}
$\Res_{I_G^0(\sigma)}^{I_G(\sigma)} \widetilde{\sigma}$ is a smooth representation of finite length.
\end{remark}
\begin{proof}
Note that $I_G^0(\sigma)/H$, $I_G(\sigma)/H$ both are open closed subgroups of $G/H$, so the indices  $[G: I_G^0(\sigma) ]$,   $[G: I_G(\sigma) ]$ both are finite.
\end{proof}
\begin{lemma}\label{theopensubgroup}
There is  an open normal subgroup $J_G(\sigma)$ of $I_G(\sigma)$ such that $H \subseteq J_{G}(\sigma) \subseteq I_G^0(\sigma)$.
\end{lemma}
\begin{proof}
 Notice that $I_G^0(\sigma)/H$ is  an open compact subgroup of $G/H$. We let $\overline{K_0} = \cap_{\overline{g} \in I_G(\sigma)/H} \overline{g} \tfrac{I_G^0(\sigma)}{H} \overline{g}^{-1}$. By Lmm.\ref{twocompactsubgroups} (2), $\overline{K_0}$ is an open normal subgroup of $I_G(\sigma)/H$, and  we denote its inverse image in $I_G(\sigma)$ or $I_G^0(\sigma)$ by $J_G(\sigma)$.
\end{proof}

\subsubsection{}
\label{hypothesisBresult}
In the following, we shall rewrite some results of \S 11 in \cite{CuRe} to our situation. We write $\pi_{[\sigma]} =\cInd_{J_G(\sigma)}^{I_G(\sigma)} W$. Let $\Delta=\left\{ g_i\in I_G(\sigma)\right\}_{i\in I}$ containing $1$,  be a set of representatives for $I_G(\sigma)/{J_G(\sigma)}$, and $\mathcal{W}$ the  canonical image of  $W $ in $\cInd_{J_G(\sigma)}^{I_G(\sigma)} W$.(cf.Lmm.\ref{therestriction1}) Following \cite[\S 11]{CuRe}, we let $D=\End_{I_G(\sigma)}(\pi_{[\sigma]})$, and write the map $
\varphi \in D$ on the right-hand side, i.e. $v\in \cInd_{J_G(\sigma)}^{I_G(\sigma)} W$, $v\longrightarrow (v) \varphi$. Notice:
\begin{itemize}
\item[(1)] $(\pi_{[\sigma]}, J_G(\sigma), \pi_{[\sigma]}(g) \mathcal{W})$ is an irreducible representation of $J_{G}(\sigma)$, isomorphic to $(\pi_{[\sigma]}, J_G(\sigma), \mathcal{W})$, for $g\in \Delta$.
\item[(2)] Let $\epsilon_g: \mathcal{W} \longrightarrow \pi_{[\sigma]}(g) \mathcal{W}$ be an intertwining operator between $(\pi_{[\sigma]}, J_G(\sigma), \mathcal{W})$ and $(\pi_{[\sigma]}, J_G(\sigma), \pi_{[\sigma]}(g) \mathcal{W})$.
\item[(3)] $\epsilon_g$ can extend uniquely to an element $\mathcal{E}_g$ in $D$, given by $[\pi_{[\sigma]}(x) f_w]\mathcal{E}_g:= \pi_{[\sigma]}(x) [(f_w)\epsilon_g]$ for $x\in \Delta$, $f_w\in \mathcal{W}$.
\item[(4)] $\mathcal{E}_{g_1} \circ \mathcal{E}_{g_2} = \alpha(g_1,g_2) \mathcal{E}_{g_3}$, for $g_i \in \Delta$\footnote{$\Delta$ is a discrete set of finite cardinality.}, where $\alpha(g_1, g_2) \in \C^{\times}$ and $g_1g_2J_G(\sigma)=g_3 J_G(\sigma)$.
\item[(5)] The above $\alpha(-, -)$ defines a $2$-cocycle of one class in $\Ha^2(I_{G}(\sigma)/J_G(\sigma), \C^{\times})$.
\end{itemize}
We fix  an embedding $\widetilde{W} \longrightarrow \Ind_{J_G(\sigma)}^{I_G(\sigma)} W$ such that the image of $W$ is $\mathcal{W}$, and let $\mathcal{N}=\left \{ \right.\varphi: \cInd_{J_G(\sigma)}^{I_G(\sigma)} W \longrightarrow \Ind_{J_G(\sigma)}^{I_G(\sigma)} W$, an $I_G(\sigma)$-homomorphism with  image in $\widetilde{W} \}$. Note that $\mathcal{N}$ is a left $D$-ideal. Following \cite[\S 11]{CuRe}, we define two projective smooth representations $(\rho_1, \mathcal{W})$, $(\rho_2, \mathcal{N})$
 of $I_G(\sigma)$ as follows:
 \begin{itemize}
\item[(1)]For $x=gg_0 \in I_G(\sigma)$ with $g\in \Delta$,  and $ g_0 \in J_G(\sigma)$,   $f_w \in \mathcal{W}$, $\rho_1(x) f_w :=  ( \pi_{[\sigma]}(x) f_w)\mathcal{E}^{-1}_g$.
\item[(2)] $\rho_2$ factors through $I_G(\sigma)/J_G(\sigma)$, and $(v)[\rho_2(g) \varphi]:= ( (v)\mathcal{E}_g)\varphi$, for $g\in \Delta$, $v \in \cInd_{J_G(\sigma)}^{I_G(\sigma)} W$, $\varphi \in \mathcal{N}$.
\end{itemize}

\begin{lemma}
$(\rho_2,  \mathcal{N})$ is an irreducible projective representation of $I_G(\sigma)$.
\end{lemma}
\begin{proof}
By construction, the space $\mathcal{N}$ is spanned by $ \mathcal{E}_{g_1} \circ  \varphi, \mathcal{E}_{g_2}\circ\varphi  , \cdots,   \mathcal{E}_{g_m}\circ \varphi$, for any non zero element $\varphi \in \mathcal{N}$, and some suitable $g_1, \cdots, g_m \in \Delta$(related to $\varphi$).
\end{proof}

\begin{theorem}[Clifford]\label{thetensorprojectivereps}
The irreducible representation $(\widetilde{\sigma}, \widetilde{W})$ of $I_G(\sigma)$ is linearly isomorphic with the tensor projective representation $\rho_1 \otimes \rho_2$ of $I_G(\sigma)$.
\end{theorem}
\begin{proof}
By observation, $\rho_1 \otimes \rho_2$ is a honest representation of $I_G(\sigma)$. Assume $\widetilde{W}=\oplus_{i=1}^m \pi_{[\sigma]}(g_i) \mathcal{W}$ in $\cInd_{J_G(\sigma)}^{I_G(\sigma)} W$, for different elements $\overline{g_i} \in I_{G}(\sigma)/J_{G}(\sigma)$. Let  $\varphi_i\in \mathcal{N}$, corresponding to $\epsilon_{g_i}: \mathcal{W} \longrightarrow \pi_{[\sigma]}(g_i) \mathcal{W}$ by Frobenius reciprocity. Then  $\{\varphi_1, \cdots, \varphi_m\}$  forms a basis of $\mathcal{N}$.  Let   $\digamma: \mathcal{W}  \otimes \mathcal{N} \longrightarrow \widetilde{W}; \sum_{i=1}^m f_{w_i} \otimes\varphi_i  \longmapsto \sum_{i=1}^m(f_{w_i})\varphi_i$. Firstly, if $\sum_{i=1}^m f_{w_i} \otimes \varphi_i \neq 0$, and $\sum_{i=1}^m(f_{w_i})\varphi_i=0$, then $(f_{w_i})\varphi_i=0$, and  $(\pi_{[\sigma]}(g) f_{w_i})\varphi_i=0$ for all $g\in I_G(\sigma)$, contradicting to Lmm.\ref{therestriction1}(2). So the injectivity of  $\digamma$ follows. Secondly, letting $x=gg_0$ with $g\in \Delta$, $g_0 \in J_G(\sigma)$, we then have
$$\digamma\big(\rho_1 \otimes \rho_2(x) (f_w\otimes \varphi )\big)=(\pi_{[\sigma]}(x) f_w)\varphi=\pi_{[\sigma]} (x) (f_w) \varphi=\pi_{[\sigma]}(x) \digamma(f_w\otimes \varphi  ),$$
which shows that $\digamma$ is an $I_{G}(\sigma)$-morphism, and then the surjectivity follows.
\end{proof}

\subsection{} In the third part, we do not  assume that $H$ is a normal subgroup of $G$.  First of all we assume that $H$ is an open subgroup of $G$. Let  $\Delta=\{s_i\in G\}_{i\in I}$ be a complete  set of representatives for  $H\setminus G/H$, and assume $1\in \Delta$. Let $H_s=s^{-1}Hs$. For $(\rho, W)\in \Rep(H)$, set $\rho^{s}(x)= \rho(sxs^{-1})$, $x\in H_s\cap H$.     For any $s\in \Delta$, $s\neq 1$,  assume that  the cardinality  of  bisets $ (H_s\cap H)\setminus H/(H_s\cap H)$ is infinite.
\begin{lemma}\label{ddf1}
Let  $(\sigma_i, W_i)\in \Rep(H)$. For any   $1\neq s\in \Delta$,  if  $\Res_{H_s\cap H}^H \sigma_1$ is finitely generated,  $\Hom_G(\cInd_{H}^G \sigma_1, \cInd_{H}^G \sigma_2)\simeq \Hom_H(\sigma_1, \sigma_2)$.
\end{lemma}
\begin{proof}
 By Frobenius reciprocity and Lmm.\ref{therestriction1}, $$\Hom_G(\cInd_{H}^G \sigma_1, \cInd_{H}^G \sigma_2)\simeq \Hom_H(\sigma_1,  \oplus_{s\in \Delta}\cInd^H_{H_s\cap H}(\sigma_2)^s) \hookrightarrow \prod_{s\in \Delta}\Hom_H(\sigma_1,\cInd_{H_s\cap H}^H(\sigma_2)^s).$$  For a fixed $s\in \Delta$ with $s\neq 1$,      let $\Sigma_s$ be a complete set of representatives for $  (H_s\cap H)\setminus H/(H_s\cap H)$.   Denote the representation $((\sigma_2)^s,W_2)$ of $H$ simply by $(\rho,W_2^{\rho})$. By Lmm.\ref{therestriction1},  as  $H_s\cap H$-module, we can embed $W_2^{\rho}$ in $ \cInd^H_{H_s\cap H}\rho$,  with  the image denoted by $\mathcal{W}_2$. Then by Lmm.\ref{therestriction1}, $\cInd^H_{H_s\cap H}W_2^{\rho} \simeq \oplus_{t\in \Sigma_s}  \mathcal{W}_{2,t}$,    $ \mathcal{W}_{2,t}= \oplus_{g\in (H_s\cap H)/[(H_s\cap H)_t\cap (H_s\cap H)]} gt^{-1}\mathcal{W}_2$, $\mathcal{W}_{2, t}\simeq  \cInd_{(H_s\cap H)_t\cap (H_s\cap H)}^{H_s\cap H}\rho^{t}$.

 Assume $W_1$ is generated by $w_1, \cdots, w_l$ as an $H\cap H_s$-module.   If $0\neq B \in \Hom_H(\sigma_1, \cInd_{H_s\cap H}^H(\sigma_2)^{s})$,   there exists a finite natural number $m$, such that all $B(w_i)\in \oplus_{j=1}^m \mathcal{W}_{2, t_j}\simeq \oplus_{j=1}^m \cInd_{(H_s\cap H)_t\cap (H_s\cap H)}^{H_s\cap H}\rho^{t_j} $.     Note that  for $t\in H$,  $w\in W_1$, $B(tw)=tB(w)\in \oplus_{j=1}^m t\mathcal{W}_{2, t_j}$.  However  $tw= \sum_{j=1}^m c_{i}h_{i} w_i$, for some $c_{i} \in \mathbb{C}$, $h_{i} \in H_s\cap H$, and $    B(tw)\in\oplus_{j=1}^m    \mathcal{W}_{2, t_j}$.  Now asume $e_1\in W_1$, $0\neq B(e_1)= \sum_{j=1}^m c_j  w_{2, j}$,   for  some non-zero $w_{2,j} \in \mathcal{W}_{2, t_j}$, and some $c_j\in \mathbb{C}$, with $c_{j'}\neq 0$.  Assume $ w_{2, j}=\oplus_{k=1}^{n_j} g_{kj} t_j^{-1} w_{k,j}$, for some non-zero $w_{k,j}\in  \mathcal{W}_2$. Then
 $B(t_{m+1}^{-1} t_{j'}g_{1j'}^{-1}  e_1)=[\oplus_{j=1, j\neq j'}^m\oplus_{k=1}^{n_j} c_jt_{m+1}^{-1} t_{j'}g_{1j'}^{-1}  g_{kj} t_j^{-1} w_{k,j}]\oplus c_{j'}t_{m+1}^{-1} w_{1,j'}\oplus [\oplus_{k\neq 1} c_{j'}t_{m+1}^{-1} t_{j'}g_{1j'}^{-1}  g_{kj'} t_{j'}^{-1} w_{k,j'}]$. Since $ c_{j'}t_{m+1}^{-1} w_{1,j'} \notin  \oplus_{j=1}^m \mathcal{W}_{2, t_j}$, a contradiction.  Therefore $\Hom_H(\sigma_1, \cInd_{H_s\cap H}^H(\sigma_2)^{s})=0$, for any $1 \neq s\in \Delta$, and the first  result follows.
    \end{proof}
    If $K$ is an open compact subgroup of $G$, for each positive integer  $n$, we let $\mathcal{N}(K)_n=\{ K^i \mid K^i \lhd K, [K: K^i]=n\}$.

    \begin{lemma}\label{ddf2}
Let  $(\sigma_i, W_i)\in \Rep(H)$.  For any $1\neq s\in \Delta$,   if assume (1)  up to $H_s\cap H$-conjugacy  there exists and  only exists  a finite number of maximal open compact groups in $H$,  (2) for  each maximal open compact subgroup $K$ of $H_s\cap H$, and each $n$, the set $\mathcal{N}(K)_n$ is finite,  then  $\Hom_G(\cInd_{H}^G \sigma_1, \cInd_{H}^G \sigma_2)\simeq \Hom_H(\sigma_1, \sigma_2)$, for any  admissible representation  $(\sigma_1, W_1)$   of $H$.
\end{lemma}
\begin{proof}
Keep   the notations of the first paragraph  in the proof of the foregoing lemma.
  Let us choose  $\{K_1, \cdots, K_m\}$ to be a total set of  maximal open compact subgroups of $H$,  up to $H_s\cap H$-conjugacy.  Let $K$ be an open compact subgroup of $H_s\cap H$, such that $W_1^K\neq 0$. By Lmm.\ref{twocompactsubgroups}, we assume that $K$ is a normal subgroup of  each  $K_i$.  Assume  $0\neq B \in \Hom_H(\sigma_1, \cInd_{H_s\cap H}^H(\sigma_2)^{s})$, and $B(W_1^K)\subseteq\oplus_{\alpha=1}^m   \mathcal{W}_{2, t_{\alpha}}\simeq \oplus_{\alpha=1}^m \rho^{t_{\alpha}}$.
 Under the condition (2) we  let  $\mathcal{L}_i$ denote  the total set of  normal  open compact subgroups $L_i$ of  $K_i$, satisfying  $[K_{i} : L_{i}]=[K_{1}: K]$, and let $\mathcal{L}=\cup_{i}   \mathcal{L}_i$.

 For a fixed   $t\in H$,   there exists $h_t\in H_s\cap H$, such that $K_{t}=t^{-1} Kt\subseteq (K_1)_t =h_t K_j h_t^{-1}$, for certain $j$.     So $K_t \lhd (K_1)_t=(K_j)_{h^{-1}_t}$, $K_{th_t} \lhd K_j$, and $[K_j : K_{th_t} ]=[(K_j)_{h^{-1}_t} : K_t]=[(K_1)_t: K_t  ]=[K_1 : K]$.  Hence $K_{th_t} =L_t$, for some $L_t\in \mathcal{L}$.  Set $D_t=K_{th_{t}}\cap K=L_t\cap K$. Then    $\epsilon_{D_t h^{-1}_t t^{-1} K} \in \mathcal{H}(H, D_t)$.  For $0\neq w\in W_1^{K}$,  $B(\epsilon_{D_t h^{-1}_t t^{-1} K} w)=B(\epsilon_{h_t^{-1}t^{-1}} \ast \epsilon_{th_tD_th^{-1}_tt^{-1}} \ast \epsilon_Kw)=h_t^{-1}t^{-1}B(w)\in \oplus_{\alpha=1}^m  h_t^{-1}t^{-1} \mathcal{W}_{2, t_{\alpha}}$.  Moreover $0\neq \epsilon_{D_{t} h_t^{-1}t^{-1} K} w\in W_1^{D_{t}}$.  Now let $\widetilde{W}_1=\sum_{L\in \mathcal{L}} W_1^{L\cap K}  \subseteq W_1$, then   $\widetilde{W}_1$ has finite dimension,  and  $W_1^K \subseteq \widetilde{W}_1$,  $W_1^{D_t} \subseteq \widetilde{W}_1$. Hence $B( \widetilde{W}_1)$ belongs to a direct sum of  finite number of   $\rho^{t_{\beta}}$'s.  This makes a contradiction similar to  the above proof.  Therefore $\Hom_H(\sigma_1, \cInd_{H_s\cap H}^H(\sigma_2)^{s})=0$, for any $1 \neq s\in \Delta$, and the second   result holds.
    \end{proof}

    \subsection{}In the fourth part we  interfere with    unitary representations of locally profinite groups.  Our main references are \cite{KT}, \cite{Ma}. The results in  them  are   mainly about representations of locally compact groups,  so let us first rewrite some of them  to fit  us well.

We call  a smooth representation  $(\rho, W)$ of $H$ \emph{preunitary}   if there exists a non-degenerate hermitian form $\langle, \rangle$ on $W$, such that $\langle \rho(h)v, \rho(h)w\rangle= \langle v, w\rangle$, for $v, w\in W$, $g\in H$. Here  $W$ is  not  required to be a  complete   vector space.

Until the end of this section,  we will let  $(\rho, \langle, \rangle, W)$  be   a smooth preunitary  representation  of $H$,  and    let $\mathcal{W}_{\rho}$ or $\mathcal{W} $ denote its complete vector space.
\begin{lemma}
$(\rho, \mathcal{W} )$ is a unitary representation of $H$ in the usual sense(cf. \cite{Ma}).
\end{lemma}
\begin{proof}
 Let  $h_0\in H$, $w_0\in W^{K_1}$, $K_1$ being an open compact subgroup of $H$. For any $\epsilon>0$, when $\Vert w-w_0\Vert<\epsilon$,  and $h\in h_0 K_1$,  we have
 $\Vert \rho(h)w-\rho(h_0) w_0\Vert\leq \Vert \rho(h)w-\rho(h) w_0\Vert+ \Vert \rho(h)w_0-\rho(h_0) w_0\Vert =\Vert w- w_0\Vert<  \epsilon$. So $\rho: H \times W \longrightarrow  W; (h, w) \longmapsto  \rho(h)w$ is continuous, and it can extend well  to  a unitary  representation  $\rho: H \times \mathcal{W} \longrightarrow  \mathcal{W}$.
   \end{proof}
\subsubsection{Admissible case}
In this subsection we will  assume $(\rho,W)$ is  admissible  unless specific  illustration.
   \begin{lemma}
For any open compact subgroup $K_1$ of $H$, let $W \simeq  \oplus_{\sigma\in \hat{K}_1} W^{\sigma}$  be the   direct sum of  its $ K_1$-isotypic components. (cf. \cite[p.15, Pro.]{BushH}). Then:
 \begin{itemize}
\item[(1)] $ W^{\sigma_i}  \bot W^{\sigma_j}$,  for different $\sigma_i, \sigma_j \in \hat{K}_1$;
\item[(2)] For each $(\sigma, U) \in \hat{K}_1$, $W^{\sigma}$ is an algebraic direct sum of its mutually orthogonal  $H$-subspaces $W^{\sigma}_i$ such that each $W^{\sigma}_i$ is  isomorphic to $U$ as $K_1$-modules.
\end{itemize}
\end{lemma}
\begin{proof}
1) For  non-zero vectors $v_i \in W^{\sigma_i}$, $v_j \in W^{\sigma_j}$,  the vector spaces $ K_1v_i$, $K_1v_j$ generated by $v_i$, $v_j$, both have finite dimension. Finally  it reduces to study  a unitary representation $K_1v_i \oplus K_1v_j $ of a finite group,  so the result holds.\\
2) Let $e_1, \cdots, e_n$ be a basis of $U$. Then we can find an open compact subgroup $K_2 \subseteq \cap_{i=1}^n \Stab_{K_1}(e_i)$  such that $K_2 \rhd K_1$.  Hence  $W^{\sigma}$ is a preunitary representation  of a finite group $\frac{K_1}{K_2}$ of finite dimension; the result holds.
\end{proof}

    Let $(\overline{\rho}, \overline{W})$ denote the complex  conjugate representation of $(\rho, W)$.
\begin{lemma}\label{K1in}
      $\overline{\rho} \simeq \check{\rho}$ and $\mathcal{W}^{K_1}=W^{K_1}$, for any open compact subgroup $K_1$ of $H$. In this case, $(\check{\rho}, \check{W})$ is  a preunitary representation of $H$.  \footnote{If $\rho$ is not admissible, we can't  ensure  that $\check{\rho}$ is also preunitary. }

    \end{lemma}
   \begin{proof}
 1)  Any non-zero vector  $\overline{w} \in \overline{W}$ defines a non-trival $\mathbb{C}$-linear function on $W$ as $w\longrightarrow \langle w, \overline{w}\rangle$, for $w\in W$. Moreover  it induces a $\mathbb{C}$-linear and  $H$-monomorphism
    $\overline{W} \longrightarrow \check{W}$; by considering their  $K_1$-invariant parts we see $\overline{W} \simeq \check{W}$ as $H$-modules.

  2) Assume $\Res_{K_1}^H \rho\simeq \oplus_{i\in I} m_i\pi_i$, for mutually  orthogonal irreducible representations $\pi_i$ of $K_1$.
  \footnote{If assume that $G$ is a second-countable  group, then it contains a countable  neighbourhood basis $\{K_i\}$ of $1_G$; we can assume each $K_i$ is an open compact subgroup of $G$.   So $V=\cup V^{K_i}$ has countable dimension.}Let $e_{i}^{1}, \cdots, e_i^{n_i}$ be an orthonormal basis of $m_i\pi_i$.  Then every element $\widetilde{a} \in \mathcal{W}$ has the following form:
  $\widetilde{a}=\sum_{i\in I}\sum_{j=1}^{n_i} a_{ij}e_i^{j}$, such that $\sum_{i\in I}\sum_{j=1}^{n_i}\mid a_{ij}\mid^2< +\infty$. If $k\cdot\widetilde{a}=\widetilde{a}$,  for any $k\in K_1$, then $k\cdot \sum_{j=1}^{n_i} a_{ij} e_i^j= \sum_{j=1}^{n_i} a_{ij} e_i^j$, in other words, $m_i\pi_i$ has a $K_1$-invariant vector $ \sum_{j=1}^{n_i} a_{ij} e_i^j$, so only a finite number of such vectors is non-zero; thus   $\widetilde{a} \in W^{K_1}$.
     \end{proof}
   \begin{lemma}\label{semisimpleu}
 $W$ is an algebraic direct sum of its irreducible and mutually orthogonal  $H$-subspaces.
  \end{lemma}
 \begin{proof}
 For any $H$-subspace $W_1$ of $W$,  the orthogonal complement  $W_1^{\bot}$ in $W$ is also $H$-invariant.  Since $(\rho, W)$ is admissible,  $W=W_1\oplus W_1^{\bot}$. So by \cite[p.14, Prop.]{BushH},  $(\rho, W)$ is semi-simple.  We order  the set $\mathcal{R}$  of all sets $\mathcal{S}_I=\{V_i\}_{i\in I}$ by set inclusion, where $\{V_i\}_{i\in I}$ consists of  mutually orthogonal and irreducible  $H$-subspaces $V_i$ of $W$. By the above discussion,  $\mathcal{R}$ is non-empty and each chain $\mathcal{C}=\{ \mathcal{S}_I\}$ in $\mathcal{R}$ has an upper bounded given by the union $\cup_{I} \mathcal{S}_I$.  Then Zorn's Lemma yields a maximal element $\{V_j\}_{j\in J}$ in $\mathcal{R}$.  Let $W'=\oplus_{j\in J} V_j$; if $W'\neq W$, then $W'^{\bot}$( not zero)  is also an $H$-space and  contains an irreducible $H$-subspace $V'$. Now  $\{V_j\}_{j\in J}\cup \{V'\}$ is also in $\mathcal{R}$,  contradicting  to the maximality of $\{V_j\}_{j\in J}$. Therefore  $W=\oplus_{j\in J} V_j$, and we are done.
 \end{proof}
 \begin{corollary}\label{semisimpleu2}
   If $\rho$ is finitely generated , then $W$ is a finite  direct sum of its irreducible and mutually orthogonal  $H$-subspaces.
   \end{corollary}

       \begin{lemma}\label{simlem}
Let $(\pi_1, \langle, \rangle_1, V_1)$, $(\pi_2, \langle, \rangle_2, V_2)$  be  two admissible preunitary smooth representations  of $H$, with the complete vector spaces $\mathcal{V}_1$, $\mathcal{V}_2$ respectively.
\begin{itemize}
\item[(1)]  If $\pi_1$ has finite length, then  every $0\neq F\in \Hom_H(V_1, V_2)$  is  continuous;
\item[(2)]  If both   $\pi_i$ are representations of finite length, then $\Hom_H(V_1, V_2) \simeq B_H(\mathcal{V}_1, \mathcal{V}_2)$.
\end{itemize}
\end{lemma}
\begin{proof}
(1) By Lmm.\ref{semisimpleu}(2), it is sufficient to assume that $\pi_1$ is irreducible and $f$ is surjective; in this case  $V_2$ is isomorphic to $V_1$ as $H$-modules. Assume $V_1^{K_1}\neq 0$, for an  open compact subgroup $K_1$ of $H$.
Then  $F: V_1^{K_1} \longrightarrow V_2^{K_1}$  is a bijective linear map between two norm spaces of finite dimension. Let $\{e_1, \cdots, e_n, \cdots\}$ be a complete orthonormal  basis of $V_1$, such that $\{e_1, \cdots, e_m\}$ forms a complete orthonormal  basis of $V_1^{K_1}$. Let $\{f_1, \cdots, f_m\}$ be a  complete orthonormal  basis of $V_2^{K_1}$.

For an element $v_1=(e_1, \cdots, e_m)\begin{pmatrix}
b_1\\
\vdots\\
b_m
\end{pmatrix}\in V_1^{K_1}$,  let us write $F(v_1)=(f_1, \cdots, f_m)A\begin{pmatrix}
b_1\\
\vdots\\
b_m
\end{pmatrix}$, where $A$ is the matrix corresponding to the linear map $F$.  It is known that there exists a unitary matrix $U$ such that $\overline{U}^{T} \overline{A}^T AU=\diag(a_1, \cdots, a_m)$ for some positive real  numbers $a_i$.  By changing the orthonormal basis of $V_i^{K_1}$, henceforth we simply assume $\overline{A}^T A=\diag(a_1, \cdots, a_m)$.

 For any  $v\in V_1$,  assume $v=\sum_{i=1}^n c_i \pi_1(h_i) e_1$, for some $c_i\in \mathbb{C}$, $h_i\in H$, and write  $\pi_1(h_j^{-1}h_i)e_1=v_{ji}+w_{ji}$ for some  $v_{ji}=\sum_{k=1}^m d_{jik}e_k\in  V_1^{K_1}$, $ w_{ji}\in\oplus_{1\neq \tau \in \Irr(K_1) } V_1^{\tau}$(here $v_{ji}\bot w_{ji}$).  Then $\Vert v\Vert_1^2=\sum_{i,j=1}^n c_i\overline{c_j}\langle \pi_1(h_j^{-1}h_i)e_1,  e_1\rangle_1=\sum_{i,j=1}^n c_i\overline{c_j} \langle v_{ji}, e_1\rangle_1=\sum_{i,j=1}^n c_i\overline{c_j} d_{ji1}$.
Note that $$\langle \pi_2(h_j^{-1}h_i)F(e_1),  F(e_1)\rangle_2 =\langle F(v_{ji}),  F(e_1)\rangle_2=\langle(f_1, f_2, \cdots, f_m)A\begin{pmatrix}
d_{ji1}\\
d_{ji2}\\
\vdots\\
d_{jim}
\end{pmatrix} , (f_1, f_2,  \cdots, f_m)A\begin{pmatrix}
1\\
0\\
\vdots\\
0
\end{pmatrix} \rangle_2$$ $$=\Tr A\begin{pmatrix}
d_{ji1}\\
d_{ji2}\\
\vdots\\
d_{jim}
\end{pmatrix} \begin{pmatrix}
1, 0, \cdots, 0\end{pmatrix} \overline{A}^T  =\Tr \begin{pmatrix}
1, 0, \cdots, 0\end{pmatrix} \overline{A}^TA\begin{pmatrix}
d_{ji1}\\
d_{ji2}\\
\vdots\\
d_{jim}
\end{pmatrix} =a_1d_{ji1}.$$
Consequently,
\[ \Vert F(v)\Vert_2^2= \Vert \sum_{j=1}^n c_i \pi_2(h_i) F(e_1)\Vert_2^2=\sum_{i,j=1}^n c_i\overline{c_j} \langle \pi_2(h_j^{-1}h_i)F(e_1), F(e_1)\rangle_2=\sum_{i,j=1}^n c_i\overline{c_j} d_{ji1}a_1=\Vert v\Vert_1^2 a_1.\]
Hence  $F$ is continuous. \\
(2) Any  $F\in \Hom_{H}(V_1, V_2)$ can extend uniquely to an element $\widetilde{F}\in B_{H}(\mathcal{V}_1, \mathcal{V}_2)$.  Conversely, the restriction  of any $\widetilde{F}\in \Hom_{H}(\mathcal{V}_1, \mathcal{V}_2)$ to $V_1$ defines an $H$-morphism $F: V_1 \longrightarrow \mathcal{V}_2^{\infty}=V_2$.
\end{proof}

From the above proof,  we obtain a result in Casselman's note,  \cite[p.23, Prop.2.1.15]{Cass}:
\begin{corollary}\label{caca}
For an irreducible  (admissible) representation $(\rho, W)$ of $H$, up to scalar multiplication there is at most  one non-degenerate $H$-invariant Hermitian inner product on $W$.
\end{corollary}
\begin{proof}
See also Bernstein's  unpublished  note on representation.
\end{proof}
\begin{remark}
There exists an equivalence between the category of unitary representations of $H$ of finite length and the category of smooth preunitary representations of $H$ of finite length.
\end{remark}
\begin{proof}
Let $(\pi, V)$ be the smooth part of an irreducible  unitary representation $(\Pi, \mathcal{V})$ of $H$. By investigating its restriction to open compact subgroups, we see that $V\neq 0$.   If $\pi$ contains a non-zero subrepresentation $\rho$, then the  completions of $\pi$ and $\rho$  must be  equal;  by the admissible condition, $\rho=\pi$.   We leave the rest details to the reader.
\end{proof}
  \subsubsection{Non-admissible case}\label{nonad}
  Let us  investigate   the general  case that $(\rho, W)$ is only a preunitary smooth representation of $H$. Assume $W$ is a \emph{second-countable} space. For the complex conjugate representation $(\overline{\rho}, \overline{W})$,  let us write the corresponding  scalar multiplication  by $\odot$,  namely $c\odot w:= \overline{c}w$, for $c\in \mathbb{C}$, $w\in \overline{W}=W$.
   \begin{lemma}\label{item44}
  \begin{itemize}
  \item[(1)]  There exists an orthonormal basis $\{e_1, \cdots, e_n, \cdots\}$ of $\mathcal{W}$ such that $e_i \in W$, and $\{ e_1, \cdots, e_n, \cdots \}$ forms an algebraic basis of $W$;
   \item[(2)]  For any open compact subgroup $K_1$ of $H$, $W^{K_1}$ is dense in $\mathcal{W}^{K_1}$;
  \item[(3)] As $H$-modules,  $\overline{W} \hookrightarrow \overline{\mathcal{W}}^{\infty} \hookrightarrow\check{W}$;
  \item[(4)] Let $(\pi, V)$ be another preunitary smooth representation of $H$, $\mathcal{V}$  the completion of $V$, and  assume $V$ is second-countable.  Then
  \begin{itemize}
  \item[(a)] $\Hom_H(W, V) \simeq \Hom_H(\overline{W}, \overline{V}); f \longrightarrow \overline{f}=f$.
  \item[(b)] Let $f: W\longrightarrow V$ be a non-zero  \emph{continuous} $H$-morphism. Then it will induce the following canonical $H$-morphisms: (I)  $f: \mathcal{W}^{\infty} \longrightarrow \mathcal{V}^{\infty}$ or $ \overline{\mathcal{W}}^{\infty} \longrightarrow \overline{\mathcal{V}}^{\infty} $, (II) $ \check{f}: \check{V} \longrightarrow \check{W}$,  (III)  $f^{\ast}:  \overline{\mathcal{V}}^{\infty} \longrightarrow \overline{\mathcal{W}}^{\infty}$.
        \end{itemize}
    \end{itemize}
  \end{lemma}
  \begin{proof}
   Part (1) comes from \cite[Chapitre V 23, Prop.6]{NB}.  For (2) assume $W=\oplus_{\sigma\in \Irr(K_1)} W^{\sigma}$, and let $\{h^{\sigma}_{1}, \cdots, h^{\sigma}_{n}, \cdots \}$ be an  orthonormal basis of $W^{\sigma}$. Note that for different $\sigma_i, \sigma_j\in \Irr(K_1)$, $W^{\sigma_i} \bot W^{\sigma_j}$. Thus $\{h^{\sigma}_i\}$ forms an orthonormal basis  of $W$.  For any $x=\sum_{i, \sigma} c_i^{\sigma} h_i^{\sigma}\in \mathcal{W}^{K_1}$, with $\sum_{i, \sigma} \vert c^{\sigma}_i\vert^2 < + \infty$, we have $k h_{i}^{\sigma}\in W^{\sigma}$, for $k\in K_1$. Hence $x=\sum_i c_i^{1_{K_1}} h_i^{1_{K_1}} $ with $h_i^{1_K}\in W^{K_1}$, i.e. $W^{K_1}$ is dense in $\mathcal{W}^{K_1}$.
   The rest parts are straightforward.
           \end{proof}
           \begin{corollary}\label{semiirrr}
           Keep the notations. If $(\pi, V)$ is an irreducible  subrepresentation of $(\rho, W)$, then $(\pi, V)$ is a direct summand of $(\rho, W)$.
           \end{corollary}
           \begin{proof}
           By Cor.\ref{caca}, we can find    a  unitary embedding $\iota: V\hookrightarrow W$, which  will introduce $\iota=\overline{\iota} :  \overline{V} \hookrightarrow \overline{W}$ and  $\overline{\iota}^{\ast} :  W \longrightarrow \mathcal{V}^{\infty}\simeq V$.  For $v_1, v_2\in V$, we have $\langle \overline{\iota}^{\ast}\circ \iota(v_1), v_2\rangle_V=\langle \iota(v_1), \iota(v_2)\rangle_W=\langle v_1, v_2\rangle_V$, so $ \overline{\iota}^{\ast}\circ \iota(v_1)=v_1$, $W=\iota(V)\oplus \ker(\overline{\iota}^{\ast})$.
  \end{proof}
  Let $B_H(W,V)$ denote the set of all continuous $H$-morphisms from $W$ to $V$.
 \begin{lemma}
      Keep the notations of Lmm.\ref{item44}.  If  $(\pi, V)$ is an irreducible representation, and $\dim B_H(W, V)<+\infty$, then  $f^{\ast} (\overline{V}) \subseteq \overline{W}$.
    \end{lemma}
\begin{proof}
First we have an orthogonal decomposition $\overline{\mathcal{W}}=f^{\ast}(\overline{\mathcal{V}})\oplus [f^{\ast}(\overline{\mathcal{V}})]^{\bot}$, and a short exact sequence $0\longrightarrow f^{\ast}(\overline{\mathcal{V}}) \longrightarrow \overline{\mathcal{W}} \stackrel{p}{\longrightarrow} [f^{\ast}(\overline{\mathcal{V}})]^{\bot} \longrightarrow 0$. If  $p=0$, then $ f^{\ast}(\overline{\mathcal{V}}) \simeq \overline{\mathcal{W}} $, the result is clearly right.  Assume now $p\neq 0$. As $\overline{W}$ is dense in $\overline{\mathcal{W}}$, the restriction of $p$ to $\overline{W}$ is non-zero. Hence $0 \longrightarrow \ker p \cap \overline{W} \longrightarrow \overline{W} \stackrel{p}{\longrightarrow} \{[f^{\ast}(\overline{\mathcal{V}})]^{\bot}\}^{\infty}$. If $\overline{W} \cap \ker p=0$, then $ \overline{W}$ is a subspace of  $\{[f^{\ast}(\overline{\mathcal{V}})]^{\bot}\}^{\infty}$; considering their completions, we get $ \overline{\mathcal{W}} \hookrightarrow [f^{\ast}(\overline{\mathcal{V}})]^{\bot}$; considering their $\overline{\pi}$-components, we get a contradiction.  Therefore $\overline{W} \cap \ker p \simeq [f^{\ast}(\overline{\mathcal{V}})]^{\infty}$, i.e. $f^{\ast}(\overline{V})\subseteq \overline{W}$.
\end{proof}
\begin{corollary}
Under the above condition, $(\pi, V)$ is a  direct summand of $(\rho, W)$.
\end{corollary}
\begin{proof}
Note that $\dim B_H(\overline{W}, \overline{V})<+\infty$. Then applying the above result to $f=\overline{f}: \overline{W} \longrightarrow \overline{V}$,  we get $\overline{f}^{\ast}: V \longrightarrow W$.  Then the result follows from Cor.\ref{semiirrr}.
\end{proof}

        \begin{lemma}
            Keep the notations of Lmm.\ref{item44}.  If $(\pi, V)$ is an irreducible representation  and  $\dim B_H(W,V)=\infty$, then there exists an element $g\in B_{H}(W, V)$ such that $g^{\ast} (\overline{V}) \not\subseteq \overline{W}$.
            \end{lemma}
            \begin{proof}
            Let $\mathcal{W}_{\pi}$ denote the $(\pi, \mathcal{V})$-isotypic component of $(\rho, \mathcal{W})$.  Since $\mathcal{W}$ is a second-countable vector space and $\dim B_H(W, V) =+\infty$,   we have $\mathcal{W}_{\pi} \simeq \oplus_{i=1}^{\infty}\mathcal{V}_i$, with $\mathcal{V}_i\simeq \mathcal{V}$;  let $P_i$ be the  projection on  its  $i$-component $\mathcal{V}_i$.  Note that the restriction of $P_i$ to $W$ is non-trivial, and it is surjective onto $V$.  Clearly there exist  two exact sequences:
            $0 \longrightarrow \mathcal{W}[\pi] \longrightarrow \mathcal{W}  \stackrel{P=\sum_{i=1}^{\infty}P_i}{\longrightarrow }\mathcal{W}_{\pi} \simeq \oplus_{i=1}^{\infty}\mathcal{V} \longrightarrow 0,$
            and
            $0 \longrightarrow \mathcal{W}[\pi] \cap W\longrightarrow W \stackrel{P}{\longrightarrow }\mathcal{W}_{\pi}^{\infty}$. For a finite set $\{1, \cdots, l\}$, $P_l=\oplus_{i=1}^l p_i:\mathcal{W}  \longrightarrow \oplus_{i=1}^l \mathcal{V}_i$ is surjective.  By Lmm.\ref{item44}(2), for any $K$, $P_l(W^K)$ is dense in  $ [\oplus_{i=1}^l \mathcal{V}_i]^K=\oplus_{i=1}^l V_i^K$; the later vector space has finite dimension, so they are equal. Finally  the image of $P|_{W}$ contains $\sum_i V_i$.

           Now we define $g=\sum_{i=1}^{\infty} \frac{1}{2^{i}} P_i$. As $\Vert P_i\Vert\leq 1$,   $\Vert g\Vert\leq 1$, so $g\in B_H(\mathcal{W}, \mathcal{V})$. Note that $g$ factors through $\mathcal{W} \longrightarrow \mathcal{W}_{\pi} \simeq \oplus_{i=1}^{\infty} \mathcal{V}$, and $g\neq 0$. Hence $g: W \longrightarrow V$ is surjective, and it factors through $W \longrightarrow P(W)$. Let $K$ be an open compact subgroup of $H$ such that $V^{K}\neq 0$ with a linear orthonormal  base, say  $\{ h_1, \cdots, h_n\}$; let $h_{1, i}, \cdots, h_{n,i}$ be the corresponding  respective elements in the $i$-component $V$ of $\oplus_{i=1}^{\infty} V $.  For each $h_{j,i}$, let $e_{j,i}$ be one preimage of it in $W^K$. Then $g(e_{j,i})=g(h_{j,i})=\frac{1}{2^{i}}h_j \neq 0$.

               Now  assume $\{ e_1, \cdots, e_i, \cdots \}$ is  an orthonormal basis of $W^{K}$.  Then there exists infinite $i$'s such that  $g(e_i) \neq 0$.     Let us  write  $g(e_i)=\sum_{j} c_{ji} h_j$. Since $g: W^{K} \longrightarrow V^{K}$ is surjective, there exists $j\in \{1, \cdots, n\}$, such that  $c_{ji} \neq 0$, for infinite $i$'s.   Then for such $j$,
    $g^{\ast}(h_j)=\sum_{i} \langle g^{\ast}(h_j), e_i\rangle_{\overline{W}} \odot e_i= \sum_{i} \langle e_i, g^{\ast} (h_j)\rangle_W  \odot e_i = \sum_{i} \langle g(e_i),  h_j\rangle_W  \odot e_i =\sum_i c_{ji} \odot  e_i \notin \overline{W}$.
  \end{proof}

 We can let $(\rho_{semi}, W_{semi})$ be  the sum of all irreducible  subrepresentations of $(\rho, W)$. Then $(\rho_{semi}, W_{semi})$ is the maximal semi-simple sub-representation  of $(\rho, W)$.

    \begin{remark}\label{contablydiem}
 Assume the category $\Rep(H)$ is locally noetherian. Then  $W/W_{semi}$ has  no  irreducible subrepresentation.
  \end{remark}
  \begin{proof}
Assume that there exists    an irreducible $H$-module $\frac{W_1}{W_{semi}}$ of $\frac{W}{W_{semi}}$.  Let $p: W_1 \longrightarrow \frac{W_1}{W_{semi}}$ be the canonical  projection.  For any $u\in W_1$, with  $p(u)\neq 0$,  let $W_u$ denote the $H$-module generated by $u$. Then there exists a short exact sequence $0 \longrightarrow W_u \cap W_{semi} \longrightarrow W_u \stackrel{p}{\longrightarrow} \frac{W_1}{W_{semi}} \longrightarrow 0$. Now  $W_u \cap W_{semi} $ is finitely generated, and then it is admissible, semi-simple. Hence $W_u$ is admissible, and  semi-simple. So $W_u \subseteq W_{semi}$, a contradiction.
          \end{proof}

   Let $\mathcal{W}_{semi}$ be the completion of $W_{semi}$.  By the general theory on unitary representations of locally compact groups,  $ \mathcal{W}=\mathcal{W}_d\oplus \mathcal{W}_c$, for the   discrete component $\mathcal{W}_d$, and the  continuous component $\mathcal{W}_c$. Here  $\mathcal{W}_c$ has no irreducible subrepresentation.     The following results are straightforward.
    \begin{lemma}
    \begin{itemize}
    \item[(1)] There exists an orthonormal basis $\{e_1, \cdots, e_n, \cdots\}$ of $\mathcal{W}_{semi}$ such that  $e_i \in W_{semi}$, and  $\{ e_1, \cdots, e_n, \cdots \}$ forms an algebraic basis of $W_{semi}$.
    \item[(2)]  If $W_{semi}=\oplus_{i\in \mathbb{N}} V_i$, for $(\rho_i, V_i ) \in \Irr_u(H)$, with the completion $ (\rho_i, \mathcal{V}_i) \in \widehat{H}$,   then $\mathcal{W}_{semi}=\widehat{\oplus}_{i\in \mathbb{N}} \mathcal{V}_i$.
      \item[(3)] $\mathcal{W}_d^{\infty} \supseteq W_{semi}$.
      \item[(4)] $\mathcal{W}^{\infty}_c$ has no irreducible subrepresentation.
        \end{itemize}
    \end{lemma}
   \begin{proof}
For (4), if there exists an irreducible  subrepresentation $(\pi,V)$, then $V\hookrightarrow \mathcal{W}^{\infty}_c$ is a continuous map, and it will  induce an $H$-morphism on their completions, a contradiction.
      \end{proof}

\subsubsection{Unitary induced representation}
        Let us recall some results of unitary  induced  representations in \cite{Ma} (cf. \cite{KT}).  Let  $\delta_{H\setminus G}=\frac{\Delta_G}{\Delta_H}$.  Let $\nu_{H\setminus G}$ be  a  positive \emph{semi-invariant}  measure on $ H\setminus G$(cf. \cite[p.32]{BushH}).  In this  text, we define the  \emph{unitary induced representation} $(\Pi=\mathfrak{Ind}_{H}^G \rho, \mathcal{V}=\mathfrak{Ind}_{H}^G \mathcal{W})$ of $G$ as follows:

        Let $\cInd_H^G [\delta_{H\setminus G}^{1/2}\otimes \mathcal{W}]$ denote the space of  continuous functions $f$ on $G$ with values in $ \mathcal{W}$ having compact support modulo $H$, such that $f(hx)=\delta^{1/2}_{H\setminus G}(h)\rho(h)f(x)$ for $h\in H$, $x\in G$; let $ \mathcal{V}=\mathfrak{Ind}_{H}^G \mathcal{W}$ be the completion of  $\cInd_H^G [\delta_{H\setminus G}^{1/2} \otimes \mathcal{W}]$ under the  norm defined as $\Vert f\Vert^2=\int_{ H\setminus G} \Vert f(\dot{x})\Vert^2  d\nu_{H\setminus G} (\dot{x})$, for $f(x)\in \cInd_H^G [\delta_{H\setminus G}^{1/2} \otimes \mathcal{W}]$. The scalar product is given by $\langle f_1, f_2\rangle=\int_{H \setminus G}\langle f_1(\dot{x}), f_2(\dot{x})\rangle_W d\nu_{H\setminus G} (\dot{x})$, for $ f_1, f_2\in \cInd_H^G [\delta_{H\setminus G}^{1/2} \otimes  \mathcal{W}]$; the action of  $G$ on the space $ \mathcal{V}$  is given by right translation, i.e. $\Pi(g)f(x)=f(xg)$, for $x,g \in G$.    \footnote{The definition given above is a slight difference from  \cite{KT}, \cite{Ma} at   the action of  $G$ on the space $ \mathcal{V}$.}
\begin{remark}
One can  refer to \cite[Section 2.3]{KT}, \cite[Sections 2, 3]{Ma} for the exact description of the space $\mathcal{V}$ and its certain subspaces.
Loosely speaking,  $ \mathcal{V}$ can be viewed as a space of all classes of  measure functions $f$ from $G$ to $\mathcal{W}$, such that (1) $f(hx)=\delta^{1/2}_{H\setminus G}(h) f(x)$ for all $h\in H$, and  almost all $x\in G$;  (2)  $\Vert f\Vert<+ \infty$.
\end{remark}
       \begin{lemma}\label{compactindas}
 \begin{itemize}
 \item[(1)] $\cInd_H^G [\delta_{H\setminus G}^{1/2} \otimes W]$ is dense in $\cInd_H^G  [\delta_{H\setminus G}^{1/2}\otimes \mathcal{W}]$, and consequently it is dense  in $\mathcal{V}$;
 \item[(2)] If $G/H$ is compact, and $(\rho, W)$ is an admissible representation of $H$, then  $\cInd_H^G [\delta_{H\setminus G}^{1/2} \otimes W]$ is just the smooth part of $\mathcal{V}$.
  \end{itemize}
 \end{lemma}
\begin{proof}
1) For $f\in \cInd_H^G  [\delta_{H\setminus G}^{1/2} \otimes \mathcal{W}]$,  assume $\supp f \subseteq HK$, and $K \subseteq \cup_{j=1}^l y_j K_j$, for some  open  compact subgroups $K_j$ of $G$.  Let $M=\sum_{j=1}^l \int_{H\setminus [H y_jK_j]} \delta_{H\setminus G}(\dot{x}) d\nu_{H\setminus G}(\dot{x})$.  For any $\epsilon >0$, and $x\in K$, there exists an
open compact subgroup $K_x$ of $G$ such that $\Vert f(xk)-f(x)\Vert_W< \frac{\epsilon}{\sqrt{6M}}$ for  any  $k\in K_x$, and $xK_x\subseteq \cup_{j=1}^l y_jK_j$.  As  $K\subseteq \cup_{x\in K} [xK_x]$,  we can choose a finite subcover, say $\{ x_iK_{x_i}, i=1, \cdots, n\}$.

Note that $W$ is dense in $\mathcal{W}$, so there exists $v_i\in W$ such that $\Vert v_i-f(x_i)\Vert_W<  \frac{\epsilon}{\sqrt{6M}}$. For each $i$,  we assume $v_i \in \rho^{x_iJ_ix_i^{-1}\cap H, \delta_{H\setminus G}^{-1/2}}$, for an open compact subgroup $J_i\subseteq K_{x_i}$.  By Lmm.\ref{twocompactsubgroups}, we choose an open compact subgroup $K_{\epsilon}\subseteq \cap_{i=1}^n J_i$ satisfying $K_{\epsilon} \lhd K_{x_i}$ for  $i=1, \cdots, n$.

Let $\Delta=\{ s_1, \cdots, s_r\}$ be a subset of the complete  representatives  for $H\setminus G/K_{\epsilon}$ such that $HK \subseteq \cup_{t=1}^r Hs_t K_{\epsilon}$, and $Hs_t K_{\epsilon} \cap HK \neq\emptyset$. If  $Hs_t K_{\epsilon} \cap Hx_i K_{x_i} \neq \emptyset$, we can replace $s_t$ by $x_ik_{it}$, for some $k_{it} \in K_{x_i}$.  By reordering the index,  we assume $\Delta=\{ x_ik_{it}, i=1, \cdots, m; t=1, \cdots, n_i\}$ with $k_{it} \in K_{x_i}$ and $m\leq n$.

  Now we  define $f_{\epsilon} \in  \cInd_H^G [\delta_{H\setminus G}^{1/2} \otimes W]$ as follows: $\supp f_{\epsilon} \subseteq \sqcup_{i=1}^m\sqcup_{t=1}^{n_i}Hx_i k_{it} K_{\epsilon}$,  and $f_{\epsilon}(hx_i k_{it} k)=\delta_{H\setminus G}^{1/2}(h) \rho(h) v_i$ for $h\in H$, $k\in K_{\epsilon}$; here $v_i \in W^{[x_i K_{\epsilon}x_i^{-1} ]\cap H, \delta_{H\setminus G}^{-1/2}}=W^{[x_ik_{it} K_{\epsilon}k_{it}^{-1} x_i^{-1}] \cap H, \delta_{H\setminus G}^{-1/2}}$.  Moreover, for $hx_i k_{it} k \in Hx_i k_{it} K_{\epsilon}\subseteq  Hx_i K_{x_i}\subseteq \cup_{j=1}^l Hy_j K_j$, we have
  \[ \Vert f(x_i k_{it} k)-f_{\epsilon}(x_i k_{it} k)\Vert_W=\Vert f(x_i k_{it}k)- v_i\Vert_W\]\[\leq\Vert f(x_i k_{it}k)-  f(x_i)\Vert_W+\Vert  f(x_i)-v_i\Vert_W \leq \frac{2\epsilon}{ \sqrt{6M}}, \]

    \[\int_{H\setminus G} \Vert f (\dot{x})-f_{\epsilon}(\dot{x})\Vert_W^2 d\nu_{H\setminus G}(\dot{x})=\sum_{i=1}^m\sum_{t=1}^{n_i}\int_{H\setminus [Hx_i k_{it} K_{\epsilon}]}  \Vert f (\dot{x})-f_{\epsilon}(\dot{x})\Vert_W^2 d\nu_{H\setminus G}(\dot{x})\]
    \[\leq \sum_{i=1}^m\sum_{t=1}^{n_i} \int_{H\setminus [Hx_i k_{it} K_{\epsilon}]} \delta_{H\setminus G}(\dot{x}) d \nu_{H\setminus G}(\dot{x}) \  \sup_{k\in K_{\epsilon}  }\Vert f(x_i k_{it} k)-f_{\epsilon}(x_i k_{it} k)\Vert^2_W\]
     \[\leq \frac{2\epsilon^2}{3M}    \sum_{i=1}^m\sum_{t=1}^{n_i}  \int_{ H\setminus [Hx_i k_{it} K_{\epsilon}]} \delta_{H\setminus G}(\dot{x}) d\nu_{H\setminus G}(\dot{x})< \epsilon^2.\footnote{ Here $\delta_{H\setminus G}(\dot{x})(hx_i k_{it} k)= \delta_{H\setminus G}(h)$, for $h\in H$, $k\in K_{\epsilon}$.}\]

     2) The second statement is a corollary of Lemmas \ref{compactadm}, \ref{K1in}.
                   \end{proof}
 \begin{remark}\label{zerospace}
$V_0=\{f\in \cInd_H^G [\delta_{H\setminus G}^{1/2} \otimes W]\mid \Vert f\Vert=0\} $ is a zero vector space.
\end{remark}
\begin{proof}
For $f\in V_0$, assume it's $K$-invariant, and $\supp(f) \subseteq \sqcup_{i=1}^n Hg_i K$.  Then $$0=\int_{H\setminus G} \Vert f(\dot{x}) \Vert^2d\nu_{H\setminus G}(\dot{x})=\sum_{i=1}^n \int_{H\setminus Hg_i K}  \Vert f(\dot{x}) \Vert^2d\nu_{H\setminus G}(\dot{x})= \sum_{i=1}^n \Vert f(g_i)\Vert^2 \int_{H\setminus Hg_i K} \delta_{H\setminus G}(\dot{x})d\nu_{H\setminus G}(\dot{x}).$$ So all $f(g_i)=0$, and  $f=0$.
 \end{proof}
 \begin{example}
$\cInd_{1_G}^G 1$  is a preunitary representation, with the completion $L^2(G,  \nu_G)$, for a right Haar measure $\nu_G$ of $G$.
\end{example}
\begin{question}
How to compare  the induced topology on $\mathcal{H}(G)$  with  the topology on it  introduced  in \cite{BernD}.
\end{question}

\begin{remark}
For any admissible irreducible $(\pi, V)$ of $G$,  if $V^K \neq 0$, then by Frobenius reciprocity, $(\pi, V)\in \mathcal{R}_G(\cInd_{K}^G 1_K) $. Thus   $(\pi, V)$ can be  a quotient of a finite-generated preunitary representation, but  $(\pi, V)$ may not be a preunitary representation.
\end{remark}
\begin{lemma}
Keep the notations. If $(\pi, V)$ is a preunitary irreducible representation, and the map  $\cInd_{K}^G 1_K \longrightarrow V$ is continuous, then  $(\pi,V)$ of  $G$ can embed into $L^2(G, \nu_G)^{\infty}$ as  $G$-module.
\end{lemma}
\begin{proof}
    It is not hard to show that the canonical  map $\cInd_1^K 1 \longrightarrow 1_K$ is  continuous, and then  the map  $\cInd_{K}^G \cInd_1^K 1\longrightarrow \cInd_K^G 1_K$ is also  continuous(the norm definitions). By the algebraic and topological  isomorphism, $\cInd_{K}^G \cInd_1^{K} 1 \simeq \cInd_{1_G}^G 1$, we get a continuous  $G$-morphism $\cInd_{1_G}^G 1 \longrightarrow V$.  So the result follows from Lmm.\ref{item44}(4)(b).
        \end{proof}
        The above result is not always right for all irreducible preunitary representations, see  \cite[p.120, Coro.]{Ma}.   However    we can get an alternate result by  going into  $L^p$-space not just $L^2$-space.  These results will   not be used later.

For an open compact subgroup $K$ of $G$, let $\mu_K$ denote the normalized Haar measure of $K$, i.e. $\mu_K(K)=1$, and we always  choose a semi-invariant measure $\nu_{K\setminus G}$ such that $\int_G f(x) \Delta^{-1}_G(x) d\mu_G(x)=\int_{K\setminus G} f(\dot{x}) d\nu_{K\setminus G}(\dot{x})$, for any  left $K$-invariant $f(x) \in C_c^{\infty}(G)=C_c^{\infty}(G, \delta_{K\setminus G})$. Here $\mu_G$ is a fixed left Haar measure of $G$.   In the following lemma, we will treat $ \cInd_K^G 1_K$ as a topological subspace of $L^1(K\setminus G, \nu_{K\setminus G})$.
\begin{lemma}
Keep the notations.   If  $(\pi, V)$  is  a preunitary irreducible representation,  then  any non-zero $G$-morphism $f: \cInd_K^G 1_K \longrightarrow V$ is continuous.
\end{lemma}
\begin{proof}
1) Let $K_1$ be an open compact subgroup of $K$.  Note  that $\mathcal{H}(G, K_1)=[\cInd_{K_1}^G 1_{K_1}]^{K_1}$, which is a topological subspace of $L^1(K_1\setminus G, \nu_{K_1\setminus G})$. Set $\sigma_1=\cInd_{K_1}^G 1_{K_1} \subseteq L^1(K_1\setminus G, \nu_{K_1\setminus G})$. Firstly let us   show that $ \pi: \mathcal{H}(G, K_1)  \times V^{K_1} \longrightarrow V^{K_1}$ is continuous. For $g\in \mathcal{H}(G, K_1), v\in V^{K_1}$,  $$\pi(g) v=\int_G g(x) \pi(x) v d\mu_G(x)=\mu_G(K_1) \sum_{x\in G/K_1}  g(x) \pi(x)v.$$ Then
$$\Vert \pi(g)v\Vert_{\pi}\leq \mu_G(K_1)(\sum_{t\in G/K_1} \vert g(t)\vert \Vert \pi(t)v\Vert_{\pi})=\mu_G(K_1) (\sum_{t\in G/K_1} \vert g(t)\vert )\cdot \Vert v\Vert_{\pi}$$
$$=\int_{G} \vert g(x)\vert d\mu_{G} (x)\cdot \Vert v\Vert_{\pi}= \Vert v\Vert_{\pi}\int_{K_1\setminus G} \vert g(\dot{x})\vert \Delta_G(\dot{x})d\nu_{K_1\setminus G}(\dot{x})= \Vert v\Vert_{\pi} \cdot\Vert g \Delta\Vert_{\sigma_1}.$$

2) Secondly  set $\sigma=\cInd_K^G 1_K$. For $w\in [\cInd_K^G 1_K]^{K_1}\subseteq \mathcal{H}(G, K_1)$,  we have $w=\sigma(g) 1_K$, for certain right $K$-invariant $g\in \mathcal{H}(G, K_1)$.  Then $$\Vert w\Vert_{\sigma}=\Vert \sigma(g) 1_K\Vert_{\sigma}=  \int_{K\setminus G}  \vert [\sigma(g) 1_K](\dot{x})\vert d\nu_{K\setminus G} (\dot{x}) =\int_{K\setminus G}  \vert \int_{G} g(t) 1_K(\dot{x}t) d\mu_G(t)\vert  d\nu_{K\setminus G} (\dot{x}) $$
$$=\int_G \vert \int_{G} g(t) 1_K(xt) d\mu_G(t)\vert  \Delta_G^{-1}(x) d\mu_G (x) = \int_G \vert  \int_K g(x^{-1}t) d\mu_G(t) \vert \Delta_G^{-1}(x) d\mu_G(x)$$
$$= \mu_G(K) \int_G  \vert g(x^{-1} )\vert \Delta_G^{-1}(x) d\mu_G(x) =\mu_G(K) \int_G \vert g(x )\vert  d\mu_G(x)=\mu_G(K)\Vert g\Delta_G\Vert_{\sigma_1}.$$
So  $\Vert f(w)\Vert_{\pi}=\Vert \pi(g) f(1_K)\Vert_{\pi} \leq\Vert g\Delta_G \Vert_{\sigma_1} \Vert f(1_K)\Vert_{\pi}\leq \Vert w\Vert_{\sigma} \Vert f(1_K)\Vert_{\pi}\frac{1}{\mu_G(K)}$;  $f$ is continuous.
\end{proof}

By the knowledge of Functional Analysis, we can identify  $L^{\infty}(G, \nu_G)$ as the topological dual space of $L^1(G, \nu_G)$.  As before, let  $[L^{\infty}(G, \nu_G)]^{\infty}$ denote   the $G$-smooth part of  $L^{\infty}(G, \nu_G)$.
\begin{corollary}
Every irreducible preunitary representation $(\pi,V)$ of  $G$ can embed into $[L^{\infty}(G, \nu_G)]^{\infty}$ as $G$-module.
\end{corollary}
    \begin{proof}
Assume $(\pi, V)\in \mathcal{R}_G(\cInd_{K}^G 1_K) $. We treat $\cInd_1^K 1$ as a topological subspace of $L^1(K, \nu_K)$.   Then  the  canonical  map $\cInd_1^K 1 \longrightarrow 1_K$ is  continuous, and the map  $\cInd_{K}^G \cInd_1^K 1\longrightarrow \cInd_K^G 1_K$ is also  continuous.    It is not hard to show that the algebraic isomorphism $\cInd_{K}^G \cInd_1^{K} 1 \simeq \cInd_{1_G}^G 1$ is also  a homeomorphism.   Finally  we get a continuous $G$-morphism $\alpha: \cInd_{1_G}^G 1 \longrightarrow V$. For any $\overline{v} \in \overline{V}$, $g\in\cInd_{1_G}^G 1$, the map $g\longrightarrow \langle \alpha(g), \overline{v}\rangle $ is a continuous linear map.  Since $\cInd_{1_G}^G 1$ is dense in $L^1(G, \nu_G)$, by duality there exists a unique $\alpha^{\ast}_{\overline{v}}\in L^{\infty}(G, \nu_G)$, such that $\int_G  g(x)\alpha_{\overline{v}}^{\ast}(x) d\nu_G(x) =\langle \alpha(g), \overline{v}\rangle $.  Then  $\alpha^{\ast}: \overline{V} \longrightarrow \overline{L^{\infty}(G, \nu_G)}; \overline{v} \longrightarrow \alpha_{\overline{v}}^{\ast}$ is a well-defined, $\mathbb{C}$-linear, $G$-morphism. Hence $ \overline{V} \simeq \check{V} \hookrightarrow [\overline{L^{\infty}(G, \nu_G)}]^{\infty}$. Similarly, $V \hookrightarrow [L^{\infty}(G, \nu_G)]^{\infty}$.
  \end{proof}
    \subsubsection{Direct sum decompositions}

 Go back to the unitary induced representations.  Let $J$ be another closed subgroup of $G$. Let $\Delta=\{ s\in G\}$ be a  complete set of double coset representatives for  $H\setminus G/J$.  Assume the cardinality  of $\Delta$ is \emph{countable}.   For $s\in \Delta$,  let   $\mathcal{V}_s$ denote the space of all measure functions $f$ from $HsJ$ to $\mathcal{W}$ such that: (1) $f(hx)=\delta^{1/2}_{H\setminus G}(h)\rho(h) f(x)$, for all $h\in H$, and almost all $x\in HsJ$, (2) $\Vert f\Vert^2=\int_{H \setminus [HsJ]}\Vert f(\dot{x})\Vert_W^2 d\nu_{H\setminus G} (\dot{x})<+\infty$. Note that it is possible that $\Vert f\Vert=0$ for any $f\in \mathcal{V}_s$, or  $\mathcal{V}_s=0$;  now let $\Delta'$ be the subset of $\Delta$ by riding  of all those $s$.  Then $\mathcal{V}\simeq  \widehat{\oplus}_{s\in \Delta'} \mathcal{V}_s$ as $J$-modules.

  For a fixed $s\in \Delta'$, set $H_s=(s^{-1}Hs)$ and let $(\rho^s, \mathcal{W})$ denote the unitary representation of $H_s\cap J$.   Similar to  lemma \ref{homeo}, it can be shown that the canonical mapping $\iota_s: (H_s\cap J)\setminus J \longrightarrow H\setminus (HsJ); [H_s\cap J x] \longmapsto  [Hsx]  $,  is   homeomorphisc.
For   $f\in  \mathcal{V}_s  $, define a function $A_s(f)$  on $J$ as  $A_s(f)(h)=f(sh)$, for $h\in J$.  Note that  for $h_1\in H_s\cap J$, all almost $h\in J$, $$A_s(f)(h_1h)=f(sh_1h)=\delta^{1/2}_{H\setminus G}(sh_1s^{-1}) \rho(sh_1s^{-1})f( sh) =\delta^{1/2}_{H\setminus G}(sh_1s^{-1}) \rho(sh_1s^{-1})A_s(f)( h).$$
Let $\mathcal{U}_s$ denote the space of all functions $A_s(f)$ on $J$.  Then $\mathcal{U}_s \supseteq \cInd_{H_s\cap J }^J (\theta^{1/2}\otimes W)$, where $\theta^{1/2}(h_1)=\delta^{1/2}_{H\setminus G}(sh_1s^{-1})$ for $h_1\in H_s\cap J $. We endow a norm on $\mathcal{U}_s$ defined as $\Vert A_s(f)\Vert^2= \int_{H\setminus HsJ} \Vert f(\dot{x})\Vert_W^2 d\nu_{H\setminus G} (\dot{x})$.  Then it will induce a non-zero linear functional  $I_{\theta}$ on  $\cInd_{H_s\cap J}^{J} \theta$ satisfying the   two conditions in \cite[p.31, Coro.]{BushH} for $H_s\cap J \setminus J$. Hence corresponding to  $\theta$, there exists  a positive semi-invariant measure on $H_s\cap J \setminus J$,  denoted by  $\nu_{H_s\cap J \setminus J}$, such that
$\Vert A_s(f)\Vert^2=\int_{H_s\cap J\setminus J}  \Vert A_s(f)(\dot{h})\Vert_W^2d\nu_{H_s\cap J\setminus J}(\dot{h})$.  If   the action of $J$ on $\mathcal{U}_s$ is given by right translation, we  indeed  obtain  the   unitary representation $(\mathfrak{Ind}_{H_s\cap J}^{J}\rho^s, \mathcal{U}_s)$ of $J$ induced from  $(\rho^s, \mathcal{W})$. As a consequence,    we obtain
\begin{lemma}[{\cite[p.116, Lmm.6.1]{Ma}}]\label{UnINd}
$\Res_{J}^G \mathfrak{Ind}_{H}^G \mathcal{W} \simeq \widehat{\oplus}_{s\in \Delta'} \mathfrak{Ind}_{H_s\cap J}^{J}\rho^s$(unitary equivalence).
\end{lemma}
\begin{example}
Let  $G=\GL_2(F)\supseteq B=\{\begin{pmatrix} t_1 & n\\ 0& t_2\end{pmatrix}\} \supseteq T=\{ \begin{pmatrix} t_1 & 0 \\ 0& t_2\end{pmatrix} \}\supseteq \omega=\begin{pmatrix} 0 & 1 \\ 1&0\end{pmatrix}$, $\Delta=\{1, \omega\}$.  Consider  $(\rho, W)=$the trivial representation of $B$,  $H=J=B$,  $\delta_{H\setminus G}(g)=\Delta_B^{-1}(g)= \Vert \frac{t_1}{t_2}\Vert_F$ for $g=\begin{pmatrix} t_1 & n \\ 0& t_2\end{pmatrix}\in B$.  Then  $\Res_{B}^G\mathfrak{Ind}_{B}^G \rho \simeq  \mathfrak{Ind}_{T}^{B}\rho^{\omega}$. (Notice that not all irreducible representations of $B$  are admissible.)
\end{example}
\subsubsection{}
  Let us now consider $J=H$; assume $G/H$ is compact,  and $1\in \Delta$. We  want to get the similar result  analogue of Lmm.\ref{ddf2}. Let us  first present some lemmas for later use. Recall that  $\mu_H$   stands for  a left Haar measure of $H$.

\begin{lemma}
There exists  a locally constant  left (resp. right) rho-function $\rho_{H\setminus G}$ (resp. $\rho_{G/H}$) for $(G,H)$ such that it is  everywhere strictly positive on $G$,  $\rho_{H  \setminus G}(h^{-1}x)=\tfrac{\Delta_G(h)}{\Delta_H(h)} \rho_{H\setminus G}(x)$  (resp. $\rho_{G/H}(xh^{-1})=\tfrac{\Delta_G(h)}{\Delta_H(h)} \rho_{G/H}(x)$),   for $h\in H$, $x\in G$.
\end{lemma}
\begin{proof}
Without loss of generality, we will only show the existence of $\rho_{H\setminus G}$.  Now let $p: G \longrightarrow H\setminus G$ be  the canonical projection.    It is known that $H\setminus G$ is paracompact.  For an open compact subgroup $U$ of $G$,  $\{ p(xU)\}_{x\in G}$  forms a family of   open compact subset coverings  of $H\setminus G$.  Let  $\{V_i\}$ be   a locally  finite open-compact refinements of this covering.  For each $i$, $p^{-1}(V_i)$ is an open subset of $G$ with  an open-compact  subset covering, say $\{ W_{ij}\}$. Clearly $\{p(W_{ij})\}$ covers $V_i$ and  has a finite subcover $\{V_{ij}\}_{j=1}^m$. Let us write $W_i=\cup_{j=1}^m W_{ij}$. Then $W_i$ is an open compact set of $G$, and $p(W_i)=V_i$.

Let $g_i$ denote the characteristic function of $W_i$, a locally constant function.  Note that for $x\in G$, there is  at most  a  finite number of $i$ such that $g_i(x)\neq 0$.  We now set $g=\sum g_i$.  For any open compact set $K$ of $G$, $p(HK)$ is  compact and intersects with only a finite number of $V_i$'s, say $V_1, \cdots, V_n$. Then $HK\cap \supp(g) $ belongs to $\cup_{i=1}^n W_i$, and  it is a compact set.   Now we define $\rho_{H\setminus G}(x)=\int_H \frac{\Delta_G(h)}{\Delta_H(h)} g(hx) \Delta^{-1}_{H}(h)d\mu_H(h)$.
It is not hard to show that  $\rho_{H\setminus G}$ satisfies all the required conditions except for the locally constant condition.  Note that for the above $K$,   $HK\cap \supp(g)$ only  intersects with  $W_1, \cdots, W_n$. Then $\rho_{H\setminus G}(x)= \sum_{i=1}^n\int_H \frac{\Delta_G(h)}{\Delta_H(h)}g_i(hx)\Delta^{-1}_H(h)d\mu_H(h)$, for $x\in K$, so  $\rho_{H\setminus G}$ is locally constant at $K$.
\end{proof}
By following the above proof, we can also show that there exists  a  left-right rho-function $\rho_{H\setminus G/H}$, which is locally  constant  and  everywhere strictly positive on $G$.
 \begin{corollary}\label{tranun}
 $\Hom_{G} (\cInd_{H}^G\sigma_1, \cInd_{H}^G \sigma_2)\simeq  \Hom_G(\cInd_{H}^G(\delta_{H\setminus G}^{1/2}\otimes \sigma_1), \cInd_{H}^G (\delta_{H\setminus G}^{1/2}\otimes \sigma_2))$, for two smooth representations  $(\sigma_1, W_1)$, $ (\sigma_2, W_2)$  of $H$.
   \end{corollary}
\begin{proof}
By Frobenius reciprocity, $\Hom_{G} (\cInd_{H}^G\sigma_1, \cInd_{H}^G \sigma_2) \simeq \Hom_{H}(\cInd_{H}^G\sigma_1,  \sigma_2)$, and $\Hom_G(\cInd_{H}^G(\delta_{H\setminus G}^{1/2}\otimes \sigma_1), \cInd_{H}^G (\delta_{H\setminus G}^{1/2}\otimes \sigma_2)) \simeq \Hom_{H}(\cInd_{H}^G(\delta_{H\setminus G}^{1/2}\otimes\sigma_1),  \delta_{H\setminus G}^{1/2}\otimes\sigma_2)$.  So it reduces to show the above  two right-hand  $\Hom_H$-vector spaces are isomorphic. For $f\in \cInd_H^G W_1$,  it can be shown that $\rho^{-1/2}_{H\setminus G/H} f \in \cInd_{H}^G(\delta_{H\setminus G}^{1/2}\otimes W_1)$. Then  the isomorphism can be given by  $\varphi \longrightarrow  \rho^{-1/2}_{H\setminus G/H} \varphi(\rho^{1/2}_{H\setminus G/H} -)$, for $\varphi \in \Hom_{H} (\cInd_{H}^G\sigma_1,\sigma_2)$.
 \end{proof}

 Recall  that  a quasi-invariant measure  on $H\setminus G$ is  a regular  Borel (real) measure $\mu$ on $ H\setminus G$ such that for a Borel subset $[X]$ of $H\setminus G$,   $\mu([X])=0$ iff $\mu([X]g)=0$ for any  $g\in G$.

\begin{lemma}
Keep the notations,  $\rho^{-1}_{H\setminus G} \nu_{H\setminus G} $  defines a quasi-invariant measure on $H\setminus G$.
\end{lemma}
\begin{proof}
Let $C_c(H\setminus G)$ denote the space of continuous functions on  $H\setminus G$ with compact support, provided with the topology of uniform convergence. Then  $C_c^{\infty}(H\setminus G)$, the underlying set of $\cInd_H^G 1$, is dense in $C_c(H\setminus G)$.   Denote $\theta=\delta_{H\setminus G}$. Through  the bijective  mapping  $ \iota: C^{\infty}_c(H\setminus G) \longrightarrow C_c^{\infty}(H\setminus G, \delta_{H\setminus G}); f \longrightarrow \rho_{H\setminus G}^{-1}f $, we obtain  a non-zero positive  linear functional $I_{\theta}\circ \iota$ on $ C^{\infty}_c(H\setminus G)$, which is invariant under the right translation of $G$. By Risez's theorem, $ I_{\theta}\circ \iota(f)=\int_{H\setminus G} f(\dot{x})\rho_{H\setminus G}^{-1}(\dot{x}) d\nu_{H\setminus G}(\dot{x})$, for $f\in C_c^{\infty}(H\setminus G)$.(cf. \cite[pp. 30-31]{BushH})
\end{proof}
  For  $f\in C_c^{\infty}(G)$, let $f_H(g)=\int_{H}\delta^{-1}_{H\setminus G}(h) f(hg)\Delta^{-1}_H(h)d\mu_H(h)=\int_{H} f(hg) \Delta_G^{-1}(h) d\mu_H(h)$; then $f_H\in C_c^{\infty}(H\setminus G, \delta_{H\setminus G})$.
  \begin{lemma}\label{triples}
There exists a triple  $(\mu_H, \mu_G, \nu_{H\setminus G})$ such that  $$\int_G f(g)\Delta_G^{-1}(g)d\mu_G(g)=\int_{H\setminus G}d\nu_{H\setminus G}(\dot{x})  \int_{H} f(hx) \Delta_G^{-1}(h) d\mu_H(h), \quad \quad f\in C_c^{\infty}(G).$$
\end{lemma}
\begin{proof}
Note that the right-hand side defines a right $G$-invariant  $\mathbb{C}$-linear map on $C_c^{\infty}( G)$, so we can find such $\mu_G$ satisfying the condition.
\end{proof}
In the following, we will always fix one  such triple.
\begin{lemma}\label{decom}
\begin{itemize}
\item[(1)]  For any open compact non-zero subset $K$ of $G$,  $\int_{H\setminus HK} \rho_{H\setminus G}^{-1}(\dot{x}) \nu_{H\setminus G} (\dot{x}) \neq 0$;
\item[(2)] In Lmm.\ref{UnINd}, if  let   $J$ just be   the above $K$,  then the set $\Delta'=\Delta$.
\end{itemize}
\end{lemma}
\begin{proof}
1) Let $\mu$ denote the measure $\rho^{-1}_{H\setminus G} \nu_{H\setminus G}$ on $H\setminus G$. Assume the converse.  Then for some  open compact subset $K_1$ of $K$, $\mu([H\setminus HK_1])=0$; $\mu([H\setminus HK_1x])=0$ for any $x\in G$;  $\mu([H\setminus HC])=0 $ for any compact set $C$ of $G$. Since $\mu$ is a regular measure, finally we see that  $\mu$ is the  zero measure on $H\setminus G$, a contradiction! \\
2) Keep the notations of  the lemma \ref{UnINd}.  Let  $\Delta_{s, H, K}$ be a measure section of  $HsK$ with respect to   $H$(cf. \cite[Lmm.1.1]{Ma}). For one  $0\neq w\in W $, we define $f_w(hx)=\rho^{-1/2}_{H\setminus G}(x) \delta_{H\setminus G}^{1/2}(h)  \rho(h)w$, for $h\in H$,  $x\in \Delta_{s, H, K}$.  Then  $f_w$ is a measure function from $HsK$ to $\mathcal{W}$, and satisfies the first condition on the definition of $\mathcal{V}_s$.  Moreover $\int_{H\setminus HsK} \Vert f_w(\dot{x})\Vert^2\nu_{H\setminus G}(\dot{x})= \Vert w\Vert^2 \int_{H\setminus HsK} \rho^{-1}_{H\setminus G} (\dot{x}) \nu_{H\setminus G}(\dot{x}) \neq 0$. Hence $f_w\in \mathcal{V}_s \neq 0$.
\end{proof}
\subsubsection{}
Keep the assumption  that $G/H$ is compact. Assume now the category $\Rep(H)$ is locally noetherian;  for any open compact subgroup $K_1$ of $H$, assume $\mathcal{H}(H,K_1)$ is an algebra which can be  generated by  $\epsilon_{K_1}$ and a finite number of $\epsilon_{x_1}$, $\cdots$, $\epsilon_{x_n}$.
\begin{lemma}\label{comm}
  $\Hom_G\big(\cInd_{H}^G(\delta_{H\setminus G}^{1/2}\otimes \sigma_1), \cInd_{H}^G (\delta_{H\setminus G}^{1/2}\otimes \sigma_2)\big)  \simeq \Hom_H\big(\delta_{H\setminus G}^{-1/2}\otimes \sigma_1, \cInd_{H}^G (\delta_{H\setminus G}^{1/2}\otimes \sigma_2)\big)$, for an  admissible representation  $(\sigma_1, W_1)$   of $H$,   and  an irreducible  preunitary  representation $(\sigma_2,W_2)$ of $H$.
\end{lemma}
\begin{proof}
By Lemmas \ref{compactadm}, \ref{semisimpleu}, $ \cInd_{H}^G(\delta_{H\setminus G}^{1/2}\otimes \sigma_2)$ is an admissible preunitary semisimple representation. Assume
$\cInd_{H}^G(\delta_{H\setminus G}^{1/2}\otimes \sigma_2) \simeq \oplus_{i \in I }  m_i\pi_i$, for different $(\pi_i, V_i)\in \Irr(G)$, and positive  integers $m_i$. By Remark \ref{notcompactfrb} there exists $$ \alpha_i: \Hom_H\big(\delta_{H\setminus G}^{-1/2}\otimes \sigma_1,  \pi_i\big)  \simeq \Hom_G\big(\cInd_{H}^G(\delta_{H\setminus G}^{1/2}\otimes \sigma_1), \pi_i \big). $$

   Let  $f$ be a $K$-invariant vector  in $\cInd_{H}^G(\delta_{H\setminus G}^{1/2}\otimes W_1)$, and assume $H\setminus G/ K=\sqcup_{t=1}^l Hg_t K$.  Assume $V_i^{[g_t K g_t^{-1}] \cap H}=\{v_{it1}, \cdots, v_{itr_t}\}$,  $K_i\lhd[\cap_{j,k} \Stab_{G}(g_t^{-1} v_{ijk})\cap K]$,  and $Hg_tK=\sqcup_{j=1}^{n_i}Hg_t a_{ij}K_i=\sqcup_{j=1}^{n_i}Hg_tK_i a_{ij}$. By the discussion in \cite[p. 24]{BernZ}, for $A_i \in\Hom_H\big(\delta_{H\setminus G}^{-1/2}\otimes \sigma_1,  \pi_i\big) $, $\alpha_i$ can be given as follows:
   \begin{align*}
   [\alpha_i(A_i)](f) &=\int_{H\setminus G} \pi_i(g^{-1}) [A_if(g)] d\nu_{H\setminus G}(g)\\
 & =\sum_{t=1}^l\sum_{j=1}^{n_i} \pi_i(a_{ij}^{-1} ) [A_if(g_t)] \int_{H\setminus Hg_tK_ia_{ij}} \delta^j_{H\setminus G} d\nu_{H\setminus G}\\
  &=\sum_{t=1}^l\sum_{j=1}^{n_i} \pi_i(a_{ij}^{-1} ) [A_if(g_t)] \int_{H\setminus Hg_tK_i} \delta^0_{H\setminus G} d\nu_{H\setminus G}.
  \end{align*}
  Here,  $\delta_{H\setminus G}^j, \delta_{H\setminus G}^0 \in C_c^{\infty}(H\setminus G, \delta_{H\setminus G})$, $\delta^j_{H\setminus G} (hg_t a_{ij} k)=\delta_{H\setminus G}(h)$, $\delta^0_{H\setminus G} (hg_t  k)=\delta_{H\setminus G}(h)$, for $h\in H$, $k\in K$, and  $f(g_t)  \in [\delta_{H\setminus G}^{1/2}\otimes W_1]^{[g_t K g_t^{-1}] \cap H}$, only dependent on $f$, $K$, $g_t$. Note that there exists
\begin{align*}
 \alpha: \Hom_H\big(\delta_{H\setminus G}^{-1/2}\otimes \sigma_1, \cInd_{H}^G (\delta_{H\setminus G}^{1/2}\otimes \sigma_2)\big)& \hookrightarrow \prod_{i\in I} m_i  \Hom_H\big(\delta_{H\setminus G}^{-1/2}\otimes \sigma_1, \pi_i\big) \\
 & \simeq \Hom_G\big(\cInd_{H}^G(\delta_{H\setminus G}^{1/2}\otimes \sigma_1), \prod_{i\in I} m_i\pi_i \big).
 \end{align*}

  Let  $A\in \Hom_H\big(\delta_{H\setminus G}^{-1/2}\otimes \sigma_1, \cInd_{H}^G (\delta_{H\setminus G}^{1/2}\otimes \sigma_2)\big)$ with the projection $\oplus_{j=1}^{m_i}A_{ij}$  in $\Hom_H\big(\delta_{H\setminus G}^{-1/2}\otimes \sigma_1, m_i \pi_i\big)$;  since $\delta_{H\setminus G}^{-1/2}\otimes \sigma_1$ is admissible, for any open compact subgroup $K_H$ of $H$, $A_{ij}\big( [\delta_{H\setminus G}^{-1/2}\otimes W_1]^{K_H}\big)=0$, for almost all $i$. Therefore $[\alpha(A)](f)=\prod_{i, j}\int_{H\setminus G} \pi_i(g^{-1}) [A_{ij}f(g)] d\nu_{H\setminus G}(g)=\int_{H\setminus G} \sum_i m_i\pi_i(g^{-1}) [Af(g)] d\nu_{H\setminus G}(g)$, i.e. $\alpha$ gives the required isomorphism.
        \end{proof}
    \begin{corollary}
    Assume all  irreducible representations of $H$, $G$ are admissible, and $G/H$ is compact;  then $\delta_{H\setminus G}=1$. \footnote{ We follow the notations of \cite[p.44]{BernZ}. For the  parabolic subgroup $P_n$ of $\GL_n(F)$,  since  $\delta_{P_n}$ is non-trivial,  it always exists a non-admissible irreducible  smooth representation of $P_n$. (cf. Remark \ref{notcompactfrb},  \cite[p.51, 5.22 Coro.]{BernZ}).  Question: does the result also hold for  the other parabolic groups? (Rodier + Bernstein+ Zelevinsky?sufficient?)}
    \end{corollary}
    \begin{proof}
    We take the above $\sigma_1=\sigma_2=$ the trivial representation of $H$.  Then $0\neq m_H\big(\delta_{H\setminus G}^{-1/2}\otimes \sigma_1, \cInd_{H}^G (\delta_{H\setminus G}^{1/2}\otimes \sigma_2)\big)$.  So $\delta_{H\setminus G}^{-1/2}\otimes \sigma_1$ is also a preunitary representation. Hence $\delta_{H\setminus G}=1$.            \end{proof}

Let  $\Delta=\{s_i\in G\}_{i\in I}$ be a complete  set of representatives for  $H\setminus G/H$;  assume $1\in \Delta$, and $\Delta$ is a countable set. Let $H_s=s^{-1}Hs$. For $(\sigma, W)\in \Rep(H)$, set $\sigma^{s}(x)= \sigma(sxs^{-1})$, $x\in H_s\cap H$.     For any $s\in \Delta$, $s\neq 1$,  assume that   $H_s\cap H$ is a normal subgroup of $H$ and   $\frac{H}{H_s\cap H}$ is not compact. Recall the notation   $\mathcal{N}(K)_n$ in Lmm.\ref{ddf2}.
\begin{lemma}\label{ddff2}
If for any $1\neq s\in \Delta$,   assume: (1) up to $H_s\cap H$-conjugacy there  exists at least one and at most  a finite number of maximal open compact subgroups in $H$,  (2)  for  each maximal open compact subgroup $K$ of $ H$, and each $n$, the set $\mathcal{N}(K)_n$ is finite,  then  $m_G(\cInd_{H}^G \sigma_1, \cInd_{H}^G \sigma_2)\leq m_H(\sigma_1, \sigma_2)$, for an  admissible representation  $(\sigma_1, W_1)$   of $H$,    an admissible preunitary  representation $(\sigma_2,W_2)$ of $H$.
\end{lemma}
\begin{proof}
By Lemmas \ref{simlem}, \ref{UnINd}, \ref{comm}, \begin{align*}
\Hom_G(\cInd_{H}^G \sigma_1, \cInd_{H}^G \sigma_2)\simeq  \Hom_H(\sigma_1, \cInd_{H}^G  \sigma_2)\\
\hookrightarrow \Hom_H(\sigma_1,\mathfrak{Ind}_{ H}^G \sigma_2)
\hookrightarrow  \prod_{s\in \Delta'}\Hom_H(\sigma_1,\mathfrak{Ind}_{H_s\cap H}^H(\sigma_2)^s).
\end{align*}

Now let  us choose $\{K_1, \cdots, K_m\}$ to  be a total set of maximal  open compact subgroups of $ H$,  up to $H_s\cap H$-conjugacy.   Let $K$ be an open compact subgroup of $ H$, such that $W_1^K\neq 0$. By Lmm.\ref{twocompactsubgroups}, we assume that $K$ is a normal subgroup of  each  $K_i$.  For a fixed $s\in \Delta'$ with $s\neq 1$,      let $\Sigma_s$ be a complete set of representatives for $ H_s\cap H\setminus H/ K$.  Since $H$  is $\sigma$-compact(cf. Section \ref{notation}), the cardinality of $\Sigma_s$ is denumerable.  For  simplicity, write  $\tau$ for $(\sigma_2)^s$.  Assume $0\neq B \in \Hom_H( \sigma_1, \mathfrak{Ind}_{H_s\cap H}^H\tau)$. For simplicity,  assume $B(W_1^{K})\neq 0$.

Under the condition (2) we  let  $\mathcal{L}_i$ denote  the total set of  normal  open compact subgroups $L_i$ of  $K_i$, satisfying  $[K_{i} : L_{i}]=[K_{1}: K]$, and let $\mathcal{L}=\cup_{i}   \mathcal{L}_i$.    For a fixed   $t\in \Sigma_s^{-1}=\{ r^{-1}\mid r\in \Sigma_s\}$,   there exists $h_t\in H_s\cap H$, such that $K_{t}=t^{-1} Kt\subseteq (K_1)_t =h_t K_j h_t^{-1}$, for certain $j$.     So $K_t \lhd (K_1)_t=(K_j)_{h^{-1}_t}$, $K_{th_t} \lhd K_j$, and $[K_j : K_{th_t} ]=[(K_j)_{h^{-1}_t} : K_t]=[(K_1)_t: K_t  ]=[K_1 : K]$.  Hence $K_{th_t} =L_t$, for certain $L_t\in \mathcal{L}$.  Set $D_t=K_{th_{t}}\cap K=L_t\cap K$. Then    $\epsilon_{D_t h^{-1}_t t^{-1} K} \in \mathcal{H}(H, D_t)$.  For $0\neq w\in W_1^{K}$,  $B(\epsilon_{D_t h^{-1}_t t^{-1} K} w)=B(\epsilon_{h_t^{-1}t^{-1}} \ast \epsilon_{th_tD_th^{-1}_tt^{-1}} \ast \epsilon_Kw)=h_t^{-1}t^{-1}B(w)$.  Moreover $0\neq \epsilon_{D_{t} h_t^{-1}t^{-1} K} w\in W_1^{D_{t}}$.  Now let $\widetilde{W}_1=\sum_{L\in \mathcal{L}} W_1^{L\cap K}  \subseteq W_1$; then   $\widetilde{W}_1$ has finite dimension,  and  $W_1^K \subseteq \widetilde{W}_1 $,  $W_1^{D_t} \subseteq \widetilde{W}_1$. Replacing  $th_t $ by $t$, we may assume $K_t \lhd K_j$ for some $j$, and $K_t\in \mathcal{L}$. Let us choose an open compact subgroup  $K_0\subseteq \cap_{L\in \mathcal{L}} L$ such that  $K_0\lhd K$, $K_0\lhd K_t$. Notice that for $t\in \Sigma_s$, $K_0 \lhd K_{t^{-1}}$.  Let  $m=\max_{L \in \mathcal{L}} [L: K_0]$.

Assume that $\{e_1=B(w_1), \cdots, e_n=B(w_n)\}$ forms an orthonormal  basis of $B( \widetilde{W}_1)$. By Lemmas \ref{UnINd}, \ref{decom}, there exists a unitary equivalence  $A=\widehat{\oplus}_{r\in \Sigma_s} A_s: \mathfrak{Ind}_{H_s\cap H}^H\tau\simeq \widehat{\oplus}_{r\in \Sigma_s} \mathfrak{Ind}_{(H_{s}\cap H)\cap K}^{K} \tau^r$.  Then $A(e_i)=\sum_{r\in \Sigma_s} e_{ir}$, for some $e_{ir} \in \mathfrak{Ind}_{(H_{s}\cap H)\cap K}^{K} \tau^r$.

     Choose  $0\neq w_0\in W_1^{K}$ such that  $B(w_0) =v_0\neq 0$, assume  $v_0=\sum_{i=1}^{n} c_i e_i$, with $ c_i=\langle v_0, e_i\rangle$ and  $\Vert v_0\Vert^2=\sum_{i=1}^{n}|c_i|^2 $.  Note that $A(v_0)=\sum_{r\in \Sigma_s}  v_{0r}= \sum_{r\in \Sigma_s} \sum_{i=1}^n c_i e_{ir}$. Let $m\epsilon^2=\Vert v_{0r_0} \Vert^2=\Vert \sum_{ i=1}^n  c_i e_{ir_0}\Vert^2>0$, for some $r_0\in \Sigma_s$.  For such $\epsilon$, there exists a finite subset $\delta\subseteq \Sigma_s$ such that $\sum_{r\notin \delta}  \Vert e_{ir}\Vert^2 < \frac{\epsilon^2}{n\Vert v_0\Vert^2}$, for each $i=1, \cdots, n$.    For  $l \in \Sigma_s\setminus \delta$,   $\cup_{t\in \Sigma_s}(H_s\cap H)lK (H_s\cap H)t^{-1}K=\cup_{t\in \Sigma_s}(H_s\cap H)lK t^{-1}(H_s\cap H)K=(H_s\cap H)l[\cup_{t\in \Sigma_s}K t^{-1}(H_s\cap H)]K= (H_s\cap H)lHK=H$.
      So there exist $t\in \Sigma_{s}$, $l \in \Sigma_s\setminus \delta$, such that $(H_s\cap H)lK (H_s\cap H)t^{-1}K=\sqcup_{j=0}^{n_0} (H_s\cap H)r_j K\supseteq (H_s\cap H)r_0 K$.

   Assume $v_{t}=B(\epsilon_{D_t  t^{-1} K}  w_0)= \sum_{i=1}^n d_{ti} e_i$ with $d_{ti}=\langle v_{t}, e_i\rangle \in \mathbb{C}$ and  $\sum_{i=1}^n |d_{ti}|^2= \Vert v_{t}\Vert^2$.  On  the other hand, $B(\epsilon_{D_{t} t^{-1} K}  w_0)=  t^{-1}v_0=\sum_{i=1}^n c_i t^{-1} e_i$, and  $\Vert v_{t}\Vert^2=\Vert v_0\Vert^2$.  Assume $A(v_t)=\sum_{r\in \Sigma_s} v_{tr}$. Then
   \begin{align*}
   \sum_{r\notin \delta} \Vert v_{tr}\Vert^2=\sum_{r\notin \delta} \Vert d_{t1} e_{1r}+\cdots + d_{tn}e_{nr}\Vert^2
   \leq\sum_{r\notin \delta} (\sum_{i=1}^n \vert d_{ti}\vert^2)(\sum_{i=1}^n \vert e_{ir}\vert^2)
   \\
   \leq \Vert v_0\Vert^2\sum_{i=1}^n ( \sum_{r\notin \delta}  \Vert e_{ir}\Vert^2 ) <\epsilon^2.
   \end{align*}
 For each $1\neq s\in \Delta$, we will fix a triple $( \mu_{H_s\cap H}, \mu_{H}, \nu_{(H_s\cap H)\setminus H})$ as in Lmm.\ref{triples}.  For $k\in K$, if we write $A( t^{-1} v_0)= A(t^{-1}kv_0)=\sum_{r\in \Sigma_s} f_{rk}$, then
 \begin{align*}
 \mu_{H}(K)\sum_{r\notin \delta} \Vert v_{tr}\Vert^2  &  = \mu_{H}(K) \sum_{r\notin \delta}  \Vert f_{rk}\Vert^2  \geq  \mu_{H}(K) \Vert f_{lk}\Vert^2 \\
  & =\mu_{H}(K)\int_{(H_s\cap H)\setminus (H_s\cap H)lK} \Vert \sum_{i=1}^n c_i t^{-1}ke_i(\dot{x}) \Vert^2d\nu_{(H_s\cap H)\setminus H}(\dot{x})\\
  & =\mu_{H}(K)\int_{(H_s\cap H)\setminus (H_s\cap H)lKt^{-1}} \Vert \sum_{i=1}^n c_i e_i(\dot{x}) \Vert^2d\nu_{(H_s\cap H)\setminus H}(\dot{x})\\
  & \geq \frac{\mu_{H}(K)}{m}  \int_{(H_s\cap H)\setminus (H_s\cap H)lKt^{-1}K} \Vert  \sum_{i=1}^n c_i e_i(\dot{x}) \Vert^2d\nu_{(H_s\cap H)\setminus H}(\dot{x}) \quad (\textrm{the next lemma \ref{lllll}} )\\
  & =\frac{\mu_{H}(K)}{m}  \sum_{j=0}^{n_0} \Vert \sum_{i=1}^nc_ie_{ir_{j}}\Vert^2\geq \frac{\mu_{H}(K)}{m}  \Vert \sum_{i=1}^nc_ie_{ir_{0}}\Vert^2 =\mu_H(K)\epsilon^2.
  \end{align*}
  This makes a contradiction!  Therefore $\Hom_H( \sigma_1,\mathfrak{Ind}_{H_s\cap H}^H(\sigma_2)^s)=0$, for any $1 \neq s\in \Delta'$;  hence  the    result holds.
  \end{proof}
\begin{lemma}\label{lllll}
Keep the above notations.
\begin{itemize}
\item[(1)] $\mu_H(K)=\mu_{H_s\cap H}(K\cap (H_s\cap H)) \nu_{H_s\cap H \setminus H}(\frac{(H_s\cap H) K}{ H_s\cap H})$, for any open compact subgroup $K$ of $H$.
\item[(2)]  Let $C=lKt^{-1}$ be an open compact subset of $H$.   Then for any $K$-right invariant $f(\dot{x}) \in C_c^{\infty}(\frac{H}{H_s\cap H})$, we  have
$\mu_{H}(K)\int_{\frac{(H_s\cap H)C}{H_s\cap H}} \vert f(\dot{x})\vert   \nu_{(H_s\cap H)\setminus H}(\dot{x})\geq \frac{\mu_{H}(K)}{m}\int_{\frac{(H_s\cap H)CK}{H_s\cap H}} \vert  f(\dot{x}) \vert  \nu_{(H_s\cap H)\setminus H}(\dot{x})$.
\end{itemize}
\end{lemma}
\begin{proof}
1) Since $H_s\cap H \lhd H$, we may assume $\nu_{(H_s\cap H)\setminus H}=\nu_{\frac{H}{H_s\cap H}}$, a right Haar measure.   Then
$$\mu_H(K)  = \int_{H} 1_{K}(x)  d\mu_H(x)=\int_{H} 1_{K}(x) \Delta_{H}(x)^{-1} d\mu_H(x)$$
  $$= \int_{\frac{H}{H_s\cap H}}  d\nu_{\frac{H}{H_s\cap H}}(\dot{x}) \int_{H_s\cap H} 1_K(hx) \Delta_{H}(h)^{-1}d\mu_{H_s\cap H}(h)$$
  $$= \int_{\frac{(H_s\cap H)K}{H_s\cap H}}  d\nu_{\frac{H}{H_s\cap H}}(\dot{x}) \int_{K\cap (H_s\cap H)} 1_K(hx) d\mu_{H_s\cap H}(h)$$
  $$= \nu_{\frac{H}{H_s\cap H}}(\frac{(H_s\cap H) K}{ H_s\cap H})\mu_{H_s\cap H}(K\cap (H_s\cap H)).$$
2) Assume $\frac{(H_s\cap H )C}{H_s\cap H}=\frac{(H_s\cap H )lt^{-1}K_{t^{-1}}}{H_s\cap H}=\sqcup_{i=1}^{m_2} \frac{(H_s\cap H ) lt^{-1}a_iK_0}{H_s\cap H}=\sqcup_{i=1}^{m_2}\frac{(H_s\cap H ) lt^{-1}K_0a_i}{H_s\cap H}$,     $\frac{(H_s\cap H)K}{H_s\cap H}=\sqcup_{j=1}^{m_1} \frac{(H_s\cap H )K_0 b_j}{H_s\cap H}$,
 for some $a_i \in K_{t^{-1}}$, $b_j \in K$.  Clearly $m_1\leq m$. Then
 $\frac{(H_s\cap H)C K}{H_s\cap H}=\frac{(H_s\cap H )lt^{-1}K_{t^{-1}}K}{H_s\cap H}=\cup_{i, j} \frac{(H_s\cap H ) lt^{-1}a_ib_jK_0}{H_s\cap H}=\cup_{i, j} \frac{(H_s\cap H ) lt^{-1}K_0a_ib_j}{H_s\cap H} $, so
\begin{align*}
 &\mu_{H}(K)\int_{\frac{(H_s\cap H)C}{H_s\cap H}} \vert f(\dot{x})\vert   \nu_{\frac{H}{H_s\cap H}}(\dot{x})\\
 & = \mu_{H_s\cap H}(K\cap (H_s\cap H)) \nu_{\frac{H}{H_s\cap H}}(\frac{(H_s\cap H) K}{ H_s\cap H})\int_{\frac{(H_s\cap H)C}{H_s\cap H}} \vert f(\dot{x})\vert   \nu_{\frac{H}{H_s\cap H}}(\dot{x})\\
& =\mu_{H_s\cap H}(K\cap (H_s\cap H))  \nu_{\frac{H}{H_s\cap H}}(\frac{(H_s\cap H) K}{ H_s\cap H}) \sum_{i=1}^{m_2} \vert f(\dot{l}\dot{t}^{-1}\dot{a_i})\vert \nu_{\frac{H}{H_s\cap H}}(\frac{(H_s\cap H)lt^{-1} K_0 a_i}{ H_s\cap H})\\
 &=\mu_{H_s\cap H}(K\cap (H_s\cap H))  \sum^{m_2, m_1}_{i,j=1} \vert f(\dot{l}\dot{t}^{-1}\dot{a_i})\vert   \nu_{\frac{H}{H_s\cap H}}(\frac{(H_s\cap H) K_0 b_j}{ H_s\cap H})\nu_{\frac{H}{H_s\cap H}}(\frac{(H_s\cap H) lt^{-1}K_0 a_i}{ H_s\cap H})\\
&=\mu_{H_s\cap H}(K\cap (H_s\cap H))    [\nu_{\frac{H}{H_s\cap H}}(\frac{(H_s\cap H) K_0 }{H_s\cap H})\nu_{\frac{H}{H_s\cap H}}(\frac{(H_s\cap H) lt^{-1}K_0 }{ H_s\cap H})] \sum^{m_2, m_1}_{i,j=1} \vert f(\dot{l}\dot{t}^{-1}\dot{a_i})\vert \\
 &=\mu_{H_s\cap H}(K\cap (H_s\cap H))    \nu_{\frac{H}{H_s\cap H}}(\frac{(H_s\cap H) K_0 }{H_s\cap H}) \sum^{m_2, m_1}_{i,j=1}     \nu_{\frac{H}{H_s\cap H}}(\frac{(H_s\cap H)lt^{-1}K_0 a_ib_j}{H_s\cap H}) \vert f(\dot{l}\dot{t}^{-1}\dot{a}_i\dot{b}_j)\vert \\
 &\geq \mu_{H_s\cap H}(K\cap (H_s\cap H))    \nu_{\frac{H}{H_s\cap H}}(\frac{(H_s\cap H) K_0 }{H_s\cap H}) \int_{\frac{(H_s\cap H )CK}{H_s\cap H}}\vert  f(\dot{x}) \vert  \nu_{(H_s\cap H)\setminus H}(\dot{x})\\
& =\frac{\mu_{H}(K)}{ m_1} \int_{\frac{(H_s\cap H )CK}{H_s\cap H}}\vert  f(\dot{x}) \vert  \nu_{(H_s\cap H)\setminus H}(\dot{x})
 \geq \frac{\mu_{H}(K)}{ m} \int_{\frac{(H_s\cap H )CK}{H_s\cap H}}\vert  f(\dot{x}) \vert  \nu_{(H_s\cap H)\setminus H}(\dot{x})
 \end{align*}
\end{proof}

  \begin{corollary}\label{notm}
   Under the conditions of Lmm.\ref{ddff2},   $m_{G}(\cInd_{H}^G\sigma_2,\cInd_{H}^G \sigma_1)\leq m_H(\sigma_2, \sigma_1)$.
   \end{corollary}
   \begin{proof}
 By \cite[p.25, Exercise]{BushH}, $m_{G}(\cInd_{H}^G\sigma_2,\cInd_{H}^G \sigma_1)= m_{G}(\cInd_{H}^G\sigma_2,(\cInd_{H}^G \check{\sigma}_1)^{\vee})= m_{G}(\cInd_{H}^G \check{\sigma}_1,\cInd_{H}^G\check{\sigma}_2) \leq m_H( \check{\sigma}_1,  \check{\sigma}_2)=m_{H}(\sigma_2, \sigma_1)$.
  \end{proof}
\subsection{ }
In this last subsection, we let $G$ be a locally profinite group with a normal subgroup $H$.  Assume $G$ is a second-countable group.  Let $\Irr_u(H)$ denote  the set of all equivalence classes of   irreducible preunitary representations of $H$, and  let $\widehat{H}$ denote    the set of all equivalence classes of irreducible unitary smooth  representations of $H$. Clearly there exists a conjugate action of $G$ on   $\Irr_u(H)$ or $\widehat{H}$, given by $\rho^g(h)=\rho(ghg^{-1})$, for $g\in G$, $\rho\in \Irr_u(H)$ or $\widehat{H}$.  Let $\mathbb{T}$ denote the unit circle group in $\mathbb{C}^{\ast}$.

Assume (I) $G$, $H$ are   groups of \emph{type I}, (II) $\widehat{H}/G$ is \emph{ countably separated}(cf. \cite[p.186]{Ma2}), \footnote{The  condition (II) is equivalent to say that $H$ is regularly embedded in $G$.( see also \cite[p.277, footnote]{Ma1}).}(III) For any $\omega \in \widehat{H}$, the orbit $\{ \omega^g \mid g\in G\}$ is \emph{countable}, (IV) there exists an open  subgroup $O$ of $G$, such that   $\Ha^2(O, \C^{\times})$ only contains  elements of finite order.    Let $(\pi, V)$ be  an irreducible preunitary representation of $G$, and $(\Pi, \mathcal{V})$ its  corresponding  unitary representation of $G$.

\begin{theorem}[Clifford-Mackey, a unitary version]\label{unitary}
\begin{itemize}
\item[(1)] $\Res_{H}^{G} \Pi$ is semi-simple.
\item[(2)] There exists an  integer  $m=1, \cdots, n$, or infinite, such that $\Res_{H}^{G} \Pi\simeq \widehat{\oplus}_{\Sigma\in \mathcal{R}_{H}(\Pi) }m\Sigma$.
\item[(3)] Let $(\Sigma, \mathcal{U}) $ be an irreducible subrepresentation of $\Res_{H}^{G}\Pi$. Then $I_{G}(\Sigma)=\left\{ g\in G \mid \Sigma^g \simeq \Sigma\right\}$ is an open  subgroup of $G$.
\item[(4)]  There exists  an irreducible  representation $(\widetilde{\Sigma}, \widetilde{\mathcal{U}})$ of $I_{G}(\Sigma)$, such that:
 \begin{itemize}
 \item[(a)] $\Res_{H}^{I_G(\Sigma)} \widetilde{\Sigma} \simeq m\Sigma$,
 \item[(b)] $\mathfrak{Ind}_{I_{G}(\Sigma)}^{G} \widetilde{\Sigma}\simeq \Pi$.
\end{itemize}
 \item[(5)]  There exists a projective irreducible unitary representation  $(\widetilde{\Phi}_1, \widetilde{\mathcal{W}})$  of $I_{G}(\Sigma)$  associated to a $2$-cocycle $c(-,-)$ with respect to the measurable cohomology group $ \Ha^{2}(I_{G}(\Sigma)/H, \mathbb{T})$,  such that
 \begin{itemize}
 \item[(a)] $ \Sigma =  \Res_{H}^{I_{G}(\Sigma)} \widetilde{\Phi}_1$,
\item[(b)] $\widetilde{\Phi}_1(g) \Sigma(h) \widetilde{\Phi}_1(g^{-1}) = \Sigma(ghg^{-1})  $, for $h\in H$, $g \in I_{G}(\Sigma)$.
\end{itemize}
Moreover,   $\widetilde{\mathcal{W}}$ is  uniquely  determined by  $\widetilde{\mathcal{U}}$  up to projective equivalence (Schur's  Lemma.).
 \item[(6)] There exists an irreducible projective unitary representation  $(\widetilde{\Phi}_2, \widetilde{\mathcal{N}})$  of $I_{G}(\Sigma)/H$   associated to the $2$-cocycle $c^{-1}(-,-)$ (or write  $\overline{c}(-,-)$) such that $(\widetilde{\Phi}_1\widehat{\otimes} \widetilde{\Phi}_2,   \widetilde{\mathcal{W}}\widehat{\otimes} \widetilde{\mathcal{N}} )$ is  linearly isomorphic to $(\widetilde{\Sigma},\widetilde{\mathcal{U}})$. Moreover, $\widetilde{\mathcal{N}}$  is uniquely  determined by $\widetilde{\mathcal{U}}$ up to projective equivalence.
 \end{itemize}
 \end{theorem}
 \begin{proof}
 These results are essentially due to Mackey and his heredes. One can refer to \cite{Ma}, \cite[Section 3.8]{Ma2}, \cite[Section 4.8]	{KT}, \cite[p.460]{KL},  \cite[pp.214-224, Theorems V.9, V.14, V.15, V.16]{Fa}.  Only the assertions (1) (3) did not  directly appear in the references. By our assumption (III),  and  the result  in \cite[p.279]{Bag}, we know that  $G/I_{G}(\Sigma)$ has  countable cardinality. Then  applying  the theorem 7.1 in \cite{Ma} to $\Pi$ gives   the assertion (1), and also  shows that $\nu_G(I_G(\Sigma)) >0$, or   $\mu_G(I_G(\Sigma)) > 0$.  Hence $I_G(\Sigma)$ is an open subgroup of $G$.
 \end{proof}

\begin{lemma}[{\cite[Theorem A]{AM}}]
The measurable cohomology group $ \Ha^{2}(I_{G}(\sigma)/H, \mathbb{T})$ is isomorphic to the continuous cohomology group $ \Ha^{2}(I_{G}(\sigma)/H, \mathbb{T})$.
\end{lemma}
 We can assume the $c(-,-)$ in Thereom  \ref{unitary} is a continuous $2$-cocycle. Let $(\sigma, U)$, $(\widetilde{\sigma}, \widetilde{ U})$,  $(\phi_1, \mathcal{W})$, $ (\phi_2, \mathcal{N})$ be the corresponding smooth  parts  of  $(\Sigma, \mathcal{U})$, $(\widetilde{\sigma}, \widetilde{\mathcal{U}})$,    $(\widetilde{\Phi}_1, \widetilde{\mathcal{W}})$, $(\widetilde{\Phi}_2, \widetilde{\mathcal{N}})$ respectively.
\begin{lemma}
$I_{G}(\sigma)=\left\{ g\in G \mid \sigma^g \simeq \sigma\right\}=I_{G}(\Sigma)$.
\end{lemma}
\begin{proof}
For  $g\in I_{G}(\Sigma)$, as  $\sigma$, $\sigma^g$ are the smooth parts of $\Sigma$, $\Sigma^g$ respectively,  $\sigma\simeq \sigma^g$. Conversely if $g\in I_G(\sigma)$,  by Lmm. \ref{simlem}(2) we obtain $g\in I_G(\Sigma)$.
\end{proof}
As a consequence, $\delta_{I_G(\sigma)\setminus G}^{1/2} =1$.
\begin{lemma}
$\cInd_{I_{G}(\sigma)}^G \widetilde{\sigma} \simeq \pi$.
\end{lemma}
\begin{proof}
By Lmm.\ref{K1in}, $\pi$ is just the smooth part of $\Pi$, and $\pi$ is an irreducible representation.
\end{proof}
\begin{lemma}
$(\widetilde{\sigma}, \widetilde{U})$ is an admissible representation of $I_{G}(\sigma)$.
\end{lemma}
\begin{proof}
Let $K$  be an open compact subgroup of $G$,  and  let $ \Delta$ be a complete set of representatives for $I_G(\sigma) \setminus G/K$.   By Lmm.\ref{therestriction1},  $\Res_{K}^{G} \pi \simeq \oplus_{s\in \Delta}\cInd_{[I_G(\sigma)]_s\cap K}^K \widetilde{\sigma}^{s}$. Since $\dim \pi^K< +\infty$,  each $m_{K\cap [I_G(\sigma)]_{s}}( \widetilde{\sigma})^{s}, \mathbb{C}) $ is finite or zero, in particular $m_{K\cap I_G(\sigma)}(\widetilde{\sigma}, \mathbb{C}) <+\infty$, which implies the result.
\end{proof}
Notice that $\widetilde{U}$ is the $I_G(\sigma)$-smooth part of $\widetilde{\mathcal{U}}$, not just the $H$-smooth part.

\begin{lemma}
$\Res_H^G \pi$ is semi-simple.
\end{lemma}
\begin{proof}
By Lmm.\ref{therestriction1}(1), $\widetilde{\sigma} \hookrightarrow \cInd_{I_G(\sigma)}^G \widetilde{\sigma}$ as $I_G(\sigma)$-modules, consequently $\sigma \hookrightarrow \cInd_{I_G(\sigma)}^G \widetilde{\sigma}$ as $H$-modules.  The rest proof is similar to that of Theorem \ref{cliffordadmissible}(1).
\end{proof}
 Under the condition (IV), the restriction of the class $[c(-,-)]$ to some open compact subgroup $K$ of $I_{G}(\sigma)$ is trivial, which guarantees that $\mathcal{W} \neq 0$, $\mathcal{N}\neq 0$. Finally we can conclude:
\begin{lemma}\label{inu}
\begin{itemize}
\item[(1)] $\Res_{H}^{G} \pi$ is semi-simple.
\item[(2)] There exists an  integer  $m=1, \cdots, n$, or infinite, such that $\Res_{H}^{G} \pi\simeq \oplus_{\sigma\in \mathcal{R}_{H}(\pi) }m\sigma$.
\item[(3)] Let $(\sigma, U) $ be an irreducible  constituent  of $\Res_{H}^{G}\pi$. Then:
 \begin{itemize}
 \item[(a)]$I_{G}(\sigma)=\left\{ g\in G \mid \sigma^g \simeq \sigma\right\}=I_{G}(\Sigma)$,
\item[(b)] $(\widetilde{\sigma}, \widetilde{U})$ is just the  isotypic component $m\sigma$ of $\sigma$ in $\Res_{H}^{G} \pi$.
\end{itemize}
\item[(4)] $\pi\simeq \cInd_{I_{G}(\sigma)}^{G} \widetilde{\sigma}\simeq \Ind_{I_{G}(\sigma)}^{G} \widetilde{\sigma}$.
\item[(5)] $(\phi_1, \mathcal{W})$, $ (\phi_2, \mathcal{N})$   are irreducible,  projective preunitary smooth  representations of $I_{G}(\sigma)$.
\item[(6)]  $(\phi_1\otimes\phi_2,  \mathcal{W}\otimes  \mathcal{N} )$ is  linearly isomorphic to $(\widetilde{\sigma},\widetilde{U})$ as $I_{G}(\sigma)$-modules. Moreover, $\phi_1$, $\phi_2$  are uniquely  determined by  $\widetilde{\sigma}$ up to projective  equivalence.
\end{itemize}
\end{lemma}
\begin{proof}
Parts (1)(3)(a) are proved above. For   (2): Assume $\widetilde{\sigma}|_{H} \simeq m_1 \sigma$. If the $m$ in Theorem \ref{unitary} is finite, $[\widetilde{\Sigma}^{\infty}]|_{H}$ is an admissible representation of $H$, so is $\widetilde{\sigma}|_{H}$. By Lmm.\ref{simlem}(2), $m=m_1$. If $m=\infty$, and $m_1<+\infty$, then $\widetilde{\sigma}$ is an admissible representation of $H$; applying the same lemma again, we get $m=m_1$, a contradiction. Hence $m_1=m=\infty$. Parts (3)(b),(4) can be deduced from $\pi\simeq  \cInd_{I_{G}(\sigma)}^G \widetilde{\sigma}$,  similar to the proofs of theorem \ref{cliffordadmissible}. For (6):   for any $w\in \mathcal{W}, u\in \mathcal{N}$,
let $U_w, U_u$,   $\chi_u$, $\chi_w$ be the corresponding notations in Definition \ref{thedePro} for $w, u$.  Let $K\subseteq U_w\cap U_u$ be an open compact subgroup of $I_{G}(\sigma)$. Then $\chi_w\otimes \chi_u$ is a character of $K$, which is trivial on certain open compact subgroup $K_0 $ of $K$. So $w\otimes u \in \widetilde{U}$. By irreducibility, $ \mathcal{W}\otimes  \mathcal{N}=\widetilde{U}$. Part (5) can be obtained by using   the admissible conditions.
\end{proof}
\subsubsection{}Our next propose  is to give  a smooth version of the main theorem in \cite[p.283]{Bag} for later use.
Some definitions in this text are different from Baggett's in \cite{Bag}.  So  we will rewrite   some results in that paper. Note that the  open subgroup $I_G(\sigma)$ of $G$ is second countable. Let $X=\frac{I_G(\sigma)}{H}$. By Lemma \ref{crossse}, there exists a continuous cross section $\kappa: X\longrightarrow I_G(\sigma)$.

Let $L^2(X, \mathcal{U})$ be the Hilbert space of measurable, $\mathcal{U}$-valued, square-integrable functions on $X$. By \cite[pp.281-282]{Bag}, there exists an isometry $\alpha$ from $L^2(X, \mathcal{U})$ onto $\mathfrak{Ind}_{H}^{I_G(\Sigma)} \mathcal{U}$;  the map  $\alpha$ is given as follows: for $x\in X$, $h\in H$, $F\in L^2(X, \mathcal{U})$, $\alpha(F)(h\kappa(x))=\Sigma(h) F(x)$.  Moreover through the isometry $\alpha^{-1}$,  the action of $G$ on $\mathfrak{Ind}_{H}^{I_G(\Sigma)} \mathcal{U}$ can be transferred  onto $L^2(X, \mathcal{U})$ in the following way:  for $ F\in L^2(X, \mathcal{U})$, $x\in X$, $g\in I_G(\Sigma)$ with the image $\dot{g}\in X$,
$$[g\cdot F](x)=\alpha^{-1}[ g\cdot \alpha(F)](x)=g\cdot \alpha(F)(\kappa(x))=\alpha(F)(\kappa(x)g)=\alpha(F)[\kappa(x)g \kappa(x\dot{g})^{-1} \cdot \kappa(x\dot{g})]$$ $$=\Sigma(\kappa(x)g \kappa(x\dot{g})^{-1} ) \alpha(F)( \kappa(x\dot{g}))=\Sigma(\kappa(x)g \kappa(x\dot{g})^{-1} ) F(x\dot{g}).$$
As  Hilbert spaces, $\widetilde{\mathcal{W}} \widehat{\otimes} \mathfrak{Ind}_{H, c^{-1}(-,-)}^{I_G(\sigma), c^{-1}(-,-)}  \mathbb{C}\simeq \mathcal{U}  \widehat{\otimes} L^2(X) \simeq  L^2(X, \mathcal{U}) \simeq \mathfrak{Ind}_H^{I_G(\sigma)} \mathcal{U}$. By following \cite[p.283]{Bag}, we can give a  composite isomorphism $\beta$ as follows: for $u\in \widetilde{\mathcal{W}}$, $F\in \mathfrak{Ind}_{H, c^{-1}(-,-)}^{I_G(\sigma), c^{-1}(-,-)}  \mathbb{C}$, $x\in X$,  let  $\beta(u\otimes F)(\kappa(x))=
 F(\kappa(x)) \cdot \widetilde{\Phi_1}(\kappa(x))(u)$.
\begin{theorem}[{\cite[p.283, Theorem]{Bag}}]
As unitary representations of $I_G(\Sigma)$,   $(\widetilde{\Phi_1} \widehat{\otimes} \mathfrak{Ind}_{H, c^{-1}(-,-)}^{I_G(\Sigma), c^{-1}(-,-)}  1, \widetilde{\mathcal{W}} \otimes \mathfrak{Ind}_{H, c^{-1}(-,-)}^{I_G(\Sigma), c^{-1}(-,-)}  \mathbb{C})\stackrel{\beta}{\simeq}(\mathfrak{Ind}_H^{I_G(\Sigma)} \Sigma,\mathfrak{Ind}_H^{I_G(\Sigma)} \mathcal{U})$.
\end{theorem}
\begin{proof}
Let us write $\psi=\mathfrak{Ind}_{H, c^{-1}(-,-)}^{I_G(\Sigma), c^{-1}(-,-)} 1$, and $\Psi=\mathfrak{Ind}_H^{I_G(\Sigma)} \Sigma$.  Keep the above notations. For $h_1\in H$, $y\in X$, we have:
 \begin{align*}
&  \Psi(h_1\kappa(y))\beta(u\otimes F)(\kappa(x))\\
 & = [\beta(u\otimes F)]( \kappa(x)h_1\kappa(y))\\
&  = [\beta(u\otimes F)]\big( \kappa(x)h_1\kappa(y) \kappa(xy)^{-1} \cdot \kappa(xy)\big)\\
& = \Sigma\big(\kappa(x)h_1\kappa(y) \kappa(xy)^{-1}\big)[\beta(u\otimes F)](  \kappa(xy))\\
 &=F\big(\kappa(xy)\big)\cdot \widetilde{\Phi_1}\big(\kappa(x)h_1\kappa(y) \kappa(xy)^{-1}\big)\widetilde{\Phi_1}\big(\kappa(xy)\big)(u)\\
  &=c(1, xy)F\big(\kappa(xy)\big)\cdot \widetilde{\Phi_1}\big(\kappa(x)h_1\kappa(y)\big)(u)\\
    &=c\big (\kappa(x)h_1 \kappa(y)\kappa(xy)^{-1} , \kappa(xy) \big)  1_{\mathbb{C}}\big(\kappa(x)h_1 \kappa(y)\kappa(xy)^{-1} \big)F\big(\kappa(xy)\big)\cdot \widetilde{\Phi_1}\big(\kappa(x) h_1\kappa(y)\big)(u)\\
    &= F\big([\kappa(x) h_1 \kappa(y)\kappa(xy)^{-1}] \kappa(xy)\big) \cdot \widetilde{\Phi_1}\big(\kappa(x)h_1\kappa(y)\big)  (u)\\
    &= F\big(\kappa(x) h_1 \kappa(y)\big) \cdot \widetilde{\Phi_1}\big(\kappa(x)h_1\kappa(y)\big)  (u)\\
     &=c^{-1}\big(\kappa(x), h_1\kappa(y)\big) F\big(\kappa(x) h_1 \kappa(y)\big) \cdot c\big(\kappa(x), h_1\kappa(y)\big)\widetilde{\Phi_1}\big(\kappa(x)h_1\kappa(y)\big)  (u)\\
& =[\psi(h_1\kappa(y))F](\kappa(x)) \cdot \widetilde{ \Phi_1}(\kappa(x))[\widetilde{ \Phi_1}(h_1\kappa(y))u]\\
&=\beta[\widetilde{\Phi_1}\big(h_1\kappa(y)\big)u\otimes \psi\big(h_1\kappa(y)\big)F](\kappa(x)).
\end{align*}
The remainder of the argument is analogous to that in \cite{Bag},  and we will not reproduce here.
\end{proof}
 The next result is our main  consequence of Baggett \cite{Bag}.
 \begin{corollary}\label{semisimplein}
As smooth $I_G(\sigma)$-modules, $ \phi_1\otimes \cInd_{H, c^{-1}(-,-)}^{I_G(\sigma), c^{-1}(-,-)}  1\simeq  \cInd_H^{I_G(\sigma)} \sigma$.
\end{corollary}
\begin{proof}
 By use of Remark \ref{zerospace}, and the above expression of $\beta$, we see that $\beta$ sends $\widetilde{\phi_1}\otimes \cInd_{H, c^{-1}(-,-)}^{I_G(\sigma), c^{-1}(-,-)}  1$ into $\cInd_H^{I_G(\sigma)} \sigma$.
For any open compact subgroup $K$ of $I_{G}(\sigma)$, let $ \Delta$ be a complete set of representatives for $H \setminus I_G(\sigma)/K$. According to Lmm.\ref{decomMa},
$ \Res_{K}^{I_G(\sigma)} [\cInd_{H, c^{-1}(-,-)}^{I_G(\sigma), c^{-1}(-,-)} 1]\simeq \oplus_{s\in \Delta} \cInd_{H_s\cap K, c^{-1}(-,-)}^{K, c^{-1}(-,-)} [1^s]_{\chi_s}$, where  for $k\in H_s\cap K $, $[1^s]_{\chi_s}(k)= \chi_s(k)^{-1}= c(ks^{-1},s)  c^{-1}(s, ks^{-1})$. Note that  $c^{-1}(k_1, k_2)=\chi_s^{-1}(k_1) \chi_s(k_2)^{-1} \chi_s(k_1k_2)$,  for $k_1,k_2\in  H_s\cap K$.
 Similarly, $\Res_{K}^{I_G(\sigma)}   [\phi_1\otimes \cInd_{H, c^{-1}(-,-)}^{I_G(\sigma), c^{-1}(-,-)} 1 ]\simeq \oplus_{s\in \Delta}\phi_1\otimes    \cInd_{H_s\cap K, c^{-1}(-,-)}^{K, c^{-1}(-,-)}[1^s]_{\chi_s}$. Now $\beta$ sends
 $  \phi_1\otimes    \cInd_{H_s\cap K, c^{-1}(-,-)}^{K, c^{-1}(-,-)}[1^s]_{\chi_s}$ into $ \cInd_{H_s\cap K}^{K}  (\phi_1)_{\chi_s}$, here $c(-,-)_{\chi_s}|_{ (H_s\cap K) \times (H_s\cap K)}=1$.  For $k=s^{-1} hs \in H_s\cap K$,
 \begin{align*}
 & (\phi_1)_{\chi_s}(s^{-1}hs)= \phi_1(s^{-1}hs) \chi_s^{-1}(s^{-1}hs)\\
  &= \phi_1(s^{-1} h) \phi_1(s) c^{-1}(s^{-1}h,s)\chi_s^{-1}(s^{-1}hs)\\
 & =\phi_1(s^{-1}) \phi_1(h) \phi_1(s) c^{-1}(s^{-1}, h)  c^{-1}(s^{-1}h,s)\chi_s^{-1}(s^{-1}hs) \\
 &= \phi_1(s)^{-1} \phi_1(h) \phi_1(s) c(s, s^{-1})c^{-1}(s^{-1}, h)  c^{-1}(s^{-1}h,s)\chi_s^{-1}(s^{-1}hs)\\
 &=  \phi_1(s)^{-1} \phi_1(h) \phi_1(s)   c(s, s^{-1}h)c^{-1}(s^{-1}h,s)\chi_s^{-1}(s^{-1}hs)\\
 &=\phi_1(s)^{-1} \phi_1(h) \phi_1(s).
      \end{align*}
Therefore  $(\phi_1)_{\chi_s}|_{K\cap H_s} \simeq \phi_1^s|_{K\cap H_s} \simeq \sigma^s|_{K\cap H_s}$. So it reduces to show  the compact case.
 By  \cite{AM} or Remark \ref{trivalK}, we assume that $K$ is much small so that the restriction of $[c(-,-)]$ to $K$ is trivial.  For simplicity,  modifying the action of $I_{G}(\sigma)$ by a continuous function, we may assume $c(-,-)|_{K\times K}=1$.  Assume
  $\phi_1|_{K}  \simeq \phi_1^s|_{K} \simeq \oplus_{\rho\in \hat{K}} m_{\rho} \rho$, for some $m_{\rho}<+\infty$.   Finally it reduces to show $\rho \otimes \cInd_{H\cap K}^{K} 1 \simeq \cInd_{H\cap K}^{K} \rho$.  Since $\rho$ is a unitary representation of finite dimension, $ \cInd_{H\cap K}^{K} 1 \simeq [\mathfrak{Ind}_{H\cap K}^{K} 1]^{\infty}$, $\rho \otimes  \cInd_{H\cap K}^{K} 1 \simeq [\rho \otimes \mathfrak{Ind}_{H\cap K}^{K} 1]^{\infty}$, and  $\cInd_{H\cap K}^{K} \rho \simeq [\mathfrak{Ind}_{H\cap K}^{K} \rho]^{\infty}$. By following Baggett's proof of the main result ( or cf. \cite[Theorem 2.8.6]{KT}), $\rho \otimes \mathfrak{Ind}_{H\cap K}^{K} 1
 \simeq \mathfrak{Ind}_{H\cap K}^{K} \rho$, so the result holds.
\end{proof}
\subsubsection{Semi-simple case}
\begin{lemma}\label{semiInd}
If assume the complementary condition (V):  for any $(\Sigma, \mathcal{W}) \in \widehat{H}$, the cardinality of $\mathcal{O}_{\Sigma}=\{ \Pi\in \widehat{G} \mid m_H(\Pi, \Sigma) \neq 0\}$ is countable, then $\mathfrak{Ind}_H^G\Sigma$, $\mathfrak{Ind}_H^{I_G(\Sigma)} \Sigma$ both are  semi-simple, and $\mathfrak{Ind}_H^G \Sigma\simeq \widehat{\oplus}_{\Pi\in \mathcal{O}_{\Sigma}}  m(\Pi) \Pi$, for $m(\Pi)=m_H(\Pi, \Sigma)$.
\end{lemma}
 \begin{proof}
 See \cite[p.500, Lmm.9.8]{KL}.
 \end{proof}

 \begin{corollary}\label{semicInd}
Let $\pi$ denote the smooth part of an element $\Pi$ in    $\mathcal{O}_{\Sigma}$. Then $\mathfrak{Ind}_H^G \Sigma\hookrightarrow  \prod_{\Pi\in \mathcal{O}_{\Sigma}}  \Pi^{m(\Pi)}$, and
$\cInd_H^G \sigma\hookrightarrow \prod_{\Pi\in \mathcal{O}_{\Sigma}} \pi^{m(\Pi)}$, $\prod_{\Pi\in \mathcal{O}_{\Sigma}}  \pi^{m(\Pi)}\twoheadrightarrow   \Ind_H^G \sigma$.
\end{corollary}
\begin{proof}
For the last assertion, we can consider the contragredient dual of the second inclusion, and obtain $\prod_{\Pi\in \mathcal{O}_{\Sigma}} \check{\pi}^{m(\Pi)}  \twoheadrightarrow   \Ind_H^G \check{\sigma}$; replacing  both sides by their complex conjugate representations give the result.
 \end{proof}
However we can not ensure that $\cInd_H^G \sigma$ is a semi-simple smooth representation. To achieve that situation, we can strengthen the condition (V), and  in addition assume that there exists at least one $\Sigma$ such that $m(\Pi)=m_H(\Pi, \Sigma)$ is finite. We take the corresponding notations in Theorem \ref{unitary} for  this $\Pi$.  Then $m(\Pi)=\dim \widetilde{\Phi_2} $.  By the results of \cite[pp.487-488]{KL} or \cite{Ma1}, the projective $\overline{c}$-representation $(\widetilde{\Phi_2}, \widetilde{\mathcal{N}})$ of $I_{G}(\Sigma)$ or $\frac{ I_{G}(\Sigma)}{H}$, corresponding to an ordinary irreducible unitary representation of $\frac{I_G(\Sigma)}{H}[\overline{c}]$, where $\frac{I_G(\Sigma)}{H}[\overline{c}]$, a locally compact group(cf. \cite[p.270]{Ma1}), is a central extension  of $\frac{I_G(\Sigma)}{H}$ by $\mathbb{T}$ attached to the $2$-cocycle $\overline{c}(-,-)$. Under our assumptions, $\mathfrak{Ind}_{H, c^{-1}(-,-)}^{I_G(\sigma), c^{-1}(-,-)} \Sigma$ contains a finite dimensional discrete irreducible  component. By the discussion in \cite[p.487]{Bag},  the right regular unitary representation of $\frac{I_G(\Sigma)}{H}[\overline{c}]$ contains  finite dimensional discrete summands. Applying the  corollary in \cite[p.120]{Ma},  we know that $\frac{I_G(\Sigma)}{H}[\overline{c}]$ is a compact group.  Hence $\cInd_{H, c^{-1}(-,-)}^{I_G(\sigma), c^{-1}(-,-)}  1$ is a semi-simple representation, so is $\cInd_H^{I_G(\sigma)} \sigma$ by Cor.\ref{semisimplein}.
  \begin{corollary}\label{semisimplerep}
Under the condition (V), assume  that there exists at least one $\Sigma$ such that $m(\Pi)=m_H(\Pi, \Sigma)$ is finite. Then  $\cInd_H^{I_G(\sigma)} \sigma$ is semi-simple; consequently, $\cInd_H^G \sigma$ is semi-simple as well.
 \end{corollary}

\section{The theta representation I }\label{stronglygraphreI}
In the next sections \ref{stronglygraphreI}, \ref{stronglygraphreII}, \ref{stronglygraphreIII},   we will let $G_1, G_2$ designate locally profinite groups with normal subgroups $H_1$ and $H_2$  respectively such that   $G_1/H_1 \simeq G_2/H_2$ under a mapping $\gamma$  with  the graph $\Gamma/{(H_1\times H_2)}$  of $ (G_1 \times G_2) /{(H_1 \times H_2)}$. Assume that all irreducible smooth representations of $G_i$, $H_i$ are admissible, $i= 1, 2$, and let $(\rho, W)$ be a smooth representation of $\Gamma$.

In this section, assume $H_1$ is an open subgroup of $G_1$,  $G_1/H_1$ is abelian,  and $\mathcal{R}_{H_i}(\pi_i) \neq \emptyset$, for any $\pi_i \in \Irr(G_i)$. Set $\pi= \cInd_{\Gamma}^{G_1 \times G_2} \rho$, $V= \cInd_{\Gamma}^{G_1 \times G_2} W$. Our  main result of this section is the following:
\begin{theorem}\label{graphrepresentation}
  \begin{itemize}
  \item[(1)]  If  the representation $\Res_{H_1 \times H_2}^{\Gamma} \rho$  of $H_1 \times H_2$ is a theta  representation, then so is  the representation $\cInd_{\Gamma}^{G_1 \times G_2} \rho$ of $G_1 \times G_2$.
   \item[(2)] If the representation $\cInd_{\Gamma}^{G_1 \times G_2} \rho$ of $G_1 \times G_2$ is a  theta representation of finite length,
    then the representation   $\Res_{H_1 \times H_2}^{\Gamma} \rho$ of $H_1 \times H_2$ satisfied the graphic property.  Moreover for each $i=1, 2$,  assume (a) $\Rep(H_i)$ is locally noetherian, (b) for any  $\pi_1\otimes \pi_2\in \mathcal{R}_{G_1\times G_2}(\pi)$,   $\Ext_G^1(\pi_i, \pi_i)=0$, then  $\Res_{H_1 \times H_2}^{\Gamma} \rho$ of $H_1 \times H_2$  is a  theta representation of finite length.
    \end{itemize}
\end{theorem}
We shall prove this theorem in the following two subsections.
\subsection{}\label{part1}

\begin{lemma}\label{principallemmedegraphe}
 In the above  theorem \ref{graphrepresentation}(1), if    $(\pi_1, V_1) \in \Irr(G_1)$ and $(\pi_2, V_2) \in \Irr(G_2)$, such that $\pi_1 \otimes \pi_2 \in \mathcal{R}_{G_1  \times G_2}(\pi)$, then:
 \begin{itemize}
\item[(1)] For any  $\sigma \in \mathcal{R}_{H_1}(\pi_1)$, there exists a unique element $\delta \in \mathcal{R}_{H_2}(\pi_2)$ such that $\sigma\otimes \delta\in \mathcal{R}_{H_1 \times H_2}(\rho)$.
\item[(2)] Let  $\widetilde{H_1}=\left\{ g_1 \in G_1\mid \sigma^{g_1} \simeq \sigma\right\}$ and $\widetilde{H_2}=\left\{ g_2 \in G_2 \mid \delta^{g_2} \simeq \delta \right\}$.  Then $\gamma$ induces a bijective map from $\widetilde{H_1}/H_1$ to $\widetilde{H_2}/H_2$ with the graph   $[\Gamma \cap (\widetilde{H_1} \times \widetilde{H_2})]/{(H_1 \times H_2)}$, and  a bijective map from $G_1/{\widetilde{H_1}}$ to $G_2/{\widetilde{H_2}}$ with the graph $[\Gamma\cdot(\widetilde{H_1} \times \widetilde{H_2})] /{(\widetilde{H_1} \times \widetilde{H_2})}$.
\end{itemize}
\end{lemma}
\begin{proof}
 1) By Frobenius reciprocity,  as is easy to see that $\Hom_{\Gamma}(\rho, \pi_1 \otimes \pi_2) \neq 0$. A priori, we can find $\sigma_1 \otimes \delta_1 \in \mathcal{R}_{H_1 \times H_2}(\rho) \cap \mathcal{R}_{H_1 \times H_2}(\pi_1 \otimes \pi_2)$. By Theorem \ref{cliffordadmissible}, there is an element $t H_1 \in G_1/{H_1}$ such that $\sigma_1^t \simeq \sigma$. Let $\gamma(tH_1)=s H_2 \in G_2/{H_2}$ with $(t, s) \in \Gamma$. Then $\sigma \otimes \delta \simeq \sigma_1^t\otimes \delta_1^s \in \mathcal{R}_{H_1 \times H_2}(\rho^{(t, s)})=\mathcal{R}_{H_1 \times H_2}(\rho)$.
The uniqueness is clear.

2) Assume  $g_1 H_1 \in G_1/{H_1}$, and let $\gamma(g_1 H_1) =g_2 H_2 \in G_2/{H_2}$. We then have
$\sigma^{g_1} \otimes \delta^{g_2}  \in \mathcal{R}_{H_1 \times H_2}(\rho)$, which implies that $\sigma^{g_1} \simeq \sigma$ iff  $\delta^{g_2} \simeq \delta$, in other words,  $g_1 \in \widetilde{H_1}$ iff $g_2 \in \widetilde{H_2}$. This means that $\gamma$ maps $\widetilde{H_1}/{H_1}$ onto $\widetilde{H_2}/{H_2}$ with the graph $ [\Gamma \cap (\widetilde{H_1}\times \widetilde{H_2})]/{(H_1 \times H_2)}$ and  induces a bijective mapping $\overline{\gamma}$ from $G_1/{\widetilde{H_1}}$ to $G_2/{\widetilde{H_2}}$ with the graph  $[\Gamma \cdot(\widetilde{H_1} \times \widetilde{H_2})]/{(\widetilde{H_1} \times \widetilde{H_2})}$.
\end{proof}
\begin{lemma}\label{therestrictionfinitelengthII}
$\cInd_{H_2}^{G_2} (\rho_{\sigma} )\simeq (\cInd_{H_2}^{G_2} \rho)_{\sigma}$  as $H_1 \times G_2$-modules, for all $\sigma\in \Irr(H_1)$.
\end{lemma}
\begin{proof}
Assume $\sigma \in \mathcal{R}_{H_1}(\rho)$; otherwise both sides vanish. Write $\Pi= \cInd_{H_2}^{G_2} \rho$, and $ (\Pi)_{\sigma}= \cInd_{H_2}^{G_2} \rho_{\sigma}$. By Lmm. \ref{therestriction1}, $ (\Pi)_{\sigma}|_{H_1 \times H_2} =\oplus_{g_2\in \Delta_2}  [(\Pi)_{\sigma}](g_2)(\rho_{\sigma})$, and $\Pi_{\sigma}|_{H_1 \times H_2} \simeq  (\oplus_{g_2\in \Delta_2} \Pi(g_2)(\rho))_{\sigma}$, where $\Delta_2$ is a  set of coset representatives of $G_2/H_2$ in $G_2$. Since $\Hom_{H_1}(\Pi(g_2)(\rho), \sigma) \simeq \Hom_{H_1}(\rho, \sigma)$, we know that
 \begin{align}
 & (\oplus_{g_2\in \Delta_2} \Pi(g_2)(\rho))_{\sigma}  \simeq  \frac{\oplus_{g_2\in \Delta_2} \Pi(g_2)(\rho)}{\cap_{f\in \Hom_{H_1}\big(\oplus_{g_2\in \Delta_2}\Pi(g_2)(\rho), \sigma\big)} \Ker f} \label{equ1}\\
 & \simeq  \frac{\oplus_{g_2\in \Delta_2} \Pi(g_2)(\rho)}{\oplus_{g_2\in \Delta_2}\Pi(g_2)\Big( \cap_{f\in \Hom_{H_1}(\rho, \sigma)}  \Ker f \Big) } \simeq \oplus_{g\in \Delta_2}  \Pi(g_2)(\rho_{\sigma})\label{equ2}
   \end{align}
  Hence an $H_1 \times H_2$-morphism $\rho_{\sigma} \longrightarrow  (\cInd_{H_2}^{G_2}\rho)_{\sigma}$ comes. By Frobenius reciprocity, we obtain an $H_1 \times G_2$-morphism $ \cInd_{H_2}^{G_2} \rho_{\sigma} \longrightarrow  (\cInd_{H_2}^{G_2}\rho)_{\sigma}$, which  is a bijection by the above (\ref{equ1}) (\ref{equ2}).
\end{proof}

If  $\pi_1 \in \Irr(G_1)$,  $\sigma \prec \pi_1|_{H_1}$, we will   let  $\widetilde{\sigma}$ denote the irreducible representation of $\widetilde{H_1}=\left\{ g_1 \in G_1 \mid \sigma^{g_1} \simeq \sigma\right\}$ as  defined  in Theorem \ref{cliffordadmissible} (4) (b). Suppose $\pi_{\pi_1} \simeq \pi_1 \otimes \Theta_{\pi_1}$ as $G_1 \times G_2$-modules, and $\rho_{\sigma} \simeq \sigma \otimes \Theta_{\sigma}$ as $H_1 \times H_2$-modules. For the time being, we write   $\widetilde{\Gamma}=\Gamma \cdot \big( \widetilde{H_1} \times \widetilde{H_2}\big)$, $\widetilde{\rho}=\cInd_{\Gamma}^{\widetilde{\Gamma}} \rho$, and $\widetilde{\rho}_{\widetilde{\sigma}} \simeq \widetilde{\sigma}\otimes \Theta_{\widetilde{\sigma}}$ as $\widetilde{H_1}\times \widetilde{H_2}$-modules.
\begin{lemma}\label{therestriction2}
\begin{itemize}
\item[(1)] $\Theta_{\pi_1} \simeq \cInd_{\widetilde{H_2}}^{G_2} \Theta_{\widetilde{\sigma}}$ as $G_2$-modules.
\item[(2)] If  $\widetilde{\sigma}|_{H_1} \simeq m\sigma$,  then  there exists an embedding $\Theta_{\sigma} \hookrightarrow  \Theta_{\widetilde{\sigma}}|_{H_2}$ as $H_2$-modules.
    \item[(3)] If the above $m=1$, then $\Theta_{\sigma} \simeq   \Theta_{\widetilde{\sigma}}|_{H_2}$ as $H_2$-modules.
\end{itemize}
\end{lemma}
\begin{proof}
1) By  the above lemma, we have $\pi_{\widetilde{\sigma}} \simeq \widetilde{\sigma} \otimes \cInd_{\widetilde{H_2}}^{G_2}  \Theta_{\widetilde{\sigma}}$ as $\widetilde{H_1} \times G_2$-modules. By \cite[p.18]{BushH},   there exists a  $\widetilde{H_1} \times G_2$-morphism  $p: \pi_{\pi_1}\longrightarrow \pi_{\widetilde{\sigma}}$. Then a $G_1\times G_2$-morphism $ \Ind_{\widetilde{H_1}}^{G_1}p: \pi_{\pi_1}\longrightarrow \Ind_{\widetilde{H_1}}^{G_1}\pi_{\widetilde{\sigma}}$ follows.  By Lmm.\ref{quotientdedeuxgroupes} (2), we get a $G_2$-morphism $ \iotaup:\Theta_{\pi_1}\longrightarrow\cInd_{\widetilde{H_2}}^{G_2} \Theta_{\widetilde{\sigma}} $.   For any representation $(\sigmaup_2, U_2)$ of $G_2$, we  have
\begin{equation}\label{cisomorphis}
\Hom_{G_2}( \Theta_{\pi_1}, \sigmaup_2) \simeq  \Hom_{G_1 \times G_2}( \pi, \pi_1 \otimes \sigmaup_2)\simeq \Hom_{\widetilde{H_1} \times G_2}( \pi_{\widetilde{\sigma}}, \widetilde{\sigma} \otimes \sigmaup_2)\simeq \Hom_{G_2}\big( \cInd_{\widetilde{H_2}}^{G_2} \Theta_{\widetilde{\sigma}}, \sigmaup_2\big),
\end{equation}
compatible with the above $\iota$. In particular, if let  $\sigmaup_2= \Theta_{\pi_1}$, then we can find  a $G_2$-morphism $ \varrho$ from $\cInd_{\widetilde{H_2}}^{G_2} \Theta_{\widetilde{\sigma}}$ to $\Theta_{\pi_1}$ such that $ \varrho \circ \iotaup= 1$, So $\iotaup$ is injective. Applying  $\Hom_{G_2}\big( -, \sigmaup_2\big)$ to the short exact sequence  $ \Theta_{\pi_1}\stackrel{\iotaup}{\hookrightarrow }\cInd_{\widetilde{H_2}}^{G_2} \Theta_{\widetilde{\sigma}} \stackrel{\tauup}{\twoheadrightarrow } \frac{\cInd_{\widetilde{H_2}}^{G_2} \Theta_{\widetilde{\sigma}} }{\Im \iotaup}$ shows  that $ \Hom_{G_2}\big( \frac{\cInd_{\widetilde{H_2}}^{G_2} \Theta_{\widetilde{\sigma}} }{\Im \iotaup}, \sigmaup_2\big)=0$; hence $\iotaup$ is also surjective.\\
2) As $H_1 \times H_2$-modules, we have  $ \sigma \otimes \Theta_{\sigma} \simeq \rho_{\sigma}\simeq \frac{\rho}{\cap_{f\in \Hom_{H_1}(\rho, \sigma)} \Ker f }\simeq \frac{\rho}{\cap_{f\in \Hom_{H_1}(\rho, \widetilde{\sigma})} \Ker f }\hookrightarrow \frac{\widetilde{\rho}}{\cap_{f\in \Hom_{H_1}(\rho, \widetilde{\sigma})} \Ker f } \longrightarrow  \frac{\widetilde{\rho}}{\cap_{\widetilde{f}\in \Hom_{\widetilde{H_1}}(\widetilde{\rho} , \widetilde{\sigma})} \Ker \widetilde{f} }\simeq   \widetilde{\rho}_{\widetilde{\sigma}} \simeq  \widetilde{\sigma} \otimes \Theta_{\widetilde{\sigma}} $. So we get an  $H_1 \times H_2$-morphism $\kappa_{\sigma}: \sigma\otimes \Theta_{\sigma} \longrightarrow \sigma \otimes \Theta_{\widetilde{\sigma}}$, and then  an  $ H_2$-morphism $\kappa: \Theta_{\sigma} \longrightarrow \Theta_{\widetilde{\sigma}}$. For any smooth representation $(\sigmaup_2, W_2)$ of $H_2$, by Frobenius reciprocity, we  have
   \begin{equation}\label{sevaliso}
   \begin{aligned} \Hom_{H_2}( \Theta_{\widetilde{\sigma}}, \sigmaup_2) \simeq \Hom_{\widetilde{H_1}\times H_2}( \widetilde{\rho}_{\widetilde{\sigma}}, \widetilde{\sigma}\otimes \sigmaup_2)  \simeq \Hom_{\widetilde{H_1}\times H_2}( \widetilde{\rho}, \widetilde{\sigma}\otimes \sigmaup_2) \simeq \Hom_{H_1\times H_2}(\rho, \widetilde{\sigma}\otimes \sigmaup_2) \\\twoheadrightarrow
    \Hom_{H_1\times H_2}(\rho, \sigma\otimes \sigmaup_2) \simeq  \Hom_{H_1\times H_2}(\rho_{\sigma}, \sigma\otimes \sigmaup_2)\simeq \Hom_{H_2}(\Theta_{\sigma},  \sigmaup_2) , \end{aligned}
   \end{equation}
   i.e. we get a surjective morphism $ \Hom_{H_2}( \Theta_{\widetilde{\sigma}}, \sigmaup_2)
 \longrightarrow \Hom_{H_2}(\Theta_{\sigma},  \sigmaup_2) $ compatible with the above $\kappa$, so the result holds.\\
 3) In the above (\ref{sevaliso}), $ \Hom_{H_2}( \Theta_{\widetilde{\sigma}}, \sigmaup_2)
 \simeq\Hom_{H_2}(\Theta_{\sigma},  \sigmaup_2) $, so $\Theta_{\sigma} \simeq   \Theta_{\widetilde{\sigma}}|_{H_2}$ as $H_2$-modules.
   \end{proof}

\subsection*{The proof of Theorem \ref{graphrepresentation}(1)}
$ $
\\
 Step I. Suppose that both $\pi_1 \otimes \pi_2',  \pi_1 \otimes \pi_2 \in \mathcal{R}_{G_1 \times G_2}(\pi)$. Assume that $\Res_{H_1}^{G_1} \pi_1$, $\Res_{H_2}^{G_2} \pi_2$, $\Res_{H_2}^{G_2} \pi_2'$ all are multiplicity-free.
By virtue of  Frobenius reciprocity, we have
$\Hom_{G_1 \times G_2} (\pi, \pi_1 \otimes \pi_2)  \simeq \Hom_{\Gamma} (\rho, \pi_1 \otimes \pi_2)$, being equal to $ \Hom_{H_1 \times H_2}(\rho, \pi_1 \otimes \pi_2)^{\Gamma/{(H_1 \times H_2)}}$
 for the canonical action of $\Gamma/{(H_1 \times H_2)}$ on $\Hom_{H_1 \times H_2}(\rho, \pi_1 \otimes \pi_2)$.

Now
$\Hom_{H_1 \times H_2}(\rho, \pi_1 \otimes \pi_2) \hookrightarrow \prod_{i, j} \Hom_{H_1 \times H_2} (\rho, \sigma_i \otimes \delta_j) = \prod_{\sigma_{\alpha} \otimes \delta_{\alpha} \in \mathcal{R}_{H_1 \times H_2}(\rho)} \Hom_{H_1 \times H_2}(\rho, \sigma_{\alpha} \otimes \delta_{\alpha});$
every component  of the last term is of dimension one, and $\Gamma$ permutes transitively them. Hence
$1 \leq  m_{G_1 \times G_2} (\pi, \pi_1 \otimes \pi_2) \leq m_{H_1 \times H_2}(\rho, \sigma_{\alpha} \otimes \delta_{\alpha})=1$ as required.

 Let $(\sigma_1, \delta_1) \in \mathcal{R}_{H_1 \times H_2} (\rho) \cap \mathcal{R}_{H_1 \times H_2}(\pi_1 \otimes \pi_2) \textrm{ and } (\sigma_2, \delta_2) \in \mathcal{R}_{H_1 \times H_2} (\rho) \cap \mathcal{R}_{H_1 \times H_2}(\pi_1 \otimes \pi_2').$
 Then there exists an element $g_1 \in G_1$, such that $\sigma_2 \simeq \sigma_1^{g_1}$. If  we write $\gamma(g_1 H_1)= g_2 H_2 \in G_2/{H_2}$, then
 $\sigma_1^{g_1} \otimes \delta_1^{g_2} \simeq \sigma_2 \otimes \delta_1^{g_2} \in \mathcal{R}_{H_1 \times H_2}(\rho).$
 By the property of graph, we get $\delta_2 \simeq \delta_1^{g_2}.$
  Hence $\mathcal{R}_{H_1 \times H_2}(\pi_1 \otimes \pi_2) \cap \mathcal{R}_{H_1 \times H_2}(\pi_1 \otimes \pi_2') \cap \mathcal{R}_{H_1 \times H_2}(\rho) \neq \emptyset,$ and $\mathcal{R}_{H_2}(\pi_2) \cap \mathcal{R}_{H_2}(\pi_2') \neq \emptyset.$

By Lmm.\ref{principallemmedegraphe}, there exists  $\sigma \otimes \delta \in \mathcal{R}_{H_1 \times H_2}(\rho) \cap \mathcal{R}_{H_1 \times H_2}(\pi_1 \otimes \pi_2) \cap \mathcal{R}_{H_1 \times H_2}(\pi_1 \otimes \pi_2')$. Let  $\widetilde{H_1}=\left\{ g_1 \in G_1 \mid \sigma^{g_1} \simeq \sigma\right\}$, $\widetilde{H_2}=\left\{ g_2\in G_2 \mid \delta^{g_2} \simeq \delta\right\}$.  Let $\widetilde{\sigma} \in \mathcal{R}_{\widetilde{H_1}}(\pi_1)$, $\widetilde{\delta} \in \mathcal{R}_{\widetilde{H_2}}(\pi_2)$,  $\widetilde{\delta}' \in \mathcal{R}_{\widetilde{H_2}}(\pi_2')$, such that $\widetilde{\sigma}|_{H_1}\simeq \sigma$, $\widetilde{\delta}|_{H_2}\simeq \delta \simeq \widetilde{\delta}'|_{H_2}$. By Clifford-Mackey theory, $\widetilde{\delta}' \simeq \widetilde{\delta} \otimes \nu $ for certain character $\nu$ of $\widetilde{H_2}/H_2$.
Now let us denote by $\widetilde{\rho}:=\cInd_{\Gamma \cap (\widetilde{H_1} \times \widetilde{H_2})}^{\widetilde{H_1} \times \widetilde{H_2}} \rho$. Then $1=m_{G_1 \times G_2}(\pi, \pi_1 \otimes \pi_2)=m_{\widetilde{H_1} \times G_2}(\pi, \widetilde{\sigma} \otimes \pi_2) = m_{\widetilde{H_1} \times \widetilde{H_2}}( \widetilde{\rho}, \widetilde{\sigma} \otimes \pi_2)$.  For any non-zero $f\in \Hom_{\widetilde{H_1} \times \widetilde{H_2}}( \widetilde{\rho}, \widetilde{\sigma} \otimes \pi_2)$, it also lies in $\Hom_{H_1 \times \widetilde{H_2}}( \widetilde{\rho}, \widetilde{\sigma} \otimes \pi_2)$, which is isomorphic to $\Hom_{H_1 \times H_2}(\rho, \sigma \otimes \pi_2)$. So the image of $f$ belongs to $\widetilde{\sigma}\otimes \widetilde{\delta}$. Therefore $1=m_{\widetilde{H_1} \times \widetilde{H_2}}(\widetilde{\rho}, \widetilde{\sigma} \otimes \widetilde{\delta})=m_{\Gamma \cap (\widetilde{H_1} \times \widetilde{H_2})}(\rho, \widetilde{\sigma} \otimes \widetilde{\delta})=m_{H_1 \times H_2}(\rho, \sigma \otimes \delta)=1.$ Similarly, we have $ m_{\Gamma \cap (\widetilde{H_1} \times \widetilde{H_2})}(\rho, \widetilde{\sigma} \otimes \widetilde{\delta}')=m_{H_1 \times H_2}(\rho, \sigma \otimes \delta)=1$.

For simplicity, we assume that $\widetilde{\sigma}|_{H_1} =\sigma$, $\widetilde{\sigma}|_{H_2} =\delta =\widetilde{\delta}'|_{H_2}$. A  non-trivial element  $T\in \Hom_{H_1  \times H_2}(\rho, \sigma \otimes \delta)$ can  extend uniquely to  $\widetilde{T} \in \Hom_{\Gamma \cap (\widetilde{H_1} \times \widetilde{H_2})}(\rho, \widetilde{\sigma} \otimes \widetilde{\delta})$ and to $\widetilde{T}' \in \Hom_{\Gamma \cap (\widetilde{H_1} \times \widetilde{H_2})}(\rho, \widetilde{\sigma} \otimes \widetilde{\delta}')$.   We may and do suppose   $T=\widetilde{T}=\widetilde{T}'$. Let  $(g,h) \in \Gamma \cap (\widetilde{H_1} \times \widetilde{H_2})$, $v\in W$.  Then
$$\widetilde{T}(\rho(g,h) v)= T(\rho(g,h) v)=\widetilde{\sigma} \otimes \widetilde{\delta}(g,h) T(v)$$
and
$$\widetilde{T}'(\rho(g,h) v)= T(\rho(g,h) v)=\widetilde{\sigma} \otimes \widetilde{\delta}'(g,h) T(v)=\widetilde{\sigma} \otimes \widetilde{\delta}(g,h) T(v)\nu(h).$$
It turns out that $\nu(h)=1$. As  the map $\gamma: [(\widetilde{H_1} \times \widetilde{H_2}) \cap \Gamma ]/{(H_1 \times H_2)} \longrightarrow  \widetilde{H_2}/{H_2}$ is surjective, it is clear that $\nu=1$, and $\widetilde{\delta}' \simeq \widetilde{\delta}$. By  Theorem \ref{cliffordadmissible} (6), $\pi_2 \simeq \cInd_{\widetilde{ H_2}}^{G_2} \widetilde{\delta}$,  $\pi_2' \simeq \cInd_{\widetilde{ H_2}}^{G_2} \widetilde{\delta}'$.   It then follows that    $\pi_2 \simeq \pi_2'$.
Making use of the results of Lmm.\ref{therestriction2} and Prop.\ref{finitegenerated} (1), we can assert that   $\pi_{\pi_1}$ is a finitely generated representation of $G_1 \times G_2$.\ \\
 Step II: the general case. Suppose now that $\pi_1 \otimes \pi_2\in \mathcal{R}_{G_1 \times G_2}(\pi) $  and $\pi_1 \otimes \pi_2' \in \mathcal{R}_{G_1 \times G_2}(\pi)$. Similarly as Step I, we can find $\sigma \otimes \delta \in \mathcal{R}_{H_1 \times H_2}(\rho) \cap \mathcal{R}_{H_1 \times H_2}(\pi_1 \otimes \pi_2) \cap \mathcal{R}_{H_1 \times H_2}(\pi_1 \otimes \pi_2')$. For $\pi_1$, applied Lmm.\ref{towerofnormalsubgroups}, we find a tower of normal subgroups of $G_1$:
$$H_1=H_1^{(0)} \lhd H_1^{(1)} \lhd \cdots \lhd  H_1^{(k)} \lhd  H_1^{(k+1)}=G_1,$$ such  that
 \begin{itemize}
\item[(1)] $H_1^{(k)}\subseteq \widetilde{H_1}$, and $H_1^{(i+1)}/{H_1^{(i)}}$ is a cyclic group, for $ i=0, \cdots  k-1$,
\item[(2)] $\mathcal{R}_{H_1^{(i)}}(\pi_1)\neq \emptyset$,  for $ i= 0, \cdots k$,
\item[(3)] $\Res_{H_1^{(i)}}^{H_1^{(i+1)}}\sigma_{i+1}$ is multiplicity-free, for any $\sigma_{i+1} \in \mathcal{R}_{H_1^{(i+1)}}(\pi_1)$ as  $i$ runs through $0, \cdots, k$.
\end{itemize}
 Let $H_2^{(i)}/H_2=\gamma(H_1^{(i)}/{H_1})$,  for  some  $H_2^{(i)} \subseteq G_2$.  By Lmm.\ref{principallemmedegraphe}, $H_2^{(k)}\subseteq \widetilde{H_2}$,  and $H_2^{(i+1)}/{H_2^{(i)}}$ is a cyclic group for $i=0,  \cdots, k-1$. Note that  according to Cor.\ref{semisimple}, for $i=0, \cdots, k$, $\Res_{H_2^{(i)}}^{G_2}\pi_2$ and $ \Res_{H_2^{(i)}}^{G_2}\pi_2'$  all are semi-simple.  Similarly, for $\pi_2$, applied Lmm.\ref{towerofnormalsubgroups}, there exists a  tower of normal subgroups: $H_2^{(k)} \lhd H_2^{(k+1)} \lhd \cdots \lhd  H_2^{(k+l)} \lhd  G_2,$ satisfying the similar properties as above; then for $\pi_2'$, there exists a similar  tower of normal subgroups: $H_2^{(k+l)} \lhd H_2^{(k+l+1)} \lhd \cdots \lhd  H_2^{(n)} \lhd  H_2^{(n+1)}=G_2$.   Let $H_1^{(i)}/H_1$ be the inverse image of $\gamma(H_2^{(i)}/{H_2})$ in $G_1/H_1$. So finally  we succeed in  constructing a tower of normal subgroups $H_{l}^{(i)}$ of $G_l$, $i=0, \cdots, n$,  $l=1,2$, such that
 \begin{itemize}
 \item[(1)] $\mathcal{R}_{H_1^{(i)}}(\pi_1) \neq \emptyset$, $\mathcal{R}_{H_2^{(i)}}(\pi_2 \oplus \pi_2')\neq \emptyset$,
\item[(2)] $\Res_{H_1^{(i)}}^{H_1^{(i+1)}}\sigma_{i+1}$ and $\Res_{H_2^{(i)}}^{H_2^{(i+1)}}\delta_{i+1}$ both are multiplicity-free, for each $\sigma_{i+1} \in \mathcal{R}_{H_1^{(i+1)}}(\pi_1)$, each $\delta_{i+1} \in  \mathcal{R}_{H_2^{(i+1)}}(\pi_2 \oplus \pi_2')$,
\item[(3)] $\gamma$ induces a bijective map
$\gamma^{(i+1)}: H_1^{(i+1)}/{H_1^{(i)}} \longrightarrow  H_2^{(i+1)}/{H_2^{(i)}}$
with  the graph $\Gamma^{(i+1)}/{(H_1^{(i)} \times H_2^{(i)})}$, where $\Gamma^{(i+1)}=[\Gamma\cap (H_1^{(i+1)} \times H_2^{(i+1)})] \cdot (H_1^{(i)} \times H_2^{(i)})$.

\end{itemize}
For each $1\leq i\leq n+1$, we introduce two representations
$\rho^{(i)}=\cInd_{\Gamma\cap (H_1^{(i)} \times H_2^{(i)})}^{H_1^{(i)} \times H_2^{(i)}} \rho  \textrm{ and }\Delta^{(i)}=\cInd_{\Gamma \cap (H_1^{(i)} \times H_2^{(i)})}^{\Gamma^{(i)}} \rho.$
Then:
  \begin{itemize}
\item[(a)]  $\Res_{H_1^{(i-1)} \times H_2^{(i-1)}}^{\Gamma^{(i)}}\Delta^{(i)}\simeq \Res_{H_1^{(i-1)} \times H_2^{(i-1)}}^{\Gamma^{(i)}} \Big( \cInd_{\Gamma \cap (H_1^{(i)} \times H_2^{(i)})}^{\Gamma^{(i)}} \rho\Big) \simeq \rho^{(i-1)}$.
\item[(b)]  $\rho^{(i)}\simeq \cInd_{\Gamma^{(i)}}^{H_1^{(i)} \times H_2^{(i)}} \Big( \cInd_{\Gamma \cap( H_1^{(i)} \times H_2^{(i)})}^{\Gamma^{(i)}}\rho\Big)\simeq\cInd_{\Gamma^{(i)}}^{H_1^{(i)} \times H_2^{(i)}} \Delta^{(i)}$.
\item[(c)] $\rho^{(n+1)}\simeq\cInd_{\Gamma}^{G_1 \times G_2} \rho\simeq \pi$.
\end{itemize}
By induction, the result of Step I shows that each $\rho^{(i)}$ is a  theta representation with respect to $\mathcal{R}_{H_1^{(i)}}(\pi_1)$ and $\mathcal{R}_{H_2^{(i)}}(\pi_2 \oplus \pi_2')$.  Finally by considering $\rho^{(n+1)}$ we obtain that $m_{G_1 \times G_2}(\pi, \pi_1 \otimes \pi_2)=1$, and $\pi_2 \simeq \pi_2'$. The finiteness conditions on the greatest isotypic quotients arise from the induction.

In view of the proof, we obtain an  analogous  result  of  Roberts  Brooks' Lmm.4.2 in \cite{Rob1}.

\begin{corollary}\label{larelationdebigrapheI}
In  Theorem   \ref{graphrepresentation}(1), if   $\pi_1 \in \Irr(G_1)$, $\pi_2 \in \Irr(G_2)$ with the decompositions
$$ \pi_1=\oplus_{\sigma_i \in \mathcal{R}_{H_1}(\pi_1)} m_1 \sigma_i, \quad \textrm{ and } \quad   \pi_2=\oplus_{\delta_i \in \mathcal{R}_{H_2}(\pi_2)} m_2 \delta_i$$
such that $\pi_1\otimes \pi_2 \in \mathcal{R}_{G_1\times G_2}(\pi),$ then
\begin{itemize}
\item[(1)] there  exists a bijective map   $\theta_{\rho}: \mathcal{R}_{H_1}(\pi_1) \longrightarrow \mathcal{R}_{H_2}(\pi_2); \sigma_{\alpha} \longmapsto \delta_{\alpha}$ such that  $\sigma_{\alpha} \otimes \delta_{\alpha} \in \mathcal{R}_{H_1 \times H_2}(\rho)$ and
$\sigma_{\alpha} \otimes \delta_{\beta} \notin \mathcal{R}_{H_1\times H_2}(\rho)$ for $\alpha \neq \beta$.
\item[(2)] $m_1=1$ if and only if $m_2=1$.
\end{itemize}
\end{corollary}

\subsection{}\label{part52}
In this subsection, we attempt  to prove the second part of Theorem \ref{graphrepresentation} in several steps.  We adopt the beginning  definitions and notations.   Suppose that $\sigma \otimes \delta\in \mathcal{R}_{H_1 \times H_2}(\rho)$. So  we can find $(\pi_1, V_1) \in \Irr(G_1)$, $(\pi_2, V_2)\in \Irr(G_2)$ such that $\sigma$, $\delta$ occur in $\Res_{H_1}^{G_1} \pi_1$, $\Res_{H_2}^{G_2} \pi_2$ as sub-representations with multiplicities $m_1$, $m_2$ respectively.
  Let $\widetilde{\sigma}$, resp. $\widetilde{\delta}$  be the representations of $\widetilde{H_1}$, resp. $\widetilde{H_2}$  as  defined  in Theorem \ref{cliffordadmissible} (4) (b) for $\sigma$ in $\Res_{H_1}^{G_1} \pi_1$, resp. $\delta$ in $\Res_{H_2}^{G_2} \pi_2$. We write $\widetilde{H_2}'$ for the inverse image of $\gamma(\widetilde{H_1}/{H_1})$  in $G_2$, and let $\widetilde{H_1}'$ be the analogous subgroup of $G_1$. Set $\widetilde{\Gamma}=\Gamma(\widetilde{H_1}\times \widetilde{H_2}')$, and $\widetilde{\rho}=\cInd_{\Gamma}^{\widetilde{\Gamma}} \rho$. Then $\pi\simeq \cInd_{\widetilde{\Gamma}}^{G_1 \times G_2} \widetilde{\rho}$.

\subsubsection{ Step 1} Let  us  first find out  $\pi_1, \pi_2$ such that $\pi_1 \otimes \pi_2\in \mathcal{R}_{G_1 \times G_2}(\pi)$.  Consider \begin{equation}\label{equattionss1}
 \begin{aligned}\Hom_{H_1 \times H_2}\big( \rho, \widetilde{\sigma} \otimes \pi_2\big) \simeq \Hom_{H_1 \times G_2} \big( \cInd_{H_1 \times H_2}^{H_1 \times G_2} \rho, \widetilde{\sigma} \otimes \pi_2\big) \simeq \Hom_{H_1 \times G_2} \big( \Res_{H_1 \times G_2}^{G_1 \times G_2}\cInd_{\Gamma}^{G_1 \times G_2} \rho, \widetilde{\sigma} \otimes \pi_2\big)\\
 \simeq \Hom_{G_1 \times G_2} \big( \cInd_{\Gamma}^{G_1 \times G_2} \rho, \Ind_{H_1}^{G_1} \widetilde{\sigma} \otimes \pi_2\big) \simeq \Hom_{H_1\times G_2}\big( \pi_{\pi_2}, \widetilde{\sigma}\otimes \pi_2\big)\simeq \Hom_{H_1}\big( \Theta_{\pi_2}, \widetilde{\sigma}\big)
 \end{aligned}\end{equation}
 The last term  has finite dimension because (1)  $\Theta_{\pi_2}$ is a smooth representation of $G_1$ of finite length, (2) for every $\kappaup\in \Irr(G_1)$, $m_{H_1}( \kappaup, \sigma) <\infty$,  (3) $\widetilde{\sigma}|_{H_1} \simeq  m_1\sigma$.  Then the proof of  Prop.\ref{twograph} shows that $\Hom_{H_1 \times H_2}\big( \rho, \widetilde{\sigma} \otimes \pi_2\big)$ is a smooth representation of $[\Gamma \cap (\widetilde{H_1}\times \widetilde{H_2}')]/(H_1\times H_2)$ via the canonical action, and  it contains at least an irreducible subrepresentation, say $(\psi^{-1}, \mathbb{C} F)$, so that  $F\in \Hom_{\Gamma \cap (\widetilde{H_1}\times \widetilde{H_2}')}\big( \rho, \psi\otimes \widetilde{\sigma} \otimes \pi_2\big)$. Let $\Psi$ be a character of $G_1/H_1$ extending $\psi$. \footnote{Here, the $ \psi$ can extend to a character of $G_1/{H_1}$, because $\cInd_{\widetilde{H_1}/H_1}^{G_1/H_1} \psi$ is finitely generated(\emph{cf}. Prop.\ref{finitegenerated}).} Then
\begin{equation}\label{equattionss7}
1\leq m_{\Gamma \cap (\widetilde{H_1}\times \widetilde{H_2}')}\big( \rho, \psi\otimes \widetilde{\sigma} \otimes \pi_2\big) = m_{\widetilde{H_1} \times  \widetilde{H_2}'}\big( \widetilde{\rho}, \Psi\otimes\widetilde{\sigma} \otimes \pi_2\big) =m_{G_1 \times G_2} \big( \pi, \Psi\otimes \pi_1 \otimes \pi_2\big)\leq1
 \end{equation}

  Clearly $\sigma \prec(\Psi\otimes \pi_1)|_{H_1}$. By replacing $\pi_1$ with $\Psi\otimes \pi_1$, we can  assume the beginning $\pi_1\otimes \pi_2 \in \mathcal{R}_{G_1 \times G_2}(\pi)$.

\subsubsection{Setp 2} Let us  consider the simple case that  $G_i/H_i$ is a finite abelian group.  By induction, we  can even assume that $G_i/H_i$ is a cyclic group. In this case, $m_1=m_2=1$. Consider
 $\Hom_{H_1 \times H_2}\big( \rho, \widetilde{\sigma} \otimes \pi_2\big) \simeq \Hom_{H_1}\big( \Theta_{\pi_2}, \widetilde{\sigma}\big)$
 which   has finite dimension, and  it can be decomposed as a direct sum of one-dimensional vector spaces, say $\sum_{i=1}^n \C F_i$, such that $\widetilde{H_1}/{H_1}$ acts on each $\C F_i$ via a character $\chi_i^{-1}$ of $\widetilde{H_1}/{H_1}$. Immediately, $F_i \in \Hom_{\widetilde{H_1}}\big( \Theta_{\pi_2}, \widetilde{\sigma} \otimes \chi_i\big)$. By Frobenius reciprocity,
 $\Hom_{\widetilde{H_1}} \big( \Theta_{\pi_2}, \widetilde{\sigma} \otimes \chi_i\big) \simeq \Hom_{G_1} \big( \Theta_{\pi_2}, \Ind_{\widetilde{H_1}}^{G_1} \widetilde{\sigma} \otimes \chi_i\big)$.
 By the property of graph of $\pi$ and Cor.\ref{IndCidn},  $\Ind_{\widetilde{H_1}}^{G_1} \widetilde{\sigma} \otimes \chi_i=\cInd_{\widetilde{H_1}}^{G_1} \widetilde{\sigma} \otimes \chi_i \simeq \pi_1$, for $i=1, \cdots, n$.   By  Theorem \ref{cliffordadmissible} we have  $\widetilde{\sigma} \otimes \chi_i \simeq \widetilde{\sigma}^{g_i}$ as $\widetilde{H_1}$-modules, for a representative $g_i\in G_1$ of some $\overline{g_i} \in G_1/{\widetilde{H_1}}$. So $\widetilde{\sigma}^{g_i}|_{H_1}\simeq \widetilde{\sigma} \otimes \chi_i |_{H_1} \simeq\widetilde{\sigma}|_{H_1}$, which implies that $g_i \in \widetilde{H_1}$ and $\widetilde{\sigma} \otimes \chi_i \simeq \widetilde{\sigma}$. Since $m_{\widetilde{H_1}}\big( \Theta_{\pi_2}, \widetilde{\sigma}\big)=1$, we can assert that the kernels of these $F_i$ are the same, and  $F_i$ are linearly independent (Here,  all $F_i$ are   $H_1$-morphisms from $\Theta_{\pi_2}$ to $\sigma$),  hence  $n=1$, i.e., $m_{H_1 \times H_2}\big( \rho, \widetilde{\sigma} \otimes \pi_2\big)=1$; hence  $\mathcal{R}_{H_1 \times H_2}(\rho) \cap \mathcal{R}_{H_1 \times H_2}(\sigma \otimes \pi_2)= \left\{ \sigma \otimes \delta\right\}$, and $m_{H_1 \times H_2}(\rho, \sigma \otimes \delta)=1$.
 If there is another $\delta'\in \Irr(H_2)$ such that $\sigma\otimes \delta' \in \mathcal{R}_{H_1 \times H_2}(\rho)$, then we can find $\pi_2'\in \Irr(G_2)$ such that $\delta'\prec \pi_2'|_{H_2} $, and $\pi_1 \otimes \pi_2' \in \mathcal{R}_{G_1 \times G_2}(\pi)$. Hence $\pi_2'\simeq \pi_2$, and we can assume $\delta'\prec \pi_2$. By the above discussion, we obtain $\delta'\simeq \delta$.
   \subsubsection{Setp 3: $\widetilde{H_v}=\widetilde{H_v}'$.}
 \begin{lemma}\label{semisimpleand}
 The restriction of $\pi_v$ to $\widetilde{H_v} \widetilde{H_v}'$ is semi-simple and multiplicity-free, for $v=1,2$.
 \end{lemma}
 \begin{proof}
 Assume $v=1$.  $\cInd_{\widetilde{H_1}}^{\widetilde{H_1} \widetilde{H_1}'} \widetilde{\sigma}$ is irreducible because $\cInd_{\widetilde{H_1}\widetilde{H_1'}}^{G_1}$ is an exact functor, and $\cInd_{\widetilde{H_1}\widetilde{H_1'}}^{G_1}\big(\cInd_{\widetilde{H_1}}^{\widetilde{H_1} \widetilde{H_1}'} \widetilde{\sigma}\big)\simeq\pi_1$. Let $\widetilde{\Delta_1}'  \subseteq G_1$ denote a complete  set of representatives for $G_1/{\widetilde{H_1} \widetilde{H_1}'}$. As  $\Res_{\widetilde{H_1} \widetilde{H_1}'}^{G_1}\pi_1\simeq  \sum_{g\in \widetilde{\Delta_1}'} \pi_1(g) \big( \cInd_{\widetilde{H_1}}^{\widetilde{H_1}\widetilde{H_1}'} \widetilde{\sigma}\big)$,  the representation   $\Res^{G_1}_{\widetilde{H_1} \widetilde{H_1}'}\pi_1$ is  semi-simple. The multiplicity-free property arises from $\widetilde{H_1}\widetilde{H_1'} \supseteq \widetilde{H_1}$.
 \end{proof}
 \begin{remark}
 $\cInd_{\widetilde{H_1}}^{\widetilde{H_1} \widetilde{H_1}'} \widetilde{\sigma}\simeq \Ind_{\widetilde{H_1}}^{\widetilde{H_1} \widetilde{H_1}'} \widetilde{\sigma}$, and $\cInd_{\widetilde{H_2}}^{\widetilde{H_2} \widetilde{H_2}'} \widetilde{\delta}\simeq\Ind_{\widetilde{H_2}}^{\widetilde{H_2} \widetilde{H_2}'} \widetilde{\delta}$.
 \end{remark}
 \begin{proof}
 Combing Theorem \ref{cliffordadmissible} (7) and the facts that both $\cInd$, $\Ind$  are exact functors, give the results.
 \end{proof}
  For the time being, we let $\widetilde{\Gamma}'= \Gamma \cdot \big( \widetilde{H_1} \widetilde{H_1}' \times \widetilde{H_2} \widetilde{H_2}'\big)$, and $\widetilde{\rho}'= \cInd_{\Gamma}^{\widetilde{\Gamma}'} \rho$.
  \begin{lemma}\label{therestrictionHH'}
  $\widetilde{\rho}'|_{\widetilde{H_1} \widetilde{H_1}' \times \widetilde{H_2} \widetilde{H_2}'} $ is a theta representation with respect to $\mathcal{R}_{ \widetilde{H_1} \widetilde{H_1}'}(\pi_1)$ and $\mathcal{R}_{ \widetilde{H_2} \widetilde{H_2}'}(\pi_2)$.
  \end{lemma}
  \begin{proof}
  If $\widetilde{\sigma}' \in \mathcal{R}_{\widetilde{H_1}\widetilde{H_1}'}(\pi_1)$, $\widetilde{\delta}' \in \mathcal{R}_{\widetilde{H_2}\widetilde{H_2}'}(\pi_2)$, then $\widetilde{\sigma}' \simeq \big( \cInd_{\widetilde{H_1}}^{\widetilde{H_1}\widetilde{H_1}'} \widetilde{\sigma}\big)^{g_1}$, for some $g_1\in G_1$, and $\cInd_{\widetilde{H_1}\widetilde{H_1}'}^{G_1} \widetilde{\sigma}' \simeq \pi_1$. Similar results also hold   for $\widetilde{\delta}'$. In case  $\widetilde{\rho}'_{\widetilde{\sigma}'} \simeq \widetilde{\sigma}' \otimes \Theta_{\widetilde{\sigma}'}$, $\cInd_{\widetilde{H_2}\widetilde{H_2}'}^{G_2} \Theta_{\widetilde{\sigma}'} \simeq \Theta_{\pi_1}$ by Lmm.\ref{therestriction2}. Hence $\Theta_{\widetilde{\sigma}'}$ is a $\widetilde{H_2}\widetilde{H_2}'$-module of finite length, and $m_{\widetilde{H_2}\widetilde{H_2}'} (\Theta_{\widetilde{\sigma}'}, \pi_2) =m_{G_2}(\Theta_{\pi_1}, \pi_2) \leq 1$. By  symmetry, the result holds.
  \end{proof}
 \begin{lemma}
 $\Res_{\widetilde{H_2}'}^{G_2} \pi_2$ is semi-simple and multiplicity-free.
 \end{lemma}
\begin{proof}
By the above result, $\cInd_{\widetilde{H_2}}^{\widetilde{H_2}\widetilde{H_2}'} \widetilde{\delta}$ is a direct summand of $\Res_{\widetilde{H_2}\widetilde{H_2}'}^{G_2} \pi_2$, so $\cInd_{\widetilde{H_2}\cap \widetilde{H_2}'}^{\widetilde{H_2}'} \widetilde{\delta}$ is a direct summand of $\Res_{\widetilde{H_2}'}^{G_2} \pi_2$. By Prop.\ref{finitegenerated}(1), $\mathcal{R}_{\widetilde{H_2}'}\Big( \cInd_{\widetilde{H_2}\cap \widetilde{H_2}'}^{\widetilde{H_2}'} \widetilde{\delta}\Big) \neq \emptyset$, and then $\mathcal{R}_{\widetilde{H_2}'}(\pi_2)\neq \emptyset$. By Theorem \ref{cliffordadmissible}, $\Res_{\widetilde{H_2}'}^{G_2} \pi_2$ is semi-simple.  On the other hand, by Frobenius reciprocity,
 \begin{equation}\label{equattionss}
\begin{aligned}
m_{\widetilde{H_1} \times  \widetilde{H_2}'}\big( \widetilde{\rho}, \widetilde{\sigma} \otimes \pi_2\big)= m_{\widetilde{H_1} \times G_2} \big( \cInd_{\widetilde{H_1} \times \widetilde{H_2}'}^{\widetilde{H_1} \times G_2}  \widetilde{\rho}, \widetilde{\sigma} \otimes \pi_2\big) = m_{\widetilde{H_1} \times G_2} \big( \Res_{\widetilde{H_1} \times G_2}^{G_1 \times G_2}\cInd_{\widetilde{\Gamma}}^{G_1 \times G_2} \widetilde{\rho}, \widetilde{\sigma} \otimes \pi_2\big)\\=m_{G_1 \times G_2} \big( \cInd_{\widetilde{\Gamma}}^{G_1 \times G_2} \widetilde{\rho}, \Ind_{\widetilde{H_1}}^{G_1} \widetilde{\sigma} \otimes \pi_2\big) =m_{G_1\times G_2}(\pi, \pi_1 \otimes \pi_2)=1
\end{aligned}\end{equation}
\end{proof}
Let $\widetilde{\delta}'\in \Irr(\widetilde{H_2}')$ such that $\widetilde{\sigma}\otimes \widetilde{\delta}' \in \mathcal{R}_{\widetilde{H_1} \times  \widetilde{H_2}'}(\widetilde{\rho})\cap \mathcal{R}_{\widetilde{H_1} \times  \widetilde{H_2}'}(\widetilde{\sigma} \otimes \pi_2) $.
\begin{lemma}
$\delta \prec \widetilde{\delta}'|_{H_2}$.
\end{lemma}
\begin{proof}
Assume $\widetilde{\delta}''\in \mathcal{R}_{\widetilde{H_2}'}(\pi_2)$, such that $\delta \prec \widetilde{\delta}''|_{H_2}$. Consider the $[\Gamma \cap (\widetilde{H_1}\times \widetilde{H_2}')]/(H_1\times H_2)$-module  $\Hom_{H_1 \times H_2}\big( \rho, \widetilde{\sigma} \otimes \widetilde{\delta}''\big)$. Similarly as the above step 1, there exist $\psi\in \Irr(\widetilde{H_1}/H_1)$, and $\Psi \in \Irr(G_1/H_1)$, such that $\Psi|_{\widetilde{H_1}/H_1} =\psi$, $\psi \otimes\widetilde{\sigma} \otimes \widetilde{\delta}'' \in \mathcal{R}_{\widetilde{H_1} \times \widetilde{H_2}'}(\widetilde{\rho})$, and $(\Psi\otimes \pi_1) \otimes \pi_2 \in \mathcal{R}_{G_1 \times G_2}(\pi)$.  Hence $\Psi \otimes \pi_1 \simeq \pi_1$, i.e. $ \cInd_{\widetilde{H}_1}^{G_1} (\widetilde{\sigma} \otimes \psi)\simeq\cInd_{\widetilde{H}_1}^{G_1} \widetilde{\sigma} $.  Consequently $\widetilde{\sigma} \otimes \psi \simeq \widetilde{\sigma}$, and $\widetilde{\sigma} \otimes \widetilde{\delta}''  \in \mathcal{R}_{\widetilde{H_1} \times \widetilde{H_2}'}(\widetilde{\rho})$. By (\ref{equattionss7}), we obtain $\widetilde{\delta}' \simeq \widetilde{\delta}''$, and $\delta \prec \widetilde{\delta}'|_{H_2}$.
\end{proof}
Note that $\{ g\in \widetilde{H_2}'\mid \delta^g\simeq \delta\}=\widetilde{H_2} \cap \widetilde{H_2}'$. Let $\widetilde{\widetilde{\delta}}$ denote the $\delta$-isotypic component in $\widetilde{\delta}'|_{H_2}$. Then $\widetilde{\delta}'\simeq \cInd_{\widetilde{H_2} \cap \widetilde{H_2}'}^{\widetilde{H_2}'} \widetilde{\widetilde{\delta}}$.
\begin{lemma}\label{2quots}
$\widetilde{H_2}'/\widetilde{H_2}\cap \widetilde{H_2}'$ is a finite abelian group.
\end{lemma}
\begin{proof}
For any $g\in \widetilde{H_2}'/\widetilde{H_2}\cap \widetilde{H_2}'$, we have $\sigma \otimes \delta^g \in \mathcal{R}_{H_1 \times H_2}(\rho)$; for different $\widetilde{H_2}\cap \widetilde{H_2}'$-cosets $g_1 \widetilde{H_2}\cap \widetilde{H_2}'$, $g_2 \widetilde{H_2}\cap \widetilde{H_2}'$, we know $\delta^{g_1} \ncong \delta^{g_2}$. By above $(\ref{equattionss1})$, we obtain the result.
\end{proof}
\begin{corollary}
For each $i$, $\widetilde{H_i}\widetilde{H_i}'/[\widetilde{H_i}\cap \widetilde{H_i}']$  is a finite abelian group.
\end{corollary}
\begin{proof}
By symmetry, the analogue result of the above lemma \ref{2quots} also holds for $\widetilde{H_1}'/[\widetilde{H_1}\cap \widetilde{H_1}']$, so  $\#\frac{\widetilde{H_i}\widetilde{H_i}'}{\widetilde{H_i}\cap \widetilde{H_i}'}=\#\frac{\widetilde{H_1}'}{\widetilde{H_1}\cap \widetilde{H_1}'}\cdot \#\frac{\widetilde{H_2}'}{\widetilde{H_2}\cap \widetilde{H_2}'} < +\infty $.
\end{proof}

 Set $\widetilde{\Gamma}''= [\Gamma\cap \big(\widetilde{H_1} \widetilde{H_1}' \times \widetilde{H_2} \widetilde{H_2}'\big)] \cdot  [( \widetilde{H_1}\cap \widetilde{H_1}')\times  (\widetilde{H_2}\cap \widetilde{H_2}')]$, and   $\widetilde{\rho}''= \cInd_{\Gamma \cap [(\widetilde{H_1} \widetilde{H_1}' ) \times (\widetilde{H_2}  \widetilde{H_2}')]}^{\widetilde{\Gamma}''} \rho$.
 \begin{remark}
\begin{itemize}
\item[(1)]$\widetilde{\rho}''|_{(\widetilde{H_1} \cap \widetilde{H_1}' )\times( \widetilde{H_2} \cap \widetilde{H_2}')} \simeq \cInd_{\Gamma\cap [(\widetilde{H_1} \cap \widetilde{H_1}' )\times( \widetilde{H_2} \cap \widetilde{H_2}')]}^{(\widetilde{H_1} \cap \widetilde{H_1}' )\times( \widetilde{H_2} \cap \widetilde{H_2}')} \rho$;
\item[(2)] $\cInd_{\widetilde{\Gamma}''}^{\widetilde{H_1}\widetilde{H_1}' \times \widetilde{H_2}\widetilde{H_2}'} \widetilde{\rho}'' \simeq \Res^{\widetilde{\Gamma}'}_{\widetilde{H_1} \widetilde{H_1}' \times \widetilde{H_2} \widetilde{H_2}'} \widetilde{\rho}' $. \end{itemize}
\end{remark}
\begin{proof}
1) It follows from that $\frac{\widetilde{\Gamma}''}{\Gamma\cap (\widetilde{H_1} \widetilde{H_1}' \times \widetilde{H_2}  \widetilde{H_2}')} \simeq \frac{(\widetilde{H_1} \cap \widetilde{H_1}' )\times( \widetilde{H_2} \cap \widetilde{H_2}')}{\Gamma\cap [(\widetilde{H_1} \cap \widetilde{H_1}' )\times( \widetilde{H_2} \cap \widetilde{H_2}')]}$, and $ [(\widetilde{H_1} \cap \widetilde{H_1}' )\times( \widetilde{H_2} \cap \widetilde{H_2}')]\cap  \Gamma \cap [(\widetilde{H_1} \widetilde{H_1}' ) \times (\widetilde{H_2}  \widetilde{H_2}')]= \Gamma\cap [(\widetilde{H_1} \cap \widetilde{H_1}' )\times( \widetilde{H_2} \cap \widetilde{H_2}')]$;

2) Both  sides are  isomorphic to $\cInd_{\Gamma \cap [(\widetilde{H_1} \widetilde{H_1}' ) \times (\widetilde{H_2}  \widetilde{H_2}')]}^{\widetilde{H_1}\widetilde{H_1}' \times \widetilde{H_2}\widetilde{H_2}'}\rho $.
\end{proof}
 Hence we can apply the result of Step 2 to $\widetilde{\rho}''|_{(\widetilde{H_1} \cap \widetilde{H_1}') \times (\widetilde{H_2} \cap \widetilde{H_2}')} $, and obtain:
  \begin{lemma}\label{thethetaofHcapH'}
  $\widetilde{\rho}''|_{(\widetilde{H_1} \cap \widetilde{H_1}' )\times( \widetilde{H_2} \cap \widetilde{H_2}')} $ satisfies the property of graph with respect to $\mathcal{R}_{ \widetilde{H_1} \cap \widetilde{H_1}'}(\pi_1)$ and $\mathcal{R}_{ \widetilde{H_2} \cap \widetilde{H_2}'}(\pi_2)$.
  \end{lemma}
 Suppose now that  $\sigma \otimes \delta' \in \mathcal{R}_{H_1 \times H_2}(\rho)$.  By the result in  Step 1, we  can assume  $\delta' \prec \Res_{H_2}^{G_2} \pi_2$. Let $\widetilde{\delta}'$ be the representation of $\widetilde{H_2}$ as defined in Theorem  \ref{cliffordadmissible} (4) (b) for $\delta'$ in $\Res_{H_2}^{G_2} \pi_2$.

Let $\widetilde{\sigma}' \in \mathcal{R}_{\widetilde{H_1} \cap \widetilde{H_1}'}( \pi_1)$ such that $\sigma \prec \widetilde{\sigma}' $.  By considering $ \Hom_{\Gamma\cap [(\widetilde{H_1} \cap \widetilde{H_1}' )\times( \widetilde{H_2} \cap \widetilde{H_2}')]}\big( \rho, \widetilde{\sigma}'\otimes \widetilde{\delta}\big)\simeq \big[\Hom_{H_1 \times H_2}\big(\rho, \widetilde{\sigma}'\otimes \widetilde{\delta}\big)\big]^{\Gamma\cap [(\widetilde{H_1} \cap \widetilde{H_1}' )\times( \widetilde{H_2} \cap \widetilde{H_2}')]/{(H_1 \times H_2)}}, $ we assert that $\mathcal{R}_{(\widetilde{H_1} \cap \widetilde{H_1}' )\times( \widetilde{H_2} \cap \widetilde{H_2}')}\big(
\widetilde{\sigma}' \otimes [\widetilde{\delta}\otimes\chi_2]\big) \cap  \mathcal{R}_{(\widetilde{H_1} \cap \widetilde{H_1}' )\times( \widetilde{H_2} \cap \widetilde{H_2}')}(\widetilde{\rho}'') \neq \emptyset$, for some character $\chi_2\in \Irr(\frac{\widetilde{H_2} \cap \widetilde{H_2}'}{H_2})$. Similarly, $\mathcal{R}_{(\widetilde{H_1} \cap \widetilde{H_1}' )\times( \widetilde{H_2} \cap \widetilde{H_2}')}\big(\widetilde{\sigma}'\otimes [\widetilde{\delta}'\otimes\chi_2'] \big) \cap  \mathcal{R}_{(\widetilde{H_1} \cap \widetilde{H_1}' )\times( \widetilde{H_2} \cap \widetilde{H_2}')}(\widetilde{\rho}'') \neq \emptyset$, for some character $\chi_2'\in \Irr(\frac{\widetilde{H_2} \cap \widetilde{H_2}'}{H_2})$.  By Lmm.\ref{thethetaofHcapH'}, $\mathcal{R}_{ \widetilde{H_2} \cap \widetilde{H_2}'}(\widetilde{\delta} \otimes \chi_2) \cap \mathcal{R}_{\widetilde{H_2} \cap \widetilde{H_2}'}(\widetilde{\delta}' \otimes \chi_2') \neq \emptyset$, and then $\delta \simeq \delta'$ as $H_2$-modules.

\begin{corollary}
The restriction $\rho|_{H_1 \times H_2}$ satisfies the property of graph.
\end{corollary}

 \begin{lemma}
 $\widetilde{H_2}' \subseteq \widetilde{H_2}$.
 \end{lemma}
 \begin{proof}
 If $(g,h) \in \Gamma \cap (\widetilde{H_1} \times \widetilde{H_2}')$, we have $\delta^h \simeq \delta$, so $\widetilde{H_2}' \subseteq \widetilde{H_2}$.
 \end{proof}
By considering the other side, we can assert $\widetilde{H_1}' \subseteq \widetilde{H_1}$, and then $\widetilde{H_2}' =\widetilde{H_2}$, $\widetilde{H_1}' =\widetilde{H_1}$.

\subsection{}
Continue  the above notations and  remove the superfluous $'$ if possible.  In this last subsection we will prove  the rest part of Theorem \ref{graphrepresentation}(2).

   \begin{lemma}\label{indco}
If $(\omega, U)$ is an indecomposable representation of $G_1/H_1$ of finite dimension $m$, then the Jordan-H\"older set  $\JH(\omega)=\{ \chi\}$,  for certain one-dimensional irreducible representation $\chi$ of $G_1/H_1$.
\end{lemma}
\begin{proof}
Assume $U=U_1 \supseteq U_2\supseteq \cdots \supseteq U_m\supseteq U_{m+1}=0$ is a complete composite series of $U$ as $G_1/H_1$-module such that  $G_1/H_1$ acts on $U_i/U_{i+1} $ via  a character $\chi_i$.  Then after choosing a  proper basis of $U$,  $\omega(h)$ acts on $U$ via an upper triangular matrix $\begin{bmatrix}
\chi_1(h)& \ast&  \ast\\
     &\ddots &   \ast\\
    &           & \chi_m(h)
\end{bmatrix}$.  If   $\chi_i\neq \chi_{i+1}$,  there exists $g\in G_1$ such that $\chi_i(g)\neq \chi_{i+1}(g)$. According to the result  in linear algebra, there exists a primary decomposition $V=\oplus_{i=1}^{n_1} V_i$ with respect to different eigenvalues  of $\omega(g)$.   Then $V_i$ is $G_1/H_1$-invariant; thus $n_1=1$, and all  $\chi_i(g)$ are equal, a contradiction.
\end{proof}
\begin{lemma}\label{npipi}
If $n\pi_i$ is a $G_i$-module  of length $n$ with the Jordan-H\"older set  $\JH(n\pi_i)=\{\pi_i\}$, then $n\pi_i$ is semi-simple.
\end{lemma}
\begin{proof}
We prove the result by induction on $n$. Since $\Ext_{G_i}^1(\pi_i, \pi_i)=0$,  the statement holds for $n=2$. For $n>2$,  there exists at least a short exact sequence of $G_i$-modules: $0 \longrightarrow 2\pi_i \longrightarrow n\pi_i \longrightarrow (n-2)\pi_i \longrightarrow 0$ $(\ast)$, which is determined by an element in $\Ext^1_{G_i}((n-2)\pi_i, 2\pi_i) \simeq \prod \Ext^1_{G_i}(\pi_i, \pi_i)=0$.  Hence the sequence $(\ast)$ is split, and $n\pi_i \simeq  \pi_i \oplus \cdots \oplus \pi_i$.
\end{proof}
   \subsubsection{In case $m_1=m_2=1$.}\  \\

A. \emph{Multiplicity-free property.}
In this case $\Theta_{\pi_2} \simeq \cInd_{\widetilde{H_1}}^{G_1} \Theta_{\widetilde{\delta}}$. Let $\Delta_1$ be a complete coset representatives of $G_1/\widetilde{H_1}$.  Then
\begin{align}
\Hom_{H_1}\big( \Theta_{\pi_2}, \widetilde{\sigma}\big) & \simeq \prod_{s\in \Delta_1} \Hom_{H_1}\big( \Theta_{\widetilde{\delta}}^s, \widetilde{\sigma}\big)
\simeq  \prod_{s\in \Delta_1}  \Hom_{H_1}\big(\Theta_{\widetilde{\delta}}, \sigma^{s^{-1}}\big) \\
& \simeq  \prod_{s\in \Delta_1}  \Hom_{H_1}\big(\Theta_{\delta}, \sigma^{s^{-1}}\big)  \simeq \Hom_{H_1}\big( \Theta_{\delta}, \sigma\big)
\end{align}  Now
 $\Hom_{H_1}\big( \Theta_{\pi_2}, \widetilde{\sigma}\big)$ is  a $\widetilde{H_1}/H_1$-module of finite length. By Krull-Schmidt theorem,  it can be decomposed as  a direct sum of indecomposable modules, say $\mathcal {V}_1\oplus \mathcal {V}_2\oplus \cdots \oplus \mathcal {V}_r$. Each $\mathcal {V}_i$ contains at least an irreducible $\widetilde{H_1}/H_1$-module, say $(\chi_i^{-1}, \mathbb{C} F_i)$. Then $F_i \in \Hom_{\widetilde{H_1}}\big( \Theta_{\pi_2}, \widetilde{\sigma} \otimes \chi_i\big)$.  Similarly as the argument in Setp 2 we can assert that the cardinality $r=1$, and  $\Hom_{H_1}\big( \Theta_{\pi_2}, \widetilde{\sigma}\big)$ is an indecomposable  $\widetilde{H_1}/H_1$-module. Let its contragredient representation denoted by $(\check{\omega_1}, \check{\mathcal{V}_1})$.
  \begin{lemma}\label{sigmadelta1}
$\Hom_{\widetilde{H_1}}\big( \Theta_{\pi_2}, \check{\omega_1}\otimes\widetilde{\sigma}\big)\simeq \Hom_{G_1}\big(\Theta_{\pi_2}, \Ind_{\widetilde{H_1}}^{G_1}( \check{\omega_1}\otimes\widetilde{\sigma})\big)\neq 0$.
 \end{lemma}
 \begin{proof}
 Let $\left\{F_1, \cdots, F_k\right\}$ be a basis of $\mathcal{V}_1$. Let $F_t$ be the dual base of $F_t^{\ast}$ in  $\check{\mathcal{V}_1}$. Then the mapping $\mathbbm{v}=\sum_{j=1}^k F_t^{\ast}\otimes F_t\in\Hom_{H_1}\big( \Theta_{\pi_2}, \check{\omega_1}\otimes\widetilde{\sigma}\big) $, sending $v\in V$ to $ \sum_{j=1}^k F_t^{\ast}\otimes F_t(v)$,  is $\widetilde{H_1}/{H_1}$-invariant.
 \end{proof}
By the above lemma \ref{indco}, we assume $\JH(\check{\omega_1})=\{ \psi\}$. Let $\Psi$ be a character of $G_1/H_1$ extending $\psi$.
\begin{lemma}
$\psi$  is the trivial character.
\end{lemma}
\begin{proof}
It is not hard to see that the Jordan-H\"older set  $\JH(\Ind_{\widetilde{H_1}}^{G_1}( \check{\omega_1}\otimes\widetilde{\sigma}))=\{ \Psi\otimes \pi_1 \}$. By Lmm. \ref{sigmadelta1}, $\Psi\otimes \pi_1\simeq \pi_1$, i.e. $ \Ind_{\widetilde{H_1}}^{G_1}(\psi \otimes\widetilde{\sigma}) \simeq \Ind_{\widetilde{H_1}}^{G_1}\widetilde{\sigma}$; $\psi \otimes \widetilde{\sigma}\simeq \widetilde{\sigma}^{g}$, for some $g\in G_1/\widetilde{H_1}$, $\sigma^{g}\simeq \sigma$; $g\in \widetilde{H_1}$; hence $\psi \otimes \widetilde{\sigma}\simeq \widetilde{\sigma}$ as $\widetilde{H_1}$-modules. Consequently $\Hom_{\widetilde{H_1}}(    \psi \otimes \widetilde{\sigma}, \widetilde{\sigma})\simeq  \Hom_{\widetilde{H_1}}(   \check{\widetilde{\sigma}}\otimes \widetilde{\sigma}, \check{\psi}) \simeq \Hom_{\widetilde{H_1}/H_1}((\check{\widetilde{\sigma}}\otimes \widetilde{\sigma})_{H_1}, \check{\psi})\simeq \Hom_{\widetilde{H_1}/H_1}(1, \check{\psi})\neq 0$. Hence $\check{\psi}$ is the trivial character.
\end{proof}
\begin{lemma}
\begin{itemize}
\item[(1)] $ \cInd_{\widetilde{H_1}}^{G_1}( \check{\omega_1}\otimes\widetilde{\sigma})$ is  semi-simple.
\item[(2)] $\check{\omega_1}\otimes\widetilde{\sigma}$ is semi-simple.
\end{itemize}
\end{lemma}
\begin{proof}
The first statement follows from Lmm.\ref{npipi}. So $ \cInd_{\widetilde{H_1}}^{G_1}( \check{\omega_1}\otimes\widetilde{\sigma}) \simeq \cInd_{\widetilde{H_1}}^{G_1}( \widetilde{\sigma}\oplus\cdots \oplus\widetilde{\sigma} ) $; by considering their $\sigma$-isotypic components we get the second statement.
\end{proof}
If  we have the decomposition: $\check{\omega_1}\otimes\widetilde{\sigma} \simeq \oplus_{i=1}^t  \widetilde{\sigma}_i$, then $\End_{\widetilde{H_1}}\big( \check{\omega_1}\otimes\widetilde{\sigma}\big)\simeq M_{t\times t}(\mathbb{C})$.  On the other hand, $\Hom_{\widetilde{H_1}}\big(  \check{\omega_1}\otimes\widetilde{\sigma}, \check{\omega_1}\otimes\widetilde{\sigma}\big)\simeq \Hom_{H_1}\big(  \check{\omega_1}\otimes\widetilde{\sigma}, \check{\omega_1}\otimes\widetilde{\sigma}\big)^{\widetilde{H_1}/H_1} \simeq \End_{\widetilde{H_1}/H_1}(\check{\omega_1})$, a local ring.  Therefore $t=1$, and $\dim\check{\omega_1}=1=m_{H_1}(\Theta_{\delta}, \sigma)$.

B. \emph{The finiteness condition.}
Before  proving the result, let us present some consequences of Casselman's results  on $\Ext^{\ast}(-,-)$ in  \cite[Appendix]{Cass2}.
 \begin{lemma}\label{proj1}
 $\cInd_{K_2}^{H_2} 1$ is  projective  in $\Rep(H_2)$, for any open compact subgroup $K_2$ of $H_2$.
 \end{lemma}
 \begin{proof}
 Given  a diagram       $\xymatrix{             &  \cInd_{K_2}^{H_2} 1 \ar[d]^F     & \\
                        U  \ar[r]^p  &    V  \ar[r]  & 0}$,  assume $F$ arises from  a $K_2$-morphism $f:  \mathbb{C} \longrightarrow V$, let $v_0=f(1)=p(u_0)$, for some $u_0\in U^{K_2}$, define a $K_2$-morphism  $g: \mathbb{C} \longrightarrow U^{K_2} \hookrightarrow U$ by $g(1)=u_0$, and let  $G : \cInd_{K_2}^{H_2} 1 \longrightarrow U$ be the corresponding  $H_2$-morphism by Frobenius reciprocity. It is not hard to see that $G$ lifts $F$.
  \end{proof}
 \begin{lemma}\label{proj2}
  Assume that $\Rep(H_2)$ is locally noetherian.  For a finitely generated representation $(\lambda, U)$  of $H_2$,  there exists a projective resolution $U_{\cdot} \longrightarrow U$, such that each $U_i$ is finitely generated.
 \end{lemma}
\begin{proof}
Assume $U$ is finitely generated by $u_1, \cdots, u_n$, and assume  an open compact subgroup $K_2\subseteq \cap_{i=1}^n\stab_{H_2} (u_i)$. Let $\iota_i: \cInd_{K_2}^{H_2}1 \longrightarrow U$, arising from a $K_2$-morphism  $\mathbb{C} \longrightarrow U; 1\longrightarrow u_i$. Then $\iota=\oplus_{i=1}^n \iota_i: \oplus_{i=1}^n \cInd_{K_2}^{H_2} 1 \longrightarrow U$ is a surjective $H_2$-morphism, and $\oplus_{i=1}^n\cInd_{K_2}^{H_2} 1$ is  a finitely generated projective object in $\Rep(H_2)$. Since $\Rep(H_2)$ is locally noetherian, we can continue this process, and obtain a required resolution.
\end{proof}

Go back to our proof. Applying the results of Lmm.\ref{therestriction2} to our situation shows that $\Theta_{\pi_1} \simeq \cInd_{\widetilde{H_2}}^{G_2} \Theta_{\widetilde{\sigma}}$ and $\Theta_{\sigma} \simeq \Theta_{\widetilde{\sigma}}|_{H_2}$.  By the property of  the exact functor $\cInd_{\widetilde{H_2}}^{G_2}$,  $\Theta_{\widetilde{\sigma}}$   is an indecomposable finite-length representation of  $\widetilde{H_2}$  (\emph{cf.} Lmm.\ref{mulitione}).
 Let its  Jordan-H\"older multiset  be recorded by $\left\{\widetilde{\delta}=\widetilde{\delta_1}, \cdots, \widetilde{\delta_k}\right\}$.

 \begin{lemma}
Let $\widetilde{\delta_i}$, $\widetilde{\delta_j}$ be two admissible representations of  $\widetilde{H_2}$ such that $\widetilde{\delta_i}|_{H_2} \simeq \oplus_{\nu \in I} \tau_{\nu} $, and $\widetilde{\delta_{j}}|_{H_2} \simeq \oplus_{\mu \in J} \tau_{\mu} $, for finite-length $H_2$-modules $\tau_{\nu}$ and $\tau_{\mu}$.
Let  \begin{equation}\label{ccqe}
0 \longrightarrow \widetilde{\delta_i} \longrightarrow  \widetilde{\Theta} \longrightarrow \widetilde{\delta_j} \longrightarrow 0
\end{equation}
be  an exact sequence of $\widetilde{H_2}$-modules. If  the cardinality of $J$ is finite,  $\widetilde{\Theta}|_{H_2}$ is  a direct sum of finite-length $H_2$-modules.
\end{lemma}
\begin{proof}

By Yodeda's extension theory(\emph{cf}. \cite[Appendix]{Cass2}, \cite[Chapter III]{Mac}),  the above sequence (\ref{ccqe})  is determined by  a class $\xi \in \Ext^1_{\widetilde{H_2}}(\widetilde{\delta_j}, \widetilde{\delta_i})$.
 Let $\xi_1$  be its image in $\Ext^1_{H_2}(\widetilde{\delta_j}, \widetilde{\delta_i})$ under the canonical mapping: $\Ext^1_{\widetilde{H_2}}(\widetilde{\delta_j}, \widetilde{\delta_i}) \longrightarrow  \Ext^1_{H_2}(\widetilde{\delta_j}, \widetilde{\delta_i})$. It is clear that $\widetilde{\Theta}|_{H_2}$ is taken  in charge by  $\xi_1$ and  there exists  $  \prod_{\nu\in I} p_{\nu}: \Ext^1_{H_2}(\widetilde{\delta_j},  \widetilde{\delta_i}) \hookrightarrow  \prod_{\nu\in I} \Ext^1_{H_2}(\widetilde{\delta_j}, \tau_{\nu})$.    Moreover by Lmms. \ref{proj1},\ref{proj2}, $p_{\nu}(\xi_1)=0$ for all $\nu$ but a finite number of $\nu \in I_0 $.  Let $0\longrightarrow \oplus_{\nu \in I_0} \tau_{\nu} \longrightarrow  \widetilde{\Theta}_{I_0} \longrightarrow \widetilde{\delta_j} \longrightarrow 0 \cdots  (\ast\ast)$ be a short exact sequence corresponding to $\prod_{v\in I_0} p_{\nu}(\xi_1) \in \Ext_{H_2}^1\big(\widetilde{\delta_j}, \oplus_{\nu \in I_0} \tau_{\nu} \big)$.  By Yodeda's theory,
  $\widetilde{\Theta }\simeq \widetilde{\Theta}_{I_0}\oplus (\oplus_{\nu\notin I_0} \tau_{\nu})$ as $H_2$-modules. Now  $\widetilde{\Theta}_{I_0}$ has finite length; by Krull-Schmidt theorem the result holds.
\end{proof}

 As $\widetilde{H_2}$-modules, there exists an exact sequence $0\longrightarrow \Theta_{\widetilde{\sigma}, 1} \longrightarrow \Theta_{\widetilde{\sigma}} \longrightarrow \widetilde{\delta_1} \longrightarrow 0$. By reordering the index, we assume   $\widetilde{\delta_2}$ is a quotient of $\Theta_{\widetilde{\sigma}, 1}$.  Then there exists an $\widetilde{H_2}$-module $\Theta^1_{\widetilde{\sigma}}$ such that the following diagram
  \[
\begin{array}{ccccccccccc}
0& \longrightarrow  &   \Theta_{\widetilde{\sigma}, 1}        & \longrightarrow     & \Theta_{\widetilde{\sigma}}  & \longrightarrow & \widetilde{\delta_1}   & \longrightarrow  & 0       \\
  &                  & \downarrow  &                     &  \downarrow      &                 &  \parallel   &                  &          \\
0 & \longrightarrow  &       \widetilde{\delta_2}         & \longrightarrow     &  \Theta^1_{\widetilde{\sigma}} & \longrightarrow & \widetilde{\delta_1}      & \longrightarrow  & 0
\end{array}
\]
is commutative. Moreover $ \Theta_{\widetilde{\sigma}}  \longrightarrow \Theta^1_{\widetilde{\sigma}}$ is surjective.  By the above lemma,  $\Theta^1_{\widetilde{\sigma}}|_{H_2}$ is a direct sum of finite-length $H_2$-modules. Since $\Theta_{\widetilde{\sigma}}|_{H_2}$ has only one quotient representation $\delta$ with multiplicity one, $\Theta^1_{\widetilde{\sigma}}|_{H_2} $ must be an indecomposable module. We can   repeat the above process by replacing $\widetilde{\delta_1}$ with  $\Theta^1_{\widetilde{\sigma}} $,  and obtain an $I_{G_2}(\delta)$-module $\Theta^2_{\widetilde{\sigma}}$ such that  the Jordan-H\"older multiset  of $\Theta^2_{\widetilde{\sigma}}$  is just $\{\widetilde{\delta_1}, \widetilde{\delta_2}, \widetilde{\delta_3}\}$, and $\Theta^2_{\widetilde{\sigma}}|_{H_2} $ is an indecomposable module. After a finite step, finally we can see that  $\Theta_{\widetilde{\sigma}}|_{H_2}\simeq \Theta_{\sigma}$ is  an indecomposable module of finite length.

   C. $\Ext^1_{H_i}=0$.  Applying the exact functor $\cInd_{\widetilde{H_1}}^{G_1}$ to a short exact sequence  of $\widetilde{H_1}$-modules  $0\longrightarrow \widetilde{\sigma} \longrightarrow 2\widetilde{\sigma}    \longrightarrow  \widetilde{\sigma}     \longrightarrow 0$, we obtain $0\longrightarrow \pi_1 \longrightarrow \cInd_{\widetilde{H_1}}^{G_1}2\widetilde{\sigma}    \longrightarrow \pi_1     \longrightarrow 0$. Hence   $\cInd_{\widetilde{H_1}}^{G_1}2\widetilde{\sigma} \simeq\cInd_{\widetilde{H_1}}^{G_1}(\widetilde{\sigma} \oplus \widetilde{\sigma})$.  By considering  their $\sigma$-isotypic components we obtain $ 2\widetilde{\sigma}  \simeq  \widetilde{\sigma} \oplus \widetilde{\sigma}$. Hence $\Ext^1_{\widetilde{H_1}}(\widetilde{\sigma}, \widetilde{\sigma})=0$.

 Assume  $0\longrightarrow \widetilde{\sigma} \stackrel{f}{\longrightarrow} 2\widetilde{\sigma} \stackrel{g}{ \longrightarrow}  \widetilde{\sigma}     \longrightarrow 0$  $(\ast\ast\ast)$  is a short exact sequence of $H_1$-modules. The $\widetilde{H_1}/H_1$-module $\Hom_{H_1}(\widetilde{\sigma}, \Im(f))$ has one dimension, so there exists a character $\chi\in \Irr(\widetilde{H_1}/{H_1})$ such that $f$ defines an $\widetilde{H_1}$-morphism from $\chi\otimes \widetilde{\sigma} $ to $ \widetilde{\sigma}$. Consequently $\Hom_{\widetilde{H_1}}(    \chi\otimes \widetilde{\sigma}, \widetilde{\sigma})\simeq  \Hom_{\widetilde{H_1}}(   \check{\widetilde{\sigma}}\otimes \widetilde{\sigma}, \check{\chi}) \simeq \Hom_{\widetilde{H_1}/H_1}((\check{\widetilde{\sigma}}\otimes \widetilde{\sigma})_{H_1}, \check{\chi})\simeq \Hom_{\widetilde{H_1}/H_1}(1, \check{\chi})\neq 0$. Hence $\check{\chi}$ is the trivial character. Similarly, $g$ also defines an $\widetilde{H_1}$-morphism. Hence the sequence $(\ast\ast\ast)$ is split, and then $\Ext^1_{H_1}(\sigma, \sigma)=0$.  By symmetry, $\Ext^1_{H_2}(\delta, \delta)=0$.
         \subsubsection{In case $m_1m_2>1$}

Invoking the result of above Step II in the proof of Theorem \ref{graphrepresentation}(1), we have a tower  of normal subgroups of $G_i$:
$H_i = H_i^{(0)} \subseteq H_i^{(1)} \subseteq  \cdots \subseteq H_i^{(n)} = \widetilde{H_i} \subseteq G_i$ satisfying the described property there. Using the result in the case $m_1=m_2=1$, inductively we obtain the result.

\section{The theta representation II }\label{stronglygraphreII}
In this section, assume that $G_i/H_i$ is a compact group,  and  the category $\Rep(H_i)$ is locally noetherian,  for $i=1, 2$.  Set $\pi= \cInd_{\Gamma}^{G_1 \times G_2} \rho, V= \cInd_{\Gamma}^{G_1 \times G_2} W$. Our  main purpose of this section is to prove the following result:
 \begin{theorem}\label{themainlemma1}
  \begin{itemize}
   \item[(1)] If the representation $\rho$ of $H_1 \times H_2$ is a theta  representation, then so is  the representation $\pi$ of $G_1 \times G_2$.
   \item[(2)] Suppose that  $\mathcal{L}_{G_i}(\Ind_{H_i}^{G_i}\sigma_i) \neq \emptyset$,   for every $\sigma_i \in \Irr(H_i)$,  $i=1,2$. If  the representation $\pi$   of $G_1 \times G_2$ is a theta representation, then so is the representation  $\rho$ of $H_1 \times H_2$.
   \end{itemize}
 \end{theorem}
  Before proving the results let us present a lemma analogue of  Lmms. \ref{therestrictionfinitelengthII}, \ref{therestriction2}. Assume $\sigma\otimes \delta \in \mathcal{R}_{H_1\times H_2}(\rho)$, and   $ \sigma\prec \pi_1|_{H_1}$, $\delta\prec \pi_2|_{H_2}$, for some $(\pi_i, V_i) \in \Irr(G_i)$.  Let $I_{G_1}(\sigma)=\{g\in G_1\mid \sigma^g \simeq \sigma\}$, $I_{G_2}(\delta)=\{ g\in G_2\mid \delta^g \simeq \delta\}$, and let $\widetilde{\sigma}$ denote the $\sigma$-isotypic component of $\pi_1|_{H_1}$. Let us write  $I'_{G_2}(\delta)$ to be  the inverse image  of $\gamma(\frac{I_{G_1}(\sigma)}{H_1})$ in $G_2$, and let  $\pi_{(\sigma, \delta)}=\cInd_{\Gamma\cap [I_{G_1}(\sigma) \times I'_{G_2}(\delta)]}^{I_{G_1}(\sigma) \times I'_{G_2}(\delta)} \rho$, $[\pi_{(\sigma, \delta)}]_{\widetilde{\sigma}} \simeq \widetilde{\sigma} \otimes \Theta_{\widetilde{\sigma}}$.

\begin{lemma}\label{therestriction23}
\begin{itemize}
\item[(1)] $\cInd_{H_2}^{G_2} (\rho_{\sigma} )\simeq (\cInd_{H_2}^{G_2} \rho)_{\sigma}$  as $H_1 \times G_2$-modules, for all $\sigma\in \Irr(H_1)$.
\item[(2)]
\begin{itemize}
\item[(a)] $\Theta_{\pi_1} \simeq \cInd_{I'_{G_2}(\delta)}^{G_2} \Theta_{\widetilde{\sigma}}$ as $G_2$-modules.
\item[(b)] If  $\widetilde{\sigma}|_{H_1} \simeq m\sigma$,  then  there exists a surjection  $ \Theta_{\widetilde{\sigma}}|_{H_2} \twoheadrightarrow \Theta_{\sigma}  $ as $H_2$-modules.
    \item[(c)] If the above $m=1$, then $\Theta_{\sigma} \simeq   \Theta_{\widetilde{\sigma}}|_{H_2}$ as $H_2$-modules.
\end{itemize}
\end{itemize}
\end{lemma}
\begin{proof}
 1) The canonical map $\Ind_{H_2}^{G_2} \rho \twoheadrightarrow \rho$ will induce an $H_1\times H_2$-morphism $\big( \Ind_{H_2}^{G_2} \rho\big)_{\sigma} \longrightarrow \rho_{\sigma}$, and  an $H_1 \times G_2$-morphism $\big( \Ind_{H_2}^{G_2} \rho\big)_{\sigma} \stackrel{\kappa_{\sigma}}{\longrightarrow}\Ind_{H_2}^{G_2} \rho_{\sigma}$.  For any open compact subgroup $K_2$ of $G_2$, let $\Delta=\{ s_1, s_2, \cdots, s_n\}$ be a complete set of representatives for $H_2\setminus G_2 / K_2$, and  let $H_{2, s}=s^{-1} H_2 s$.  By lemma \ref{therestriction1},
$\Hom_{H_1\times K_2} \big( (\Ind_{H_2}^{G_2} \rho)_{\sigma} ,  \sigma \otimes \mathbb{C}\big) \simeq \Hom_{H_1\times K_2} ( \Ind_{H_2}^{G_2} \rho,  \sigma \otimes \mathbb{C}) \simeq \Hom_{H_1\times K_2} (\oplus_{s\in \Delta}  \cInd_{H_{2,s} \cap K_2}^{K_2} \rho^s,  \sigma \otimes \mathbb{C})
\simeq \oplus_{s\in \Delta}   \Hom_{H_1\times ( H_{2,s} \cap K_2)} (\rho^s,  \sigma \otimes \mathbb{C}) \simeq \oplus_{s\in \Delta}   \Hom_{H_1\times (H_{2,s} \cap K_2)} \big(\rho^s_{\sigma},  \sigma \otimes \mathbb{C}\big)
 \simeq \oplus_{s\in \Delta}   \Hom_{H_1\times K_2} \big( \cInd_{H_{2,s} \cap K_2}^{K_2} \rho^s_{\sigma},  \sigma \otimes \mathbb{C}\big) \simeq \Hom_{H_1\times K_2} ( \Ind_{H_2}^{G_2} \rho_{\sigma}, \sigma \otimes \mathbb{C}) $.
Hence $\kappa_{\sigma}$ is an isomorphism by Lmm.\ref{thedualityequa}.\\
(2)(a) By  the above result(not need the normal condition), we have $\pi_{\widetilde{\sigma}} \simeq \widetilde{\sigma} \otimes \cInd_{I'_{G_2}(\delta)}^{G_2}  \Theta_{\widetilde{\sigma}}$ as $I_{G_1}(\sigma)\times G_2$-modules. By \cite[p.18]{BushH},   there exists an  $I_{G_1}(\sigma) \times G_2$-morphism  $p: \pi_{\pi_1}\longrightarrow \pi_{\widetilde{\sigma}}$. Then a $G_1\times G_2$-morphism $ \Ind_{I_{G_1}(\sigma)}^{G_1}p: \pi_{\pi_1}\longrightarrow \Ind_{I_{G_1}(\sigma)}^{G_1}\pi_{\widetilde{\sigma}}$ follows, and then we get a $G_2$-morphism $ \iotaup:\Theta_{\pi_1}\longrightarrow\cInd_{I'_{G_2}(\delta)}^{G_2} \Theta_{\widetilde{\sigma}} $.   For any representation $(\sigmaup_2, U_2)$ of $G_2$, we  have
\begin{equation}\label{cisomorphis}
\Hom_{G_2}( \Theta_{\pi_1}, \sigmaup_2) \simeq  \Hom_{G_1 \times G_2}( \pi, \pi_1 \otimes \sigmaup_2)\simeq \Hom_{I_{G_1}(\sigma) \times G_2}( \pi_{\widetilde{\sigma}}, \widetilde{\sigma} \otimes \sigmaup_2)\simeq \Hom_{G_2}\big( \cInd_{I'_{G_2}(\delta)}^{G_2} \Theta_{\widetilde{\sigma}}, \sigmaup_2\big),
\end{equation}
Similarly, $\iota$ is an isomorphism.\\
(2)(b)   There   exists a canonical morphism $q: \pi_{(\sigma, \delta)} \twoheadrightarrow \rho$ as $\Gamma \cap [I_{G_1}(\sigma) \times I'_{G_2}(\delta)]$-modules. Moreover,
\begin{equation}\label{sevaliso22}
   \begin{aligned}
   \Hom_{I_{G_1}(\sigma)} (\pi_{(\sigma, \delta)} , \widetilde{\sigma}) \simeq \Hom_{I_{G_1}(\sigma) \times 1} (\pi_{(\sigma, \delta)} , \widetilde{\sigma} \otimes \mathbb{C}) \simeq \Hom_{I_{G_1}(\sigma) \times H_2} ( \pi_{(\sigma, \delta)}, \widetilde{\sigma} \otimes \Ind_{1}^{H_2} 1)\\
\simeq \Hom_{I_{G_1}(\sigma) \times H_2} (\cInd_{H_1 \times H_2}^{I_{G_1}(\sigma) \times H_2} \rho, \widetilde{\sigma} \otimes \Ind_{1}^{H_2} 1) \simeq \Hom_{H_1\times H_2}(\rho, \widetilde{\sigma} \otimes \Ind_{1}^{H_2} 1) \simeq \Hom_{H_1}(\rho, \widetilde{\sigma})
\end{aligned}
   \end{equation}
By following these isomorphisms, for any  $f\in \Hom_{I_{G_1}(\sigma)}(\pi_{(\sigma, \delta)}, \widetilde{\sigma}) $,  as an  $H_1$-module morphism,  it needs to decompose as   $ \pi_{(\sigma, \delta)}  \stackrel{q}{ \twoheadrightarrow} \rho \stackrel{f_1}{\longrightarrow}  \widetilde{\sigma} $, for some $f_1\in  \Hom_{H_1}(\rho, \widetilde{\sigma})$.  The converse also holds. Hence there exists a canonical morphism $  \frac{\pi_{(\sigma, \delta)} }{\cap_{\widetilde{f}\in \Hom_{I_{G_1}(\sigma)}\big(\pi_{(\sigma, \delta)}, \widetilde{\sigma}\big)} \Ker \widetilde{f} }
\stackrel{q}{ \twoheadrightarrow}  \frac{\rho}{\cap_{f\in \Hom_{H_1}(\rho, \widetilde{\sigma})} \Ker f }=\frac{\rho}{\cap_{f\in \Hom_{H_1}(\rho, \sigma)} \Ker f }$, which introduces  an  $H_1 \times H_2$-morphism $ \kappa_{\sigma}: \widetilde{\sigma}\otimes \Theta_{\widetilde{\sigma}} \longrightarrow \sigma \otimes \Theta_{\sigma}$, and then  an  $ H_2$-morphism $\kappa: \Theta_{\widetilde{\sigma}}\longrightarrow \Theta_{\sigma} $. For any smooth representation $(\sigmaup_2, W_2)$ of $H_2$, by Frobenius reciprocity, we  have

   \begin{equation}\label{sevalisomm}
      \begin{aligned}
      \Hom_{H_2}(\Theta_{\sigma},  \sigmaup_2)\simeq \Hom_{H_1\times H_2}(\rho_{\sigma}, \sigma\otimes \sigmaup_2) \hookrightarrow
           \Hom_{H_1\times H_2}(\rho, \widetilde{\sigma}\otimes \sigmaup_2) \\
           \simeq    \Hom_{I_{G_1}(\sigma)\times H_2}( \pi_{(\sigma, \delta)}, \widetilde{\sigma}\otimes \sigmaup_2) \simeq
                     \Hom_{I_{G_1}(\sigma)\times H_2}([ \pi_{(\sigma, \delta)} ]_{\widetilde{\sigma}}, \widetilde{\sigma}\otimes \sigmaup_2) \simeq
                     \Hom_{H_2}( \Theta_{\widetilde{\sigma}}, \sigmaup_2)
                           \end{aligned}
   \end{equation}
   i.e. we get  a injective  morphism $ \Hom_{H_2}(\Theta_{\sigma},  \sigmaup_2)  \longrightarrow \Hom_{H_2}( \Theta_{\widetilde{\sigma}}, \sigmaup_2)
 $ compatible with the above $\kappa$, so the result holds.\\
 (2)(c) In the above (\ref{sevalisomm}), $
 \Hom_{H_2}(\Theta_{\sigma},  \sigmaup_2) \simeq \Hom_{H_2}( \Theta_{\widetilde{\sigma}}, \sigmaup_2)$, so $\Theta_{\sigma} \simeq   \Theta_{\widetilde{\sigma}}|_{H_2}$ as $H_2$-modules.
\end{proof}

\subsection{The proof of the   part (1).}
\begin{lemma}
If $(\pi_1, V_1) \in \Irr(G_1)$, and $(\pi_2, V_2) \in \Irr(G_2)$, such that $\pi_1 \otimes \pi_2 \in \mathcal{R}_{G_1 \times G_2}(\pi)$, then:
\begin{itemize}
\item[(1)] For $\sigma \in \mathcal{R}_{H_1}(\pi_1)$,   there exists a unique element $\delta \in \mathcal{R}_{H_2}(\pi_2)$ such that $\sigma \otimes \delta \in \mathcal{R}_{H_1 \times H_2}(\rho)$.
\item[(2)] If $\sigma\otimes \delta\in \mathcal{R}_{H_1 \times H_2}(\rho)$, then $\gamma$ induces a bijective map from $I_{G_1}( \sigma)/{H_1}$ to $I_{G_2}(\delta)/{H_2}$ with the graph $\Gamma_{(\sigma,\delta)} /{(H_1 \times H_2)}$, where $\Gamma_{(\sigma,\delta)}= \Gamma \cap \big( I_{G_1}(\sigma) \times I_{G_2}(\delta)\big)$.
 \item[(3)]    For two irreducible constituents $(\sigma, \mathbb{U})$, $(\delta, \mathbb{W})$   of $\Res_{H_1}^{G_1} \pi_1$ and $\Res_{H_2}^{G_2} \pi_2$ respectively, we let $I_{G_1}^0 \big( \sigma, \delta\big)=\left\{ g_1\in G_1 \mid g_1(\mathbb{U}) \subseteq \mathbb{U} \textrm{ and } \gamma(g_1)(\mathbb{W}) \subseteq \mathbb{W} \right\}$, and  $I_{G_2}^0 \big( \sigma, \delta\big)=\left\{ g_2\in G_2 \mid g_2(\mathbb{W}) \subseteq \mathbb{W}, \textrm{ and } \gamma^{-1}(g_2)(\mathbb{U}) \subseteq \mathbb{U} \right\}$. Then:
     \begin{itemize}
    \item[(a)] $ I_{G_1}^0 \big( \sigma, \delta\big)$, $I_{G_2}^0 \big( \sigma, \delta\big)$ are open subgroups of $G_1, G_2$ respectively;
    \item[(b)]$\gamma$ maps $I_{G_1}^0 \big( \sigma, \delta\big)/H_1$ onto $I_{G_2}^0 \big( \sigma, \delta\big)/H_2$.
        \end{itemize}
    \end{itemize}
\end{lemma}
\begin{proof}
1)  Let us write $ I'_{G_2}(\delta)/{H_2}=\gamma(I_{G_1}(\sigma)/{H_1})$, and $\widetilde{\sigma}$  the $\sigma$-isotypic component of $\pi_1$. Then $\pi_1\simeq \cInd_{I_{G_1}(\sigma)}^{G_1} \widetilde{\sigma}$.  By Frobenius reciprocity,  \footnote{ By Lmm.\ref{homeo}, $(G_1\times G_2) /\Gamma$ is homeomorphic  to $\frac{G_1}{H_1}$, compatible with the $G_1$-action. Note that $\frac{G_1}{H_1}$ is a compact group having a Haar measure; thus there exists a  left quasi-invariant measure on $(G_1\times G_2)/ \Gamma$, which implies $\Delta_{G_1\times G_2}|_{\Gamma} =\Delta_{\Gamma}$.}
\begin{equation}\label{frobeniusrecomplex}
\begin{aligned}
m_{G_1 \times G_2} \big( \pi, \pi_1 \otimes \pi_2\big)= m_{G_1 \times G_2}\big(\pi , \Ind_{I_{G_1}(\sigma)\times G_2}^{G_1 \times G_2} \widetilde{\sigma} \otimes \pi_2\big)
= m_{ I_{G_1}(\sigma)\times I'_{G_2}(\delta)}\big(\pi_{(\sigma, \delta)},\widetilde{\sigma} \otimes \pi_2\big)\\ \leq m_{ I_{G_1}(\sigma)\times H_2}\big(\pi_{(\sigma, \delta)},\widetilde{\sigma} \otimes \pi_2\big)
=m_{ I_{G_1}(\sigma)\times H_2}\big(\cInd^{I_{G_1}(\sigma)\times H_2}_{H_1\times H_2}\rho,\widetilde{\sigma} \otimes \pi_2\big)\stackrel{}{=}m_{H_1 \times H_2}(\rho,\widetilde{\sigma} \otimes \pi_2)
\end{aligned}
\end{equation}
 So  we can find $\delta \in \mathcal{R}_{H_2}(\pi_2)$ such that $ \sigma \otimes \delta \in \mathcal{R}_{H_1 \times H_2}(\rho) \cap \mathcal{R}_{H_1 \times H_2}(\pi_1 \otimes \pi_2)$. And the uniqueness is clear.\\
2) Assume $g_1 H_1 \in I_{G_1}(\sigma)/{H_1}$, and $\gamma(g_1H_1)=g_2H_2 \in G_2/{H_2}$. We then have $\sigma^{g_1} \otimes \delta^{g_2} \simeq \sigma \otimes \delta^{g_2} \in \mathcal{R}_{H_1 \times H_2}(\rho)$, which implies that $\delta^{g_2} \simeq \delta$, and then $g_2 \in I_{G_2}(\delta)$. The converse also holds, so $\gamma$ maps $I_{G_1}(\sigma)/{H_1}$ onto $I_{G_2}(\delta)/{H_2}$ with the graph $\Gamma\cap \big( I_{G_1}(\sigma) \times I_{G_2}(\delta)\big)/{(H_1 \times H_2)}$.\\
3) The results arise from Lmm.\ref{opennormal}.
\end{proof}

 Keep the notations. We take an open normal subgroup $J_{G_1}(\sigma) $ of $I^0_{G_1}(\sigma, \delta )$ as defined in Lmm.\ref{theopensubgroup}, and write its image in $I_{G_2}(\delta)/H_2$ by $J_{G_2}(\delta)/H_2$ through $\gamma$. Let $(\mathfrak{n}_1,  \mathcal{N}_1)$, $(\mathfrak{n}_2, \mathcal{N}_2)$, resp. $(\mathfrak{m}_1, \mathcal{M}_1)$ and $(\mathfrak{m}_2, \mathcal{M}_2)$ be two projective representations related to $(\widetilde{\sigma}, \widetilde{\mathbb{U}})$, and $(\widetilde{\delta}, \widetilde{\mathbb{W}})$ respectively in Theorem \ref{thetensorprojectivereps}. Let $\Delta_1$, $\Delta_2$ be the relative sets of representatives for $I_{G_1}(\sigma)/J_{G_1}(\sigma)$ and $I_{G_2}(\delta)/J_{G_2}(\delta)$ respectively. On $\Hom_{H_1 \times H_2}(\rho, \widetilde{\sigma} \otimes \widetilde{\delta})$, we impose a  natural left $\Gamma_{(\sigma, \delta)}/{(H_1 \times H_2)}$-action defined as follows:
$$\overline{a} \cdot\varphi(\widetilde{v}):= \varphi^{\overline{a}}(\widetilde{v})=\widetilde{\sigma} \otimes \widetilde{\delta} (a) \varphi\big( \rho(a^{-1}) \widetilde{v}\big)$$
for $\overline{a} \in \Gamma_{(\sigma, \delta)}/{(H_1 \times H_2)}$, $\varphi \in \Hom_{H_1 \times H_2}(\rho, \widetilde{\sigma} \otimes \widetilde{\delta})$, $\widetilde{v} \in \widetilde{\mathbb{U}} \otimes \widetilde{\mathbb{W}}$, and a representative $a$ of  $\overline{a}$ in $\Gamma_{(\sigma, \delta)}$. So $\Hom_{\Gamma_{(\sigma, \delta)}}(\rho, \widetilde{\sigma} \otimes \widetilde{\delta}) \simeq \Hom_{H_1 \times H_2} ( \rho, \mathcal{N}_1 \otimes \mathcal{N}_2 \otimes \mathcal{M}_1 \otimes \mathcal{M}_2 )^{\frac{\Gamma_{(\sigma, \delta)}}{H_1 \times H_2}}$. Recall that $m_{H_1 \times H_2}(\rho, \mathfrak{n}_1 \otimes\mathfrak{m}_1)=1$.  Let us now  fix a nonzero element $F \in \Hom_{H_1 \times H_2} (\rho, \mathfrak{n}_1 \otimes \mathfrak{m}_1)$.  In view of Theorem \ref{thetensorprojectivereps}, we have
$$\Hom_{H_1 \times H_2} (\rho, \mathfrak{n}_1 \otimes \mathfrak{n}_2 \otimes \mathfrak{m}_1 \otimes \mathfrak{m}_2)\simeq \Hom_{H_1 \times H_2}(\rho, \mathfrak{n}_1 \otimes \mathfrak{m}_1)  \otimes \mathcal{N}_2 \otimes \mathcal{M}_2\simeq  \mathcal{N}_2 \otimes \mathcal{M}_2.$$
The action of $\frac{\Gamma_{(\sigma, \delta)}}{H_1 \times H_2}$ on $ \Hom_{H_1 \times H_2}(\rho, \mathfrak{n}_1 \otimes \mathfrak{m}_1)\otimes \mathcal{N}_2 \otimes \mathcal{M}_2$ is described as follows:
\begin{itemize}
\item If $(x_1, x_2) \in \Gamma \cap \big( J_{G_1}(\sigma) \times J_{G_2}(\delta)\big)$, with the projection $(\overline{x_1}, \overline{x_2})$ in $\Gamma_{(\sigma, \delta)}/{(H_1 \times H_2)}$,
 then $(\overline{x_1}, \overline{x_2}) \cdot F(v)\otimes \varphi \otimes \psi  =\widetilde{\sigma}(x_1) \otimes \widetilde{\delta} (x_2) F\big( \rho(x_1^{-1}, x_2^{-1}) v\big) \otimes  \varphi \otimes \psi  =\beta(\overline{x_1}, \overline{x_2})F (v)\otimes  \varphi \otimes \psi $, for $\varphi \in \mathcal{N}_2$, $\psi \in \mathcal{M}_2$, and  suitable $\beta(\overline{x_1}, \overline{x_2}) \in \C^{\times}$.
\item If $(x_1, x_2) \in \Gamma \cap \big( I_{G_1}(\sigma) \times I_{G_2}(\delta)\big)$, with the decomposition $x_1=gg_0$, $x_2=hh_0$, for $g\in \Delta_1$, $g_0 \in J_{G_1}(\sigma)$, $h\in \Delta_2$, $h_0 \in J_{G_2}(\delta)$. Then
    $$(\overline{x_1}, \overline{x_2}) \cdot F(v) \otimes \varphi \otimes\psi=   [\widetilde{\sigma}(x_1) \otimes \widetilde{\delta}(x_2) F\big( \rho(x_1^{-1}, x_2^{-1}) v\big)]\big( \mathcal{E}_g^{-1}  \otimes \mathcal{E}_h^{-1}\big)\otimes \big( \mathcal{E}_g \circ  \varphi\otimes \mathcal{E}_h \circ\psi \big)$$
    for $v\in V$. Note that $ [\widetilde{\sigma} (x_1) \otimes \widetilde{\delta}(x_2)F\big( \rho(x_1^{-1}, x_2^{-1})-\big)](\mathcal{E}_g^{-1} \otimes \mathcal{E}_h^{-1})$ also lies in $\Hom_{H_1 \times H_2} (\rho, \mathbb{U} \otimes \mathbb{W})$, so it equals to $\beta(\overline{x_1}, \overline{x_2}) F(-)$ for some $\beta(\overline{x_1}, \overline{x_2}) \in \C^{\times}$. Finally we conclude that $(\overline{x_1}, \overline{x_2}) \cdot  F\otimes \varphi \otimes \psi= \beta(\overline{x_1}, \overline{x_2}) F \otimes ( \mathcal{E}_g \circ\varphi ) \otimes ( \mathcal{E}_h\circ \psi)$.
\end{itemize}
By use of the isomorphism $\Hom_{H_1 \times H_2}(\rho, \widetilde{\sigma} \otimes \widetilde{\delta}) \simeq  \mathcal{N}_2 \otimes \mathcal{M}_2$, let us denote the induced  representation  of $\frac{\Gamma_{(\sigma, \delta)}}{(H_1 \times H_2)}$ on $ \mathcal{N}_2 \otimes \mathcal{M}_2$ by $(\iota, \mathcal{N}_2 \otimes \mathcal{M}_2)$. Then $\iota$ has the following properties:
\begin{lemma}\label{thedecompsotion1}
\begin{itemize}
\item[(1)] $(\iota,  \mathcal{N}_2 \otimes \mathcal{M}_2)$ is a smooth representation of $\Gamma_{(\sigma, \delta)}/{(H_1 \times H_2)}$.
\item[(2)] $(\iota, \mathcal{N}_2 \otimes \mathcal{M}_2)$ is projectively isomorphic to $(\Res_{\frac{\Gamma_{(\sigma, \delta)}}{H_1\times H_2}}^{\frac{I_{G_1}(\sigma)}{H_1} \times \frac{I_{G_2}(\delta)}{H_2}} \mathfrak{n}_2 \otimes \mathfrak{m}_2, \mathcal{N}_2 \otimes \mathcal{M}_2)$.
    \end{itemize}
\end{lemma}
\begin{proof}
Note that any non-trivial element in $\Hom_{H_1 \times H_2}(W, \mathcal{N}_1 \otimes \mathcal{M}_1)$ has the same kernel, just as $\Ker F$, so that $\Gamma \cap \big(J_{G_1}(\sigma) \times J_{G_2}(\delta)\big)$ fixes $\Ker F$. Let $0 \neq  \overline{v} \in W/ {\Ker F}$, and $F(\overline{v})=\overline{u} \in \mathcal{N}_1 \otimes \mathcal{M}_1$. Let $U_{\overline{v}}$ and $U_{\overline{u}}$ denote their stabilizers in $\Gamma \cap \big( J_{G_1}(\sigma) \times J_{G_2}(\delta)\big)$. Clearly, $\widetilde{\sigma} \otimes \widetilde{\delta} \big( (x_1,x_2)\big) F\big(\rho(x_1^{-1}, x_2^{-1})-\big) =F(-)$, for $(x_1,x_2) \in U_{\overline{v}} \cap U_{\overline{u}}$. So the first statement is proved. The second statement arises from the explicit action of  $\frac{\Gamma_{(\sigma, \delta)}}{(H_1 \times H_2)}$ described above.
\end{proof}
\begin{lemma}\label{thedimension1pro}
Notations being as above, we then have:
\begin{itemize}
\item[(1)] $m_{\Gamma_{(\sigma, \delta)}}(\rho, \widetilde{\sigma} \otimes \widetilde{\delta})=1$.
\item[(2)] $m_{G_1 \times G_2}(\pi, \pi_1 \otimes \pi_2)=1$.
\end{itemize}
\end{lemma}
\begin{proof}
Let $(\widetilde{\sigma}, \widetilde{\mathbb{U}})$ (resp. $(\widetilde{\delta}, \widetilde{\mathbb{W}})$) be the smooth irreducible representation of $I_{G_1}(\sigma)$ (resp. $I_{G_2}(\delta)$) as defined in Lmm.\ref{finitelengths} relative to $\sigma$(resp. $\delta$), so that $\pi_1 \simeq \cInd_{I_{G_1}(\sigma)}^{G_1} \widetilde{\sigma}$ and $\pi_2 \simeq \cInd_{I_{G_2}(\delta)}^{G_2} \widetilde{\delta}$. By the equality (\ref{frobeniusrecomplex}), we have
$\Hom_{G_1 \times G_2}\big(\cInd_{\Gamma}^{G_1 \times G_2} \rho, \pi_1 \otimes \pi_2\big) \simeq \Hom_{ I_{G_1}(\sigma)\times I_{G_2}(\delta)} \big(  \pi_{(\sigma, \delta)}, \pi_1 \otimes \widetilde{\delta}\big)$. Any non-zero element $f$ in the last Hom-vector space, a fortiori,  belongs to $ \Hom_{H_1\times I_{G_2}(\delta)} \big(  \pi_{(\sigma, \delta)}, \pi_1 \otimes \widetilde{\delta}\big) \simeq \Hom_{H_1 \times H_2} \big( \rho, \pi_1 \otimes \widetilde{\delta}\big)$. By the property of graph, it factors through $\rho \longrightarrow \widetilde{\sigma}\otimes \widetilde{\delta}$. Then $ 1\leq m_{I_{G_1}(\sigma) \times I_{G_2}(\delta)} \big(  \pi_{(\sigma, \delta)}, \pi_1 \otimes \widetilde{\delta}\big)=m_{I_{G_1}(\sigma) \times I_{G_2}(\delta)} \big(  \pi_{(\sigma, \delta)}, \widetilde{\sigma} \otimes \widetilde{\delta}\big) =m_{\Gamma_{(\sigma, \delta)}} \big( \rho, \widetilde{\sigma}\otimes \widetilde{\delta}\big)$, which is smaller than $1$ by Lmm.\ref{thedecompsotion1}(2) and Lmms.\ref{froben}, \ref{twoprojectiverepreses}. Hence both results hold.
\end{proof}
\begin{corollary}\label{finitelenghtmodule}
Keep the notations. There exist  a decreasing  complete chaining of $\Gamma_{(\sigma, \delta)}$-modules: $\widetilde{\mathbb{U}} \otimes \widetilde{\mathbb{W}}  =\mathbb{V}_m \supseteq \cdots \supseteq \mathbb{V}_1 =0$, and a nonzero $\Gamma_{(\sigma, \delta)}$-morphism $f: W \longrightarrow \mathbb{V}_{i+1}$, for some $i$, such that $\Im(f) \cap \mathbb{V}_i =0$, and $\Im(f) \simeq \mathbb{U}\otimes \mathbb{W}$ as $H_1\times H_2$-modules.
\end{corollary}
\begin{proof}
Let $f$ be a non-zero $\Gamma_{(\sigma, \delta)}$-morphism from $W$  to $\widetilde{\mathbb{U}}\otimes \widetilde{\mathbb{W}}$. It is clear that $\Im(f)  \simeq n \mathbb{U} \otimes \mathbb{W}$ as $H_1 \times H_2$-modules. Making use of  $m_{H_1 \times H_2}(W, \mathbb{U} \otimes \mathbb{W}) =1$ shows $n=1$. Hence $\Im(f)$ is an irreducible $\Gamma_{(\sigma, \delta)}$-module. By Lmm.\ref{finitelengths}, the restriction of $\widetilde{\sigma} \otimes \widetilde{\delta}$ to $\Gamma_{(\sigma, \delta)}$ is a smooth representation of finite length, afforded a decreasing chaining of $\Gamma_{(\sigma, \delta)}$-modules, say
$\widetilde{\mathbb{U}} \otimes \widetilde{\mathbb{W}}  =\mathbb{V}_m \supseteq \cdots \supseteq \mathbb{V}_1 =0.$
So there is a nonzero $\Gamma_{(\sigma, \delta)}$-homomorphism $f: W \longrightarrow \mathbb{V}_{i+1}$, for certain $i\in \left\{ 0, \cdots, m-1\right\}$ such that $\Im(f) \cap \mathbb{V}_{i}=0$.
\end{proof}
\subsubsection{Proof of Theorem \ref{themainlemma1}(1)}
The multiplicity-free property has been verified in Lmm.\ref{thedimension1pro} (2).
  We  assume $\pi_1 \otimes \pi_2$, $\pi_1 \otimes\pi_2' \in\mathcal{R}_{G_1 \times G_2}(\pi)$, and will prove that $\pi_2 \simeq \pi_2'$. Keep the above notations, and use  the analogous notations relative to $\pi_2'$ by adding the symbol $'$. Therefore it is sufficient  to show that $\widetilde{\delta} \simeq \widetilde{\delta'}$. To simply the discussion, we identify  $(\delta, \mathbb{W})$ and $(\delta', \mathbb{W}')$.  Since $\# I_{G_2}(\delta)/J_{G_2}(\delta), \#I_{G_2}(\delta)/ J_{G_2}(\delta')$ both are finite, the cardinality of $I_{G_2}(\delta)/J_{G_2}(\delta) \cap J_{G_2}(\delta')$ is also finite. Hence we can use $J_{G_2}(\delta) \cap J_{G_2}(\delta') $ instead of $J_{G_2}(\delta)$, $J_{G_2}(\delta')$ in both cases. Combing Lmm.\ref{thedecompsotion1}(2) with Lmm.\ref{thedimension1pro} (1) shows that the two projective representations $(\mathfrak{n}_2\circ \gamma^{-1}, \mathcal{N}_2)$  and $(\mathfrak{m}_2, \mathcal{M}_2)$ of $I_{G_2}(\delta)/{H_2}$ are projective  isomorphic, and then $(\mathfrak{m}_2, \mathcal{M}_2)\simeq (\mathfrak{m}_2', \mathcal{M}_2')$ as  projective representations of $I_{G_2}(\delta)/{H_2}$. Notice that in Section \ref{hypothesisBresult}, the definition of the projective representation $(\mathfrak{m}_1, \mathcal{M}_1)$ of $I_{G_2}(\delta)/{H_2}$ is  only dependent on the choice of the classes $\left\{ \mathcal{E}_g \mid g\in \Delta_2\right\}$. So we can  identify the two projective smooth representations $(\mathfrak{m}_1, \mathcal{M}_1)$ and $(\mathfrak{m}_1', \mathcal{M}_1')$ of $I_{G_2}(\delta)/{H_2}$, and the actions of $\frac{\Gamma_{(\sigma, \delta)}}{H_1 \times H_2}$ on $ \Hom_{H_1 \times H_2}(\rho, \mathfrak{n}_1 \otimes \mathfrak{m}_1)$,  $\Hom_{H_1 \times H_2}(\rho, \mathfrak{n}_1 \otimes \mathfrak{m}'_1)$. Therefore $(\mathfrak{m}_2, \mathcal{M}_2)$ is linearly isomorphic to $(\mathfrak{m}_2', \mathcal{M}_2')$ by Lmms. \ref{thedecompsotion1}, \ref{thedimension1pro}. Consequently $\widetilde{\delta} \simeq \widetilde{\delta'}$, and $\pi_2 \simeq \pi_2'$. Recall that  $[G_2: I_{G_2}(\delta)]$  has finite index, so  by Lmm.\ref{therestriction23}, $\pi_{\pi_1}$ is a finitely generated representation of $G_1 \times G_2$.

In view of the proof, we obtain an  analogue  result  of  Roberts  Brooks' Lmm.4.2 in \cite{Rob1}.
\begin{corollary}\label{mutiplicity}
In  Theorem \ref{themainlemma1}, if   $\pi_1 \in \Irr(G_1)$, $\pi_2 \in \Irr(G_2)$ with the decompositions
$$ \pi_1|_{H_1}\simeq \oplus_{\sigma_i \in \mathcal{R}_{H_1}(\pi_1)} m_1 \sigma_i, \quad \textrm{ and } \quad   \pi_2|_{H_2}\simeq \oplus_{\delta_i \in \mathcal{R}_{H_2}(\pi_2)} m_2 \delta_i$$
such that $\pi_1\otimes \pi_2 \in \mathcal{R}_{G_1\times G_2}(\pi),$ then
\begin{itemize}
\item[(1)] there  exists a bijective map   $\theta_{\rho}: \mathcal{R}_{H_1}(\pi_1) \longrightarrow \mathcal{R}_{H_2}(\pi_2); \sigma_{\alpha} \longmapsto \delta_{\alpha}$ such that  $\sigma_{\alpha} \otimes \delta_{\alpha} \in \mathcal{R}_{H_1 \times H_2}(\rho)$ and
$\sigma_{\alpha} \otimes \delta_{\beta} \notin \mathcal{R}_{H_1\times H_2}(\rho)$ for $\alpha \neq \beta$.
\item[(2)] $m_1=m_2$.
\end{itemize}
\end{corollary}
\begin{proof}
We follow the  notations in the above proof. Then  the second statement follows from the fact that the two projective representations $(\mathfrak{n}_2\circ \gamma^{-1}, \mathcal{N}_2)$  and $(\mathfrak{m}_2, \mathcal{M}_2)$ of $I_{G_2}(\delta)/{H_2}$ are isomorphic.
\end{proof}

\subsection{The proof of the   part (2).}\label{stronglygraphreIV2}
Assume that $\mathcal{L}_{G_i}(\Ind_{H_i}^{G_i}\sigma_i) \neq \emptyset$,   for any $\sigma_i \in \Irr(H_i)$ as $i=1,2$.
Suppose now that   $\sigma\otimes \delta\in \mathcal{R}_{H_1 \times H_2}(\rho)$, for $(\sigma, U)\in \Irr(H_1)$, $(\delta, W)\in \Irr(H_2)$. Then there exist irreducible representations $(\pi_1, V_1)$ of $G_1$, $(\pi_2,V_2)$ of $G_2$, such that $\sigma \prec \Res_{H_1}^{G_1} \pi_1$, $\delta\prec \Res_{H_2}^{G_2} \pi_2$. Let $I'_{G_2}(\delta)$ (resp. $I'_{G_1}(\sigma)$ ) be the inverse image of $\gamma(\tfrac{I_{G_1}(\sigma)}{H_1})$(resp. $\gamma^{-1}(\tfrac{I_{G_2}(\delta)}{H_2})$) in $G_2$(resp. $G_1$). Let us denote $\Gamma_{(\sigma, \delta)}'= \Gamma \cap (I_{G_1}(\sigma) \times I'_{G_2}(\delta))$, and $\pi_{(\sigma, \delta)}= \cInd_{\Gamma'_{(\sigma, \delta)}}^{I_{G_1}(\sigma) \times I'_{G_2}(\delta)} \rho$.
\subsubsection{Case I. $\tfrac{G_1}{H_1}$ is a finite group. } We first seek out $\pi_1, \pi_2$  such that $\pi_1 \otimes \pi_2 \in \mathcal{R}_{G_1 \times G_2}(\pi)$. The similar procedure as Step 1 in Section \ref{part52} yields,  $\Hom_{H_1 \times H_2}\big( \rho, \widetilde{\sigma} \otimes\pi_2\big) \simeq \Hom_{H_1\times G_2}\big( \pi_{\pi_2}, \widetilde{\sigma}\otimes \pi_2\big)$, which has   finite dimension(Prop.\ref{typeimpliquequotientadmissible}). Hence $\Hom_{H_1 \times H_2}\big( \rho, \widetilde{\sigma} \otimes\pi_2\big)$ is  a smooth  $\Gamma_{(\sigma, \delta)}'/{(H_1 \times H_2)}$-module; it can be   decomposed as $\mathcal{V}_1^{\ast} \oplus \cdots\oplus \mathcal{V}_k^{\ast}$, for  some irreducible representations $(\widetilde{\omega_i}^{\ast}, \mathcal{V}_i^{\ast})$ of  $\Gamma'_{(\sigma,\delta)}/{(H_1 \times H_2)}$. Then the contragredient representation $(\widetilde{\omega_i},\mathcal{V}_i)$  of $(\widetilde{\omega_i}^{\ast}, \mathcal{V}_i^{\ast})$ can be also viewed as an irreducible representation of  $I_{G_1}(\sigma)/H_1$ or $I'_{G_2}(\delta)/H_2$.
 \begin{lemma}\label{sigmadelta}
 $\Hom_{\Gamma_{(\sigma, \delta)}'}\big(\rho, \widetilde{\omega_i}\otimes\widetilde{\sigma}\otimes \pi_2\big)\neq 0$.
 \end{lemma}
 \begin{proof}
 Let $\left\{F_1^{\ast}, \cdots, F_k^{\ast}\right\}$ be a basis of $\mathcal{V}_i^{\ast}$. Let $F_t$ be the dual base of $F_t^{\ast}$ in  $\mathcal{V}_i$. Then the mapping $\mathbbm{v}_i=\sum_{j=1}^k F_t\otimes F_t^{\ast}\in\Hom_{H_1 \times H_2}\big(\rho, \widetilde{\omega_i}\otimes\widetilde{\sigma}\otimes \pi_2\big) $, sending $v\in V$ to $ \sum_{j=1}^k F_t\otimes F_t^{\ast}(v)$,  is $\Gamma'_{(\sigma,\delta)}/{(H_1 \times H_2)}$-invariant.
 \end{proof}

\begin{lemma}\label{ssss}
\begin{itemize}
\item[(1)] $\Ind_{I_{G_1}(\sigma)}^{G_1} \widetilde{\omega_i} \otimes \widetilde{\sigma} $ is a semi-simple representation  of finite length.
\item[(2)] $\cInd_{\Gamma'_{(\sigma, \delta)}}^{I_{G_1}(\sigma)\times G_2} \rho \simeq \Res_{I_{G_1}(\sigma) \times G_2}^{G_1 \times G_2} \pi$.
\item[(3)] There exists $\pi_1 \in \mathcal{R}_{G_1}\big(\Ind_{I_{G_1}(\sigma)}^{G_1} \widetilde{\omega_i}\otimes\widetilde{\sigma}\big) $ such that $\pi_1\otimes \pi_2 \in \mathcal{R}_{G_1\times G_2}(\pi)$, and $\sigma\in \mathcal{R}_{H_1}(\pi_1)$.
\end{itemize}
\end{lemma}
\begin{proof}
 1)Notice that $I_{G_1}(\sigma)$ is an open subgroup of $G_1$ of finite index, and $\widetilde{\omega_i} \otimes \widetilde{\sigma}\hookrightarrow \cInd_{H_1}^{I_{G_1}(\sigma_1)}( \widetilde{\omega_i} \otimes \widetilde{\sigma}) $. Hence $\widetilde{\omega_i} \otimes \widetilde{\sigma}$ is semi-simple,  so the first argument  holds  by \cite[p. 21, Lmm.]{BushH}.\\
 2) $\Gamma\backslash (G_1 \times G_2)$, $\Gamma'_{(\sigma, \delta)} \backslash \big(I_{G_1}(\sigma)\times G_2\big)$ both are homeomorphic with $H_2\backslash G_2$, and $( I_{G_1}(\sigma)\times G_2 ) \cap \Gamma =\Gamma'_{(\sigma, \delta)}$, so  the result follows from Prop.\ref{restricition}.\\
 3) $0 \neq \Hom_{\Gamma'_{(\sigma, \delta)}}\big( \rho, (\widetilde{\omega_i}\otimes \widetilde{\sigma})\otimes \pi_2\big)\simeq \Hom_{I_{G_1}(\sigma) \times I'_{G_2}(\delta)} \big( \cInd_{\Gamma'_{(\sigma, \delta)}}^{I_{G_1}(\sigma)\times I'_{G_2}(\delta)} \rho, (\widetilde{\omega_i}\otimes \widetilde{\sigma})\otimes \pi_2\big)\simeq \Hom_{I_{G_1}(\sigma) \times G_2} \big( \cInd_{\Gamma'_{(\sigma, \delta)}}^{I_{G_1}(\sigma)\times G_2} \rho, (\widetilde{\omega_i}\otimes \widetilde{\sigma})\otimes \pi_2\big) \simeq \Hom_{I_{G_1}(\sigma) \times G_2} \big( \Res_{I_{G_1}(\sigma) \times G_2}^{G_1 \times G_2} \pi, (\widetilde{\omega_i}\otimes \widetilde{\sigma})\otimes \pi_2\big) \simeq \Hom_{G_1\times G_2} \big( \pi, \Ind_{I_{G_1}(\sigma)}^{G_1} (\widetilde{\omega_i} \otimes \widetilde{\sigma})\otimes \pi_2\big)$. By the property of graph of $\pi$, the first statement  is clear. Moreover $\Hom_{I_{G_1}(\sigma)}(\widetilde{\omega_i}\otimes\widetilde{\sigma}, \pi_1) \neq 0$, a fortiori, $\Hom_{H_1}(m\sigma, \pi_1) \neq 0$.
 \end{proof}
Let us  show $I_{G_2}(\delta) =I'_{G_2}(\delta)$ in the following:
From now on we take up one such pair $(\pi_1, \pi_2)$;  consequently  $m_{I_{G_1}(\sigma) \times I'_{G_2}(\delta)}\big(  \pi_{(\sigma, \delta)}, \widetilde{\sigma}\otimes \pi_2\big)\simeq m_{\Gamma_{(\sigma, \delta)}'}\big(\rho, \widetilde{\sigma} \otimes \pi_2\big)\simeq m_{G_1\times G_2} \big( \pi, \pi_1\otimes \pi_2\big)=1$. So   $\pi_2|_{I'_{G_2}(\delta)}$ contains only one   $\widetilde{\delta}' \in \Irr (I'_{G_2}(\delta))$ such that $\widetilde{\sigma} \otimes \widetilde{\delta}' \in \mathcal{R}_{I_{G_1}(\sigma)\times I'_{G_2}(\delta)}\big( \pi_{(\sigma, \delta)}\big)$, and $m_{I_{G_1}(\sigma)\times I'_{G_2}(\delta)}\big(\pi_{(\sigma, \delta)}, \widetilde{\sigma} \otimes \widetilde{\delta}'\big)=1$.
\begin{lemma}\label{deltauinique}
$ \delta\prec \widetilde{\delta}' |_{H_2}$.
\end{lemma}
\begin{proof}
Assume   $\delta \in\mathcal{R}_{H_2}\big( \widetilde{\delta}''\big)$, for some $ \widetilde{\delta}'' \in \mathcal{R}_{I'_{G_2}(\delta)}(\pi_2)$. Then $\Hom_{I_{G_1}(\sigma) \times I'_{G_2}(\delta)}\big( \pi_{(\sigma, \delta)}, \widetilde{\sigma}\otimes ( \widetilde{\delta}'' \otimes \widetilde{\omega})\big) \neq 0$, for certain  suitable irreducible representation $\widetilde{\omega}$ of $I'_{G_2}(\delta)/{H_2}$.   Decompose $(\widetilde{\delta}'' \otimes \widetilde{\omega})|_{I'_{G_2}(\delta)}$ into irreducible components as $\sum_{i=1}^{k} \widetilde{\delta_i}''$. By the proof of Lmm.\ref{ssss}(3), we assert  that $\widetilde{\sigma} \otimes \widetilde{\delta_j}''\in \mathcal{R}_{I_{G_1}(\sigma) \times I'_{G_2}(\delta)}(\pi_{(\sigma, \delta)})$ and  $\pi_2 \prec \Ind_{I'_{G_2}(\delta)}^{G_2} \widetilde{\delta_j}'' $, for some $j$; consequently $\widetilde{\delta_j}'' \in\mathcal{R}_{I'_{G_2}(\delta)} (\pi_2)$. Hence $\widetilde{\delta_j}'' \simeq \widetilde{\delta}'$,   and $0\neq m_{I'_{G_2}(\delta)}(\widetilde{\delta}'' \otimes \widetilde{\omega}, \widetilde{\delta}')=m_{I'_{G_2}(\delta)}(\widetilde{\delta}'' , \widetilde{\delta}'\otimes \check{\widetilde{\omega}})$. So  $\widetilde{\delta}''$ is a direct summand of $\widetilde{\delta}'\otimes \check{\widetilde{\omega}}$, and then $\delta \in \mathcal{R}_{H_2}(\widetilde{\delta}'\otimes \check{\widetilde{\omega}})=\mathcal{R}_{H_2}(\widetilde{\delta}')$.
\end{proof}
\begin{remark}\label{eq1}
If $\Hom_{I_{G_1}(\sigma) \times I'_{G_2}(\delta)} \big( \pi_{(\sigma,\delta)}, \widetilde{\sigma} \otimes  \widetilde{\delta}''\big) \neq 0$, for some $\widetilde{\delta}'' \in \Irr(I'_{G_2}(\delta))$, then $\widetilde{\delta}'' \simeq \widetilde{\delta}'$.
\end{remark}
\begin{proof}
Assume $  \widetilde{\delta}''\prec \pi_2' |_{ I'_{G_2}(\delta)}$, for some  $\pi_2'\in \Irr(G_2)$. Then $\Hom_{I_{G_1}(\sigma) \times I'_{G_2}(\delta)} \big( \pi_{(\sigma,\delta)}, \widetilde{\sigma} \otimes  \pi_2'\big) \simeq \Hom_{G_1 \times G_2} \big( \pi, \pi_1\otimes \pi_2'\big)$. By the property of graph of $\pi$, we have $\pi_2'\simeq \pi_2$. So we can  assume $\widetilde{\delta}''  \prec  \pi_2 |_{ I'_{G_2}(\delta)}$.  By  $m_{I_{G_1}(\sigma) \times I'_{G_2}(\delta)}\big(  \pi_{(\sigma, \delta)}, \widetilde{\sigma}\otimes \pi_2\big)=1$, we obtain $\widetilde{\delta}'' \simeq \widetilde{\delta}'$.
\end{proof}
Note that $\big(\widetilde{\sigma} \otimes \widetilde{\delta}'\big)|_{\Gamma_{(\sigma, \delta)}'}$ is semi-simple. Assume $ \mathcal{R}_{\Gamma'_{(\sigma, \delta)}}(\rho) \cap \mathcal{R}_{\Gamma'_{(\sigma, \delta)}}(\widetilde{\sigma} \otimes \widetilde{\delta}')=\left\{\lambdaup\right\} $. Set     $\gamma_{(\sigma, \delta)}=\cInd_{\Gamma'_{(\sigma, \delta)}}^{I_{G_1}(\sigma) \times I'_{G_2}(\delta)} \lambdaup$.
\begin{lemma}\label{thetamapp1}
$\gamma_{(\sigma, \delta)}$
is a theta representation with respect to $\mathcal{R}_{I_{G_1}(\sigma)}(\pi_1)$ and $\mathcal{R}_{I'_{G_2}(\delta)}(\pi_2)$.
\end{lemma}
\begin{proof}
 A non-zero element $f\in \Hom_{\Gamma'_{(\sigma, \delta)}}\big( \rho,  \lambdaup\big)$ by composing with $\lambda \hookrightarrow \gamma_{(\sigma, \delta)}$,  will induce a surjective morphism $\pi_{(\sigma, \delta)} \longrightarrow \gamma_{(\sigma, \delta)}$, and then $\mathcal{R}_{I_{G_1}(\sigma) \times I'_{G_2}(\delta)}\big(\gamma_{(\sigma, \delta)}\big) \subseteq \mathcal{R}_{I_{G_1}(\sigma) \times I'_{G_2}(\delta)}\big(\pi_{(\sigma, \delta)}\big)$. If $\widetilde{\phi } \otimes  \widetilde{ \varphi} \in \mathcal{R}_{I_{G_1}(\sigma) \times I'_{G_2}(\delta)}\big(\gamma_{(\sigma, \delta)}\big) \cap  \mathcal{R}_{I_{G_1}(\sigma) \times I'_{G_2}(\delta)}\big(\pi_1\otimes \pi_2\big)$, then $\sigma\prec \widetilde{\phi } |_{H_1} $ and $\widetilde{\phi }\prec\cInd_{H_1}^{I_{G_1}(\sigma) } \sigma $. By Clifford theory, the irreducible components  of $ \cInd_{H_1}^{I_{G_1}(\sigma) } \sigma$  correspond bijectively to the irreducible representations of $G_1$ extending $\sigma$(\emph{cf}. \cite[p.82, Theorem 6.11]{I}). Hence $\mathcal{R}_{I_{G_1}(\sigma)}\big(\cInd_{H_1}^{I_{G_1}(\sigma) } \sigma\big) \cap \mathcal{R}_{I_{G_1}(\sigma)}(\pi_1) = \big\{\widetilde{\sigma}\big\}$, and $ \widetilde{\phi}\simeq\widetilde{\sigma} $. Since $m_{I_{G_1}(\sigma) \times I'_{G_2}(\delta)}\big(  \pi_{(\sigma, \delta)}, \widetilde{\sigma}\otimes \pi_2\big)=1$, we obtain $\widetilde{\varphi} \simeq \widetilde{\delta}'$. Of course, $m_{I_{G_1}(\sigma) \times I'_{G_2}(\delta)}\big( \gamma_{(\sigma, \delta)}, \widetilde{\sigma}\otimes \widetilde{\delta}'\big) = 1$.
\end{proof}

\begin{remark}\label{eq2}
If $\Hom_{I_{G_1}(\sigma) \times I'_{G_2}(\delta)} \big( \gamma_{(\sigma,\delta)}, \widetilde{\sigma}' \otimes  \widetilde{\delta}'\big) \neq 0$, for some $\widetilde{\sigma}' \in \Irr(I_{G_1}(\sigma))$ such that $ \sigma\prec \widetilde{\sigma}'|_{H_1}$, then $\widetilde{\sigma}' \simeq \widetilde{\sigma}$.
\end{remark}
\begin{proof}
 $0\neq \Hom_{I_{G_1}(\sigma) \times I'_{G_2}(\delta)} \big( \gamma_{(\sigma,\delta)}, \widetilde{\sigma}' \otimes  \pi_2\big) \hookrightarrow \Hom_{G_1 \times G_2} \big( \pi,\Ind_{ I_{G_1}(\sigma)}^{G_1}\widetilde{\sigma}' \otimes  \pi_2\big) $. Note that $\widetilde{\sigma}' \prec\cInd_{H_1}^{I_{G_1}(\sigma)} \sigma$. By Clifford theory, $\Ind_{ I_{G_1}(\sigma)}^{G_1}\widetilde{\sigma}'$ is an irreducible representation of $G_1$. Hence  $\Ind_{ I_{G_1}(\sigma)}^{G_1}\widetilde{\sigma}' \simeq \pi_1$, and $\widetilde{\sigma}' \simeq \widetilde{\sigma}$.
\end{proof}
\begin{remark}\label{eq3}
The results of Remarks \ref{eq1}, \ref{eq2} hold for $\pi_{(\sigma, \delta)}$, and $\gamma_{(\sigma, \delta)}$.
\end{remark}
 Suppose now $\widetilde{\sigma}|_{H_1}=n\sigma$, $ m_{H_2}(\widetilde{\delta}', \delta)=  m_1 \neq 0$, $m_{H_1 \times H_2}\big(\lambda, \sigma \otimes \delta\big)=k$, and $t=m_{H_1 \times H_2} \big(\lambdaup, \sigma \otimes \widetilde{\delta}'\big) $. Then
 \begin{align}
& \Hom_{H_1 \times H_2} \big(\lambda , \sigma \otimes \widetilde{\delta}'\big) \simeq \Hom_{I_{G_1}(\sigma)  \times I'_{G_2}(\delta)} \big( \gammaup_{(\sigma, \delta)}, \cInd_{H_1}^{I_{G_1}(\sigma) } \sigma \otimes \widetilde{\delta}' \big) \label{equations11}  \\
& \Hom_{H_1 \times H_2} \big(\lambdaup, \widetilde{\sigma }\otimes \delta\big) \simeq \Hom_{I_{G_1}(\sigma)  \times I'_{G_2}(\delta)} \big( \gammaup_{(\sigma, \delta)}, \widetilde{\sigma} \otimes \cInd_{H_2}^{I'_{G_2}(\delta) } \delta\big)  \label{equations12}
\end{align}
By equation (\ref{equations11}), we get $km_1 \leq t=n$, and  by equation (\ref{equations12}), $kn=m_1\neq 0 $. Therefore $k=1$, $m_1=n=t$. Consequently, $\widetilde{\delta}'|_{H_2} \simeq m_1 \delta$\big(because now $\widetilde{\delta}'|_{H_2}\simeq \sum_{i=1}^{t} \delta_i$, for $\delta_i\in \Irr(H_2)$, and then $\sigma \otimes \delta_i \in \mathcal{R}_{H_1 \times H_2}(\lambda)$\big), and $I'_{G_2}(\delta)\subseteq I_{G_2}(\delta)$. By symmetry, $I'_{G_1}(\sigma)\subseteq I_{G_1}(\sigma)$. Hence $I'_{G_2}(\delta)=I_{G_2}(\delta)$. As a consequence, indeed  $\widetilde{\delta}'$ is the $\delta$-isotypic component of $\pi_2|_{H_2}$.

Replacing  $\lambdaup$ in equations (\ref{equations11}), (\ref{equations12})  by $\rho$ itself, we also obtain the same   numerical  equalities, and the similar result that $m_{H_1 \times H_2}\big(\rho, \sigma \otimes \delta\big)=1$.  Moreover,
\begin{equation}\label{eq333}
m_{H_1\times H_2}(\rho, \sigma \otimes \pi_2)=m_{I_{G_1}(\sigma)  \times I_{G_2}(\delta)}(\pi_{(\sigma, \delta)}, \Ind_{H_1}^{I_{G_1}(\sigma)} \sigma \otimes \pi_2)=m_{I_{G_1}(\sigma)  \times I_{G_2}(\delta)}(\pi_{(\sigma, \delta)}, n \widetilde{\sigma} \otimes \pi_2)=n
\end{equation}

If $\sigma \otimes \delta_1 \in \mathcal{R}_{H_1 \times H_2}(\rho)$, then there exists $\pi_2'\in \Irr(G_2)$ such that $ \delta_1\prec\pi_2'|_{H_2} $, and $\pi_1 \otimes \pi_2' \in \mathcal{R}_{G_1 \times G_2}(\pi)$. Hence $\pi_2'\simeq \pi_2$, and we can assume $\delta_1\prec \pi_2|_{H_2}$. By (\ref{eq333}), we see $\delta_1\simeq \delta$. This  completes the proof in the first case.
\subsubsection{Case II. $\tfrac{G_1}{H_1}$ is only a compact group. }
Let $J_{G_1}(\sigma)$, $J_{G_2}(\delta)$ be the subgroups of $I_{G_1}(\sigma)$, $I_{G_2}(\delta)$ respectively as defined in Lmm.\ref{theopensubgroup}, and write their  images in $I_{G_2}(\delta)/H_2$, $I_{G_1}(\delta)/H_1$ by $J'_{G_2}(\delta)/H_2$, $J'_{G_1}(\sigma)/H_1$ respectively.  Let $J_{G_1}^0(\sigma)=J_{G_1}(\sigma)\cap J'_{G_1}(\sigma)$, and $J_{G_2}^0(\delta)=J_{G_2}(\delta)\cap J'_{G_2}(\delta)$. Then:
\begin{lemma}\label{simila4}
\begin{itemize}
\item[ (1)] $J_{G_1}^0(\sigma)$, $J_{G_2}^0(\delta)$ are open normal subgroups of $G_1$,$G_2$ respectively, and $\gamma$ sends $J_{G_1}^0(\sigma)/{H_1}$ onto $J_{G_2}^0(\delta)/{H_2}$.
    \item[ (2)] $\gamma$ induces a bijective group isomorphism $\overline{\gamma}: G_1/J_{G_1}^0(\sigma) \longrightarrow G_2/J_{G_2}^0(\delta)$, with the graph $\widetilde{\Gamma^0_{(\sigma,\delta)}}/{(J_{G_1}^0(\sigma) \times J_{G_2}^0(\delta))}$, where $\widetilde{\Gamma^0_{(\sigma,\delta)}}=\Gamma \cdot (J_{G_1}^0(\sigma) \times J_{G_2}^0(\delta))$.
     \item[ (3)]  $G_1/J_{G_1}^0(\sigma)  $ is a finite group.
    \end{itemize}

\end{lemma}\label{simila5}
We now let $\Gamma^0_{(\sigma, \delta)}=\Gamma\cap \big(J_{G_1}^0(\sigma) \times J_{G_2}^0(\delta)\big)$, and $\widetilde{\rho_{(\sigma, \delta)}}=\cInd_{\Gamma}^{\widetilde{\Gamma^0_{(\sigma,\delta)}}} \rho$. Then $\pi = \cInd_{\widetilde{\Gamma^0_{(\sigma,\delta)}}}^{G_1\times G_2} \widetilde{\rho_{(\sigma, \delta)}}$.
\begin{lemma}\label{simila6}
$\pi^0=\cInd_{\Gamma^0_{(\sigma, \delta)}}^{J_{G_1}^0(\sigma) \times J_{G_2}^0(\delta)} \rho$  is a  theta representation of $J_{G_1}^0(\sigma) \times J_{G_2}^0(\delta)$.
\end{lemma}
\begin{proof}
This is a consequence of Step 1 and Lmm.\ref{simila4}.
\end{proof}
Let us write $\pi^0_{\sigma}\simeq \sigma \otimes \Theta^0_{\sigma}$ as  $J_{G_1}^0(\sigma) \times J_{G_2}^0(\delta)$-modules. Then $\Theta^0_{\sigma}$ is a finitely generated   $ J_{G_2}^0(\delta)$-module. If we write $\rho_{\sigma}\simeq \sigma \otimes \Theta_{\sigma}$ as $H_1\times H_2$-modules, then by Lmm.\ref{therestriction23}(2)(c), $\Theta^0_{\sigma}|_{H_2} \simeq \Theta_{\sigma}$.
 \begin{remark}
 By Prop.\ref{finitegenerated}(2), the restriction of  $\Theta^0_{\sigma}$  to $H_2$ is also  finitely generated.
 \end{remark}

  $\Hom_{H_1 \times H_2}(\rho, \sigma \otimes \delta)$($\simeq \Hom_{H_2}(\Theta_{\sigma}, \delta)$) is a smooth $\Gamma^0_{(\sigma, \delta)}/{(H_1 \times H_2)}$-module of finite dimension  via the canonical action, and it can be   decomposed  as $\mathcal {U}_1^{\ast}\oplus \cdots \oplus \mathcal {U}_k^{\ast}$ for some irreducible representations $(\check{\widetilde{\varpiup}}_i,  \mathcal {U}^{\ast}_i)\in \Irr\big(\Gamma^0_{(\sigma, \delta)}/{(H_1 \times H_2)}\big)$. The result of Lmm.\ref{sigmadelta} also works for this case. So $0 \neq \Hom_{\Gamma^0_{(\sigma, \delta)}}\big( \rho, \widetilde{\varpiup}_i \otimes \sigma \otimes\delta) \simeq \Hom_{J_{G_1}^0(\sigma) \times J_{G_2}^0(\delta)} \big(\pi^0, \widetilde{\varpiup}_i \otimes \sigma \otimes \delta\big)$. Hence there exists a nonzero $J_{G_2}^0(\delta)$-morphism $f: \Theta^0_{\sigma} \longrightarrow \delta\otimes \widetilde{\varpiup}_i$.
  \begin{lemma}\label{subirr}
  $\delta\otimes \widetilde{\varpiup}_i$ is an irreducible $J_{G_2}^0(\delta)$-module.
  \end{lemma}
  \begin{proof}
If  $\varsigma$ is  a nonzero subrepresentation of   $\delta\otimes \widetilde{\varpiup}_i$, then there exists a short exact sequence of   $J_{G_2}^0(\delta)$-modules  $0 \longrightarrow \varsigma \longrightarrow  \delta\otimes \widetilde{\varpiup}_i\longrightarrow \varsigma_0\longrightarrow 0$.  Note that  $ [(\check{\delta}\otimes \varsigma)_{H_2}]^{\ast} \simeq \Hom_{H_2}(\varsigma, \delta)\neq 0$ and $\dim (\check{\delta}\otimes \varsigma)_{H_2} \leq \dim\widetilde{\varpiup}_i$.   Since $\check{\delta}\otimes -$, $(-)_{H_2}$ both are right exact functors, there exists  an exact sequence of $J_{G_2}^0(\delta)/{H_2}$-modules:  $ (\check{\delta}\otimes \varsigma)_{H_2} \stackrel{\kappa}{\longrightarrow }(\check{\delta}\otimes \delta\otimes \widetilde{\varpiup}_i)_{H_2}\simeq  \widetilde{\varpiup}_i \longrightarrow (\check{\delta}\otimes \varsigma_0)_{H_2} \longrightarrow 0$,  $\kappa\neq 0$.   So we obtain  $(\check{\delta}\otimes \varsigma)_{H_2}  \simeq \widetilde{\varpiup}_i$ as $J_{G_2}^0(\delta)/H_2$-modules, and $ (\check{\delta}\otimes \varsigma_0)_{H_2} =0$. Therefore $\varsigma_0=0$ and $\varsigma = \delta\otimes \widetilde{\varpiup}_i  $.
    \end{proof}
  As a consequence, the image of the above $f$ is full. We now apply the above  approach to the representations $ \sigma$ of $J_{G_1}^0(\sigma)$ and  $\widetilde{\omega_i} \otimes \delta$ of $J_{G_2}^0(\delta)$, instead of the ones $\pi_1$ of $G_1$ and  $\pi_2$ of $G_2$. Then  there exist  open  normal subgroups $J_{G_1}^1(\sigma)$ of $J_{G_1}^0(\sigma)$ and $J_{G_2}^1(\delta)$ of $J_{G_2}^0(\delta)$ such that $\gamma$ sends $\frac{J_{G_1}^1(\sigma)}{H_1}$ onto $\frac{J_{G_2}^1(\delta)}{H_2}$ with the image $\frac{\Gamma^1_{(\sigma, \delta)}}{H_1 \times H_2}$, and $[\widetilde{\omega_i} \otimes \delta ]|_{J_{G_2}^1(\delta)} \simeq k \delta $. Set $\widetilde{\Gamma^1_{(\sigma, \delta)}}=[\Gamma\cap \big(J_{G_1}^0(\sigma) \times J_{G_2}^0(\delta)\big)]\cdot [J_{G_1}^1(\sigma) \times J_{G_2}^1(\delta)]$, and $\widetilde{\rho^1_{(\sigma, \delta)}}=\cInd_{\Gamma^0_{(\sigma, \delta)}}^{\widetilde{\Gamma^1_{(\sigma,\delta)}}} \rho$. Then $\pi^0 = \cInd_{\widetilde{\Gamma^1_{(\sigma,\delta)}}}^{J_{G_1}^0(\sigma) \times J_{G_2}^0(\delta)} \widetilde{\rho^1_{(\sigma, \delta)}}$, and $\widetilde{\rho^1_{(\sigma, \delta)}}|_{J_{G_1}^1(\sigma) \times J_{G_2}^1(\delta)} \simeq \cInd_{\Gamma^1_{(\sigma, \delta)}}^{J_{G_1}^1(\sigma) \times J_{G_2}^1(\delta)} \rho$. Hence:

 \begin{lemma}\label{simila6}
 $\pi^1=\cInd_{\Gamma^1_{(\sigma, \delta)}}^{J_{G_1}^1(\sigma) \times J_{G_2}^1(\delta)} \rho$  is a  theta representation of $J_{G_1}^1(\sigma) \times J_{G_2}^1(\delta)$.
 \end{lemma}
\begin{proof}
This is a consequence of Step 1 and the above discussion.
\end{proof}

Note that  $ 0\neq\Hom_{J_{G_1}^0(\sigma) \times J_{G_2}^0(\delta)}\big( \pi^0, \sigma \otimes \delta\otimes \widetilde{\omega_i} \big)\simeq \Hom_{\widetilde{\Gamma^1_{(\sigma,\delta)}}}\big(  \widetilde{\rho^1_{(\sigma, \delta)}},  \sigma \otimes \delta\otimes \widetilde{\omega_i}\big)\hookrightarrow \Hom_{J_{G_1}^1(\sigma) \times J^1_{G_2}(\delta)}\big( \pi^1, \sigma \otimes \delta\otimes \widetilde{\omega_i})$, so $\Hom_{J_{G_1}^1(\sigma) \times J^1_{G_2}(\delta)}\big( \pi^1,  \sigma  \otimes \delta\big)\neq 0 $.

As above,   $\Hom_{H_1 \times H_2}(\rho, \sigma \otimes \delta)$ is a smooth $\Gamma^1_{(\sigma, \delta)}/{(H_1 \times H_2)}$-module via the canonical action, being  decomposed  as $\mathcal {V}_1^{\ast}\oplus \cdots \oplus \mathcal {V}_l^{\ast}$ for some irreducible representations $(\check{\widetilde{\tau}}_i,  \mathcal {V}^{\ast}_i)\in \Irr(\Gamma^1_{(\sigma, \delta)}/{(H_1 \times H_2)})$.  So $0 \neq \Hom_{\Gamma^1_{(\sigma, \delta)}}\big( \rho,  \sigma \otimes\delta\otimes \widetilde{\tau}_i ) \simeq \Hom_{J_{G_1}^1(\sigma) \times J_{G_2}^1(\delta)} \big(\pi^1,  \sigma \otimes\delta\otimes \widetilde{\tau}_i \big)$. By the similar result of Lmm.\ref{subirr}, we know $\delta\otimes \widetilde{\tau}_i $ is irreducible.  By Lmm.\ref{simila6},  $\delta \otimes \widetilde{\tau}_i\simeq  \delta $ as $J_{G_2}^1(\delta)$-modules. Hence  $ 0 \neq \Hom_{ J_{G_2}^1(\delta)}\big( \delta\otimes\widetilde{\tau_i}, \delta\big) \simeq \Hom_{ H_2}\big((\delta\otimes \check{\delta})_{H_2}, \check{\widetilde{\tau_i}}\big)^{\frac{J_{G_2}^1(\delta)}{H_2}}$.  Since $\check{\widetilde{\tau_i}}$ is an irreducible representation of  $\frac{J_{G_2}^1(\delta)}{H_2}$,   we obtain $\check{\widetilde{\tau_i}}\simeq \C$ as $J_{G_2}^1(\delta)$-modules; every non-trivial element in $\mathcal{V}_i^{\ast}$ sits in $\Hom_{\Gamma^1_{(\sigma, \delta)}}\big(\rho, \sigma \otimes \delta\big)$, and it forces $l=1$. Consequently, we obtain
\begin{lemma}
$m_{H_1 \times H_2}( \rho,  \sigma \otimes \delta)=1$.
\end{lemma}
\begin{corollary}
There exist $(\pi_,V_1) \in \Irr(G_1), (\pi_2, V_2)\in \Irr(G_2)$ such that $\sigma \prec \pi_1|_{H_1}$, $\delta\prec \pi_2|_{H_2}$, and $\pi_1\otimes \pi_2 \in \mathcal{R}_{G_1 \times G_2}(\pi)$.
\end{corollary}
\begin{proof}
 The results of Lmms. \ref{sigmadelta}, \ref{ssss} also hold, if
 we see $\sigma$, $\delta$ as representations of $J_{G_1}^1(\sigma)$,$J_{G_2}^1(\delta)$ respectively. Hence the results hold.
\end{proof}
Finally let us check the property of graph.  If $\sigma \otimes \delta' \in \mathcal{R}_{H_1 \times H_2}(\rho)$, we can find  $\pi_2'\in \Irr(G_2)$, such that $\pi_1 \otimes \pi_2'\in \mathcal{R}_{G_1 \times G_2}(\pi)$, and $\delta' \prec\pi_2'|_{H_2}$.  Therefore $\pi_2' \simeq \pi_2$, and we can assume $\delta' \prec \pi_2|_{H_2}$. We define the analogous notion for $\delta'$, and  denote by $J_{G_2}^1(\delta, \delta')=J_{G_2}^1(\delta) \cap J_{G_2}^1(\delta')$, and by $J_{G_1}^1(\sigma, \sigma)$ its corresponding group in $G_1$. By the result of Step 2, the following result holds:
\begin{lemma}
$\pi^1_{(\sigma\sigma, \delta\delta')}=\cInd_{\Gamma\cap [J_{G_1}^1(\sigma, \sigma) \times J_{G_2}^1(\delta, \delta')]}^{J_{G_1}^1(\sigma, \sigma) \times J_{G_2}^1(\delta, \delta')} \rho$  is a theta representation of $J_{G_1}^1(\sigma, \sigma) \times J_{G_2}^1(\delta, \delta')$.
\end{lemma}
By the same discussion as above, we can see that $\sigma \otimes \delta$, $\sigma \otimes \delta'\in \mathcal{R}_{J_{G_1}^1(\sigma, \sigma) \times J_{G_2}^1(\delta, \delta')}\big(\pi^1_{(\sigma\sigma, \delta\delta')}\big)$. Hence $\delta \simeq \delta'$ as $J_{G_2}^1(\delta, \delta')$-modules.
\begin{corollary}
$\delta \simeq \delta'$ as $H_2$-modules.
\end{corollary}
\section{The theta representation III }\label{stronglygraphreIII}
In this section, let $(\rho, \langle, \rangle, W)$ be  a  preunitary smooth   representation of $\Gamma$ with the complete vector space $\mathcal{W}$.    Let $(\pi, V)=\big(\cInd_{\Gamma}^{G_1 \times G_2} (\delta_{\Gamma \setminus (G_1\times G_2)}^{1/2}\otimes \rho), \cInd_{\Gamma}^{G_1 \times G_2} (\delta_{\Gamma \setminus (G_1\times G_2)}^{1/2}\otimes W) \big)$. Let  $(\Pi, \mathcal{V})= (\mathfrak{Ind}_{\Gamma}^{G_1\times G} \rho, \mathfrak{Ind}_{\Gamma}^{G_1\times G} \mathcal{W})$, the unitary induced from $(\rho,\mathcal{W})$. Let $\Irr_u(H_i)$, $\Irr_u(G_i)$ denote the sets of all equivalent  irreducible preunitary  representations of $H_i$, $G_i$ respectively.
Assume (1) $H_i$, $G_i$  are   groups of \emph{type I}, (2) $\widehat{H_i}/G_i$ is \emph{countably separated}, (3) For any $\omega \in \widehat{H_i}$, the orbit $\{ \omega^g \mid g\in G_i\}$ is \emph{countable}, (4) For any $(\sigma_i, U_i) \in \Irr_u(H_i)$,  the cardinality of $ \{ \pi_i \in \Irr_u(G_i) \mid m_{H_i}(\pi_i, \sigma_i)\neq 0\}$ is \emph{countable}, (5) there exists an open  subgroup $O$ of $G$, such that   $\Ha^2(O, \C^{\times})$ only contains  elements of finite order.  Assume $W$ is a second countable vector space, and $G_i$, $H_i$ all are second-countable  groups.
  \begin{theorem}
     \begin{itemize}
   \item[(1)]
If $\Res_{H_1 \times H_2}^{\Gamma} \rho$  is a general theta representation of  $H_1 \times H_2$  with respect to $\Irr_u(H_1) \times \Irr_u(H_2)$, then so is  the representation $\cInd_{\Gamma}^{G_1 \times G_2}(\delta_{\Gamma \setminus (G_1\times G_2)}^{1/2}\otimes \rho)$ of $G_1 \times G_2$  with respect to $\Irr_u(G_1) \times \Irr_u(G_2)$.
   \item[(2)] Suppose that  $m_{H_i}(\lambda_i, \omega_i)<+\infty$,   for  $\lambda_i\in \Irr_u(G_i),  \omega_i \in \Irr_u(H_i)$,  $i=1,2$. If  $\cInd_{\Gamma}^{G_1 \times G_2} (\delta_{\Gamma \setminus (G_1\times G_2)}^{1/2}\otimes\rho)$   of $G_1 \times G_2$ is a general theta representation  with respect to $\Irr_u(G_1) \times \Irr_u(G_2)$, then so is   $\Res_{H_1 \times H_2}^{\Gamma} \rho$ of $H_1 \times H_2$  with respect to $\Irr_u(H_1) \times \Irr_u(H_2)$.
   \end{itemize}
     \end{theorem}
 Remark that $\delta_{\Gamma \setminus (G_1\times G_2)}^{1/2}|_{H_1\times H_2} \simeq  \delta_{(H_1\times H_2) \setminus (G_1\times G_2)}^{1/2}/  \delta_{ (H_1\times H_2)\setminus\Gamma }^{1/2} =1$. Since $I_{G_1}(\sigma) \times I_{G_2}(\delta)$ is an open subgroup of $G_1\times G_2$, $\delta_{\Gamma\setminus (G_1\times G_2)}|_{\Gamma \cap [I_{G_1}(\sigma) \times I_{G_2}(\delta)]} =\delta_{\Gamma \cap [I_{G_1}(\sigma) \times I_{G_2}(\delta)]\setminus  [I_{G_1}(\sigma) \times I_{G_2}(\delta)] } $. By Remark \ref{generalmod}, $\delta_{\Gamma\setminus (G_1\times G_2)}|_{H_1\times H_2} =\delta_{(H_1\times H_2)\setminus G_1\times H_2}=1$.
 \subsection{The proof of  the first  part}
\begin{lemma}\label{simepl}
If $(\pi_1, V_1) \in \Irr_u(G_1)$, and $(\pi_2, V_2) \in \Irr_u(G_2)$, such that $\pi_1 \otimes \pi_2 \in \mathcal{R}_{G_1 \times G_2}(\pi)$, then:
\begin{itemize}
\item[(1)] For $\sigma \in \mathcal{R}_{H_1}(\pi_1)$,   there exists a unique element $\delta \in \mathcal{R}_{H_2}(\pi_2)$ such that $ \sigma \otimes \delta \in \mathcal{R}_{H_1 \times H_2}(\rho)$.
\item[(2)] For  $\sigma \otimes \delta\in \mathcal{R}_{H_1 \times H_2}(\rho)$,  $\gamma$ induces a bijective map from $I_{G_1}( \sigma)/{H_1}$ to $I_{G_2}(\delta)/{H_2}$ with the graph $\Gamma_{(\sigma,\delta)} /{(H_1 \times H_2)}$, where $\Gamma_{(\sigma,\delta)}= \Gamma \cap \big( I_{G_1}(\sigma) \times I_{G_2}(\delta)\big)$.
\end{itemize}
\end{lemma}

\begin{proof}
1)  Let us write $ I'_{G_2}(\delta)/{H_2}=\gamma(I_{G_1}(\sigma)/{H_1})$, and let  $\widetilde{\sigma}$  be the $\sigma$-isotypic component  of $\pi_1|_{H_1}$. Then $\pi_1\simeq \cInd_{I_{G_1}(\sigma)}^{G_1} \widetilde{\sigma}$.   By Frobenius reciprocity, we have
\begin{equation}
\begin{aligned}
0& \neq m_{G_1 \times G_2} \big( \pi, \pi_1 \otimes \pi_2\big)\\
& =m_{G_1 \times G_2}\big(\cInd_{\Gamma}^{G_1 \times G_2} (\delta_{\Gamma \setminus (G_1\times G_2)}^{1/2}\otimes \rho) , \Ind_{I_{G_1}(\sigma)\times G_2}^{G_1 \times G_2} \widetilde{\sigma} \otimes \pi_2\big) \label{equations36}\\
& =m_{I_{G_1}(\sigma)\times I'_{G_2}(\delta)}\big(\cInd_{\Gamma\cap (I_{G_1}(\sigma)\times I'_{G_2}(\delta))}^{I_{G_1}(\sigma)\times I'_{G_2}(\delta)} (\delta_{\Gamma \setminus (G_1\times G_2)}^{1/2}\otimes  \rho),\widetilde{\sigma} \otimes \pi_2\big) \\
 &\leq m_{H_1 \times I'_{G_2}(\delta)}(\cInd_{H_1\times H_2}^{H_1 \times I'_{G_2}(\delta)}(\delta_{\Gamma \setminus (G_1\times G_2)}^{1/2}\otimes \rho), \widetilde{\sigma} \otimes  \pi_2)\\
& =m_{H_1 \times H_2}(\delta_{\Gamma \setminus (G_1\times G_2)}^{-1/2}\otimes \rho, \widetilde{\sigma} \otimes ( \check{ \pi}_2|_{H_2})^{\vee}) \end{aligned}
\end{equation}
 So  by Lmm.\ref{directsums}, we can find $\delta \in \mathcal{R}_{H_2}(\pi_2)$ such that $\sigma \otimes \delta \in \mathcal{R}_{H_1 \times H_2}(\rho)\cap \mathcal{R}_{H_1 \times H_2}(\pi_1 \otimes \pi_2)$. The uniqueness is clear right.\\
2) Assume $g_1 H_1 \in I_{G_1}(\sigma)/{H_1}$, and $\gamma(g_1H_1)=g_2H_2 \in G_2/{H_2}$. We then have $\sigma^{g_1} \otimes\delta^{g_2} \simeq \sigma \otimes \delta^{g_2} \in \mathcal{R}_{H_1 \times H_2}(\rho)$, which implies that $\delta^{g_2} \simeq \delta$, and then $g_2 \in I_{G_2}(\delta)$. The converse also holds, so $\gamma$ maps $I_{G_1}(\sigma)/{H_1}$ onto $I_{G_2}(\delta)/{H_2}$ with the graph $\Gamma\cap \big( I_{G_1}(\sigma) \times I_{G_2}(\delta)\big)/{(H_1 \times H_2)}$.
\end{proof}
We now fix  irreducible constituents $(\sigma, \mathbb{U})$ of $\Res_{H_1}^{G_1} \pi_1$ and  $(\delta, \mathbb{W})$   of  $\Res_{H_2}^{G_2} \pi_2$ such that $\sigma \otimes \delta\in \mathcal{R}_{H_1\times H_2}(\rho)$.
Let $(\mathfrak{n}_1,  \mathcal{N}_1)$, $(\mathfrak{n}_2, \mathcal{N}_2)$, resp. $(\mathfrak{m}_1, \mathcal{M}_1)$ and $(\mathfrak{m}_2, \mathcal{M}_2)$ be two preunitary projective representations related to $(\widetilde{\sigma}, \widetilde{\mathbb{U}})$, and $(\widetilde{\delta}, \widetilde{\mathbb{W}})$ respectively in Lmm.\ref{inu}(6). In the above equations (\ref{equations36}), any map $f\in \Hom_{\Gamma_{(\sigma,\delta)}}\big(\delta_{\Gamma \setminus (G_1\times G_2)}^{-1/2}\otimes  \rho,[\Res_{\Gamma_{(\sigma,\delta)}}^{I_{G_1}(\sigma)\times I_{G_2}(\delta)}\check{\widetilde{\sigma}} \otimes\check{\pi}_2]^{\vee}\big)$ needs to factor through $\widetilde{\sigma} \otimes \widetilde{\delta} \hookrightarrow \widetilde{\sigma} \otimes \pi_2\hookrightarrow [\Res_{\Gamma_{(\sigma,\delta)}}^{I_{G_1}(\sigma)\times I_{G_2}(\delta)}\check{\widetilde{\sigma}} \otimes\check{\pi}_2]^{\vee}\big)$.  Hence $ m_{\Gamma_{(\sigma, \delta)}}\big(\delta_{\Gamma \setminus (G_1\times G_2)}^{-1/2}\otimes  \rho,\widetilde{\sigma} \otimes \pi_2\big)=m_{\Gamma_{(\sigma, \delta)}}\big(\delta_{\Gamma \setminus (G_1\times G_2)}^{-1/2}\otimes  \rho,\widetilde{\sigma} \otimes \widetilde{\delta}\big)\geq 1$.

  On $\mathcal{V}=\Hom_{H_1 \times H_2}(\delta_{\Gamma \setminus (G_1\times G_2)}^{-1/2}\otimes  \rho, \widetilde{\sigma} \otimes \widetilde{\delta})$, we impose a  natural $\Gamma_{(\sigma, \delta)}/{(H_1 \times H_2)}$-action defined as follows:
$[\overline{a}\varphi](\widetilde{v}):=\widetilde{\sigma} \otimes \widetilde{\delta} (a) \varphi\big( \delta_{\Gamma \setminus (G_1\times G_2)}^{-1/2}(a^{-1})\rho(a^{-1}) \widetilde{v}\big)$, for $a\in \Gamma_{(\sigma, \delta)}$.  Recall that $m_{H_1 \times H_2}(\delta_{\Gamma \setminus (G_1\times G_2)}^{-1/2}\otimes \rho, \mathfrak{n}_1 \otimes\mathfrak{m}_1)=1$.  As projective  $\frac{\Gamma_{(\sigma, \delta)}}{H_1 \times H_2}$-modules, we have
\begin{equation}\label{isdenty}
\Hom_{H_1 \times H_2} (\delta_{\Gamma \setminus (G_1\times G_2)}^{-1/2}\otimes \rho, \mathfrak{n}_1 \otimes \mathfrak{n}_2 \otimes \mathfrak{m}_1 \otimes \mathfrak{m}_2)\simeq \Hom_{H_1 \times H_2}(\delta_{\Gamma \setminus (G_1\times G_2)}^{-1/2}\otimes \rho, \mathfrak{n}_1 \otimes \mathfrak{m}_1)  \otimes  \mathcal{N}_2 \otimes \mathcal{M}_2 .
\end{equation}
By Lmm.\ref{inu}(6), we can obtain likewise the result of  Lmm.\ref{thedimension1pro}, that is  $m_{G_1\times G_2}(\pi, \pi_1\otimes \pi_2)=1=m_{\Gamma_{(\sigma, \delta)}}(\delta_{\Gamma \setminus (G_1\times G_2)}^{-1/2}\otimes \rho, \widetilde{\sigma} \otimes \widetilde{\delta})$. Consequently , $(\mathfrak{m}_2, \mathcal{M}_2)\simeq (\mathfrak{n}_2\circ \gamma^{-1}, \mathcal{N}_2) $ as projective $I_{G_2}(\delta)/{H_2}$ -modules.
By symmetry we  now assume $\pi_1 \otimes \pi_2$, $\pi_1 \otimes\pi_2' \in\mathcal{R}_{G_1 \times G_2}(\pi)$, and will prove that $\pi_2 \simeq \pi_2'$. Keep the above notations, and use  the analogous notations relative to $\pi_2'$ by adding the symbol $'$. Therefore it is sufficient  to show that $\widetilde{\delta} \simeq \widetilde{\delta'}$. To simply the discussion, we identify  $(\delta, \mathbb{W})$ and $(\delta', \mathbb{W}')$, and obtain $(\mathfrak{m}_1, \mathcal{M}_1)\simeq (\mathfrak{m}_1', \mathcal{M}_1')$ as projective representations of $I_{G_2}(\delta)$ by Lmms. \ref{inu}(5)(6). Similarly we obtain  $(\mathfrak{m}_2, \mathcal{M}_2)\simeq (\mathfrak{n}_2\circ \gamma^{-1}, \mathcal{N}_2) \simeq (\mathfrak{m}_2', \mathcal{M}_2')$ as  projective representations of $I_{G_2}(\delta)/{H_2}$.  Hence $\widetilde{\delta}\simeq \widetilde{\delta}'$ as projective $I_{G_2}(\delta)$-modules, and $\widetilde{\delta}\simeq \widetilde{\delta}'\otimes\chi$ as  ordinary smooth $I_{G_2}(\delta)$-modules, for some character $\chi$ of $I_{G_2}(\delta)/H_2$. For the decompositions $\widetilde{\mathbb{W}}\simeq  \mathcal{M}_1\otimes \mathcal{M}_2$, $\widetilde{\mathbb{W}}'\simeq  \mathcal{M}'_1\otimes \mathcal{M}'_2$, by modifying a continuous function of  $I_{G_2}(\delta)/H_2$ on $\mathcal{M}_2$ or $\mathcal{M}'_2$, we can  identify $(\mathfrak{m}_1, \mathcal{M}_1)$ and $(\mathfrak{m}_1', \mathcal{M}_1')$.  Hence by (\ref{isdenty}), and $m_{\Gamma_{(\sigma, \delta)}}(\delta_{\Gamma \setminus (G_1\times G_2)}^{-1/2}\otimes \rho, \widetilde{\sigma} \otimes \widetilde{\delta})=1=m_{\Gamma_{(\sigma, \delta)}}(\delta_{\Gamma \setminus (G_1\times G_2)}^{-1/2}\otimes \rho, \widetilde{\sigma} \otimes \widetilde{\delta}')$, we obtain that $ (\mathfrak{m}_2, \mathcal{M}_2)$ is linearly isomorphic to $(\mathfrak{m}'_2, \mathcal{M}'_2)$.   Let $F: \mathcal{M}_1 \otimes \mathcal{M}_2 \longrightarrow \mathcal{M}_1 \otimes \mathcal{M}'_2$ be an $I_{G_2}(\delta)$-isomorphism between $\widetilde{\delta}$ and $\widetilde{\delta}'\otimes \chi$.  By considering $F$ as an $H_2$-morphism and Schur's Lemma, we can write $F=1 \otimes  \varphi$ with $\varphi\in \Hom_{I_{G_2(\delta)}}(\mathcal{M}_2,  \mathcal{M}_2' )$.    Hence $(\mathfrak{m}_2', \mathcal{M}_2')$ is linearly isomorphic with $(\mathfrak{m}_2'\otimes \chi, \mathcal{M}_2')$, which implies that $\widetilde{\delta}'\simeq \widetilde{\delta}'\otimes \chi\simeq \widetilde{\delta}$, $\pi_2\simeq \pi_2'$.

\subsection{The  proof of the second  part}

Assume $\sigma\otimes  \delta \in \mathcal{R}_{H_1 \times H_2}(\rho)$.
Let $(\pi_1, V_1)$, $(\pi_2, V_2)$ be irreducible preunitary representations of $G_1$, $G_2$ respectively such that $ \sigma\prec \pi_1|_{H_1}$, $\delta\prec \pi_2|_{H_2}$.  Let $\widetilde{\sigma}$ denote  the $\sigma$-isotypic component of $\sigma$ in $\pi_1|_{H_1}$, $\widetilde{\delta}$ the $\delta$-isotypic component of $\delta$ in $\pi_2|_{H_2}$,
Let $\frac{I'_{G_1}(\sigma) }{H_1}=\gamma^{-1}\big(\frac{I_{G_2}(\delta)}{H_2}\big)$,  $\frac{I'_{G_2}(\delta)}{H_2}=\gamma\big(\frac{I_{G_1}(\sigma)}{H_1}\big)$, and
 denote   $\Gamma'_{(\sigma, \delta)}=\Gamma \cap [I_{G_1}(\sigma) \times I'_{G_2}(\delta)]$,  $\Gamma_{(\sigma, \delta)}=\Gamma \cap [I_{G_1}(\sigma) \times I_{G_2}(\delta)]=\Gamma \cap [(I_{G_1}(\sigma) \cap I'_{G_1}(\sigma))\times (I_{G_2}(\delta)\cap I'_{G_2}(\delta))]$, and $\pi_{(\sigma, \delta)}=\cInd_{\Gamma'_{(\sigma, \delta)}}^{I_{G_1(\sigma)}\times I'_{G_2}(\delta)} (\delta^{1/2}_{\Gamma\setminus (G_1\times G_2)} \otimes \rho)$, a preunitary representation of $I_{G_1(\sigma)}\times I'_{G_2}(\delta)$.
  Let $(\Sigma, \mathcal{W}_1)$ be the completion of $(\sigma, W_1)$.  By Cor.\ref{semisimplerep}, $\cInd_{H_1}^{G_1}\check{\sigma} \simeq \oplus m(\check{\pi}_{\nu}) \check{\pi}_{\nu}$, for  $\check{\pi}_{\nu} \in \mathcal{R}_{G_1}( \cInd_{H_1}^{G_1}\check{\sigma})$, and finite natural  numbers  $m(\check{\pi}_{\nu}) $. Note that the result of Lmm.\ref{simepl} (2) has not yet proved.   By  Lmm.\ref{semisimpleu},  $\Res_{I_{G_2}(\delta)}^{G_2} \pi_2$, $\Res_{I'_{G_2}(\delta)}^{G_2} \pi_2$ both are semi-simple.  Note that $\frac{I_{G_1}(\sigma)}{H_1}$,  $\frac{ I_{G_2}(\delta)}{H_2}$ both are compact groups, and $\frac{I_{G_1}(\sigma)}{H_1}$,  $\frac{I_{G_2}(\delta)}{H_2}$ are  open subgroups of $\frac{G_1}{H_1}$, $\frac{G_2}{H_2}$ respectively. Hence   by Frobenius reciprocity, we have
\begin{align*}
0 &\neq   \Hom_{H_1 \times H_2}\big( \rho, \sigma \otimes (\check{\pi}_2|_{H_2})^{\vee}\big) \simeq \Hom_{H_1 \times I'_{G_2}(\delta)}(\cInd_{H_1\times H_2}^{H_1 \times I'_{G_2}(\delta)}(\delta_{\Gamma \setminus (G_1\times G_2)}^{1/2}\otimes \rho), \sigma \otimes  \pi_2)\\
&  \simeq \Hom_{I_{G_1}(\sigma) \times I'_{G_2}(\delta)}(\pi_{(\sigma, \delta)}, \Ind_{H_1}^{I_{G_1}(\sigma)}\sigma \otimes  \pi_2)
 \simeq \Hom_{I_{G_1}(\sigma) \times G_2}(\pi, \Ind_{H_1}^{I_{G_1}(\sigma)}\sigma \otimes  \pi_2)\\
  &  \simeq \Hom_{G_1 \times G_2} \big( \pi, \Ind_{H_1}^{G_1}\sigma \otimes \pi_2\big)\simeq \Hom_{G_1 \times G_2} \big( \pi, [\cInd_{H_1}^{G_1}\check{\sigma}]^{\vee} \otimes \pi_2\big) \\
  & \hookrightarrow \prod_{\pi_{\nu}}m(\check{\pi_{\nu}}) \Hom_{G_1 \times G_2} \big(\pi,   \pi_{\nu}\otimes\pi_2\big)
 \end{align*}
 Therefore there exist $\pi_i \in \Irr(G_i)$ such that $\pi_1 \otimes \pi_2\in \mathcal{R}_{G_1 \times G_2}(\pi)$, $\sigma \prec \pi_1|_{H_1}$, $\delta\prec \pi_2|_{H_2}$. Moreover $1=m_{G_1 \times G_2} \big( \pi, \pi_1 \otimes \pi_2\big)=m_{I_{G_1}(\sigma)\times I'_{G_2}(\delta)}\big( \pi_{(\sigma, \delta)},\widetilde{\sigma} \otimes \pi_2\big)$.  So   $\pi_2|_{I'_{G_2}(\delta)}$ contains only one   $\widetilde{\delta}' \in \Irr (I'_{G_2}(\delta))$ such that $\widetilde{\sigma} \otimes \widetilde{\delta}' \in \mathcal{R}_{I_{G_1}(\sigma)\times I'_{G_2}(\delta)}\big( \pi_{(\sigma, \delta)}\big)$, and $m_{I_{G_1}(\sigma)\times I'_{G_2}(\delta)}\big(\pi_{(\sigma, \delta)}, \widetilde{\sigma} \otimes \widetilde{\delta}'\big)=1$.
 \begin{lemma}
\begin{itemize}
\item[(1)] $I'_{G_2}(\delta)/[I_{G_2}(\delta) \cap I'_{G_2}(\delta)]$, $I_{G_2}(\delta)/[I_{G_2}(\delta) \cap I'_{G_2}(\delta)]$ both have finite cardinalities.
\item[(2)] $\cInd_{H_2}^{I'_{G_2}(\delta)} \delta$ is a semi-simple representation.
\end{itemize}
\end{lemma}
\begin{proof}
1) $\frac{I_{G_2}(\delta) \cap I'_{G_2}(\delta)}{H_2}$ is an  open subgroup of $\frac{I_{G_2}(\delta)}{H_2}$ or $\frac{ I'_{G_2}(\delta)}{H_2}$.\\
2) By Coro.\ref{semisimplerep}, $\cInd_{H_2}^{I_{G_2}(\delta)} \delta$ is semi-simple, so is $\Res^{I_{G_2}(\delta)}_{I_{G_2}(\delta) \cap I'_{G_2}(\delta)}\cInd_{H_2}^{I_{G_2}(\delta)} \delta$. Hence $\cInd_{H_2}^{I_{G_2}(\delta) \cap I'_{G_2}(\delta)} \delta$ is semi-simple, and so is  $\cInd_{H_2}^{I'_{G_2}(\delta)} \delta$.
\end{proof}

\begin{lemma}
$ \delta\prec \widetilde{\delta}' |_{H_2}$.
\end{lemma}
\begin{proof}
By Frobenius reciprocity, $0\neq \Hom_{H_1\times H_2}(\rho, \widetilde{\sigma}\otimes \delta) \simeq \Hom_{I_{G_1(\sigma)}\times I'_{G_2}(\delta)}\big( \pi_{(\sigma, \delta)}, \widetilde{\sigma} \otimes \cInd_{H_2}^{I'_{G_2}(\delta)} \delta\big)$. So there exists $\widetilde{\delta}''\prec \cInd_{H_2}^{I'_{G_2}(\delta)} \delta$, such that $\widetilde{\sigma}\otimes \widetilde{\delta}'' \in \mathcal{R}_{I_{G_1(\sigma)}\times I'_{G_2}(\delta)}(\pi_{(\sigma, \delta)})$. Note that
$ \cInd_{H_2}^{G_2}  \widetilde{\delta}'' \prec \cInd_{H_2}^{G_2} \delta$. By virtue of Frobenius reciprocity again, we obtain
$\widetilde{\delta}'' \prec \pi_2$. Hence $\widetilde{\delta}'' \simeq \widetilde{\delta}'$, and $\delta\prec \widetilde{\delta}' |_{H_2}$.
\end{proof}
\begin{lemma}\label{dim}
\begin{itemize}
\item[(1)] If $m_{I_{G_1}(\sigma) \times I'_{G_2}(\delta)} \big( \pi_{(\sigma,\delta)}, \widetilde{\sigma} \otimes  \widetilde{\delta}''\big) \neq 0$, for some $\widetilde{\delta}'' \in \mathcal{R}_{I'_{G_2}(\delta)}\big(\cInd_{H_2}^{I'_{G_2}(\delta)} \delta\big)$, then $\widetilde{\delta}'' \simeq \widetilde{\delta}'$.
\item[(2)] If $m_{I_{G_1}(\sigma) \times I'_{G_2}(\delta)} \big( \pi_{(\sigma,\delta)}, \widetilde{\sigma}' \otimes  \widetilde{\delta}'\big) \neq 0$, for some $\widetilde{\sigma}' \in \Irr(I_{G_1}(\sigma))$ such that $ \sigma\prec \widetilde{\sigma}'|_{H_1}$, then $\widetilde{\sigma}' \simeq \widetilde{\sigma}$.
\end{itemize}
\end{lemma}
\begin{proof}
1) Assume $  \widetilde{\delta}''\prec \pi_2' |_{ I'_{G_2}(\delta)}$, for some  $\pi_2'\in \mathcal{R}_{G_2}\big( \cInd_{H_2}^{G_2} \delta\big)$. Then $\Hom_{I_{G_1}(\sigma) \times I'_{G_2}(\delta)} \big( \pi_{(\sigma,\delta)}, \widetilde{\sigma} \otimes  \pi_2'\big) \simeq \Hom_{G_1 \times G_2} \big( \pi, \pi_1\otimes \pi_2'\big)$. By the property of graph of $\pi$, we have $\pi_2'\simeq \pi_2$. So we can  assume $\widetilde{\delta}''  \prec  \pi_2 |_{ I'_{G_2}(\delta)}$.  By  $m_{I_{G_1}(\sigma) \times I'_{G_2}(\delta)}\big(  \pi_{(\sigma, \delta)}, \widetilde{\sigma}\otimes \pi_2\big)=1$, we obtain $\widetilde{\delta}'' \simeq \widetilde{\delta}'$.\\
2)  $0\neq m_{I_{G_1}(\sigma) \times I'_{G_2}(\delta)} \big( \pi_{(\sigma,\delta)}, \widetilde{\sigma}' \otimes  \pi_2\big)=m_{G_1 \times G_2} \big( \pi,\Ind_{ I_{G_1}(\sigma)}^{G_1}\widetilde{\sigma}' \otimes  \pi_2\big) $. Note that $\widetilde{\sigma}' \prec\cInd_{H_1}^{I_{G_1}(\sigma)} (\delta_{H_1\setminus G_1}\otimes  \sigma)$, and     $ \mathcal{R}_{H_1}(\widetilde{\sigma}')=\{\sigma\}$.  Consequently $\cInd_{ I_{G_1}(\sigma)}^{G_1}\widetilde{\sigma}'$ is a semi-simple representation. By Frobenius reciprocity,  $\Hom_{G_1}(\cInd_{ I_{G_1}(\sigma)}^{G_1}\widetilde{\sigma}', \cInd_{ I_{G_1}(\sigma)}^{G_1}\widetilde{\sigma}') \simeq \Hom_{ I_{G_1}(\sigma)}( \widetilde{\sigma}',  \cInd_{ I_{G_1}(\sigma)}^{G_1}\widetilde{\sigma}') \simeq \Hom_{ I_{G_1}(\sigma)}( \widetilde{\sigma}', \widetilde{\sigma}') $, because every element in $\Hom_{ I_{G_1}(\sigma)}( \widetilde{\sigma}',  \cInd_{ I_{G_1}(\sigma)}^{G_1}\widetilde{\sigma}') $  needs  to factor through $\widetilde{\sigma}' \hookrightarrow \cInd_{ I_{G_1}(\sigma)}^{G_1}\widetilde{\sigma}'$.  Hence  $ \cInd_{ I_{G_1}(\sigma)}^{G_1}\widetilde{\sigma}'$ is an irreducible representation, and $\widetilde{\sigma}'$ is just the $\sigma$-isotypic component of it. Hence  $\cInd_{ I_{G_1}(\sigma)}^{G_1}\widetilde{\sigma}'\simeq \Ind_{ I_{G_1}(\sigma)}^{G_1}\widetilde{\sigma}' \simeq \pi_1$, and $\widetilde{\sigma}' \simeq \widetilde{\sigma}$.
\end{proof}
Suppose now $\widetilde{\sigma}|_{H_1}=n\sigma$, $ m_{H_2}(\widetilde{\delta}', \delta)=  m_1 \neq 0$, $m_{H_1 \times H_2}\big(\rho, \sigma \otimes \delta\big)=k$, and $m_{H_1 \times H_2} \big(\rho, \sigma \otimes \widetilde{\delta}'\big)=t$. Note that  $n<+\infty$.  Then
 \begin{align}
& \Hom_{H_1 \times H_2} \big(\rho , \sigma \otimes \widetilde{\delta}'\big) \simeq \Hom_{I_{G_1}(\sigma)  \times I'_{G_2}(\delta)} \big( \pi_{(\sigma, \delta)}, \Ind_{H_1}^{I_{G_1}(\sigma) } \sigma \otimes \widetilde{\delta}' \big) \label{equations1}  \\
& \Hom_{H_1 \times H_2} \big(\rho, \widetilde{\sigma }\otimes \delta\big) \simeq \Hom_{I_{G_1}(\sigma)  \times I'_{G_2}(\delta)} \big( \pi_{(\sigma, \delta)}, \widetilde{\sigma} \otimes \Ind_{H_2}^{I'_{G_2}(\delta) } \delta\big)  \label{equations2}
\end{align}
So by equation (\ref{equations1}),   $km_1 \leq t=n<+\infty$, and  by equation (\ref{equations2}), $kn\leq m_1\neq 0 $. Therefore $k=1$, $m_1=n=t$. As a consequence, we obtain $\widetilde{\delta}'|_{H_2} \simeq m_1 \delta$, and $I'_{G_2}(\delta)\subseteq I_{G_2}(\delta)$. By symmetry, $I'_{G_1}(\sigma)\subseteq I_{G_1}(\sigma)$. Hence $I'_{G_2}(\delta)=I_{G_2}(\delta)$. Consequently,   $\widetilde{\delta}'$ is the $\delta$-isotypic component of $\pi_2|_{H_2}$.  Note that $m_{H_1\times H_2}(\rho, \sigma \otimes \delta)=k=1$.

 If $\sigma \otimes \delta_1 \in \mathcal{R}_{H_1 \times H_2}(\rho)$, then there exists $\pi_2'\in \Irr(G_2)$ such that $ \delta_1\prec\pi_2'|_{H_2} $, and $\pi_1 \otimes \pi_2' \in \mathcal{R}_{G_1 \times G_2}(\pi)$. Hence $\pi_2'\simeq \pi_2$, and we can assume $\delta_1\prec \pi_2|_{H_2}$. So   $\delta\simeq\delta_1^g$, for certain $g\in G_2$. Since $\sigma \otimes \delta_1\in \mathcal{R}_{H_1\times H_2}(\rho)$,  we have  $\gamma^{-1}(g)\in I_{G_1}(\sigma)$. Hence  $g\in I_{G_2}(\delta)$, and $\delta_1\simeq \delta$.

 \section{The theta representation IV }\label{stronglygraphreIV}
In this section,   let $G_1, G_2$ be locally profinite groups with closed  subgroups $H_1$ and $H_2$  respectively.  Assume all irreducible smooth representations of $G_i$, $H_i$ are admissible, $i= 1, 2$.  Set $H=H_1\times H_2$, $G=G_1\times G_2$. Let  $\Delta=\{s=(s_1, s_2)\in G\}$, containing $1$, be a complete  set of representatives for  $H\setminus G/H$. Assume $\Delta$ is a countable set. For any $s\in \Delta$, $s\neq 1$,   assume: (1) $H_s\cap H$ is a normal subgroup of $H$, (2) $H/(H_s\cap H)$  is  not compact, (3) up to $H_s\cap H$-conjugacy there  exists at least one and at most  a finite number of maximal open compact subgroups in $H$, (4)   for each maximal open compact  subgroup $K$ of $G$, for each positive integer  $n$, the set  $\mathcal{N}(K)_n=\{ K^i \mid K^i \lhd K, [K: K^i]=n\}$ has finite cardinality. Let $(\sigma, U)$ be  a smooth representation of $H$,   set $\pi= \cInd_{H}^{G}\sigma$.
 Assume $U$ is a second countable vector space, and $G$, $H$ both are second countable groups.
For simplicity, we assume $G/H$ is compact in this text.
\subsection{} In the first  part, assume that  $H$ is an open subgroup of $G$. Note that  the  conditions of Lmm.\ref{ddf2}  hold in this case.
\begin{lemma}
 For any $\pi_i\in \Irr(G_i)$, $\mathcal{L}_{H_i}(\pi_i)=\{ \sigma_i \in \Irr(H_i)\mid m_{H_i}(\sigma_i, \pi_i)\neq 0\}\neq \emptyset$.
  \end{lemma}
  \begin{proof}
  Since $H_i$ is an open subgroup of $G_i$, $\Res_{H}^{G} \pi_i$ is also admissible.  Let  $\check{\sigma}_i\in \mathcal{R}_{H}(\check{\pi}_i)$.  Then  $ m_{H_i}(\sigma_i, \pi_i)=m_{H_i}(\check{\pi}_i, \check{\sigma}_i)\neq 0$, which means $\sigma_i\in\mathcal{L}_{H_i}(\pi_i)$.
 \end{proof}

  \begin{proposition}
If $\rho$ is a  general theta representation of $H$, then  so is  the representation $\pi$ of $G$.
\end{proposition}
\begin{proof}
 Assume $\pi_1\otimes \pi_2\in \mathcal{R}_{G}(\pi)$. Let $\check{\sigma}_i\in \mathcal{L}_{H_i}(\check{\pi}_i)$. Then $\check{\pi}_i\in \mathcal{R}_{G_i}\big(\cInd_{H_i}^{G_i}  \check{\sigma}_i\big)$.  So $1\leq m_G(\pi, \pi_1\otimes \pi_2)=m_{G}(\check{\pi}_1\otimes \check{\pi}_2, \cInd_H^G \check{\rho})\leq m_{G}(\cInd_H^G \check{\sigma}_1\otimes \check{\sigma}_2, \cInd_H^G \check{\rho})= m_H(\check{\sigma}_1\otimes \check{\sigma}_2, \check{\rho})= m_{H} (\rho, \sigma_1\otimes \sigma_2)\leq 1;$ the second equality comes from Lmm.\ref{ddf2}.
On the other hand,  if $\pi_1\otimes \pi_2'\in \mathcal{R}_{G}(\pi)$,  then $m_{H} (\rho, \sigma_1\otimes \sigma_2')= 1$, where $\check{\sigma'}_2\in \mathcal{L}_{H_2}(\check{\pi}_2')$.  By the property of graph, $\sigma_2'\simeq \sigma_2$, and $\check{\pi}_2'\in \mathcal{R}_{G_2}\big(\cInd_{H_2}^{G_2}  \check{\sigma}_2\big)$.  If $\pi_2 \ncong \pi_2' $, then  $m_{G}\big(\cInd_H^G( \check{\sigma}_1\otimes\check{\sigma}_2), \check{\pi}_1\otimes (\check{\pi}_2\oplus \check{\pi}'_2)\big)\geq 2$, $m_{G}\big(\check{\pi}_1\otimes[\check{\pi}_2\oplus \check{\pi}'_2],  \cInd_H^G \check{\rho}\big)=2$, and $m_{G}\big(\cInd_H^G \check{\sigma}_1\otimes \check{\sigma}_2, \cInd_H^G \check{\rho}\big)=1$, contradicting to Lmm.\ref{compodouble}.
\end{proof}

\subsection{}
In the second part, assume $(\rho, W)$ is  an admissible preunitary representation of $H$. Assume  the category $\Rep(H)$ is locally noetherian;  for any open compact subgroup $K_1$ of $H$, assume $\mathcal{H}(H,K_1)$ can be  generated by $\epsilon_{K_1}$ and  a finitely number of $\epsilon_{h}$'s. Note that the condition of Coroallary \ref{notm} holds in this case.
\begin{proposition}
If $\rho$ is a  general theta representation of $H$, then  so is  the representation $\pi$ of $G$.
\end{proposition}
\begin{proof}
 Assume $\pi_1\otimes \pi_2\in \mathcal{R}_{G}(\pi)$. Let $\sigma_i\in \mathcal{R}_{H_i}(\pi_i)$. Then  by Frobenius reciprocity $\pi_i \hookrightarrow \cInd_{H_i}^{G_i} \sigma_i$.   So $1\leq m_G(\pi, \pi_1\otimes \pi_2)\leq m_{G}\big(\cInd_H^G \rho, \cInd_H^G (\sigma_1\otimes \sigma_2)\big)\leq m_{H}(\rho,\sigma_1\otimes \sigma_2)\leq 1$; the third  inequality comes from Coro.\ref{notm}.
On the other hand,  if $\pi_1\otimes \pi_2'\in \mathcal{R}_{G}(\pi)$,  then $m_{H} (\rho, \sigma_1\otimes \sigma_2')= 1$, where $\sigma'_2\in \mathcal{R}_{H_2}(\pi_2')$.  By the property of graph, $\sigma_2'\simeq \sigma_2$, and $\pi_2'\hookrightarrow \cInd_{H_2}^{G_2}\sigma'_2$.  If $\pi_2 \ncong \pi_2' $, then  $m_{G}\big(\pi_1\otimes (\pi_2\oplus \pi'_2),\cInd_H^G( \sigma_1\otimes \sigma_2)\big)\geq 2$, $m_{G}\big( \cInd_H^G \rho, \pi_1\otimes(\pi_2\oplus \pi'_2)\big)=2$, and $m_{G}\big( \cInd_H^G\rho,\cInd_H^G (\sigma_1\otimes \sigma_2)\big)=1$, contradicting to Lmm.\ref{compodouble}.
\end{proof}

 \section{Howe correspondences for  the similitude groups}
In this section, we shall show how one can use the results in Sections \ref{stronglygraphreI}, \ref{stronglygraphreII} to do with  Howe correspondences for the similitude groups in the $p$-adic case.  To do so smoothly, we review some known results and methods on the classical theta correspondences and the related topics.
\subsection{Notation and conventions}\label{Notationandconventionsis}
In this last  section, we will  use the following notion and conventions(\emph{cf}. \cite{MVW},  \cite{Scha}). We will let  $F$ be  a non-archimedean local field of \emph{odd} residual characteristic with  ring of integers $\mathcal{O}_F$ and  finite residue field $k_F$.   $E$ will stand for  a \emph{separable} quadratic  field extension of $F$. $\mathbb{H}$ will  denote the unique(non-splitting) quaternion algebra over $F$.  We will  write  $D$ for  a division ring over $F$ with an involution $\tau$ such  that $F$ consists of all $\tau$-fixed points of $D$.  When $D=\mathbb{H}$,  define the reduced trace  by
$\Trd(a):= a+ \tau(a)$ and the reduced norm  by $\Nrd(a):= a \tau(a)$. We  denote by $\mathbb{H}^0$ the set of elements of  pure quaternions, i.e. those elements $a \in \mathbb{H}$ such that $\Trd(a)=0$.

Let $\varepsilon$ be the number $1$ or $-1$.  If $V$ is  a finite-dimensional non-degenerate right (resp. left) $\varepsilon$-hermitian vector space over $D$ endowed  with an $\varepsilon$-hermitian form $(-,-)_{V}: V \times V \longrightarrow D$
     satisfying $(v', v)_{V}= \varepsilon \tau( (v, v')_{V}$, for $v, v'\in V$; as usual, when $\varepsilon=1$, $1$-hermitian is called simply \emph{hermitian} and when $\varepsilon=-1$, $-1$-hermitian is called \emph{skew hermitian};
we will let $\U(V)$ be the group of isometries of $(V, (, )_{V})$, which consists of  $g\in \GL_D(V)$ such that
$ (g \cdot v, g \cdot v')_{V}=(v, v')_{V} \quad \Big( \textrm{resp.}  (v \cdot g , v' \cdot g)_{V}=(v, v')_{V}\Big)$
for all $v, v' \in V$,
 and $\GU(V)$  the group of isometries of similitudes of $(V, (,  )_{V})$, which consists of  $g\in \GL_D(V)$ such that
$ (g \cdot v, g \cdot v')_{V}=\lambda(g) (v, v')_{V} \quad \Big( \textrm{resp. }  (v \cdot g , v' \cdot g)_{V}= \lambda(g)(v, v')_{V}\Big)$
for all $v, v' \in V$, where $\lambda(g) \in F^{\times}$ depending on $g$, is called the \emph{multiplier} of $g$.

 There are two kind of canonical right (resp. left) $\varepsilon$-hermitian vector spaces over $D$.  One is of one dimension  $(D(a), \langle, \rangle)$ (resp. $((a)D, \langle, \rangle )$) for $a\in D^{\times}$ satisfying $a=\varepsilon \tau(a)$, defined as
$$\langle d_1, d_2 \rangle= \tau(d_1) a d_2 \quad \big(\textrm{ resp. } \langle d_1, d_2 \rangle= d_1 a \tau(d_2)\big),  \qquad d_1, d_2 \in D.$$
The other one is of two dimension, so-called  the right (resp. left) $\varepsilon$-hermitian  \emph{hyperbolic plane} $H, \langle, \rangle$  over $D$,  defined as
$$\langle (d_1, d_1^{\ast}), (d_2, d_2^{\ast})\rangle= \tau(d_1) d_2^{\ast} + \varepsilon \tau(d_1^{\ast}) d_2, \quad
\bigg(\textrm{ resp. } \langle (d_1, d_1^{\ast}), (d_2, d_2^{\ast})\rangle= d_1 \tau(d_2^{\ast}) + \varepsilon d_1^{\ast} \tau(d_2)\bigg),$$
for $d_1, d_2, d_1^{\ast}, d_2^{\ast} \in D$. Let $(-, -)_F$ be the Hilbert symbol defined from $F^{\times} \times F^{\times}$ to $\{ \pm 1\}$. Let $(Q, W)$ be a quadratic form defined over $F$ with the Witt decomposition $W\simeq \oplus_{i=1}^m F(a_i)$. The Hasse invariant is given in the following form:
$\epsilon(Q):=\prod_{1 \leq i< j\leq m} (a_i, a_j)_F$. We will let $\mu_n=\langle e^{\frac{2\pi i}{n}}\rangle$,  $ e^{\frac{2\pi i}{n}} \in \C^{\times}$.
\subsection{Weil index}\label{Weilindex}
Let $\psi$ be a non-trivial character of $F$. Let  $V$ be  a (left) vector space over $F$ of dimension $n$, and $V^{\ast}=\Hom(V, F)$ its dual space.  For $v\in V, v^{\ast} \in V^{\ast}$, we write $[v, v^{\ast}]$ for the value of $v^{\ast}$  at $v$. Fix a Haar measure $dv$ for $V$. The Fourier transformation of an element $f\in S(V)$ is defined by
$$ \mathcal{F}(f)(v^{\ast})= \int_{V} f(v) \psi\big( [ v, v^{\ast} ]\big) dv, \qquad  v^{\ast} \in V^{\ast}.$$
Then there is a unique Haar measure $dv^{\ast}$ assigned to $V^{\ast}$, called the duality of $dv$ such that
$$f(-v)=\int_{V^{\ast}}  \mathcal{F}(f) \big( v^{\ast} \big) \psi\big( [ v, v^{\ast}]\big) dv^{\ast}, \qquad \qquad v\in V, f\in S(V).$$
By convention, we define the Fourier transformation on $ T\in S^{\ast}(V)$ with respect to $dv, dv^{\ast}$ by
$$ [ \mathcal{F}(T), f^{\ast} ] = [ T, \mathcal{F}(f^{\ast})], \qquad \qquad  f^{\ast}\in S(V^{\ast}).$$
Recall that if $\alpha$ is an $F$-linear bijection from $V$ to $V^{\ast}$, then the module of $\alpha$ is the number $|\alpha|_F=d(v \cdot \alpha )/dv$ defined by the formula
$$\int_{V^{\ast}} f^{\ast}(v^{\ast}) dv^{\ast}= |\alpha|_F \int_{V} f^{\ast} ( v  \cdot \alpha) dv, \qquad \qquad  f^{\ast} \in S(V^{\ast}).$$
Let $(-,-)$ be a non-degenerate symmetric form on $V$, and $q$ the quadratic form associated, i.e.
$$q(v+v') -q(v) -q(v')=(v,v'),  \qquad \qquad v, v' \in V.$$
Follow above, the symmetric form $(-,-)$ can be written in the form:
 $$(v,v')=[ v , v' \cdot b ], \qquad \qquad v,v'\in V$$
for a unique  $b\in \Hom(V, V^{\ast})$. In particular, we can introduce a symmetric form on  $V^{\ast}$:
$$ ( v^{\ast}, v'^{\ast}):= [ v^{\ast} \cdot b^{-1} ,  v'^{\ast} ],  \qquad \qquad v^{\ast}, v'^{\ast} \in V^{\ast},$$
and the quadratic form associated:
$$q^{\ast}(v^{\ast} + v'^{\ast})  -q^{\ast} (v^{\ast}) -q^{\ast} (v'^{\ast})=(v^{\ast} ,v'^{\ast} ).$$
  Let $\psi (q)(v):= \psi(q(v))$ (resp. $\psi (q^{\ast})(v^{\ast}):= \psi(q^{\ast}(v^{\ast}))$) be a character of second degree of $V$(resp. $V^{\ast}$). By \cite[p. 161, Th\'eor\`eme]{Weil}, they exists a unique root of  unity of  degree $8$, called  the  \textbf{  Weil index} attached to $\psi(q)$, denoted by  $\gamma_{\psi}(q)$, such that
$$\mathcal{F}(\psi( q) dv) = \gamma_{\psi} (q) |b|_F^{-\frac{1}{2}} \psi ( q^{\ast})^{-1} dv^{\ast}, $$
for  $\psi (q) dv  \in S^{\ast}(V)$, and  $\psi (q^{\ast})^{-1} dv^{\ast}\in S^{\ast}(V^{\ast})$.

 Remark that the Weil index only depends on  the Witt class of $(q, V)$  and $\psi$.  For simplicity, we will denote  by  $\gamma_{\psi}(a)$ the Weil index attached to the quadratic form  $v \longmapsto av^2$, and  let $\gamma(a, \psi)= \tfrac{\gamma_{\psi}(a)}{\gamma_{\psi}(1)}$ be its normalizer.
\subsection{The Weil representation I}\label{TheWeilrepresentation}
Let $W$ be a symplectic space over $F$ of dimension $2n$, endowed with a symplectic form $\langle, \rangle$. The Heisenberg group $\Ha(W)$, attached to $W$ and $F$, is a topological group $W \oplus F$, with the  law
$$ (w, t) (w', t')=(w+ w', t+t' + \tfrac{\langle w,w' \rangle}{2} )$$
where $w, w' \in W, t, t' \in F$. The center  of $\Ha(W)$ is $\{ 0\} \times F$.

 Let $\Sp(W)$ be the group of isometries of $(W, \langle, \rangle)$ and $A$ a subgroup of $\C^{\times}$ containing $\{ \pm 1\}$. By \cite[p. 13, Lmm.2.3  and p. 53, Theorem 10.5]{Mo} and \cite[p. 57, Th\'er\`eme 12.1(c)]{Mat}, we know
  $$ \Ha^1(\Sp(W), A)=0 \textrm{ and } \Ha^{2}(\Sp(W), A) \simeq \Hom(\mu_F, A),$$
  where $\mu_F$ is the cyclic group of the roots of unity in $F$ (a finite group). Here,  $\Ha^1(\Sp(W), A),  \Ha^2(\Sp(W), A)$ are the measurable  cohomology groups defined in \cite{Mo}.
So there exists a unique  element in $\Ha^2(\Sp(W), A)$ of order two; this class  gives rise to a unique central topological extension
$$1 \longrightarrow  A \longrightarrow \Mp_A(W) \stackrel{p}{\longrightarrow } \Sp(W) \longrightarrow 1$$
of $\Sp(W)$ by $A$.  As usual,  $\Mp_A(W)$ is called the \emph{Metaplectic group} (w.r.t. $A$). When $A=\mu_2$, $\mu_8$ and $\C^{\times}$, we will denote it   by $\widehat{\Sp}(W)$, $\overline{\Sp}(W)$ and  $\widetilde{\Sp}(W)$ respectively. In particular, the topological groups $\widehat{\Sp}(W)$, $\overline{\Sp}(W)$    are locally profinite.

Fix a non-trivial character $\psi$ of $F$. According
to the Stone-von Neumann theorem, there is only one equivalence class of irreducible smooth complex representation $\rho_{\psi}$ of $\Ha(W)$ with central character $\psi$. Let us denote one model of  this representation by $(\rho_{\psi}, S)$.  Now we define  a semi-direct product group $ \Mp_A(W)\ltimes  \Ha(W)$ by
$$[h_1, (w_1, t_1)][  h_2, (w_2,t_2) ]:=[ h_1h_2 , (w_1 \cdot p(h_2), t_1)+ (  w_2,t_2) ]$$
for $h_1, h_2\in \Mp_A(W)$, $w_1, w_2 \in W$ and $t_1, t_2 \in F$.
\begin{theorem}[Weil]\label{Weil}
$(\rho_{\psi}, S)$ can be extended uniquely  to  a smooth representation of  $  \Mp_A(W)\ltimes \Ha(W)$ such that  $\rho_{\psi}|_{A}(\epsilon)=\epsilon \id_S$, for $\epsilon \in A$.
\end{theorem}
\begin{proof}
The existence is a well-known result,  due to  Andr\'e Weil \cite{Weil}. The  uniqueness is just an exercise, and let us do it now. If $(\pi_{\psi}, S)$, $(\pi_{\psi}', S')$ are two extensions of the representation $(\rho_{\psi}, S)$ of $A \times \Ha(W)$ to $  \Mp_A(W)\ltimes \Ha(W)$, then a $\Sp(W)$-module $\Hom_{A \times \Ha(W)}(\rho_{\psi}, \rho_{\psi})$ comes as defined by $g\cdot \phi(x)=\pi_{\psi}'(g)\phi(\pi_{\psi}(g^{-1})x)$, whence $\phi\in \Hom_{A \times \Ha(W)}(\rho_{\psi}, \rho_{\psi})$, $g\in \Sp(W)$. Since $\Sp(W)$ is perfect(the case that residual characteristic not even), and $\Hom_{A \times \Ha(W)}(\rho_{\psi}, \rho_{\psi})$ has only one dimension, $\pi_{\psi}$ and $\pi_{\psi}'$ coincide.
\end{proof}

The restriction of $\rho_{\psi}$ to $\Mp_A(W)$ is called the \textbf{Weil representation} of $\Mp_A(W)$, denoted by $\omega_{\psi}$  from now on. It is known that $\omega_{\psi}$ has two irreducible components.

Similarly, let $\chi_A^+$ be a character of $A$ given by $x\longrightarrow x^{-1}$, and $\psi^-$ another character of $F$ defined by $x\longrightarrow \psi(-x)$. Now let $(\rho_{\psi^-}, S^-)$ be the smooth representation of $\Mp_A(W)\ltimes \Ha(W)$, associated to $\psi^-$, such that $\rho_{\psi^-, \chi_A^+}(t)=\chi_A^+(t)\Id_{S^-}$, for $t\in A$. By uniqueness, we have:
\begin{corollary}
$\check{\rho}_{\psi} \simeq \rho_{\psi^-, \chi_A^+}$, and $\check{\omega}_{\psi} \simeq \rho_{\psi^-, \chi_A^+}|_{\Mp_A(W)}$.
\end{corollary}
\begin{proof}
The first statement is immediate. Since $\omega_{\psi}$ is a smooth admissible representation of $\Mp_A(W)$, the second one follows.
\end{proof}
\begin{remark}
The Weil representation $\omega_{\psi}$ of $\Mp_A(W)$ arising from a projective representation of $\Sp(W)$ is  primitive defined for $\widehat{\Sp}(W)$.
\end{remark}
\subsection{Rao's cocycle I}\label{RaoscocycleI}
The cocycles associated to $\overline{\Sp}(W)$, $\widehat{\Sp}(W)$ have been constructed by Rao \cite{Rao}, by Perrin \cite{Per}. For convenient use, we recall their results by following \cite{Kud1} and \cite{MVW}.

 Let $(X_1, X_2, X_3)$ be a triple of Lagrangians of $W$. The Levi invariant $L(X_1, X_2, X_2)$ is an isometry class of the following symmetric vector space: When $X_1, X_2, X_3$ are pairwise transversal, the two complete polarizations $W=X_2  \oplus X_1$ and $W=X_2 \oplus X_3$ will give a unique element $u \in \Sp(W)$ such that $x_1 \cdot u=x_1$ for $x_1 \in X_1$ and $X_2 \cdot u=X_3$. As   a result,
 $$(x,y)   := \langle x, y \cdot u\rangle =\langle y, x \cdot u \rangle, \qquad x, y \in X_2$$
 is a non-degenerate symmetric bilinear form on $X_2$. In this situation, set $L(X_1, X_2, X_3)= X_2, (,)$.  Otherwise, let $M= X_1 \cap X_2 + X_2 \cap X_3 + X_3 \cap X_1$,  consider the non-degenerate symplectic vector space $W_M= M^{\bot} /M$ and its pairwise transversal Lagrangians $Z_i= \big( (X_i+ M) \cap M^{\bot}\big) /M$ for $i=1, 2,3$, and then define $L(X_1, X_2, X_3)=L(Z_1, Z_2, Z_3)$. For $L(X_1, X_2, X_3)$, it has the following properties due to Rao:
 $$L(X_{\sigma(1)}, X_{\sigma(2)}, X_{\sigma(3)})=sign (\sigma)L(X_1, X_2, X_3),  \qquad\qquad \sigma \in S_3;$$
  $$  L( X_1 \cdot g,  X_2 \cdot g,  X_3 \cdot g)=L(X_1, X_2, X_3), \qquad \qquad  g\in \Sp(W);$$
  we will denote the quadratic form associated by $Q(X_1, X_2, X_3)$ (\emph{cf}. Section \ref{Weilindex}).

Now let $Y$ be a Lagrangian of $W$, and $\psi$ a non-trivial character of $F$. For $g_1, g_2 \in \Sp(W)$, set
$$q_Y(g_1, g_2):= Q(Y, Y \cdot g_2^{-1}, Y  \cdot g_1 ).$$
\begin{theorem}[Perrin, Rao]\label{RaococycleforMp}
The class of the $2$-cocycle
$c_Y(g_1,g_2)= \gamma_{\psi}\big(q_Y(g_1,g_2)\big)$ in $\Ha^2(\Sp(W), \mu_8)$ is non-trivial of order $2$.
\end{theorem}
It is immediate that
 $$c_Y(p_1 gp,p^{-1} g' p_2)=c_Y(g, g'),  \qquad\qquad p_1, p_2, p\in P, g, g' \in \Sp(W),$$
  $$ c_Y(p, g)=c_Y(g,p)=1, \qquad\qquad g\in \Sp(W), p\in P, $$
where  $P=\{ g\in \Sp(W) \mid Y \cdot g=Y\}$ is a parabolic subgroup of $\Sp(W)$.
\subsection{Rao's cocycle II}\label{RaoscocycleII}
Let $\{ e_1, \cdots, e_n; e_1^{\ast}, \cdots, e_n^{\ast} \}$ be a symplectic basis of $W$ so that $\langle e_i, e_j \rangle=\langle e_i^{\ast}, e_j^{\ast} \rangle=0$, and $\langle e_i, e_j^{\ast} \rangle=\delta_{ij}$. Let $Y$ be the Lagrangian generated by $e_1^{\ast}, \cdots, e_n^{\ast}$, and $P=\{g \in \Sp(W) \mid Y \cdot g = Y\}$. For $S \subseteq \{ 1, \cdots, n\}$, we let $\omega_{S} \in \Sp(W)$, given by
\[e_i \cdot \omega_S=\left\{\begin{array}{clrr}
     -e_i^{\ast} & i \in S \\       e_i  &  i\notin S,
     \end{array}\right. \qquad e_i^{\ast} \cdot \omega_S=\left\{\begin{array}{clrr}
     e_i & i \in S \\       e_i^{\ast}  &  i\notin S.
     \end{array}\right. \]
As is known that there exists a  decomposition (\emph{cf}. \cite[p. 54]{MVW}) $\Sp(W)= \sqcup_{j=1}^n C_j$,
where $C_j= P\omega_SP$ for any $\omega_S$ with $|S|=j$.   In \cite{Rao}, Rao defined the following functions:
$$x: \Sp(W) \longrightarrow F^{\times}/{(F^{\times})}^2; p_1\omega_Sp_2 \longmapsto \det(p_1p_2|_Y) (F^{\times })^2$$
$$t: \Sp(W) \times \Sp(W) \longrightarrow \Z; (g_1, g_2)\longmapsto \tfrac{1}{2}(|S_1| + |S_2| - |S_3| -l)$$
 where $ g_1=p_1\omega_{S_1} p_1', g_2= p_2 \omega_{S_2} p_2'$  and  $ g_1g_2= p_3 \omega_{S_3} p_3', l=\dim q_Y(g_1, g_2)= \dim Q(Y, Y \cdot g_2^{-1}, Y\cdot g_1)$, $S, S_1, S_2, S_3 \subseteq \{1, \cdots, n\}$.

 The Rao's cocycle  is defined by
$$c_{Rao, Y}(g_1, g_2)= (x(g_1), x(g_2))_F(-x(g_1)x(g_2), x(g_1g_2))_F ((-1)^t, \det(2q))_F (-1, -1)_F^{\tfrac{t(t-1)}{2}}\epsilon(2q)$$
where $t=t(g_1,g_2)$, $q=q_Y(g_1,g_2)$ for  $g_1, g_2 \in \Sp(W)$.
\begin{theorem}[Rao]\label{Raococycle}
The class of Rao's cocycle, $[c_{Rao, Y}]$, in  $\Ha^2(\Sp(W),\{ \pm 1\})$ is non-trivial of order $2$.
\end{theorem}
\begin{proof}
See \cite[p. 20, Theorem 4.5]{Kud1}.
\end{proof}

Up to isomorphism, one can think of the group $\widehat{\Sp}(W)$ as  the underlying topological  set $\Sp(W) \times \{ \pm 1 \}$ with the law
$$(g_1, \epsilon_1) \cdot (g_2, \epsilon_2)= (g_1 g_2, c_{Rao, Y}(g_1, g_2) \epsilon_1 \epsilon_2).$$
 The above constructed $2$-cocycles $c_Y$ and $c_{Rao, Y}$  give   the same  class in   $\Ha^2( \Sp(W), \mu_8)$,  so they will differ by  a coboundary.  Following \cite{Rao}, we  define the normalizing constants as
$$m_Y: \Sp(W) \longrightarrow \mu_8; g \longmapsto (x(g), \tfrac{1}{2})_F \gamma(x(g), \psi)^{-1} \gamma_{\psi}(\tfrac{1}{2})^{-j(g)}$$
 for $g=P\omega_SP$, $j(g)=|S|$.
\begin{proposition}[Rao]\label{Raococyclesmu2mu8}
 For $g_1, g_2 \in \Sp(W)$, we have $$c_Y(g_1, g_2)=m_Y(g_1g_2) m_Y(g_1)^{-1} m_Y(g_2)^{-1} c_{Rao, Y}(g_1,g_2).$$
\end{proposition}
\begin{proof}
See Kudla's  famous note  \cite[p. 20, Theorem 4.5]{Kud1}.
\end{proof}
\subsection{Rao'cocycle III}
 Suppose  $W_1$ and $W_2$ are the symplectic subspaces of $W$ generated by $\{ e_1, \cdots, e_{n_1}; e_1^{\ast}, \cdots, e_{n_1}^{\ast}\}$ and $\{ e_{n_1+1}, \cdots, e_{n}; e_{n_1+1}^{\ast}, \cdots, e_{n}^{\ast}\}$ respectively. Let $Y_1 =\spann\{e_1^{\ast}, \cdots, e_{n_1}^{\ast}\}$,  $Y_2= \spann\{ e_{n_1+1}^{\ast}, \cdots, e_{n}^{\ast}\}$, and $Y =\spann\{e_1^{\ast}, \cdots, e_{n}^{\ast}\}$.  Write $\widehat{\Sp}(W_1)$ and  $\widehat{\Sp}(W_2)$ for the metaplectic groups following  the laws
$$(g_1, \epsilon_1) \cdot(g_1', \epsilon_1')= (g_1g_1', c_{Rao, Y_1}(g_1, g_1') \epsilon_1\epsilon_1')$$
and
 $$(g_2, \epsilon_2) \cdot(g_2', \epsilon_2')= (g_2g_2', c_{Rao, Y_2}(g_2, g_2') \epsilon_2\epsilon_2')$$
respectively, for $g_i, g_i' \in \Sp(W_i)$, $\epsilon_i, \epsilon_i' \in \mu_2$.
\begin{proposition}[{\cite[pp. 245-246]{HanM}}]\label{morphismedegroupes1}
There is a group homomorphism:
$$\widehat{\Sp}(W_1) \times \widehat{\Sp}(W_2) \stackrel{\hat{p}}{\longrightarrow} \widehat{\Sp}(W)$$
$$ [(g_1, \epsilon_1), (g_2, \epsilon_2)]  \longmapsto [(g_1,g_2), \epsilon_1\epsilon_2 c_{Rao,Y}((g_1, 1), (1, g_2))]$$
In particular, considering $\hat{p}|_{\widehat{\Sp}(W_1)}$ and  $\hat{p}|_{\widehat{\Sp}(W_2)}$, we obtain
$$c_{Rao, Y_1}(g_1,g_1')=c_{Rao, Y}((g_1, 1), (g_1', 1))$$ and
$$c_{Rao, Y_2}(g_2,g_2')=c_{Rao, Y}((1,g_2), (1, g_2'))$$
for  $g_1, g_1'\in \Sp(W_1)$, $g_2, g_2'\in  \Sp(W_2)$.
\end{proposition}
Let $\psi$ be  a fixed non-trivial character  of $F$.     Let  $\overline{\Sp}(W_1)$ and  $\overline{\Sp}(W_2)$  be the metaplectic groups associated to $\psi$ by following  the laws
$$(g_1, \epsilon_1) \cdot(g_1', \epsilon_1')= (g_1g_1', c_{ Y_1}(g_1, g_1') \epsilon_1\epsilon_1')$$
and
 $$(g_2, \epsilon_2) \cdot(g_2', \epsilon_2')= (g_2g_2', c_{ Y_2}(g_2, g_2') \epsilon_2\epsilon_2')$$
respectively, for $g_i, g_i' \in \Sp(W_i)$, $\epsilon_i, \epsilon_i' \in \mu_8$.

\begin{proposition}[Rao]\label{morphismedegroupes2}
There is a group homomorphism:
$$\overline{\Sp}(W_1) \times \overline{\Sp}(W_2) \stackrel{\overline{p}}{\longrightarrow} \overline{\Sp}(W)$$
$$ [(g_1, \epsilon_1), (g_2, \epsilon_2)]  \longmapsto [(g_1,g_2), \epsilon_1\epsilon_2 ],$$
i.e. $c_{Y}((g_1, g_2), (g_1', g_2)) = c_{ Y_1}(g_1,g_1') c_{Y_2}(g_2, g_2')$ for  $g_1, g_1'\in \Sp(W_1)$, $g_2, g_2'\in  \Sp(W_2)$.
\end{proposition}
\begin{proof}
By Prop.\ref{morphismedegroupes1}, we have
$$c_{Rao, Y}\big(((g_1g_1', 1), (1, g_2g_2')\big) c_{Rao, Y_1}(g_1, g_1') c_{Rao, Y_2}(g_2,g_2')$$ $$=c_{Rao, Y}\big((g_1, g_2), (g_1',g_2')\big) c_{Rao, Y}\big((g_1, 1), (1, g_2)\big) c_{Rao, Y}\big((g_1', 1), (1, g_2')\big).$$
Applying the result of Prop.\ref{Raococyclesmu2mu8}, we get
$$c_Y\big( (g_1,g_2), (g_1', g_2')\big) c_{Y_1}^{-1}(g_1,g_1') c_{Y_2}^{-1}(g_2,g_2')= \tfrac{m_Y\big( g_1g_1', g_2g_2'\big)}{m_{Y_1}(g_1g_1')m_{Y_2}(g_2g_2')} \cdot \bigg(\tfrac{m_Y\big((g_1, g_2)\big) }{ m_{Y_1}(g_1) m_{Y_2}(g_2)}\bigg)^{-1}\cdot \bigg(\tfrac{m_Y\big((g_1', g_2')\big) }{ m_{Y_1}(g_1') m_{Y_2}(g_2')}\bigg)^{-1}$$
$$c_{Rao, Y}\big( (g_1, 1), (1, g_2)\big)^{-1} c_{Rao, Y}\big( (g_1', 1), (1, g_2')\big)^{-1}  c_{Rao, Y}\big( (g_1g_1', 1), (1, g_2 g_2')\big).$$
Note that by definition, for $s_1 \in \Sp(W_1)$, $s_2 \in \Sp(W_2)$, we have
$$ \tfrac{m_Y\big((s_1, s_2)\big)}{m_{Y_1}(s_1) m_{Y_2}(s_2)}= \tfrac{\gamma (x(s_1), \psi)\gamma(x(s_2), \psi)}{\gamma(x(s_1) x(s_2), \psi)}= (x(s_1), x(s_2))_F$$
and
$$c_{Rao, Y}\big( (s_1, 1), (1, s_2)\big)= (x(s_1), x(s_2))_F (-x(s_1)x(s_2), x(s_1)x(s_2))_F= (x(s_1), x(s_2))_F,$$
so the result follows.
\end{proof}
\subsection{The Weil representation II}
Part of the richness of the Weil representations reflects on their different realized models. Down to the earth, let us recall one so-called the Schr\"odinger model of the Weil representation constructed by  Perrin in \cite{Per}. Let us fix a complete polarisation $W=X\oplus Y$.
\subsubsection{Model for $\overline{\Sp}(W) \ltimes \Ha(W)$} The representation $\rho_{\psi}$ of $\overline{\Sp}(W) \ltimes \Ha(W)$ can be realized in $S(X)$ by the following formulas:
\begin{itemize}
\item[(1)] $\rho_{\psi}((x+y, t)) f(x')=\psi(\langle x', y\rangle +\frac{\langle x, y\rangle}{2} +t) f(x+x')$,
\item[(2)] $\rho_{\psi}((g, \epsilonup))f(x')=\epsilonup |a|_F^{\tfrac{1}{2}} \psi(\tfrac{1}{2}\langle x' \cdot a,  x' \cdot b\rangle ) f(x' \cdot a)$,
\item[(3)] $\rho_{\psi}((g', \epsilonup))f(x')=\epsilonup t(g') \int_{\ker(c')\backslash Y} |\overline{c'}|^{\tfrac{1}{2}} \psi\big(\tfrac{1}{2}\langle x' \cdot a', x'\cdot b'\rangle -\langle x' \cdot b', y \cdot c'\rangle + \tfrac{1}{2} \langle y \cdot c', y\cdot d'\rangle  \big)d\dot{y}$,
\end{itemize}
where $w=x+y\in W$, $t\in F$; $g=\begin{pmatrix}
a& b\\
0& d\end{pmatrix}$,  $g'=\begin{pmatrix}
a'& b'\\
c'& d'\end{pmatrix} \in \Sp(W)$, for $a, a'\in \End_F(X)$, $b, b'\in \Hom_F(X, Y)$, $c'\in \Hom_F(Y, X)$, $d,d' \in \End_F(Y)$, $\epsilonup\in \mu_8$, $f\in S(X)$, and $\overline{c'}$ being the isomorphism from $Y/\ker(c')$ to $[Y/\ker(c')]^*$,  $t(g')$ being a complex number of module $1$ given in \cite[Thm. 2.2]{Per}.
\subsubsection{Doubling method}
Let $P(Y)$ be the parabolic subgroup of $\Sp(Y)$ associated to $Y$ admitting a unipotent subgroup $N(Y)$. Then there is a short exact sequence:   $1 \longrightarrow N(Y) \longrightarrow P(Y) \longrightarrow \GL(Y)\longrightarrow 1$. Let $\chi_{\overline{P}(Y)}^+$ be the character of $\overline{P}(Y)$ defined by $[\begin{pmatrix}
a & b \\
0& a^{\ast-1}\end{pmatrix}, \epsilonup]\longrightarrow |\det(a|_X)|_F^{\tfrac{1}{2}} \epsilonup$.
\begin{lemma}
$\rho_{\psi}|_{\overline{P}(Y)\ltimes \Ha(W)}\simeq \cInd_{\overline{P}(Y) \ltimes Y\cdot F}^{\overline{P}(Y) \ltimes \Ha(W)}(\chi^+_{\overline{P}(Y)} \cdot 1_Y \cdot \psi)$.
\end{lemma}
\begin{proof}
It follows from the above Schr\"odinger model.
\end{proof}
As a consequence, we obtain:
\begin{proposition}
Let $\rho_{\psi}$ be the smooth representation of $\overline{\Sp}(W) \ltimes \Ha(W)$ defined as above. Then $[\rho_{\psi} \otimes \check{\rho}_{\psi}]|_{\overline{\Sp}(W) \ltimes \Ha(W)} \simeq \cInd_{\overline{\Sp}(W) \times F}^{\overline{\Sp}(W) \ltimes \Ha(W)} 1\cdot \psi$.
\end{proposition}
\begin{proof}
Let $2W=W\oplus W$ be a vector space over $F$ of dimension $4n$, equipped with the symplectic form $\langle, \rangle$ defined by
$\langle (w_1, w_2), (w_1', w_2')\rangle:=\langle w_1, w_1'\rangle-\langle w_2, w_2'\rangle$, for $ w_i, w_i'\in W$.
Then there exists the following morphism of groups:
$$(\overline{\Sp}(W)\ltimes \Ha(W)) \times (\overline{\Sp}(W)\ltimes \Ha(W))  \longrightarrow \overline{\Sp}(2W)\ltimes \Ha(2W)$$
$$[(g_1, \epsilonup_1;  w_1, t_1),  (g_2, \epsilonup_2;  w_2, t_2)]\longmapsto  [(g_1, g_2), c_{Rao}((g_1, 1), (1, g_2))\epsilonup_1 \epsilonup_2^{-1};  (w_1,w_2), t_1-t_2]$$
Let $\rho_{\psi}'$ be the smooth representation of $ \overline{\Sp}(2W)\ltimes \Ha(2W)$ as defined in Section \ref{TheWeilrepresentation}. It is known that $\rho_{\psi}'|_{\Ha(W) \times \Ha(W)} \simeq \rho_{\psi}|_{\Ha(W)} \otimes \check{\rho}_{\psi}|_{\Ha(W)}$. Applying the result of Theorem \ref{Weil}, we obtain $\rho_{\psi}'|_{(\overline{\Sp}(W)\ltimes \Ha(W)) \times(\overline{\Sp}(W)\ltimes \Ha(W)) } \simeq \rho_{\psi} \otimes \check{\rho}_{\psi}$; its restriction to the canonical diagonal subgroup $\overline{\Sp}(W) \ltimes \Ha(W)$, yields $\rho_{\psi}' |_{\overline{\Sp}(W) \ltimes \Ha(W)} \simeq (\rho_{\psi} \otimes \check{\rho}_{\psi})|_{\overline{\Sp}(W) \ltimes \Ha(W)}$. Now let us choose a Lagrangian subspace $Y=\left\{(w, -w)\mid w\in W\right\}$ of $2W$.  By definition the image of $\overline{\Sp}(W) \ltimes \Ha(W)$ in  $\overline{\Sp}(2W) \ltimes \Ha(2W)$ lies in $\overline{P}(Y)\ltimes \Ha(2W)$, so
\[(\rho_{\psi} \otimes \check{\rho}_{\psi})|_{\overline{\Sp} (W)\ltimes \Ha(W)}  \simeq\Res_{\overline{\Sp}(W)\ltimes\Ha(W)}^{\overline{P}(Y)\ltimes \Ha(2W)}\big(\rho_{\psi}'|_{\overline{P}(Y)\ltimes \Ha(2W)}\big)\]
\[\simeq \Res_{\overline{\Sp}(W)\ltimes\Ha(W)}^{\overline{P}(Y)\ltimes \Ha(2W)} \Big( \cInd_{\overline{P} (Y)\ltimes Y\cdot F}^{\overline{P}(Y) \ltimes \Ha(2W)}\chi^+_{\overline{P}(Y)} \cdot 1_Y \cdot \psi\Big)\simeq \cInd_{\overline{\Sp}(W) \times F}^{\overline{\Sp} (W)\ltimes \Ha(W)} 1_{\mu_8} \cdot \psi\]
\end{proof}
\subsection{Reductive dual pair}\label{Reductivedualpair}
Let $G_1, G_2$ be two closed subgroups of $\Sp(W)$. We call $(G_1, G_2)$ a \emph{reductive dual pair} or \emph{Howe pair}, if
\begin{itemize}
\item[(1)]  $G_1$ is the commutant of $G_2$, and vice-versa,
\item[(2)]  the action of $G_1G_2$ on $W$ is \emph{absolument} semi-simple.
\end{itemize}

A $G_1G_2$-stable orthogonal decomposition $W=\oplus_v W_v$ will yield a decomposition of the pair $(G_1,G_2)$:
$$G_1=\prod_v H_1^{(v)}, G_2=\prod_v H_2^{(v)}$$
with $(H_1^{(v)}, H_2^{(v)})$ a reductive dual pair of $\Sp(W_v)$; while there is no such non-trivial decomposition, we will call $(G_1, G_2)$ irreducible. An irreducible reductive dual pair $(G_1, G_2)$ has the following  form(\emph{cf}. \cite[p. 15]{MVW}):\\

Type I (a). $V_1, \langle, \rangle_1$ (resp. $V_2, \langle, \rangle_2$) denotes a non-degenerate symplectic (resp. orthogonal) vector space over $F$ such that
$ W \simeq V_1 \otimes V_2, \langle, \rangle \simeq \langle, \rangle_1 \otimes \langle, \rangle_2;$
$G_1 \simeq \Sp(V_1)$,  $G_2 \simeq \Oa(V_2)$, and vice-versa.\\

Type I (b). $V_1, \langle, \rangle_1$ (resp. $V_2, \langle, \rangle_2$) denotes a non-degenerate $\varepsilon_1$-hermitian (resp. $\varepsilon_2$-hermitian) vector space over $E$ such that $\varepsilon_1 \varepsilon_2=-1$, $W\simeq V_1 \otimes_E V_2$, $\langle, \rangle \simeq \tr_{E/F}\big(\langle, \rangle_1 \otimes \tau ( \langle, \rangle_2)\big); G_1  \simeq \U(V_1), G_2 \simeq \U(V_2)$.\\

Type I (c). $V_1, \langle, \rangle_1$ (resp. $V_2, \langle, \rangle_2)$ denotes a non-degenerate right $\varepsilon_1$-hermitian (resp. left $\varepsilon_2$-hermitian) vector space over $\mathbb{H}$ such that $\varepsilon_1 \varepsilon_2=-1$, $W \simeq V_1 \otimes_{\mathbb{H}}V_2$, $\langle, \rangle \simeq \Trd_{\mathbb{H}/F}(\langle, \rangle_1 \otimes \tau (\langle, \rangle_2))$; $G_1 \simeq \U(V_1)$, $G_2 \simeq \U(V_2)$ except when $\varepsilon_1=1$, $\varepsilon_2=-1$, $V_2 \simeq \mathbb{H}$.\\

Type II. There exist a division ring  $D'$ over a separable finite extension $K$ of $F$, and two vector spaces $X_1, X_2$ over $D'$ with the dual vector spaces $X_1^{\ast}$, $X_2^{\ast}$ respectively such that
$W \simeq [X_1 \otimes_{D'} X_2 ]\oplus [X_2^{\ast}  \otimes_{D'} X_1^{\ast}]$, $G_1 \simeq \GL_{D'}(X_1)$, $G_2 \simeq \GL_{D'}(X_2)$.\\

Scalar descent. There exist a nontrivial separable field extension $K$ of $F$, a symplectic vector space $V, \langle, \rangle_V$ over $K$ and $0 \neq t_{K/F} \in \Hom_F(K,F)$ (satisfying that  $t_{K/F}: K \times K \longrightarrow F; (a,b) \longmapsto t_{K/F}(ab)$ is a non-degenerate $F$-bilinear form)  such that $W \simeq V_{/F}$, $\langle, \rangle \simeq t_{K/F}(\langle, \rangle_V)$, $(G_1, G_2)$ is a non-trivial irreducible reductive dual pair  mentioned above in $\Sp(V)$. The ``non-trivial'' signifies  $G_i \ncong \{ \pm 1\} $, $\Sp(V)$.\\

Remark that the pairs listed above   all are the   irreducible reductive dual pairs  in $\Sp(W)$.
  Now we   write $\overline{G_1}$ and $\overline{G_2}$  for their inverse images in $\overline{\Sp}(W)$\footnote{When we treat it as a group, we always fix a $2$-cocycle in hand without  mention.}.  The following result  is a  modified version of the  Th\'eor\`eme  in \cite[p. 52]{MVW}     by considering the Metaplectic group $\overline{\Sp}(W)$ instead of $\widetilde{\Sp}(W)$.
\begin{theorem}\label{splittinghowedual}
The group $\overline{G_1}$ splits over  $G_1$,  except when  $W \simeq V_1 \otimes_K V_2$, $\langle, \rangle \simeq t_{K/F}(\langle, \rangle_1 \otimes \langle, \rangle_2)$ with $V_1$ being symplectic  and $V_2$ being orthogonal of odd dimension ( in this  case $\overline{G_1}\simeq \overline{\Sp}(V_1)$).
\end{theorem}

\subsection{The theta correspondence} Let $G_1$, $G_2$ be a reductive dual pair in $\Sp(W)$, and write $\overline{G_1}$, $\overline{G_2}$ for their inverse images in $\overline{\Sp}(W)$ respectively. By \cite[p. 39, Lmm.]{MVW}, $\overline{G_1}$ commutes with $\overline{G_2}$ in $\overline{\Sp}(W)$.
\begin{theorem}[Howe, Waldspurger]\label{thetacorrespondencesforreductivegroups}
Suppose that the residue characteristic of $F$ is not $2$. Then  the restriction of the Weil representation $\rho_{\psi}$ to $\overline{G_1} \times \overline{G_2}$ is a   theta representation of finite length. As usual, the corresponding bijection between $\mathcal{R}_{\overline{G_1}}^0(\rho_{\psi})$ and $\mathcal{R}_{\overline{G_2}}^0(\rho_{\psi})$ is called the local theta (or \emph{Howe}) correspondence.
\end{theorem}
 In the whole context, we assume that  the residue characteristic of $F$ is not $2$(\emph{cf}. Section \ref{Notationandconventionsis}), and the above result is sufficient to us.\footnote{We  mainly limit ourself  to those cases, because the similar results in \cite{Wang1} are not established.} However it is also worth to present  some recent progress on the classical theta correspondences by following \cite{GanT2}, \cite{GanT1}. Here we only cite some interesting results in the personal way.
\begin{remark}
\begin{itemize}
\item[(1)] By Self-reducibility property\footnote{ This proper concept comes from Gurevich and Hadani's paper \cite{GuH1}.}of $\omega_{\psi}$, to prove  the local theta correspondence, it reduces to the above discussed dual pairs of types I, II.
\item[(2)]  The classification of reductive dual pairs as described above also fits in the case where  $F$ is  a local field of  characteristic not $2$.
\item[(3)] For  $F$ being a local field of  characteristic  $2$,  the situation seems not  the same as  above. However one can consult with  L.  Blasco (\emph{cf.}\cite{Bar1}) on  the classification of reductive dual pairs , or turn to Gurevich-Hadani's paper\cite{GeL},   Genestier-Lysenko's \cite{GuH}  for the geometric approach  in this case.
\end{itemize}
\end{remark}
\begin{remark}
\begin{itemize}
\item[(1)] The local theta correspondences for the reductive dual pairs of type II have been established by Minguez(cf.\cite{Mi}) in all residue characteristic. Of course, his paper contains much  more results about this type.
\item[(2)]  For $F$ being a local field of  characteristic not $2$, the local theta correspondences have completely settled  by W. T. Gan with his cooperators B. Sun in \cite{GanT2},  S. Takeda in \cite{GanT1}.
\item[(3)] For much detailed structure  results on local Howe correspondences, one can read a series of  papers: Goran Mui\'c's \cite{Mu1}---\cite{Mu3} and  Mui\'c-Savin's paper \cite{Mu4}.
\end{itemize}
\end{remark}

\subsection{The intermediate group }\label{TheintermediategroupI}
 In  this subsection, we will define a canonical intermediate subgroup of $\Sp(W)$ associated to a reductive dual pair, and explain the  splitting of its  metaplectic form with an obvious  exception.  These results will be crucial  in the following sections in order to study Howe correspondences for the similitude groups. We follow the notations of Section \ref{Notationandconventionsis}. We now let $V$ be a right vector space over $D$. Recall that there is an exact sequence
$$1 \longrightarrow \U(V) \longrightarrow \GU(V) \stackrel{\lambda}{\longrightarrow} \Lambda_{\GU(V)} \longrightarrow 1, $$
where $\lambda$ is the similitude character and $\Lambda_{\GU(V)}\subseteq F^{\times}$.

\begin{lemma}\label{wittlambda}
Suppose that $V= V_H \oplus V^0$ is  a Witt decomposition with $V_H \simeq mH$ and  $V^0$ being anisotropic, where $H$ is an $\varepsilon$-hermitian hyperbolic plane over $D$. Then $\Lambda_{\GU(V)}= \Lambda_{\GU(V^0)}$.
\end{lemma}
\begin{proof}
Without loss of generality,  suppose that $V$ is a right $D$-vector space. For  $g\in \GU(V)$, the action of  $g$ on  $V$  will yield another Witt decomposition  $V=g \cdot  \big(V_H\big)\oplus g \cdot \big( V^0\big)$. By Witt's Theorem,  $g \cdot V^0   = g_0  \cdot V^0$ for some suitable  $g_0 \in \U(V)$. Moreover,  $  g_0^{-1}g\cdot (V^0) =V^0$. So  $ g_0^{-1}g \in \GU(V^0)$, and  $\lambda(g_0^{-1}g)=\lambda(g)$. This shows that $\Lambda_{\GU(V)} \subseteq \Lambda_{\GU(V^{0})}$. On the other hand,  recall  that the  $\varepsilon$-hermitian hyperbolic plane $H$ over $D$ is isometric to $(D \oplus D, \langle, \rangle)$, where $\langle (d_1, d_2), (d_1', d_2')\rangle:=\tau(d_1)d_2' + \varepsilon \tau(d_2)d_1'$; this implies $F^{\times} \supseteq \Lambda_{\GU(H)} \supseteq F^{\times}$. So  for $h_0 \in \GU(V^{0})$ with $\lambda=\lambda(h_0) \in F^{\times}$,  we can find an element   $g_H \in \GU(H)$ satisfying  $\lambda(g_H)=\lambda$. Then   $g:=h_0 \times \underbrace{g_H \times \cdots \times g_H  }_{m}  $,  viewed as an element of $\GU(V)$, satisfies  $\lambda(g)=\lambda(h_0)$.  This completes the proof.
\end{proof}
By this lemma,  we can determine the image of $\lambda$ in $F^{\times}$ by means of  the characteristic of  the anisotropic subspace of $V$. The following result is from \cite[p. 7]{MVW}.
\begin{lemma}\label{vigneraslemma}
  Up to isometry,
\begin{itemize}
\item[-]  an anisotropic quadratic  vector space over $F$ has the following form:
 (i) $F(a)$,  for $ a \in F^{\times} \textrm{ modulo } (F^{\times})^2$, with the canonical form; (ii)  $F_1(a)$, any  quadratic field extension $F_1$ of $F$,  for $ a \in F^{\times} \textrm{ modulo } (F^{\times})^2$ with the form $x \longmapsto a\nnn_{F_1/F}(x)$,  $x\in F_1$; (iii) $\mathbb{H}^0(a)$, with the form $x \longmapsto \tau(x)a x$ for $a \in F^{\times} \textrm{ modulo } (F^{\times})^2$; (iv) $\mathbb{H}$, with the form $x \longmapsto \Nrd(x)$.
\item[-] an anisotropic hermitian  vector space over $E$ has the following form: (i) $E(a)$,  for $ a \in F^{\times} \textrm{ modulo } (F^{\times})^2$, with the form $(x ,y)\longmapsto a\tau(x)y$,  for $x,y\in E$; (ii) $\mathbb{H}$ with the form $(x,y) \longmapsto \tau(x)y$.
\item[-] an anisotropic right hermitian  vector space over $\mathbb{H}$ has the following form: $\mathbb{H}$, with the form $(x,y) \longmapsto \tau(x)y$.
\end{itemize}
\end{lemma}

\begin{proposition}\label{lamedavalues}
Let $V$ be an $\varepsilon$-hermitian vector space over $D$ of dimension $n$.
\begin{itemize}
\item[(1)] If  $D=F$, $\varepsilon=-1$, then $ \U(V)=\Sp(V)$ and $\GU(V)=\GSp(V)$. In this case,  $ \Lambda_{\GU(V)}=  F^{\times}.$
\item[(2)] If $D=F$, $\varepsilon=1$, then $\U(V)=\Oa(V)$ and  $\GU(V)=\GO(V)$. Suppose $V=V^0\oplus mH$ is a Witt  decomposition with   $V^0$ being anisotropic  and $mH$ being  a hyperbolic space. Then
\[ \Lambda_{\GO(V)}= \left\{\begin{array}{lcr}
    F^{\times} &  &\dim{V^0}=0,4, \\
   (F^{\times})^2&  & \dim{V^0}=1, 3, \\
    \nnn_{F_1/F}(F_1^{\times}) &  &\dim{V^0}=2,  V^0=F_1(a).
    \end{array} \right. \]
In case $\dim{V^0}=2 $, $V^0=F_1(a)$ is the space mentioned in Lmm.\ref{vigneraslemma}.
\item[(3)]If $D=E$ is a separable quadratic field extension of  $F$, and  $\varepsilon= \pm 1$,
then \[\Lambda_{\GU(V)}= \left\{\begin{array}{lcr}
  F^{\times} &  & 2| n ,\\
    \nnn_{E/F}(E^{\times}) &  &2\nmid  n .
     \end{array} \right.\]
\item[(4)] If  $D$ is the unique (non-splitting) quaternion algebra $\mathbb{H}$ over $F$ and  $\varepsilon= \pm1$, then $\Lambda_{\GU(V)}= F^{\times}$.
\end{itemize}
\end{proposition}
\begin{proof}
Part (1) is well-known. For (2), when $\dim V^0=0, 1, 2, 4$, the results can be deduced from Lmm.\ref{vigneraslemma}; when $\dim V^0=3$, for $g\in \GU(V)$, $(\det g)^2= \lambda(g)^3$, so $(\lambda(g)^{-1} \det(g))^2= \lambda(g) \in (F^{\times})^2$.  For (3) --- (4), the hermitian cases  follow from  Lmm.\ref{vigneraslemma}. For (3), when $\varepsilon=-1$, according to \cite[p.2]{MVW}, multiplying the skew hermitian $\langle, \rangle$ by an element $\mu \in E^{\times}$  satisfying $-1=\mu^{\tau}/{\mu}$,  gives a hermitian form.   But in this process  the group $\GU(V)$ remains unchanged, so it reduces to the hermitian case. For (4), when $\varepsilon=-1$,  let us fix firstly $a \in F^{\times}$. Without loss of generality,   assume that $V$ is a right $D$-vector space. By Witt's decomposition,  $V\simeq \oplus_{i=1}^n \mathbb{H}(a_i)$  for some $a_i\in \mathbb{H}^0$, where $ \mathbb{H}(a_i)$ is a skew hermitian vector space over $\mathbb{H}$ of dimension 1 defined by $\langle d_1, d_2\rangle=\tau(d_1)a_id_2$. By  \cite[p.  364]{Scha},  we can find suitable  $d^i_a
\in \mathbb{H}$ satisfying  $\tau(d^i_a)
a_id^i_a
= aa_i$ for $1 \leq i \leq n$\footnote{For the proof, see also \cite[Lmm.1]{Tsuk}.}. By definition,  $d^i_a$ lies inside $\GU(\mathbb{H}(a_i))$ and its multiplier  is just $a$.    As before, the element $\delta_a = d^1_a
\times \cdots \times  d^n_a$, viewed as an element of  $\GU(V)$,  satisfies
$\lambda(\delta_a) = a$, so finally $\Lambda_{\GU(V)}=F^{\times}$.
\end{proof}

\begin{corollary}
The order of $\Lambda_{\GU(V)}/{(F^{\times})^2}$ is at most $4$.
\end{corollary}
\subsection{Split Metaplectic subgroups}\label{SplittingMetaplecticsubgroups}
Until the end of this section, we will let $(W, \langle, \rangle)$  be a symplectic space over $F$ of dimension $2n$. Let $\big(W=W_1 \otimes_{D'} W_2, \langle, \rangle = t_{K/F}( \langle, \rangle_1 \otimes \tau(\langle, \rangle_2))\big)$ be a  decomposition of tensor product, such that $(\U(W_1), \U(W_2))$ is  an irreducible reductive dual pair of $\Sp(W)$ (\emph{cf}. Section \ref{Reductivedualpair}).   We shall define a canonical   intermediate subgroup $\Gamma$ of $\Sp(W)$ by
$$\Gamma:= \{ (g_1,g_2)\mid g_1 \in \GU(W_1), g_2 \in \GU(W_2) \textrm{ such that } \lambda_1(g_1) \lambda_2(g_2)=1\},$$
where $\lambda_i$ is the similitude character from $\GU(W_i )$ to $K^{\times}$.  As before, there exists a canonical map:
$$\iota: \Gamma \longrightarrow \Sp(W_1 \otimes W_2, \langle, \rangle_1 \otimes \tau(\langle, \rangle_2)) \hookrightarrow \Sp(W, \langle, \rangle).$$
We will let $\iota(\Gamma)$ be the image of $\Gamma$ in $\Sp(W)$ and  $\overline{\Gamma}$ the inverse image of $\iota(\Gamma)$ in $\overline{\Sp}(W)$.
\begin{theorem}\label{scindagedugroupeR}
The exact sequence
\begin{equation}\label{equationsimple}
1 \longrightarrow \mu_{8} \longrightarrow \overline{\Gamma} \longrightarrow \iota(\Gamma)\longrightarrow 0
\end{equation}
 splits, except  when the reductive dual  pair  is a symplectic-orthogonal type, and the orthogonal vector space over $K$   is of odd dimension.
\end{theorem}
\begin{proof}
Note that  the restriction of any  non-trivial class of order $2$  in  $\Ha^2(\Sp(W), \mu_8)$  to   $ \Ha^2(\Sp(W_1 \otimes W_2, \langle, \rangle_1 \otimes \tau(\langle, \rangle_2)), \mu_8)$ is also  non-trivial  of order $2$. So to prove the above theorem, it is sufficient  to handle  the case $K=F$, which has been done in \cite{Wang1}.
\end{proof}
\begin{remark}\label{excludecase}
In  case  $W=W_1\otimes_{F} W_2$,  for  a symplectic
 space  $W_1$  over $F$ and  an orthogonal space $W_2$ over $F$ of odd dimension, the inverse image of $\Sp(W_1)$ in  $\overline{\Sp}(W)$  is isomorphic with $\overline{\Sp}(W_1)$ so that  the canonical extension $\overline{\Gamma}$ does not split over $\Gamma$.
\end{remark}
\begin{proposition}\label{centralextensionfoot}
In the above case, we let $\overline{\GSp}(W_1)$ be an arbitrary central extension of $\GSp(W_1)$ by $\mu_8$, such that there exists a short exact sequence
$1 \longrightarrow \overline{\Sp}(W_1) \longrightarrow \overline{\GSp}(W_1) \longrightarrow F^{\times} \longrightarrow 1$.\footnote{For the existence,  see  \cite[Theorem 1.1.A]{Bar1}.} Let  $\overline{\Gamma}^{1/2}= \left\{ (\overline{g}, h)\in \overline{\GSp}(W_1) \times \GO(W_2) \mid \overline{\lambda}(\overline{g}) \lambda(h)=1\right\}$  be a subgroup  of $\overline{\GSp}(W_1) \times \GO(W_2)$,  for $\overline{\lambda}$(reps.$\lambda$ )  being  the  similitude character  from $\overline{\GSp}(W_1)$(resp. $\GO(F)$) to $F^{\times}$. Then there exists a homomorphism  $\iota_{1/2} : \overline{\Gamma}^{1/2} \longrightarrow \overline{\Sp}(W)$ such that the following diagram
\begin{equation}\label{clubsuit2}
\begin{array}{ccccc}
 \overline{\Gamma}^{1/2}    &    \stackrel{\iota_{1/2}}{\longrightarrow} & \overline{\Sp}(W)\\
       \downarrow    &              &               \downarrow\\
 \Gamma    &     \stackrel{\iota}{\longrightarrow} & \Sp(W)
  \end{array}
\end{equation}
is commutative.
\end{proposition}
\begin{proof}
Let $\{ e_1, \cdots, e_n; e_1^{\ast}, \cdots, e_n^{\ast} \}$ be a symplectic basis of $W_1$.  Let  $X$ (resp. $X^{\ast}$) be  the Lagrangian subspace  of $W_1$ generated by those $e_i$ (resp. $e_i^{\ast}$).  Let $\{ f_1, \cdots, f_{2m-1}\}$ be an orthogonal basis of $W_2$, $\psi$ a non-trivial character of  $F$.  We will take $c_W$ to be the $2$-cocycle constructed in  Section \ref{RaoscocycleI} associated to the Lagrangian subspace $(X^{\ast}\otimes W_2)$  of $W$ and $\psi$.  By Remark \ref{excludecase}, there exists a homomorphism  from  $\overline{\Sp}(W_1)$ to $\overline{\Sp}(W)$ so that  we can choose a defining  $2$-cocycle $c_{W_1}$ of  $\overline{\Sp}(W_1)$, given by
\[ c_{W_1}(g_1,g_2)=c_W(g_1 \otimes 1, g_2 \otimes 1) \qquad \qquad g_1, g_2 \in \Sp(W_1).\]
Then  $s_1: \overline{\Sp}(W_1) \longrightarrow \overline{\Sp}(W); [g, \epsilon] \longrightarrow [g\otimes 1, \epsilon]$ is a morphism of groups.  By hypothesis, $c_{W_1}$ can extend to be a $2$-cocycle defining $\overline{\GSp}(W_1)$. We then define the map $\iota_{1/2}$ as follows:
\begin{equation}
\iota_{1/2}: \overline{\Gamma}^{1/2} \longrightarrow \overline{\Sp}(W);  \qquad ([g,\epsilon], h) \longmapsto [ g\otimes h, \epsilon]
\end{equation}
This  map satisfies the commutative diagram (\ref{clubsuit2}). Then it reduces to show that $\iota_{1/2}$  is  a homomorphism of groups.

Firstly  $\Oa(W_2)$ belongs to the parabolic subgroup $P(X^{\ast}\otimes W_2)$ of $\Sp(W)$, so $ s_2: \Oa(W_2) \longrightarrow \overline{\Sp}(W), h \longmapsto (1 \otimes h, 1)$ is a morphism of groups.
Moreover,  $s_1([g,\epsilon]) s_2(h)=\iota_{1/2}(\overline{g}, h)$, for $\overline{g}=[g, \epsilon] \in \overline{\Sp}(W_1)$, $h\in \Oa(W_2)$.
Since $s_1(\overline{g})$ commutes with $s_2(h)$ by \cite[p. 44, Lemme]{MVW},  $\iota_{1/2}|_{\overline{\Sp}(W_1) \times \Oa(W_2)}$ is a homomorphism of groups.
Consequently,
$\iota_{1/2} \big( [\overline{g_1} \overline{g_2}, h_1h_2]\big)= \iota_{1/2} \big( [\overline{g_1}, h_1]\big) \cdot \iota_{1/2}\big([\overline{g_2}, h_2]\big),$ for $\overline{g_1}=[g_1, \epsilon_1], \overline{g_2}=[g_2, \epsilon_2] \in \overline{\Sp}(W_1)$ and $h_1, h_2 \in \Oa(W_2)$.
Therefore $c_{W_1}(g_1, g_2)=c_W(g_1 \otimes h_1, g_2 \otimes h_2)$.

Next,  let $\overline{\Gamma}_0^{1/2}$ be a  subgroup of $ \overline{\Gamma}^{1/2}$ consists of  $ [\overline{g_t}, h]$ with $\overline{g_t}=(\begin{pmatrix}
1& 0\\
0&t
\end{pmatrix}, \epsilon)$, $h\in \GO(W_2)$,  for  $t \in K^{\times}, \epsilon \in \mu_8$, and $\lambda(h)=t $.
For $[\overline{g_{t_i}}, h_i]= [ (g_{t_i}, \epsilon_i), h_i]\in \overline{\Gamma}^{1/2}_0$,  $i=1, 2$,
$\iota_{1/2}\big( [\overline{g_{t_i}}, h_i]\big)= [g_{t_i}\otimes h_i, \epsilon_i]= \bigg( \begin{pmatrix}
h_i& 0\\
0&t_ih_i
\end{pmatrix}, \epsilon_i\bigg)$, and
\[ \iota_{1/2}\big( (\overline{g_{t_1}}, h_1)\big) \iota_{1/2}\big( (\overline{g_{t_2}}, h_2) \big)= (g_{t_1t_2} \otimes h_1h_2, c_{W} \bigg( \begin{pmatrix}
h_1& 0\\
0&t_1h_1
\end{pmatrix}, \begin{pmatrix}
h_2& 0\\
0&t_2h_2
\end{pmatrix} \bigg)\epsilon_1\epsilon_2)=(g_{t_1t_2} \otimes h_1h_2, \epsilon_1\epsilon_2).\]
Because of $c_{W_1}(g_{t_1}, g_{t_2})=c_{W}(g_{t_1}\otimes 1, g_{t_2}\otimes 1)=1$, we obtain
$[ \overline{g_{t_1}}, h_1] [\overline{g_{t_2}}, h_2] = [\overline{g_{t_1}} \overline{g_{t_2}}, h_1h_2] =[ (g_{t_1t_2} , \epsilon_1\epsilon_2), h_1h_2],$
and
$\iota_{1/2} \big( [( g_{t_1t_2}, \epsilon_1 \epsilon_2), h_1h_2]\big) = [ g_{t_1t_2} \otimes h_1 h_2, \epsilon_1 \epsilon_2].$
Hence   finally
$\iota_{1/2}\big( [\overline{g_{t_1}}, h_1] [\overline{g_{t_2}}, h_2] \big)= \iota_{1/2}\big( [\overline{g_{t_1}}, h_1]\big) \iota_{1/2}\big( [\overline{g_{t_2}}, h_2]\big)$. Now if  $(\overline{g}, h)= [(g, \epsilon), h] \in \overline{\Gamma}^{1/2}$ decomposed as  $[\overline{g}, h]= [\overline{g_0}, h_0] \cdot [\overline{g_t}, h_t]$, for $[\overline{g_0}, h_0]=[(g_0,\epsilon), h_0] \in \overline{\Sp}(W_1) \times \Oa(W_2)$, $[\overline{g_t}, h_t]=[(g_t, 1), h_t ] \in \overline{\Gamma}_0^{1/2}$, then
$\iota_{1/2}( [\overline{g}, h])= [g\otimes h, \epsilon]= [g_0 \otimes h_0, \epsilon] [g_t\otimes h_t, 1]= \iota_{1/2}( [\overline{g_0}, h_0]) \iota_{1/2}([\overline{g_t}, h_t])$.

Finally, in the general case,  for  $[\overline{g}^{(i)}, h^{(i)}]=\big[(g^{(i)}, \epsilon^{(i)}), h^{(i)}\big] \in \overline{\Gamma}^{1/2}$ as $i=1, 2$, if we write
$[\overline{g}^{(i)}, h^{(i)}]= [\overline{g_0}^{(i)}, h_0^{(i)}][\overline{g_t}^{(i)}, h_t^{(i)}]$ with $[\overline{g_0}^{(i)}, h_0^{(i)}]= [ (g_0^{(i)}, \epsilon^{(i)}), h_0^{(i)}]\in \overline{\Sp}(W_1) \times \Oa(W_2)$ and  $ [\overline{g_t}^{(i)}, h_t^{(i)}]=[ (g_t^{(i)}, 1), h_t^{(i)} ] \in \overline{\Gamma}_0^{1/2}$, then
$$[\overline{g}^{(1)}, h^{(1)}][\overline{g}^{(2)}, h^{(2)}]= [(g_0^{(1)}, \epsilon^{(1)}), h_0^{(1)}] [ (g_t^{(1)}, 1), h_t^{(1)}] [(g_0^{(2)}, \epsilon^{(2)}), h_0^{(2)}] [ (g_t^{(2)}, 1), h_t^{(2)}]$$
$$= [(g_0^{(1)}, \epsilon^{(1)}), h_0^{(1)}]\cdot [ \big( g_t^{(1)} g_0^{(2)} (g_t^{(1)})^{-1}, \epsilon^{(2)}\big), h_t^{(1)}h_0^{(2)} (h_t^{(1)})^{-1}] \cdot [\big(g_t^{(1)}, 1\big), h_t^{(1)}] \cdot [ \big( g_t^{(2)},1\big), h_t^{(2)}]$$
$$= [ \Big( g_0^{(1)} g_t^{(1)} g_0^{(2)} (g_t^{(1)})^{-1}, c_{W_1} \big( g_0^{(1)}, g_t^{(1)} g_0^{(2)} (g_t^{(1)})^{-1}\big) \epsilon_1^{(1)} \epsilon^{(2)}\Big), h_0^{(1)} h_t^{(1)} h_0^{(2)} (h_t^{(1)})^{-1}] \cdot [\big( g_t^{(1)}g_t^{(2)}, 1\big), h_t^{(1)} h_t^{(2)}].$$
By the above discussion,
$$\iota_{1/2}\Big( [\overline{g}^{(1)}, h^{(1)}] [ \overline{g}^{(2)}, h^{(2)}]\Big)$$
 $$= \Big[ g_0^{(1)} g_t^{(1)} g_0^{(2)} (g_t^{(1)})^{-1} \otimes h_0^{(1)} h_t^{(1)} h_0^{(2)} (h_t^{(1)})^{-1}, \epsilon^{(1)} \epsilon^{(2)} c_{W_1} (g_0^{(1)}, g_t^{(1)} g_0^{(2)} (g_t^{(1)})^{-1})\Big]  \Big[ g_t^{(1)} g_t^{(2)} \otimes h_t^{(1)} h_t^{(2)}, 1\Big]$$
 $$= \Big[ g_0^{(1)} \otimes h_0^{(1)}, \epsilon^{(1)}\Big] \Big[ g_t^{(1)} g_0^{(2)} (g_t^{(1)})^{-1} \otimes h_t^{(1)} h_0^{(2)} (h_t^{(1)})^{-1}, \epsilon^{(2)}\Big] \Big[ g_t^{(1)} g_t^{(2)} \otimes h_t^{(1)} h_t^{(2)}, 1\Big];$$
by use of $ c_{W_1}\big(g_0^{(1)}, g_t^{(1)} g_0^{(2)} (g_t^{(1)})^{-1}\big)
=c_{W} (g_0^{(1)} \otimes h_0^{(1)}, g_t^{(1)} g_0^{(2)} (g_t^{(1)})^{-1} \otimes h_t^{(1)} h_0^{(2)} (h_t^{(1)})^{-1})$, the last term  in turn equals
$ \big[ g_0^{(1)} \otimes h_0^{(1)}, \epsilon^{(1)}\big] \big[g_t^{(1)} \otimes h_t^{(1)}, 1\big] \big[g_0^{(2)} \otimes h_0^{(2)}, \epsilon^{(2)}\big] \big[ g_t^{(2)} \otimes h_2^{(2)}, 1 \big]= \iota_{1/2}\big([\overline{g}^{(1)}, h^{(1)}]\big)\iota_{1/2}\big([ \overline{g}^{(2)}, h^{(2)}]\big)$.
This finishes the proof!
\end{proof}
\subsection{Irreducible admissible representations of $\GU(V)$}
In order to obtain  the theta  correspondences for the similitude groups,   we will use the main theorems in Sections \ref{stronglygraphreI}, \ref{stronglygraphreII}.  As required there, we discuss some additional conditions  in this subsection. Throughout this  subsection, we  follow the conventions of Section \ref{Notationandconventionsis}.  In addition, we let $A$ be an abelian group of order $n$. Suppose $2 | n$ and $(p, n)=1$. For the local field $F$, we will write $U_n=\{ u \in F^{\times} \mid   u\equiv 1 \mod \mathfrak{P}^{n} \}$.  Let $U$ be the group of units in $\mathcal{O}_F$, and $\omega$ the prime element  of $F$. Clearly, $U/{U_1} \simeq k_F^{\times}$ is a cyclic group of order $q-1$; by \cite[p. 20]{Mo} , $U\simeq U_1 \times S$ for certain subgroup $S$ of $U$.
\begin{lemma}\label{tidtytyt}
There exists  an isomorphism
$\varphi: \Ha^2(F^{\times}, A) \simeq \Hom(S, A)$.
Moreover, this map can be given by  $s \longmapsto c(\omega, s) c(s,\omega)^{-1}$ for a $2$-cocycle $c$ of $Z^{2}(F^{\times}, A)$.
\end{lemma}
\begin{proof}
This arises from the result of Moore in \cite{Mo}. By Lmm.4.1 there, we get
$\Ha^2(F^{\times}, A) \simeq \Hom(S, A) \oplus \Hom(U_1, A) \oplus \Ha^2(U_1, A)$.
 The last two terms are  $p$-primary groups, and  $A$ has order prime to $p$,  so those terms must vanish. On the other hand, the explicit  map has already been constructed in \cite[ Lmm.4.1]{Mo}.
\end{proof}
\begin{corollary}\label{nullh2}
For the subgroup $(F^{\times})^n$ of $F^{\times}$, the canonical map $\Ha^2(F^{\times}, A) \longrightarrow \Ha^2\big((F^{\times})^n, A\big)$ is null.
\end{corollary}

Now let $(V, \langle, \rangle)$ be a right $\varepsilon$-hermitian vector space over $D$, $\U(V)$ the group of  isometries   of $(V, \langle, \rangle)$ and $\GU(V)$ the  group of
similitudes of  $(V, \langle, \rangle)$. To each class $[c]$ of $\Ha^2( \GU(V), A)$ is associated a central extension
$$ 1 \longrightarrow A  \longrightarrow \widetilde{\GU}^A(V) \longrightarrow \GU(V) \longrightarrow 1$$
of $\GU(V)$ by  the abelian group $A$.  We will denote the inverse image of $\U(V)$ in $\widetilde{\GU}^A(V)$ by $\widetilde{\U}^A(V)$.
\begin{lemma}\label{splittinghomomorphism}
There is an isomorphism:
$$(p_1, p_2, p_3): \Ha^2\big( F^{\times} \times \U(V), A\big)  \longrightarrow  \Ha^2( \U(V), A) \oplus \Hom\big(\U(V), \Hom(F^{\times}, A)\big) \oplus \Ha^2(F^{\times}, A),$$
where $p_1, p_3$ are the restriction homomorphisms; if $c(-,-)$ is a $2$-cocycle  of one class in $\Ha^2\big( F^{\times}\times \U(V), A\big)$, then $p_2([c])$ is given by $u \longrightarrow (x \longmapsto c(x, u) c(u, x)^{-1})$, for $u\in \U(V)$, $x\in F^{\times}$.
\end{lemma}
\begin{proof}
See \cite[Lmm.4.1]{Mo}.
\end{proof}
This lemma can derive the following results:
\begin{lemma}\label{commutativegroups}
\begin{itemize}
\item[(1)] The exact sequence
$1 \longrightarrow A \longrightarrow \widetilde{\GU}^A(V) \longrightarrow \GU(V) \longrightarrow 1$
splits at $(F^{\times})^n$. Here, we identify  $(F^{\times})^n$ as a subgroup of $\GU(V)$ via scalar multiplicities.
\item[(2)] The two subgroups $(F^{\times})^n$ and $\widetilde{\U}^A(V)$ of $\widetilde{\GU}^A(V)$ commute.
\end{itemize}
\end{lemma}
\begin{proof}
1) The homomorphism $\Ha^2(\GU(V), A) \longrightarrow \Ha^2\big((F^{\times})^n, A\big)$ factors through the null map $\Ha^2(F^{\times },A) \longrightarrow \Ha^2\big((F^{\times})^n, A\big)$ (Coro.\ref{nullh2}), so the result follows.\\
2) Let us consider the homomorphism  $(F^{\times })^n \times \U(V) \longrightarrow \GU(V)$, which yields a homomorphism  $\varphi: \Ha^2( \GU(V), A) \longrightarrow \Ha^2((F^{\times})^n \times \U(V), A)$. Note that for each $2$-cocycle $c \in Z^2( \GU(V), A) $, $\varphi([c])$ is just the class of  the restriction of $c(-,-)$ to $(F^{\times})^n \times \U(V)$. Similarly as above, $\varphi$ has to factor through $\Ha^2\Big(F^{\times} \times \U(V), A\Big) \longrightarrow \Ha^2\Big( (F^{\times})^n \times \U(V), A\Big)$, so  by Lmm.\ref{splittinghomomorphism}, $p_2 \circ \varphi([c])=0$, which means $c(x,u)=c(u,x)$ for $x \in (F^{\times})^n$, $u \in \U(V)$ by construction.
\end{proof}

\begin{theorem}
If  $\widetilde{\pi}\in \Irr(\widetilde{\GU}^A(V))$, $\widetilde{\sigma}\in \Irr(\widetilde{\U}^A(V))$,  then  $\widetilde{\pi}, \widetilde{\sigma}$ both are  admissible.
\end{theorem}
\begin{proof}
See \cite[p. 17, and pp. 25-32]{BernD}.
\end{proof}
\begin{corollary}\label{restrictionadmissible}
If $\widetilde{\pi}\in \Irr(\widetilde{\GU}^A(V))$, then $\widetilde{\pi}|_{\widetilde{\U}^A(V)}$ is admissible.
\end{corollary}
\begin{proof}
By \cite[p.142, Coro.]{Neu}, we know that $F^{\times}/{(F^{\times})^{ 2n}}$ is a finite abelian group.   Since  $\widetilde{\GU}^A(V)/[{(F^{\times})^n \widetilde{\U}^A(V)}] \hookrightarrow F^{\times}/{(F^{\times})^{ 2n}}$, the result holds.
 \end{proof}
\subsection{Howe correspondences for the similitude groups}\label{thehowecorrespondenceforsimilitude}
Let $(W, \langle, \rangle)$ be a symplectic vector space over $F$ of dimension $2m$, $(\rho_{\psi}, S)$  the Weil representation of $\overline{\Sp}(W)$ relative to $\psi$ (\emph{cf}.  Theorem \ref{Weil}). We fix  an abelian group $A$ of finite order dividing $2$ and prime to $p$.  Let  $W=W_1 \otimes_{D'} W_2$, $\langle, \rangle = t_{K/F}\big( \langle, \rangle_1 \otimes \tau(\langle, \rangle_2)\big)$ henceforth be a decomposition of tensor product
(Section \ref{SplittingMetaplecticsubgroups})  for a finite separable extension $K$ of $F$.  Let $\widetilde{\GU}^A(W_i)$ be an arbitary central extension of $\GU(W_i)$ by $A$, and $\widetilde{\U}^A(W_i)$ the inverse image of $\U(W_i)$ in $\widetilde{\GU}^A(W_i)$. To such  decomposition of tensor product  is associated a canonical intermediate subgroup $\Gamma$ of $\Sp(W)$ (\emph{cf}. Section \ref{SplittingMetaplecticsubgroups});  denote by $\overline{\Gamma}$ its inverse image in $\overline{\Sp}(W)$. We also  define  an intermediate  subgroup of $\widetilde{\GU}^A(W_1)\times \widetilde{\GU}^A(W_2)$ by $\widetilde{\Gamma}^A=\left\{ (\widetilde{g_1}, \widetilde{g_2})\mid \lambda(\widetilde{g_1}) \lambda(\widetilde{g_2})=1\right\}$.
  \begin{lemma}\label{twoexactsequences}
  \begin{itemize}
  \item[(1)]  $1 \longrightarrow \U(W_i) \longrightarrow \GU(W_i) \stackrel{\lambda}{\longrightarrow} \Lambda_{\GU(W_i)} \longrightarrow 1,$  $i=1,2$;

\item[(2)]  $1 \longrightarrow  \widetilde{\U}^A(W_i) \longrightarrow \widetilde{\GU}^A(W_i) \stackrel{\lambda}{\longrightarrow} \Lambda_{\GU(W_i)} \longrightarrow 1$, $i=1,2$;
\item[(3)]  $1 \longrightarrow \U(W_1) \times \U(W_2) \longrightarrow \Gamma \stackrel{\lambda}{\longrightarrow} \Lambda_{\Gamma} \longrightarrow 1$.
\end{itemize}
\end{lemma}
 \begin{proof}
 It suffices to verify the second exact sequence.  By definition, we have the following commutative diagram:
 \[
\begin{array}{ccccccccccc}
1 & \longrightarrow  &  A            & \longrightarrow     & \widetilde{\U}^A(W_i)  & \longrightarrow & \U(W_i)       & \longrightarrow  & 1       \\
  &                  &  \parallel    &                     &  \downarrow      &                 & \downarrow &                  &          \\
1 & \longrightarrow  &      A        & \longrightarrow     & \widetilde{\GU}^A(W_i) & \longrightarrow & \GU(W_i)      & \longrightarrow  & 1
\end{array}
\]
Using the snake's lemma, we obtain
\[
\begin{array}{ccccccccccc}
  &                  &    1          &                     &      1           &                 &   1        &                  &    \\
  &                  &   \downarrow  &                     & \downarrow       &                 &  \downarrow&                  &     \\
1 & \longrightarrow  &  A        & \longrightarrow     & \widetilde{\U}^A(W_i)  & \longrightarrow & \U(W_i)       & \longrightarrow  & 1       \\
  &                  &  \parallel    &                     &  \downarrow      &                 & \downarrow &                  &          \\
1 & \longrightarrow  & A         & \longrightarrow     & \widetilde{\GU}^A(W_i) & \longrightarrow & \GU(W_i)      & \longrightarrow  & 1     \\
  &                  &  \downarrow   &                     & \footnotesize{\lambda}\downarrow      &                 &  \footnotesize{\lambda}\downarrow  &                  &       \\
1 & \longrightarrow  &     1       &\longrightarrow &\Lambda_{\widetilde{\GU}^A(W_i)}    &  =         &\Lambda_{\GU(W_i)}& \longrightarrow & 1  \\
  &                  &   \downarrow  &                     & \downarrow       &                 &  \downarrow&                  &     \\
 &                  &    1          &                     &      1           &                 &   1        &                  &
\end{array}
\]
\end{proof}
 As a consequence of the above proof, we obtain:
\begin{lemma}\label{twolongexactsequences}
\begin{itemize}
\item[(1)] There is a short exact sequence $1 \longrightarrow \widetilde{\U}^A(W_1) \times \widetilde{\U}^A(W_2) \longrightarrow \widetilde{\Gamma}^A \stackrel{\lambda}{\longrightarrow} \Lambda_{\widetilde{\Gamma}^A}=\Lambda_{\Gamma} \longrightarrow 1$.
\item[(2)] There is a canonical morphism $\widetilde{p}: \widetilde{\Gamma}^A \longrightarrow \GU(W_1) \times \GU(W_2)$ with the image $\Gamma$.
\end{itemize}
\end{lemma}
\begin{proof}
The first statement derives from the equality: $\Lambda_{\widetilde{\GU}^A(W_i)}= \Lambda_{\GU(W_i)}$.  The second one is automatically.
\end{proof}
Notice that $\Lambda_{\GU(W_i)}=\Lambda_{\widetilde{\GU}^A(W_i)} \subseteq \Lambda_{\widetilde{\Gamma}^A}=\Lambda_{\Gamma}$. We hence  define a subgroup of $\GU(W_i)$ related to $\Lambda_{\Gamma}$ by $\Ga^{\Gamma}\Ua(W_i)=$ the inverse image of $\Lambda_{\Gamma}$ in $\GU(W_i)$, and  obtain likewise a subgroup  $\Ga^{\widetilde{\Gamma}^A}\widetilde{\U}^A(W_i)$ of $\widetilde{\GU}^A(W_i)$.
 \subsubsection{Case 1} By Theorem \ref{equationsimple}, apart from the exceptional symplectic-orthogonal cases we are in a position to obtain morphisms from $\Gamma$ to $\overline{\Sp}(W)$.  We now fix once for all one such map $\iota$. The restriction  of $\rho_{\psi}$ to $\Gamma$ (through $\iota$) shall give a smooth representation of  $\Gamma$ denoted by $\omega_{\psi}$, whose inflation, a smooth representation of $\widetilde{\Gamma}^A$  via the map $\widetilde{p}$ in Lmm.\ref{twolongexactsequences} (2) will be denoted by $\widetilde{\omega}_{\psi}$.
 \begin{theorem}\label{maintheoem1}
 \begin{itemize}
\item[(1)] $\pi_{\psi}=\cInd_{\Gamma}^{\GU^{\Gamma}(W_1) \times \GU^{\Gamma}(W_2)}\omega_{\psi}$ is a   theta representation of finite length.
\item[(2)] $\widetilde{\pi_{\psi}}=\cInd_{\widetilde{\Gamma}^A}^{\Ga^{\widetilde{\Gamma}^A}\widetilde{\U}^A(W_1) \times \Ga^{\widetilde{\Gamma}^A}\widetilde{\U}^A(W_2)}\widetilde{\omega}_{\psi}$ is a theta representation of finite length.
\end{itemize}
\end{theorem}
\begin{proof}
For (1) we take a subgroup $F^{\times} \U(W_i)$ of $\Ga^{\Gamma}\U(W_i)$, and $F^{\times}\big(\U(W_1) \times \U(W_2)\big)$ of $\Gamma$.  By Theorem \ref{thetacorrespondencesforreductivegroups} and Remark \ref{ouvertmorphisme}, the induction $\omega_{\psi}^{(1)}=\cInd_{F^{\times}( \U(W_1) \times \U(W_2))}^{F^{\times} \U(W_1) \times F^{\times} \U(W_2)} \big( \omega_{\psi}|_{F^{\times} ( \U(W_1) \times \U(W_2))} \big)$ is a  theta representation of finite length. Note that $\Ga^{\Gamma}\Ua(W_i)/{F^{\times} \U(W_i)} \simeq \Gamma/{ [F^{\times}( \U(W_1) \times \U(W_2))]} \simeq \Lambda_{\Gamma}/{(F^{\times})^2}$, and all are finite abelian groups. Without doubt, $\omega_{\psi}^{(1)}$ can extend naturally to get a smooth representation $\omega_{\psi}^{(2)}=\cInd_{\Gamma}^{\Gamma\big( F^{\times} \U(W_1) \times F^{\times} \U(W_2)\big)} \omega_{\psi}$. As is easily checked that the triple  of  groups $\big( \Ga^{\Gamma}\U(W_1) \times \Ga^{\Gamma}\U(W_2), \Gamma ( F^{\times} \U(W_1) \times F^{\times}\U(W_2))), F^{\times} \U(W_1) \times F^{\times} \U(W_2)\big)$ satisfies the conditions of Theorem \ref{graphrepresentation}; hence $\pi_{\psi}=\cInd_{\Gamma\big( F^{\times} (\U(W_1) \times \U(W_2))\big)}^{\Ga^{\Gamma}\U(W_1) \times \Ga^{\Gamma}\U(W_2)} \omega_{\psi}^{(2)}$ is a   theta representation of finite length. For (2) the proof is the same by replacing the above $F^{\times}$ with $(F^{\times})^{n}$ but using Lmm.\ref{commutativegroups}.
\end{proof}

\subsubsection{Case 2}  Let us discuss the exceptional case: $W= W_1 \otimes_K W_2$ with $V_1$ being symplectic and $V_2$ being orthogonal, in which case we assume that  the abelian group $A$ contains $\mu_8$.  We \emph{fix}   a  central extension  $\widetilde{\GSp}^A(W_1)$ of $\GSp(W_1)$ by $A$ containing at least one group $\overline{\GSp}(W_1)$  in Prop. \ref{centralextensionfoot}. As a consequence we can write  $\widetilde{\GSp}^A(W_1)=\overline{\GSp}(W_1)\otimes_{\mu_8} A$.  Now let us  also define a subgroup of $\widetilde{\GSp}^A(W_1) \times \GO(W_2)$ by
$$\widetilde{\Gamma}_{1/2}^A=\left\{(\widetilde{g}, h)\in \widetilde{\GSp}^A(W_1) \times \GO(W_2) \mid \lambda(\widetilde{g}) \lambda(h)=1\right\}.$$
\begin{lemma}
There exists a homomorphism $\iota_{A}: \widetilde{\Gamma}_{1/2}^A \longrightarrow \widetilde{\Sp}^A(W)$ such that the following diagram
\begin{equation}\label{clubsuit}
\begin{array}{ccccc}
 \widetilde{\Gamma}_{1/2}^A     &    \stackrel{\iota_{A}}{\longrightarrow} & \widetilde{\Sp}^A(W)\\
       \downarrow    &              &               \downarrow\\
 \Gamma    &    \longrightarrow & \Sp(W)
  \end{array}
\end{equation}
is commutative.
\end{lemma}
\begin{proof}
See the proof of Prop.\ref{centralextensionfoot}.
\end{proof}
 Recall that $\widetilde{\GO}^A(W_2)$ is a  central extension  of $\GO(W_2)$ by $A$, and $\widetilde{\Gamma}^A=\left\{(\widetilde{g}, \widetilde{h})\in \widetilde{\GSp}^A(W_1) \times \widetilde{\GO}^A(W_2) \mid \lambda(\widetilde{g}) \lambda(\widetilde{h})=1\right\}$. It is clear that there is an exact sequence
 $$\widetilde{\Gamma}^A \longrightarrow   \widetilde{\Gamma}_{1/2}^A \longrightarrow 0.$$
  The restriction  of $\rho_{\psi}$ to $\widetilde{\Gamma}_{1/2}^A$ (through $\iota_A$)  gives a smooth representation of  $\widetilde{\Gamma}_{1/2}^A$ denoted by $\omega_{\psi}$, and its inflation to the group $\widetilde{\Gamma}^A$ will be denoted by $\widetilde{\omega}_{\psi}$. Similarly  as Lmm.\ref{twoexactsequences}, we have:
  \begin{lemma}
  There is a short exact sequence: $1 \longrightarrow \widetilde{\Sp}^A(W_1) \times \Oa(W_2) \longrightarrow  \widetilde{\Gamma}_{1/2}^A \longrightarrow \Lambda_{\widetilde{\Gamma}^A_{1/2}}=\Lambda_{\Gamma} \longrightarrow 1$.
  Let $\Ga^{\widetilde{\Gamma}^A_{1/2}}\widetilde{\Sp}(W_1)$, $\Ga^{\widetilde{\Gamma}^A_{1/2}}\Oa(W_1)$ be the inverse images of $\Lambda_{\widetilde{\Gamma_{1/2}^A}}$ in $\widetilde{\GSp}^A(W_1)$, $\GO(W_2)$ respectively, and $\Ga^{\widetilde{\Gamma}^A}\widetilde{\Sp}(W_1)$, $\Ga^{\widetilde{\Gamma}^A}\widetilde{\Oa}(W_2)$ the analogous subgroups of   $\widetilde{\GSp}^A(W_1)$, $\widetilde{\GO}^A(W_2)$  respectively.
  \end{lemma}
 \begin{theorem}\label{maintheoem2}
 \begin{itemize}
\item[(1)] $\pi_{\psi}^{1/2}=\cInd_{\widetilde{\Gamma}_{1/2}^A}^{\Ga^{\widetilde{\Gamma}^A_{1/2}}\widetilde{\Sp}(W_1) \times \Ga^{\widetilde{\Gamma}^A_{1/2}}\Oa(W_1)} \omega_{\psi}$ is a  theta representation of finite length.
\item[(2)] $\widetilde{\pi_{\psi}}^{1/2}=\cInd_{\widetilde{\Gamma}^A}^{\Ga^{\widetilde{\Gamma}^A}\widetilde{\Sp}(W_1)\times \Ga^{\widetilde{\Gamma}^A}\widetilde{\Oa}(W_2)}\widetilde{\omega}_{\psi}$ is a   theta representation of finite length.
\end{itemize}
\end{theorem}
\begin{proof}
The proof is similar as  that of the above Theorem \ref{maintheoem1}.
\end{proof}
\subsubsection{Examples}
By aid of the explicit analysis on the case studies in Prop. \ref{lamedavalues}, we can     provide     the representations $\pi_{\psi}$ in Theorem \ref{maintheoem1}, and  $\pi_{\psi}^{1/2}$ in Theorem \ref{maintheoem2} on  different  cases as follows:  Recall the notations in Section \ref{Reductivedualpair}. Assume $W_i=W_i^0\oplus m_i H_i$ with $W_i^0$ being an anisotropic subspace and $H_i$ the hyperbolic plane.
\paragraph{Case (1)} Assume $D=F$, $\epsilon_1=-1$, $\epsilon_2=1$, $\Ua(W_1)=\Sp(W_1)$, $\Ua(W_2)=\Oa(W_2)$, and $\GU(W_1)=\GSp(W_1)$, $\GU(W_2)=\GO(W_2)$.

(i) $\dim_{F} W_2^0=0,4$, $\Gamma=\left\{(g,h) \in \GSp(W_1) \times \GO(W_2)\mid \lambda(g) \lambda(h)=1\right\}$, $\Lambda_{\Gamma}=F^{\times}$, $\Ga^{\Gamma}\Sp(W_1)=\GSp(W_1)$, $\Ga^{\Gamma}\Oa(W_2)=\GO(W_2)$. Then $\pi_{\psi}=\cInd_{\Gamma}^{\GSp(W_1) \times \GO(W_2)} \omega_{\psi}$.

(ii) $\dim_{F} W_2^0=1, 3$, $\widetilde{\Gamma}^{A}_{ 1/2}=\left\{(\widetilde{g},h) \in \widetilde{\GSp}^A(W_1)\times \GO(W_2)\mid \widetilde{\lambda}(\widetilde{g})\lambda(h)=1\right\}$, $\Lambda_{\widetilde{\Gamma}^{A}_{ 1/2}}=F^{\times 2}$, $\widetilde{\GSp}^A_+(W_1):= \Ga^{\widetilde{\Gamma}^{A}_{1/2}}\widetilde{\Sp}^A(W_1)=\left\{\widetilde{g} \in \widetilde{\GSp}^A(W_1) \mid \widetilde{\lambda}(\widetilde{g}) \in F^{\times 2}\right\}$, $\Ga^{\widetilde{\Gamma}^{A}_{1/2}}\Oa(W_2)=\GO(W_2)$. Then $\pi_{\psi}^{1/2}=\cInd_{\widetilde{\Gamma}^{A}_{1/2}}^{\widetilde{\GSp}^A_+(W_1)\times \GO(W_2)} \omega_{\psi}$.

(iii) $\dim W_2^0=2$, $W_2^0=E(f)$, where $E/F$ is a quadratic field extension, $f=1$ or $f\in F \backslash \nnn_{E/F}(E^{\times})$. Let $\Gamma=\left\{ (g, h) \in \GSp(W_1) \times \GO(W_2) \mid \lambda(g) \lambda(h)=1\right\}$, $\Lambda_{\Gamma}=\nnn_{E/F}(E^{\times})$, $\GSp_+(W_1):=\Ga^{\Gamma}\Sp(W_1)=\left\{ g\in \GSp(W_1) \mid \lambda(g)\in \nnn_{E/F}(E^{\times})\right\}$, $\Ga^{\Gamma}\Oa(W_2)=\GO(W_2)$. Then $\pi_{\psi}=\cInd_{\Gamma}^{\GSp_+(W_1) \times \GO(W_2)} \omega_{\psi}$.\\

\paragraph{Cas(2)}
Assume $D=E$  is a quadratic field extension over $F$,
$\Gamma=\left\{ (g, h) \in \GU(W_1) \times \GU(W_2) \mid \lambda(g) \lambda(h)=1\right\}$.

(i)  $\dim_{E} W_1$ ,  $\dim_{E} W_2$ both are even. Then  $\Lambda_{\Gamma}=F^{\times}$, $\Ga^{\Gamma}\Ua(W_i)=\GU(W_i)$, $\pi_{\psi}=\cInd_{\Gamma}^{\GU(W_1) \times \GU(W_2)}\omega_{\psi}$.

(ii) $\dim_{E} W_1$, $\dim_{E}W_2$ both are odd. Then $\Lambda_{\Gamma}=\nnn_{E/F}(E^{\times})$, $\Ga^{\Gamma}\Ua(W_1)=\GU(W_i)$, $\pi_{\psi}=\cInd_{\Gamma}^{\GU(W_1) \times \GU(W_2)} \omega_{\psi}$.

(iii) $\&$ (iv) By symmetry, we  assume  $\dim_{E} W_1$ is even and   $\dim_{E} W_2$ is odd.  Let  $\Lambda_{\Gamma}=\nnn_{E/F}(E^{\times})$, $\GU_+(W_1):= \Ga^{\Gamma}\Ua(W_2)=\left\{ g\in \GU(W_2)\mid \lambda(g)\in \nnn_{E/F}(E^{\times})\right\}$, $\Ga^{\Gamma}\Ua(W_2)=\GU(W_2)$. Then  $\pi_{\psi}=\cInd_{\Gamma}^{\GU_+(W_1) \times \GU(W_2)} \omega_{\psi}$.

\paragraph{Cas (3)}
Assume $D$ is the unique quaternion algebra over  $F$, $\Ga^{\Gamma}\Ua(W_i)=\GU(W_i)$. Then  $\pi_{\psi}=\cInd_{\Gamma}^{\GU(W_1) \times \GU(W_2)} \omega_{\psi}$.\\

 The work can be done similarly for the other representations $\widetilde{\pi_{\psi}}$(\emph{cf}. Theorem \ref{maintheoem1}),  $\widetilde{\pi_{\psi}}^{1/2}$(\emph{cf}. Theorem \ref{maintheoem2}). Indeed, we can also construct other kinds of theta representations as above.  Let us present two  examples.
\paragraph{Cas (1)'}
Assume $D=F$, $\epsilon_1=-1$, $\epsilon_2=1$, $\Ua(W_1)=\Sp(W_1)$, $\Ua(W_2)=\Oa(W_2)$; $\GU(W_1)=\GSp(W_1)$, $\GU(W_2)=\GO(W_2)$.

(i)' If  $\dim_{F}W_2$ is even,  let $E'/F$ be  an arbitrary  quadratic field extension. Now we define  $\Ga^{E'}\Sp(W_1)=\left\{ g\in \GSp(W_1)\mid  \lambda(g)\in  \nnn_{E'/F}(E^{'\times})\right\}$, $\Ga^{E'}\Oa(W_2)=\left\{ h\in \GO(W_2)\mid  \lambda(h)\in \nnn_{E'/F}(E^{'\times})\right\}$, and  $\Gamma^{E'}=\left\{ (g,h) \in \Ga^{E'}\Sp(W_1) \times \Ga^{E'}\Oa(W_2)\mid  \lambda(g)\lambda(h)=1\right\}$. Then $\pi^{E'}=\cInd_{\Gamma^{E'}}^{\Ga^{E'}\Sp(W_1) \times \Ga^{E'}O(W_2)} \big(\omega_{\psi}|_{\Gamma^{E'}}\big)$ is also a theta representation.

(ii)' If  $\dim_{F}W_2$ is odd,
we define $\widetilde{\GSp}^A_+(W_1)=\left\{ \widetilde{g} \in \widetilde{\GSp}^A(W_1)\mid  \widetilde{\lambda}(\widetilde{g})\in F^{\times 2}\right\}$, $\GO_+(W_2)=\left\{ h\in \GO(W_2)\mid h\in \GO(W_2),  \lambda(h) \in F^{\times 2}\right\}$,  and a subgroup  $\Gamma^A_+=\left\{ (g,h) \in \widetilde{\GSp}^A_+(W_1) \times \GO_+(W_2)\mid \widetilde{\lambda}(\widetilde{g})\lambda(h)=1\right\}$ of  $\widetilde{\Gamma}^{A}_{ 1/2}$. Then  $\pi_+=\cInd_{\Gamma^A_+}^{\widetilde{\GSp}^A_+(W_1) \times \GO_+(W_2)} \big(\omega_{\psi}|_{\Gamma^A_+}\big)$ is also a theta representation.

 \labelwidth=4em
\addtolength\leftskip{25pt}
\setlength\labelsep{0pt}
\addtolength\parskip{\smallskipamount}

\end{document}